\newif\ifllstyle
\newif\ifmbstyle
\newif\ifaap
\newif\ifdraft
\newif\iflongversion
\def\iftextcolor{black} 

\llstyletrue
\longversiontrue

\ifllstyle
\documentclass[a4paper,11pt]{amsart}
\usepackage[utf8]{inputenc}
\usepackage{lslmath}
\usepackage[colorlinks=true, allcolors=blue]{hyperref}
\usepackage{tikz}
\usetikzlibrary{automata, arrows, positioning, calc}
\usetikzlibrary{shapes}
\usepackage[top=30mm, bottom=33mm, left=28mm, right=28mm]{geometry}
\fi

\ifmbstyle
\documentclass[a4paper,11pt]{article}
\usepackage[utf8]{inputenc}
\usepackage{lslmath}
\usepackage[colorlinks=true, allcolors=blue]{hyperref}
\usepackage{tikz}
\usetikzlibrary{automata, arrows, positioning, calc}
\usetikzlibrary{shapes}
\fi

\ifaap
\documentclass[aap, preprint]{imsart}
\usepackage[utf8]{inputenc}
\RequirePackage[OT1]{fontenc}
\RequirePackage{amsthm, amsmath}
\usepackage{lslmath}
\usepackage[colorlinks=true, allcolors=blue]{hyperref}
\usepackage{tikz}
\usetikzlibrary{automata, arrows, positioning, calc}
\usetikzlibrary{shapes}

\arxiv{arXiv:0000.0000}
\startlocaldefs
\numberwithin{equation}{section}
\theoremstyle{plain}

\endlocaldefs
\fi

\usepackage[ruled, vlined]{algorithm2e}
\SetKwInput{KwInput}{Input}
\SetKwInput{KwOutput}{Output}       
\SetKwInput{KwStart}{Initialisation step}       
\SetKwInput{KwMinStep}{Minimisation step}       
\SetKwInput{KwEnd}{Final step}       
\SetKwInput{KwInitializationStep}{Initialise}
\SetKwInput{KwUpdateStep}{Update}
\SetKwInput{KwDynUpdateStep}{Dynamic update}
\SetKwInput{KwProcessStep}{Process}

\newcommand{\cmax}{c_1}
\newcommand{\gw}{\operatorname{GW}}
\renewcommand{\new}{\emph}
\renewcommand{\eqst}{\stackrel{\rm d}{=}}
\newcommand{\dbin}{\operatorname{Bin}}
\newcommand{\edeg}{\pi^{\rm deg}}
\newcommand{\aedeg}{\bar\pi^{\rm deg}}
\newcommand{\edegnu}{\pi^{\rm deg}_\nu}
\newcommand{\aedegnu}{\bar\pi^{\rm deg}_\nu}
\renewcommand{\law}{\mathcal{L}}
\newcommand{\mom}{m}

\newcommand{\ang}[1]{\langle #1 \rangle} 
\newcommand{\xmax}{\supnorm{x}}
\newcommand{\pimin}{\pi_{\rm min}}

\newcommand{\qtwo}{\bar{q}_2} 

\newcommand{\prnuto}{\stackrel{\pr_\nu}{\longrightarrow}}

\newcommand{\wprnuto}{\ \prnuto \ }

\newcommand{\gon}{\operatorname{Gon}}

\newcommand{\CondBin}{\operatorname{CondBin}}

\renewcommand{\var}{\Var}
\renewcommand{\epsilon}{\varepsilon}

\renewcommand{\set}[1]{\operatorname{set}(#1)}

\newcommand{\II}{{\mathbb I}}
\newcommand{\DD}{\mathbb{D}}

\newcommand{\CBin}{\operatorname{CBin}}

\definecolor{keynoteblue}{HTML}{00A2FF}
\def\nodelinecolor{black}

\def\nodetextcolor{white}

\definecolor{pink1}{RGB}{241,238,246}
\definecolor{pink2}{RGB}{215,181,216}
\definecolor{pink3}{RGB}{223,101,176}
\definecolor{pink4}{RGB}{206,18,86}

\newcommand{\PicModifiedExplorationFourFour}{
\begin{tikzpicture}[>=stealth', auto, semithick, scale=1.3]
\tikzstyle{every state} = [fill=blue, draw=\nodelinecolor, text=\nodetextcolor, thick, scale=1, inner sep=0pt, minimum size=5mm]
\tikzstyle{statei} = [fill=blue, draw=\nodelinecolor, circle, text=\nodetextcolor, thick, scale=1, inner sep=0pt, minimum size=5mm]
\tikzstyle{stateo} = [fill=gray, draw=\nodelinecolor, circle, text=\nodetextcolor, thick, scale=1, inner sep=0pt, minimum size=5mm]
\tikzstyle{states} = [fill=red, draw=\nodelinecolor, circle, text=\nodetextcolor, thick, scale=1, inner sep=0pt, minimum size=5mm]
\tikzstyle{myedge1} = [very thick, color=black]
\tikzstyle{myedge2} = [very thick, color=red]
\def\h{0.4}
\def\hh{0.3}
\node[statei] (V1)  at (1,4) {1};
\node[state] (V2) at (2,4)  {2};
\node[state] (V3) at (3,4) {3};
\node[state] (V4) at (4,4) {4};
\node[stateo] (V5) at (1,3) {5};
\node[state] (V6) at (2,3)  {6};
\node[state] (V7) at (3,3) {7};
\node[state] (V8) at (4,3) {8};
\node[stateo] (V9) at (1,2) {9};
\node[states] (V10) at (2,2)  {10};
\node[state] (V11) at (3,2) {11};
\node[state] (V12) at (4,2) {12};
\node[stateo] (V13) at (1,1) {13};
\node[state] (V14) at (2,1)  {14};
\node[state] (V15) at (3,1) {15};
\node[state] (V16) at (4,1) {16};

\draw [rounded corners, draw=blue, thick] (1-\h, 4+\h) rectangle (2+\h, 2-\h); 
\draw [rounded corners, draw=blue, thick] (2-\hh, 4+\hh) rectangle (3+\hh, 4-\hh); 
\draw [rounded corners, draw=blue, thick] (3-\h, 4+\h) rectangle (4+\h, 2-\h); 
\draw [rounded corners, draw=blue, thick] (4-\hh, 2+\hh) rectangle (4+\hh, 1-\hh); 
\draw [rounded corners, draw=blue, thick] (2-\hh, 2+\hh) rectangle (3+\hh, 1-\hh); 
\draw [rounded corners, draw=blue, thick] (1-\hh, 2+\hh) rectangle (1+\hh, 1-\hh); 
\node at (1.5,4.5) {\color{blue} $G_1$};
\node at (2.5,4.5) {\color{blue} $G_2$};
\node at (3.5,4.5) {\color{blue} $G_3$};
\node at (1,0.5) {\color{blue} $G_4$};
\node at (2.5,0.5) {\color{blue} $G_5$};
\node at (4,0.5) {\color{blue} $G_6$};
\path
(V1) edge [myedge1] (V2)
(V2) edge [myedge1] (V6)
(V5) edge [myedge1] (V9)
(V5) edge [myedge1] (V10)
(V9) edge [myedge1] (V10)
(V2) edge [myedge1] (V3)
(V3) edge [myedge1] (V4)
(V3) edge [myedge1] (V7)
(V7) edge [myedge1] (V8)
(V7) edge [myedge1] (V11)
(V8) edge [myedge1] (V12)
(V11) edge [myedge1] (V10)
(V9) edge [myedge1] (V13)
(V11) edge [myedge1] (V14)
(V11) edge [myedge1] (V15)
(V12) edge [myedge1] (V16)
;
\end{tikzpicture}
}

\newcommand{\PicDegreeExponent}{
\begin{tikzpicture} [xscale=4, yscale=3.0]

\def\th{0.5pt} 
\def\w{4.4} 
\def\h{1.1} 
\coordinate (A0) at (2,0);
\coordinate (A1) at (3,0);
\coordinate (A2) at (4,0);
\coordinate (A3) at (\w,0);
\coordinate (B0) at (2,1);
\coordinate (B3) at (\w,1);
\coordinate (C0) at (2,\h);
\coordinate (C3) at (\w,\h);

\fill[line width=0mm, fill=pink4!70] (B0)--(A0)--(A1)--cycle;
\fill[line width=0mm, fill=pink3!70] (B0)--(A1)--(A2)--cycle;
\fill[line width=0mm, fill=pink2!70] (B0)--(A2)--(A3)--(B3)--cycle;
\draw (B0)--(A1);
\draw (B0)--(A2);
\draw (B0)--(B3);

\draw[->, semithick] (2cm-\th, 0) -- (\w,0);
\node[below] at (\w,-\th) {$\alpha$};
\foreach \x in {2,3,4}{
     \draw (\x,-\th) -- (\x,\th)
       node[below=5*\th] {$\x$};}
\draw[->, semithick] (2,0-\th) -- (2,\h);
\node[left=5*\th] at (2,\h) {$\beta$};
\foreach \y in {0,1}{
     \draw (2cm-\th,\y) -- (2cm+\th,\y)
       node[left=5*\th] {$\y$};}

\node[align=center] at (2.35,0.25) {\parbox{15mm}{\scriptsize \centering Power law\\  ($1 < \delta < 2$)}};
\node[align=center] at (3.0,0.3) {\parbox{15mm}{\scriptsize \centering Power law\\  ($2 < \delta < 3$)}};
\node[align=center] at (3.7,0.55) {\parbox{15mm}{\scriptsize \centering Power law\\  ($\delta > 3$)}};
\end{tikzpicture}
}

\newcommand{\PicTransitivityExponent}{
\begin{tikzpicture} [xscale=4, yscale=3.0]

\def\th{0.5pt} 
\def\w{4.4} 
\def\h{1.1} 
\coordinate (A0) at (2,0);
\coordinate (A1) at (3,0);
\coordinate (A2) at (4,0);
\coordinate (A3) at (\w,0);
\coordinate (AA2) at (2.667,0.667);
\coordinate (AA3) at (\w,0.667);
\coordinate (AAA2) at (3, 0.5);
\coordinate (AAA3) at (\w, 0.5);

\coordinate (B0) at (2,1);
\coordinate (B3) at (\w,1);
\coordinate (C0) at (2,\h);
\coordinate (C3) at (\w,\h);

\fill[line width=0mm, fill=gray!60] (B0)--(A0)--(A1)--cycle;
\fill[line width=0mm, fill=gray!30] (B0)--(A1)--(A2)--cycle;
\fill[line width=0mm, fill=pink3!70] (AA2)--(AA3)--(A3)--(A2)--cycle;
\fill[line width=0mm, fill=pink4!70] (AAA2)--(AAA3)--(A3)--(A2)--cycle;
\fill[line width=0mm, fill=pink2!70] (B0)--(AA2)--(AA3)--(B3)--cycle;
\draw (B0)--(A1);
\draw (B0)--(A2);
\draw (B0)--(B3);
\draw (AA2)--(AA3);
\draw (AAA2)--(AAA3);

\draw[->, semithick] (2cm-\th, 0) -- (\w,0);
\node[below] at (\w,-\th) {$\alpha$};
\foreach \x in {2,3,4}{
     \draw (\x,-\th) -- (\x,\th)
       node[below=5*\th] {$\x$};}
\draw[->, semithick] (2,0-\th) -- (2,\h);
\node[left=5*\th] at (2,\h) {$\beta$};
\foreach \y in {0,1}{
     \draw (2cm-\th,\y) -- (2cm+\th,\y)
       node[left=5*\th] {$\y$};}

\node[align=center] at (2.25,0.3) {\parbox{15mm}{\scriptsize \centering $(P)_{21} = \infty$\\ $(P)_{32}=\infty$}};
\node[align=center] at (3.0,0.3) {\parbox{15mm}{\scriptsize \centering $(P)_{32}=\infty$}};
\node[align=center] at (3.95,0.3) {\parbox{35mm}{\scriptsize \centering Power law:  $\beta/(1-\beta) \in (0,1)$}};
\node[align=center] at (3.6,0.58) {\parbox{55mm}{\scriptsize \centering Power law: $\beta/(1-\beta) \in (1,2)$}};
\node[align=center] at (3.4,0.85) {\parbox{55mm}{\scriptsize \centering Power law with exponent 2}};
\end{tikzpicture}
}

\newcommand{\abstracttext}{
A simple but powerful network model with $n$ nodes and $m$ partly overlapping layers is generated as an overlay of independent random graphs $G_1,\dots,G_m$ with variable sizes and densities. The model is parameterised by a joint distribution $P_n$ of layer sizes and densities. When $m$ grows linearly and $P_n \to P$ as $n \to \infty$, the model generates sparse random graphs with a rich statistical structure, admitting a nonvanishing  clustering coefficient together with a limiting degree distribution and clustering spectrum with tunable power-law exponents. Remarkably, the model admits parameter regimes in which bond percolation exhibits two phase transitions: the first related to the emergence of a giant connected component, and the second to the appearance of gigantic single-layer components. 
}

\begin{document}

\title{Clustering and percolation on superpositions of Bernoulli random graphs}

\ifllstyle
\renewcommand{\mnote}[1]{\mbox{}\marginpar{\raggedright\hspace{30pt}{\color{red} \tiny #1}}}
\author{Mindaugas Bloznelis \and Lasse Leskelä}
\date{\today}
\maketitle
\begin{abstract}
\abstracttext
\end{abstract}
\fi

\ifmbstyle
\author{Mindaugas Bloznelis \and Lasse Leskelä\thanks{\aaltoaddress\lslurl}}
\date{\today}
\maketitle
\fi

\ifaap
\begin{frontmatter}
\runtitle{Superpositions of Bernoulli random graphs}
\begin{aug}
\author{\fnms{Mindaugas} \snm{Bloznelis} \ead[label=e1]{mindaugas.bloznelis@mif.vu.lt}}
\and
\author{\fnms{Lasse} \snm{Leskel\"a} \ead[label=e2]{lasse.leskela@aalto.fi}
}
\runauthor{M.\ Bloznelis and L.\ Leskel\"a}
\affiliation{Vilnius University and Aalto University}
\address{Mindaugas Bloznelis\\
Vilnius University\\
Faculty of Mathematics and Informatics\\
Naugarduko 24, LT-03225 Vilnius, Lithuania\\
\printead{e1}
}
\address{Lasse Leskel\"a\\
Aalto University\\
School of Science\\
Department of Mathematics and Systems Analysis\\
Otakaari 1, 02015 Espoo, Finland\\
\printead{e2}
}
\end{aug}

\begin{abstract}
\abstracttext
\end{abstract}




\end{frontmatter}
\fi

\ifdraft 

\clearpage

\tableofcontents

\clearpage

\section*{Key updates during 27 Dec 2019 -- 30 Dec 2019}

\begin{itemize}
\item The paper is now written in the style of Annals of Probability / Annals of Applied Probability.

\item In Theorem~\ref{the:DegreeDistribution}, for the limiting degree distribution, alternatively, $\CPoi(\lambda, g_{10}) = \CPoi(\tilde \lambda, \tilde g_{10})$ where $\tilde \lambda = \mu x_*$ with $x_* = \int x 1(x \ge 2) P(dx,dy)$, and
\[
 \tilde g_{10}(t)
 \weq \frac{\int \Bin(x-1, y, t) \, x 1(x \ge 2) P(dx,dy)}{\int x 1(x \ge 2) P(dx,dy)},
 \quad t=0,1,\dots
\]
This follows by noting that $g_{10} = a \tilde g_{10} + (1-a) \delta_0$ with $a = \frac{x_*}{(P)_{10}}$ implies $\CPoi(\lambda, g_{10}) = \CPoi( a \lambda, \tilde g_{10})$. (Compare with [MB 2019-12-27].) 

\item In Theorem~\ref{the:TransitivityGlobal} [MB 2019-12-27] allows $(P)_{32}=0$ and $(P)_{33}=0$. This is updated here as well. Note that $(P)_{21}=0$ $\implies$ $(P)_{31}=0$ $\implies$ $(P)_{32}=0$ $\implies$ $(P)_{33}=0$
(If $(P)_{21}=0$, then $(X)_2 Y = 0$ a.s., and also $(X)_3 Y = 0$ a.s., so that $(P)_{32} \le (P)_{31}=0$.)

\item In Theorem~\ref{the:TransitivityLocal}
[MB 2019-12-27] allows $(P)_{32}=0$ and $(P)_{33}=0$ but this is not good, because then the $g_{32}, g_{33}$ are ill defined. Note that $(P)_{33} > 0$ implies $(P)_{rs} > 0$ for all integers $0 \le r,s \le 3$.

\item In Theorem~\ref{the:Giant}, alternatively, $\CPoi(\lambda, g_{10}) = \CPoi(\tilde \lambda, \tilde g_{10})$ where $\tilde \lambda = \mu x_*$ with $x_* = \int x 1(x \ge 2) P(dx,dy)$, and
\[
 \tilde g_{10}(t)
 \weq \frac{\int \Bin^+(x-1,y)(t) \, x 1(x \ge 2) P(dx,dy)}{\int x 1(x \ge 2) P(dx,dy)},
 \quad t=0,1,\dots
\]
This follows by noting that $\bar g_{10} = a \tilde g_{10} + (1-a) \delta_0$ with $a = \frac{x_*}{(P)_{10}}$ implies $\CPoi(\lambda, \bar g_{10}) = \CPoi( a \lambda, \tilde g_{10})$. (Compare with [MB 2019-12-27].) 

\item There are now two percolation models discussed: Overlay percolation and layer percolation.

\item There is a new term "gigantic single-layer neighborhoods" referring to the second phase transition threshold.

\item Theorem~\ref{the:Percolation} is new. Its proof is missing some details, but maybe it could be sufficient for 1st submission. For example, it might require a bit of detail to justify the coupling argument.  Also, the proof is now only given for model~A (random layer types).

\item Discussion about power laws is now restricted to cases with $f(t) \sim a t^{-\alpha}$ instead of the more general $f(t) \sim \ell(t) t^{-\alpha}$ where $\ell(t)$ is slowly varying. This is to save space (and to keep notations and technicalities minimal and clean).
\end{itemize}

\clearpage

\begin{table}[h!]
\centering
\begin{tabular}{ll}
\toprule
Symbol & Meaning \\
\midrule
Free symbols & $a,b,c,u,v,w,z$ \\
Free symbols & $A,B,F,G,H,K,L,M,N,S,T,U,V,W$ \\
Free symbols & $\lambda,\mu,\nu, \eta, \zeta, \tau, \sigma, \phi, \psi, \xi$ \\
\\
$m$ & Number of nodes \\
$n$ & Number of layers \\
$x_k, x^{(n)}_k$ & Size of layer $k$ \\
$y_k, y^{(n)}_k$ & Strength of layer $k$ \\
\\
$i,j$ & Generic node indices \\
$k,\ell$ & Generic layer indices \\
$f,g,h$ & Generic probability distributions on $\Z_+$ \\
$g_{sr}, \hat g_{sr}$ & Mixed binomial distributions \\
$r,s,t$ & Generic integers \\
\\
$G^{(n)}$ & Overlay graph \\
$G^{(n)}_k$ & Layer-$k$ graph \\
$D, D_n$ & Degree of a node \\
\\
$\hat P_n, \rnote{P_n}$ & Empirical layer type distribution \\
$P_n, \rnote{\pi_n} $ & Layer type distribution (generative model) \\
$P, \rnote{\pi}$     & Limiting layer type distribution (both models) \\
$(P_n)_{sr}, (\hat P_n)_{sr}$ & Cross-factorial moment \\
$\mu_{sr}, \rnote{p_{sr}}$ & Normalised cross-factorial moment $\mu_{sr} = \frac{n}{(m)_s} (\hat P_n)_{sr}$ \\
$\mu_{sr}$ & Normalised cross-factorial moment $\mu_{sr} = \frac{n}{(m)_s} (P_n)_{sr}$ \\
$p_{sr}(a)$ & Covering probability of layer $a$: $p_{sr}(a) = \frac{(x_a)_s y_a^r}{(m)_s}$ \\
$\law, \law_{XY}, \law_{xy}$ & Law, conditional law \\
$\pr_{XY}, \pr_{xy}$ & Conditional distribution given layer types $(x_k, y_k)$ \\
$\pr$ & Unconditional distribution \\
\\
$\alpha$ & Power law exponent of layer sizes \\
$\beta$ & Power law exponent of layer strength \\
$\delta$ & Power law exponent of degrees \\
$\mu^{-1}$ & Inverse relative number of layers $\lim_{n \to \infty} \frac{m}{n}$ \\
$\mu$ & Relative number of layers $\lim_{n \to \infty} \frac{n}{m}$ \\
$\rho$ & GW survival probability, $\rho = 1 -\eta$ \\
$\tau$ & Transitivity (clustering coefficient) \\
\\
$\chi_M$ & Truncation map, $\chi_M(x,y) = (x \wedge M, y)$ \\
$M$ & Generic truncation parameter \\
$p$ & Layer size distribution \\
$q$ & Layer strength function \\
\bottomrule\\
\end{tabular}
\caption{Table of notations.}
\end{table}

\clearpage

\section*{Related work}
\begin{itemize}
\item \Boguna et al.\ \cite{ColomerDeSimon_Boguna_2014} have observed an interesting double percolation phenomenon over graphs where the local clustering coefficient follows power law with a heavy tail of exponent less than one.
\item \Pralat et al.\ \cite{Iskhakov_Kaminski_Mironov_Pralat_Prokhorenkova_2018} analysed local clustering coefficient in a spatial preferential attachment model.
\item Stegehuis \cite{Stegehuis_2019_thesis} has reported a lot of interesting observations about the local transitivity spectrum, also nice numerical experiments.
\item Kivelä et al.\ review \cite{Kivela_etal_2014} multilayer and multiplex networks, they refer to Gomez et al.~\cite{Gomez_etal_2013} (superposition network) and Battiston et al.~\cite{Battiston_Nicosia_Latora_2014} (overlapping network), and De Domenico et al.~\cite{DeDomenico_etal_2013} (overlay network), and \cite{Szell_Lambiotte_Thurner_2010} as most relevant to us.
\item Battiston et al.~\cite{Battiston_Nicosia_Latora_2014} discuss multiplex networks, where each layer is a graph. The adjacency matrix of the union of the graphs is called the aggregated topological matrix. 
\end{itemize}

\fi 


\section{Introduction}

Applications in natural sciences, social sciences, and technology often deal with large networks of nodes linked by pairwise interactions which involve uncertainty due to noisy observations and missing data.  Such uncertainties have been investigated using statistical models ranging from classical Bernoulli random graphs and uniform random graphs with given degree distributions to stochastic block models and more complex generative models involving various preferential attachment and rewiring mechanisms \cite{Abbe_2018_JMLR,Frieze_Karonski_2016,Janson_Luczak_Rucinski_2000,Newman_2003_Structure,VanDerHofstad_2017}.
While succeeding to obtain a good fit for degree distributions and tractable percolation analysis, most earlier models fail to capture second-order effects related to clustering and transitivity. Random intersection graphs \cite{Ball_Sirl_Trapman_2014,Bloznelis_Godehardt_Jaworski_Kurauskas_Rybarczyk_2015,Britton_Deijfen_Lageras_Lindholm_2008,Karonski_Scheinerman_Singer-Cohen_1999,Spirakis_Nikoletseas_Raptopoulos_2013}, spatial preferential attachment models \cite{Iskhakov_etal_2020,Jacob_Morters_2015,Jacob_Morters_2017}, and hyperbolic random geometric graphs \cite{Bode_Fountoulakis_Muller_2015,Fountoulakis_VanDerHoorn_Muller_Schepers_2020,Kiwi_Mitsche_2019,Krioukov_Papadopoulos_Kitsak_Vahdat_Boguna_2010} have been introduced to conduct percolation analysis on networks with nonvanishing transitivity and clustering properties.

Despite remarkable methodological advances, most sparse network models still appear somewhat rigid in what comes to modeling finer clustering properties, such as the \emph{clustering spectrum} (degree-dependent local clustering coefficient) \cite{AngelesSerrano_Boguna_2006_I,Stegehuis_VanDerHofstad_Janssen_VanLeeuwaarden_2017,Vazquez_Pastor-Satorras_Vespignani_2002}, which may significantly impact the percolation properties of the network \cite{AngelesSerrano_Boguna_2006_II,ColomerDeSimon_Boguna_2014}. 
A decreasing clustering spectrum manifests the fact that \emph{high-degree nodes tend to have sparser local neighbourhoods than low-degree nodes.} Motivated by analysing this phenomenon in a tractable quantitative framework, this article discusses a statistical network model generated as an overlay of mutually independent Bernoulli random graphs $G_1,\dots, G_m$ which can be interpreted as \emph{layers} or \emph{communities}. The layers have a variable size (number of nodes) and strength (link probability), and they may overlap each other.  A key feature of the model is that the layer sizes and layer strengths are assumed to be correlated, which allows to model and analyse a rich class of networks with a tunable frequency of strong small communities and weak large communities.

\subsection{Main contributions}

This article presents a rigorous mathematical analysis of clustering and percolation of the overlay graph model in the natural sparse limiting regime where the number of nodes $n$ tends to infinity, the number of layers $m$ is linear in the number of nodes, and the joint distribution $P_n$ of layer sizes and layer strengths converges to a limiting distribution~$P$.
We derive exact formulas for the limiting degree distribution,  clustering coefficient, clustering spectrum, and the largest component size in terms of cross-factorial moments and functional transforms of $P$.   We also investigate the model under bond and site percolation, and characterise critical parameter values of the associated phase transitions.

The descriptive power of the model is illustrated by a detailed investigation of an instance where the layer size follows a power law, and the layer strength is a deterministic function of the layer size following another power law. This setting leads to a power-law degree distribution and a power-law clustering spectrum with tunable exponents in ranges $(1,\infty)$ and [0,2], respectively. A special case in which layer strengths are inversely proportional to their sizes corresponds to layers of bounded average degree.  In this natural parameter regime we discover a remarkable \emph{double phase transition} phenomenon with two critical values: the first characterising the emergence of a giant component in the overlay graph, and the second characterising the emergence of gigantic components in layers covering a typical node.

Finally, we highlight that the modelling framework in this article covers \emph{both deterministic and random layer types}.  Our approach of characterising the regularity of layer types using averaged empirical distributions allows both cases to be treated in a uniform manner.

\subsection{Related work}

The overlay network model discussed in this article is naturally motivated and implicitly described by classical works in social networks \cite{Breiger_1974,Feld_1981}.  The explanatory power and wide applicability of the model in the context of social, collaboration, and information networks has been demonstrated in \cite{Yang_Leskovec_2012,Yang_Leskovec_2014} by experimental studies of a \emph{community-affiliation graph}, which represents an instance of the present model where the node sets of layers are nonrandom or otherwise known to the observer.  The superposition of Bernoulli random graphs considered here serves as a null model for sparse community-affiliation graphs.

The mathematical analysis in this article builds on earlier works on component evolution and clustering in inhomogeneous random graphs \cite{Bollobas_Janson_Riordan_2007} and random intersection graphs \cite{Bloznelis_2010_Largest,Bloznelis_2013}.
The special model instance with unit layer strengths reduces to the so-called \emph{passive random intersection graph} \cite{Godehardt_Jaworski_2001}, and as a byproduct, the present article also provides the first rigorous analysis of giant components in general passive random intersection graphs, extending \cite{Bradonjic_Hagberg_Hengartner_Percus_2010,Lageras_Lindholm_2008}. When layer strengths are constant but not necessarily one, clustering properties and subgraph densities of the model have been analysed in \cite{Karjalainen_Leskela_2017,Karjalainen_VanLeeuwaarden_Leskela_2018,Petti_Vempala_2018-02}, and the recovery of the layers in \cite{Epasto_Lattanzi_PaesLeme_2017}.
Another related work \cite{Vadon_Komjathy_VanDerHofstad_2019} (also part of \cite{Vadon_2020}) on percolation in overlapping community networks assumes that layers are sampled from an arbitrary distribution on the space of finite connected graphs, and the layers are assigned to nodes via a bipartite configuration model. 
The restriction to connected layers and the use of a configuration model makes the model in \cite{Vadon_Komjathy_VanDerHofstad_2019} and its analysis fundamentally different from the present one, and limits its applicability by ruling out networks composed of weak communities.

Clustering spectra with power-law exponent 1 have been shown for random intersection graph models \cite{Benson_Liu_Yin_2020,Bloznelis_2013} and spatial preferential attachment models \cite{Iskhakov_etal_2020,Krot_OstroumovaProkhorenkova_2015}, and with a tunable power-law exponent in $[0,1]$ for random intersection graphs \cite{Bloznelis_2019,Bloznelis_Petuchovas_2017} and recently also for a hyperbolic random geometric graph model \cite{Fountoulakis_VanDerHoorn_Muller_Schepers_2020}.
%
%
%
Furthermore, \cite{Stegehuis_VanDerHofstad_Janssen_VanLeeuwaarden_2017} discusses an inhomogeneous Bernoulli graph model where the clustering spectrum vanishes, but its normalised version displays evidence of a power-law behaviour with exponent in range (0,2).
%

To the best of our knowledge, the present work is the first of its kind where a nonvanishing clustering spectrum with a tunable power-law exponent in the extended range [0,2] is rigorously derived in terms of a simple statistical network model. This model admits a clear explanation of the values of power-law exponents, and introduces a new analytical framework for studying ordinary and double phase transitions in bond and site percolation on sparse networks of overlapping communities of variable size and strength.



\subsection{Outline}
In the rest of the article, Section~\ref{sec:Model} presents model details and notations, and Section~\ref{sec:MainResults} the main results. Section~\ref{sec:PowerLaws} illustrates the main results in a power-law setting, and confirms the existence of double phase transition. The remaining Sections~\ref{sec:DegreeAnalysis}--\ref{sec:PowerLawAnalysis} are devoted to proofs, with technical details postponed to Appendix~\ref{sec:Supplement}.

\section{Model description}
\label{sec:Model}
\subsection{Multilayer network}
A multilayer network model with $n$ nodes and $m$ layers is defined by a list $((G_1,X_1,Y_1), \dots, (G_m,X_m,Y_m))$ of mutually independent random variables with values in $\cG_n \times \{0,\dots,n\} \times [0,1]$, where $\cG_ n$ is the set of undirected graphs with node set contained in $\{1,\dots,n\}$. We assume that conditionally on $(X_k,Y_k)$, the probability distribution of $V(G_k)$ is uniform on the subsets of $\{1,\dots,n\}$ of size $X_k$, and conditionally on $(V(G_k), X_k, Y_k)$,
each node pair of $V(G_k)$ is linked with probability $Y_k$, independently of other node pairs. Thus, $G_k$ is a Bernoulli random graph on node set $V(G_k)$, with edge set denoted $E(G_k)$. The variables $X_k$, $Y_k$, and $(X_k,Y_k)$ are called the \emph{size}, \emph{strength}, and \emph{type} of layer $k$, respectively. Aggregation of layers produces an overlay random graph $G$ defined by
\begin{equation}
 \label{eq:OverlayGraph}
 V(G) = \{1,\dots,n\}
 \qquad \text{and} \qquad
 E(G) = \cup_{k=1}^m E(G_k).
\end{equation}
This setting includes as special cases: (i) models with deterministic layer types, and (ii) models where the layer types are independent and identically distributed random variables.

\subsection{Large networks}
A large network is analysed by considering a sequence of network models $((G^{(n)}_1,X^{(n)}_1,Y^{(n)}_1), \dots, (G^{(n)}_m, X^{(n)}_m,Y^{(n)}_m))$ indexed by the number of nodes $n=1,2,\dots$ so that the number of layers $m = m_n$ tends to infinity as $n \to \infty$. We shall focus on a sparse parameter regime where there exists a probability measure $P$ on $\{0,1,\dots\} \times [0,1]$ which approximates in sufficiently strong sense the averaged layer type distribution
\begin{equation}
 \label{eq:AveragedLayerTypeDistribution}
 P_n(A)
 \weq \frac{1}{m} \sum_{k=1}^m \pr( (X^{(n)}_k,Y^{(n)}_k) \in A).
\end{equation}
In this fundamental regime, the network features are described by limiting formulas with rich expressive power captured by cross moments and tail characteristics of $P$.

\subsection{Notations}

We denote $\Z_+ = \{0,1,\dots\}$, $(a)_+=\max \{0, a\}$, and $(x)_s = x(x-1) \cdots (x-s+1)$. The indicator function of a condition $A$ is denoted by $1(A)$ or $1_A$, whichever is more convenient. Sets of size $x$ are called \emph{$x$-sets}. Unordered pairs and triples are abbreviated as $ij = \{i,j\}$ and $ijk = \{i,j,k\}$. We write $\sum'_{i,j}$ and $\sum'_{i,j,k}$ to indicate sums over ordered pairs and ordered triples with distinct elements. We write $a_n \ll b_n$ and $a_n = o(b_n)$ when $a_n/b_n \to 0$, $a_n \lesim b_n$ and $a_n = O(b_n)$ when $\limsup \abs{a_n/b_n} < \infty$, and $a_n \sim b_n$ when $a_n/b_n \to 1$.



A graph is a pair $G=(V,E)$ where $E$ is a set of unordered pairs of elements of $V$. The degree and component of node $i$ in graph $G$ are denoted by $\deg_G(i)$ and $C_G(i)$, respectively.  The transitive closure of graph $G$ is defined as the graph $\bar G$ with $V(\bar G) = V(G)$ and $E(\bar G) = \{ij: i \in C_G(j)\}$ consisting of unordered node pairs connected by a path in $G$.

The probability distribution of a random variable $X$ is denoted by $\law(X)$. For probability measures, $\dtv(f,g)$ denotes the  total variation distance, $f \conv g$ the convolution, and $f_n \weakto f$ refers to weak convergence. On countable spaces, the same letter is used for both a probability measure $f(A)$ and its density $f(t)$ with respect to the counting measure.  The Dirac measure at $x$ is denoted by $\delta_x$. The densities of the binomial distribution $\Bin(x,y)$ and the Poisson distribution $\Poi(\lambda)$ are denoted by
\[
 \Bin(x,y)(t) \weq \binom{x}{t} (1-y)^{x-t} y^t,
 \qquad
 \Poi(\lambda)(t) \weq e^{-\lambda} \frac{\lambda^t}{t!},
\]
with the convention that the densities are zero for $t$ outside $\{0,\dots, x\}$ and $\Z_+$, respectively.
The Bernoulli distribution is denoted $\Ber(y)(t) = \Bin(1,y)(t)$. We also denote by
\begin{equation}
 \label{eq:BinPlus}
 \Bin^+(x, y)(t)
 \weq \pr( \deg_{\bar H_{x+1,y}}(1) = t )
\end{equation}
the degree distribution of any particular node in the transitive closure $\bar H_{x+1,y}$ of a Bernoulli random graph $H_{x+1,y}$ on node set $\{1,\dots, x+1\}$, where each node pair is linked with probability $y$, independently of other node pairs.
Alternatively, $\Bin^+(x, y)(t)$ equals the probability that the connected component of any particular node in $H_{x+1,y}$ has size $t+1$.  Both distributions have the same support $\{0,\dots,n\}$, and $\Bin(x,y) \lest \Bin^+(x,y)$ in the strong stochastic order.
%
No simple closed form expression is know for $\Bin^+(x, y)(t)$, but its values can be efficiently computed with the help of Gontcharoff polynomials \cite{Andersson_Britton_2000,Ball_Sirl_Trapman_2014}. The compound Poisson distribution with rate parameter $\lambda$ and increment distribution $g$ is denoted $\CPoi(\lambda,g)$; recall that this is the law of a random variable $\sum_{k=1}^\Lambda X_k$ where $\Lambda, X_1, X_2,\dots$ are mutually independent and such that $\law(\Lambda) = \Poi(\lambda)$ and $\law(X_k) = g$.

For any probability measure $P$ on $\Z_+ \times [0,1]$, any $P$-distributed random variable $(X,Y)$, and integers $r,s\ge 0$, we denote
\begin{align}
 \label{eq:CrossMoment}
 (P)_{rs}
 \weq \E (X)_r Y^s
 \weq \int (x)_r y^s \, P(dx,dy),
\end{align}
%
and when this quantity is finite and nonzero, we define mixed probability distributions $\Bin_{rs}(P)$ and $\Bin^+_{rs}(P)$ on $\Z_+$ with probability mass functions
\begin{align}
 \label{eq:MixedBin}
 \Bin_{rs}(P)(t)
 &\weq \E \left( \, \Bin(X-r, Y)(t) \, \frac{(X)_r Y^s}{(P)_{rs}} \right), \\
 \label{eq:MixedBinPlus}
 \Bin^+_{rs}(P)(t)
 &\weq \E \left( \Bin^+(X-r, Y)(t) \, \frac{(X)_r Y^s}{(P)_{rs}} \right).
\end{align}


\section{Main results}
\label{sec:MainResults}
\subsection{Degree distribution}

The model degree distribution is defined by
\begin{equation}
 \label{eq:ModelDegreeDistribution}
 f^{(n)}(t)
 \weq \frac{1}{n} \sum_{i=1}^n \pr( \deg_{G^{(n)}}(i) = t ),
\end{equation}
and represents the probability distribution of the number of neighbours of a randomly chosen node. Because $G^{(n)}$ is an exchangeable random graph, we see that $f^{(n)} = \law(\deg_{G^{(n)}}(1))$.


%
%

\begin{theorem}
\label{the:DegreeDistribution}
Assume that $\frac{m}{n} \to \mu \in (0,\infty)$ and $P_n \to P$ weakly together with $(P_n)_{10} \to (P)_{10} \in (0,\infty)$ for some probability measure $P$ on $\Z_+ \times [0,1]$. Then the model degree distribution $f^{(n)}$ converges weakly to a compound Poisson distribution
$f = \CPoi(\mu (P)_{10}, \Bin_{10}(P))$. 
\end{theorem}


The limiting degree distribution $f$ in Theorem~\ref{the:DegreeDistribution} can be represented as the law of
$D = \sum_{k=1}^\Lambda D_k$ where $\Lambda$ is Poisson distributed with mean $\mu (P)_{10}$, $D_1, D_2,\dots$ follow a mixed binomial distribution $\Bin_{10}(P)$, and the random variables in the sum are mutually independent. Here $\Lambda$ represents the number of layers covering a particular node, and $D_k$ the number of neighbours in a typical layer covering the node.
The mean equals $\E(D) = \mu (P)_{21} \le \infty$, and the variance equals $\Var(D) = \mu \big( (P)_{21} + (P)_{32} \big)$ for 
$(P_{21}) < \infty$.  Moreover, $\E(D^r) < \infty$ if and only if $(P)_{r+1,r} < \infty$.  The generating function is given by $\E(z^D) = e^{\lambda(\hat g_{10}(z)-1)}$, where $\hat g_{10}(z) = \int (1-y + y z)^{x-1} \frac{x P(dx,dy)}{(P)_{10}}$.  
The structure of $P$ determines whether or not the limiting degree distribution is light-tailed or heavy-tailed.  Section \ref{sec:PowerLaws} illustrates both cases and provides  examples of power laws with a tunable exponent.

\subsection{Clustering}

The  clustering (a.k.a.\ transitivity) coefficient of the model is defined by 
\[
 \tau^{(n)}
 \weq \frac{\sum'_{ijk} \pr(G^{(n)}_{ij}, G^{(n)}_{ik}, G^{(n)}_{jk}) }{\sum'_{ijk} \pr(G^{(n)}_{ij}, G^{(n)}_{ik})},
\]
where $G^{(n)}_{ij}$ represents the event that node pair $ij$ is linked, and the sums are taken over ordered triples of distinct nodes. We may interpret $\tau^{(n)}$ as the conditional probability that node pair $JK$ is linked given that $IJ$ and $IK$ are linked, where $(I,J,K)$ is an ordered triple of distinct nodes selected uniformly at random. 

\begin{theorem}
\label{the:TransitivityGlobal}
Assume that $(P_n)_{rs} \to (P)_{rs} < \infty$ for $rs = 21, 32, 33$, and $(P)_{21}>0$. Then the  model clustering coefficient is approximated by $\tau^{(n)} \to \tau$, where
\[
 \tau
 \weq
 \begin{cases}
  \dfrac{ (P)_{33} }{(P)_{32}} &\quad \text{when $m \ll n$ and $(P)_{32}>0$}, \\
  \dfrac{ (P)_{33} }{(P)_{32} + \mu (P)_{21}^2} &\quad \text{when $\frac{m}{n} \to \mu \in (0,\infty)$}, \\[2.5ex]
  0 &\quad \text{when $n \ll m \ll n^2$}. 
 \end{cases}
\]
\end{theorem}

\begin{remark*}[Constant layer strengths]
When $Y_k=q$ is constant for all $k$, we see that $(P)_{rs} = (p)_r q^s$ where $(p)_r$ equals the $r$-th factorial moment of the limiting layer size distribution. In this case the limiting model clustering equals
$\frac{q (p)_{3} }{(p)_{3} + \mu (p)_{2}^2 }$ and agrees with \cite{Bloznelis_2013,Karjalainen_VanLeeuwaarden_Leskela_2018}.
\end{remark*}

The clustering spectrum of the model is defined by
\[
 \sigma^{(n)}(t)
 \weq \frac{\sum_{ijk} \pr(\deg_{G^{(n)}}(i) = t, \, G^{(n)}_{ij}, G^{(n)}_{ik}, G^{(n)}_{jk} ) }{\sum_{ijk} \pr(\deg_{G^{(n)}}(i) = t, G^{(n)}_{ij}, G^{(n)}_{ik})}, \quad t \ge 2,
\]
and can be interpreted as the conditional probability that node pair $JK$ is linked given that $J$ and $K$ are neighbours of a node $I$ with degree $t$, where $(I,J,K)$ is an ordered triple of nodes selected uniformly at random.
Section~\ref{sec:PowerLaws} illustrates examples where the limiting clustering spectrum below follows a power law.

\begin{theorem}
\label{the:TransitivityLocal}
Assume that $\frac{m}{n} \to \mu \in (0,\infty)$, and $P_n \to P$ weakly together with $(P_n)_{rs} \to (P)_{rs} \in (0,\infty)$ for $rs=10,21,32,33$. Then $\sigma^{(n)} \to \sigma$ pointwise to the limit
\begin{equation}
 \label{eq:TransitivityLocal}
 \sigma(t)
 \weq \frac{(P)_{33} \, (f \conv g_{33})(t-2)}
 {(P)_{32} (f \conv g_{32})(t-2) +\mu (P)_{21}^2 (f \conv g_{21} \conv g_{21})(t-2)},
\end{equation}
where $f = \CPoi(\mu (P)_{10}, \Bin_{10}(P))$ is the limiting degree distribution in Theorem~\ref{the:DegreeDistribution}, and  the distributions $g_{rs} = \Bin_{rs}(P)$ are defined by \eqref{eq:MixedBin}.
\end{theorem}

\subsection{Connected components}

We denote by $N_1(G^{(n)}) \ge N_2(G^{(n)})$ the two largest component sizes in $G^{(n)}$. For a probability distribution $f$ on $\Z_+$, we denote by
\[
 \rho(f)
 \weq 1 - \min\Big\{s \ge 0: \sum_{x \ge 0} s^x f(x) = s \Big\}
\]
the probability of eternal survival of a Galton--Watson branching process with offspring distribution $f$.

\begin{theorem}
\label{the:Giant}
Assume that $\frac{m}{n} \to \mu \in (0,\infty)$ and $P_n \to P$ weakly together with $(P_n)_{10} \to (P)_{10} \in (0,\infty)$.
Then the largest two component sizes in $G^{(n)}$ are approximated by
\[
 \frac{N_1(G^{(n)})}{n}
 \wprto \rho(f^+)
 \quad\text{and}\quad
 \frac{N_2(G^{(n)})}{n}
 \wprto 0,
\]
where $f^+ = \CPoi(\mu (P)_{10}, \Bin^+_{10}(P))$ is a compound Poisson distribution with rate parameter $\mu (P)_{10}$ and increment distribution $\Bin^+_{10}(P)$ defined by \eqref{eq:MixedBinPlus}.
\end{theorem}

\subsection{Site percolation}

We may analyse how a subset of nodes $S_n \subset \{1,\dots,n\}$ is connected by considering a \emph{site-percolated graph} defined as the subgraph
\begin{equation}
 \label{eq:SitePercolation}
 \check G^{(n)} = G^{(n)}[S_n]
\end{equation}
of $G^{(n)}$ induced by $S_n$. The site-percolated graph is an instance of the overlay graph model \eqref{eq:OverlayGraph} with layers $(\check G_1, \check X_1, \check Y_1), \dots, (\check G_m, \check X_m, \check Y_m)$ such that the conditional distribution of $\check X_k = \abs{V(\check G_k)}$ given $X_k = V(G_k)$ is hypergeometric, and $\check Y_k = Y_k$. An approximation of the hypergeometric distribution by a binomial distribution $\Bin(X_k, \theta)$ with $\frac{\abs{S_n}}{n} \approx \theta$ suggests replacing the limiting layer type distribution $P$ by
\[
 \check P(A) \weq \int ( \Bin(x,\theta) \times \delta_y) (A) \, P(dx,dy).
\]
The following result confirms that this modification is well justified, and summarizes the results of Theorems~\ref{the:DegreeDistribution}--\ref{the:Giant} adjusted to site percolation.

\begin{theorem}
\label{the:SitePercolation}
Assume that $\frac{m}{n} \to \mu \in (0,\infty)$, $P_n \to P$ weakly together with $(P_n)_{10} \to (P)_{10} \in (0,\infty)$, and $S_n \subset \{1,\dots,n\}$ satisfies $\frac{\abs{S_n}}{n} \to \theta \in (0,1]$. Then the following approximations are valid for the site-percolated graph $\check G^{(n)} = \check G^{(n)}[S_n]$:
\begin{enumerate}[(i)]
\item \label{ite:SitePercolation1} The degree distribution converges weakly to 
$\check f = \CPoi(\mu (\check P)_{10}, \Bin_{10}(\check P))$.
\item \label{ite:SitePercolation2} The largest two component sizes are approximated by $n^{-1} N_1 \prto \rho(\check f^+)$ and $n^{-1} N_2 \prto 0$ with $\check f^+ = \CPoi(\mu (\check P)_{10}, \Bin^+_{10}(\check P))$.
\newcounter{enumit}
\setcounter{enumit}{\value{enumi}}
\end{enumerate}
If we also assume that $(P_n)_{rs} \to (P)_{rs} \in (0,\infty)$ for $rs=21,32,33$, then
\begin{enumerate}[(i)]
\setcounter{enumi}{\value{enumit}}
\item \label{ite:SitePercolation3} The  clustering coefficient converges to 
$\hat \tau = \tau$ where $\tau$ is the corresponding limit of the nonpercolated graph $G^{(n)}$.
\item \label{ite:SitePercolation4} The clustering spectrum converges pointwise to $\check\sigma$ defined by replacing $f$ and $g_{rs}$ in \eqref{eq:TransitivityLocal} by $\check f$ and $\check g_{rs} = \Bin_{rs}(\check P)$.
\end{enumerate}
\end{theorem}

\subsection{Bond percolation}
\label{sec:BondPercolation}
Bond percolation studies how well the nodes of a graph are connected along a subset of links obtained by random sampling. In a multilayer networks, we may either sample (i) a subset of links of the overlay graph, or (ii) independent subsets of links for each layer separately. To analyse these cases for the overlay graph model $G = G^{(n)}$ in \eqref{eq:OverlayGraph}, we define an \emph{overlay bond-percolated graph} by
\begin{equation}
 \label{eq:BondPercolation}
 \hat G = G \cap H,
\end{equation}
and a \emph{layerwise bond-percolated graph} $\tilde G$ by
\begin{equation}
 \label{eq:LayerPercolation}
 V(\tilde G) = \{1,\dots, n\}
 \qquad\text{and}\qquad
 E(\tilde G) = \cup_{k=1}^m E(G_k \cap H_k),
\end{equation}
where $H,H_1,\dots,H_m$ are mutually independent random graphs on $\{1,\dots,n\}$ in which each node pair is linked with probability $\theta$, independently of other node pairs, and independently of
the layers $(G_k,X_k,Y_k)$.

In an epidemic modeling context, the standard SIR epidemic model is used to model individuals who infect their neighbours with probability $\theta$, independently of each other \cite{Andersson_Britton_2000}.  The links of a graph $G$ represent social contacts, and the bond-percolated component of node $i$ corresponds to the set of eventually infected individuals in a population where node $i$ is initially infectious and the other nodes susceptible. Bond percolation on the overlay graph can be used to develop finer models to model contacts of individuals generated by social communities (households, workplaces, schools) of variable size and strength. Layerwise percolation $\hat G$ then models the case where infections occur independently inside the communities, and the overlay bond-percolation $\tilde G$ models the case where infections occur between individuals regardless of the underlying community structure.

The layerwise bond-percolated graph is an instance of the overlay model \eqref{eq:OverlayGraph} with layer types $(X_k, \theta Y_k)$. This suggests considering a modified limiting layer type distribution
\[
 \hat P(A) \weq \int ( \delta_x \times \delta_{\theta y}) (A) \, P(dx,dy).
\]
We expect the overlay bond-percolated model to behave similarly to the layerwise bond-percolated model in sparse regimes where the layers do not overlap much. The following result confirms this, and summarises the results of Theorems~\ref{the:DegreeDistribution}--\ref{the:Giant} adjusted to bond percolation.

\begin{theorem}
\label{the:BondPercolation}
Assume that $\frac{m}{n} \to \mu \in (0,\infty)$, and $P_n \to P$ weakly together with $(P_n)_{10} \to (P)_{10} \in (0,\infty)$, and $\theta_n \to \theta \in (0,1]$. Then the following approximations are valid for both the overlay bond-percolated graph $\hat G^{(n)}$ and the layerwise bond-percolated graph $\tilde G^{(n)}$:
\begin{enumerate}[(i)]
\item \label{ite:BondPercolation1} The degree distribution converges weakly to 
$\hat f = \CPoi(\mu (\hat P)_{10}, \Bin_{10}(\hat P))$.
\item \label{ite:BondPercolation2} The largest two component sizes are approximated by $n^{-1} N_1 \prto \rho(\hat f^+)$ and $n^{-1} N_2 \prto 0$ with $\hat f^+ = \CPoi(\mu (\hat P)_{10}, \Bin^+_{10}(\hat P))$.
\setcounter{enumit}{\value{enumi}}
\end{enumerate}
If we also assume that $(P_n)_{rs} \to (P)_{rs} \in (0,\infty)$ for $rs=21,32,33$, then:
\begin{enumerate}[(i)]
\setcounter{enumi}{\value{enumit}}
\item \label{ite:BondPercolation3} The  clustering coefficient converges to $\hat \tau = \theta \tau$ where $\tau$ is the corresponding limit of the nonpercolated graph $G^{(n)}$.
\item \label{ite:BondPercolation4} The clustering spectrum converges pointwise to $\hat\sigma$ defined by replacing $P$, $f$, and $g_{rs}$ in \eqref{eq:TransitivityLocal} by $\hat P$, $\hat f$, and $\hat g_{rs} = \Bin_{rs}(\hat P)$.
\end{enumerate}
\end{theorem}

\subsection{Double phase transition}

Theorem~\ref{the:BondPercolation} shows that the largest relative component size in the bond-percolated graph is approximated by the survival probability $\rho(\hat f^+)$ of a Galton--Watson process with compound Poisson offspring distribution $\hat f^+ = \CPoi(\mu (\hat P)_{10}, \Bin^+_{10}(\hat P))$. The mean of the offspring distribution can be written as\footnote{$R_0(\theta)$ can be interpreted as the basic reproduction number ``R naught'' in the epidemiological context.}
\begin{equation}
 \label{eq:R0}
 R_0(\theta)
 \weq \mu \int R(x-1,\theta y) \, x P(dx,dy).
\end{equation}
where $R(x, y) = \sum_{t \ge 0} t \Bin^+(x,y)(t)$ defined using \eqref{eq:BinPlus} represents the expected transitive degree in a homogeneous Bernoulli graph with $x+1$ nodes and link probability $y$. Classical branching process theory tells that $\rho(\hat f^+) > 0$ if and only if $R_0(\theta) > 1$. Hence the largest component in the bond-percolated graph is sublinear for $\theta < \theta_1$, and linear for $\theta > \theta_1$, where the critical threshold is defined by
\[
 \theta_1 \weq \sup\{ \theta\in [0,1]: R_0(\theta) < 1\}.
\]
%
The overlay graph model in studied in this article involves another nontrivial phase transition associated with a critical threshold value
\[
 \theta_2 \weq \sup\{ \theta\in [0,1]: R_0(\theta) < \infty\}.
\]
Section~\ref{sec:PowerLaws} describes an example where $0 < \theta_1 < \theta_2  <1$.
%
%

The first phase transition at $\theta_1$ characterises the emergence of a giant component in a bond-percolated overlay graph. To understand the second phase transition, note that $R_0(\theta)$ is proportional to the expected number of nodes which can be reached by paths within a typical bond-percolated layer covering a particular node.
The second phase transition at $\theta_2$ hence amounts to the emergence of gigantic components inside
bond-percolated layers covering a typical node.

In the epidemic context discussed in Section~\ref{sec:BondPercolation}, we note that the critical quantity $R_0(\theta)$ does \emph{not} refer to the number of individuals directly infected by a reference individual in an otherwise susceptible population, unlike in classical SIR models. Rather, $R_0(\theta)$ also counts the number of individuals indirectly infected by the reference individual via single-layer infection paths.

\section{Power-law models}
\label{sec:PowerLaws}
This section illustrates the rich statistical features of the overlay model by discussing the results of Section~\ref{sec:MainResults} in a setting where the layer strength is a deterministic function of layer size according to $Y_k = q(X_k)$ for some $q: \Z_+ \to [0,1]$, and the limiting layer type distribution factorises according to
\begin{equation}
 \label{eq:LayerTypeFactorized}
 P(dx,dy) \weq p(dx) \delta_{q(x)}(dy)
\end{equation}
where the layer size distribution $p$ is a probability on $\Z_+$.  For concreteness, we assume that the probability mass function $p(x)$ of the layer size distribution and $q(x)$ follow power laws
\begin{equation}
 \label{eq:PowerLaws}
 p(x) = (a+o(1)) x^{-\alpha}
 \quad\text{and}\quad
 q(x) = (b+O(x^{-1/2})) x^{-\beta}
\end{equation} 
as $x \to \infty$, with exponents $\alpha > 2$, $\beta \ge 0$ and constants $a,b >0$. In this case
\[
 (P)_{rs}
 \weq \sum_{x \ge 0} (x)_r q(x)^s p(x)
 \weq \sum_{x \ge 0} \big(ab^s+o(1) \big) \, x^{r-s\beta - \alpha}
\]
shows that $(P)_{rs}$ is finite if and only if $\alpha + s \beta > r+1$.


\subsection{Degree distribution and clustering spectrum}

Theorems~\ref{the:PowerLawDegree} and~\ref{the:PowerLawTransitivityLocal} below establish 
power laws for the limiting degree distribution and clustering spectrum. Figures~\ref{fig:PowerLawDegree} and~\ref{fig:PowerLawTransitivityLocal} illustrate how the associated power-law exponents relate to the corresponding exponents of layer sizes and layer strengths.
%
Remarkably, the power law of the clustering spectrum admits a tunable exponent in $[0,2]$. A similar power law with exponent 1 has earlier been established for a random intersection graph \cite{Bloznelis_2013} and for a spatial preferential attachment random graph \cite{Iskhakov_etal_2020}, and with exponent restricted to $[0,1]$ for inhomogeneous random intersection graphs \cite{Benson_Liu_Yin_2020,Bloznelis_2019,Bloznelis_Petuchovas_2017} and a hyperbolic random geometric graph model \cite{Fountoulakis_VanDerHoorn_Muller_Schepers_2020}.


\begin{theorem}
\label{the:PowerLawDegree}
Assume~\eqref{eq:PowerLaws} for some $\alpha > 2$, $\beta \ge 0$, and $a,b > 0$.
\begin{enumerate}[(i)]
\item If $\beta \in (0,1)$,
then the limiting degree distribution satisfies 
\begin{equation}
 \label{eq:PowerLawDegree}
 f(t)
 \wsim d t^{-\delta}
\end{equation}
for $\delta = 1 + \frac{\alpha-2}{1-\beta}$ and $d =\mu (1-\beta)^{-1} a b^{\delta-1}$.
\item Relation \eqref{eq:PowerLawDegree} holds also for $\beta = 0$, provided that either $b < 1$, or $b=1$ and $q(x)=1$ for all but finitely many $x$.
\item If $\beta \ge 1$, then the limiting degree distribution is light-tailed
with generating function bounded by $\sum_{t \ge 0} z^t f(t) \le e^{\mu(P)_{10}( e^{M(z-1)}-1 )}$ for all $z \ge 0$, where $M = \sup_{x \ge 1} (x-1)q(x)$.
\end{enumerate}
\end{theorem}

\begin{figure}[h]
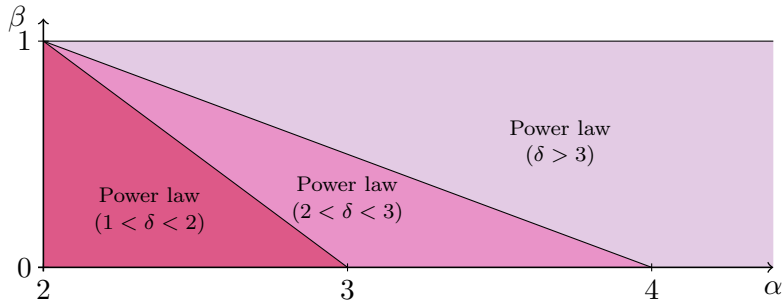

\centering
\PicDegreeExponent
\caption{\label{fig:PowerLawDegree} (Color online.) Power-law exponent of degree distribution as a function of layer size exponent $\alpha$ and layer strength exponent $\beta$.}
\end{figure}

\begin{theorem}
\label{the:PowerLawTransitivityLocal}
Assume~\eqref{eq:PowerLaws} for some $\alpha \in (2,\infty)$ and $\beta \in (0,1)$ such that $\alpha + 2\beta > 4$, and that
$(P_n)_{rs} \to (P)_{rs} \in (0,\infty)$ for $rs=10,21,32,33$. Then the limiting clustering spectrum defined by \eqref{eq:TransitivityLocal} follows a power law according to
\[
 \sigma(t)
 \wsim
 \begin{cases}
  c_1 t^{-\beta/(1-\beta)}, &\quad \beta < 2/3, \\
  c_2 t^{-2}, &\quad \beta = 2/3, \\
  c_3 t^{-2}, &\quad \beta > 2/3,
 \end{cases}
\]
where $c_1 = b^{1/(1-\beta)}$, $c_3 = \mu (P)_{33}$, and $c_2 = c_1 + c_3$.  Furthermore, if 
\eqref{eq:PowerLaws} holds for $\alpha \in (4,\infty)$ and $\beta =0$, and $q(x) = b \in (0,1]$ for all but finitely many $x$, then $\sigma(t) \sim b$.
\end{theorem}

Networks with $\sigma(t) \ll t^{-1}$ are sometimes call weakly clustered, and those with $\sigma(t) \gg t^{-1}$ strongly clustered \cite{AngelesSerrano_Boguna_2006_II}.  According to Theorem~\ref{the:PowerLawTransitivityLocal}, the overlay graph model produces weakly clustered networks for $\beta > \frac12$, and strongly clustered networks for $\beta < \frac12$. Using techniques in \cite{Bloznelis_2019}, Theorem~\ref{the:PowerLawTransitivityLocal} can be generalised to the case where $p(x)$ in \eqref{eq:PowerLaws} has a regularly varying tail, and we believe that it can be extended to more general subexponential distributions as well. We do not pursue this line here to avoid unnecessary technicalities.

\begin{figure}[h]
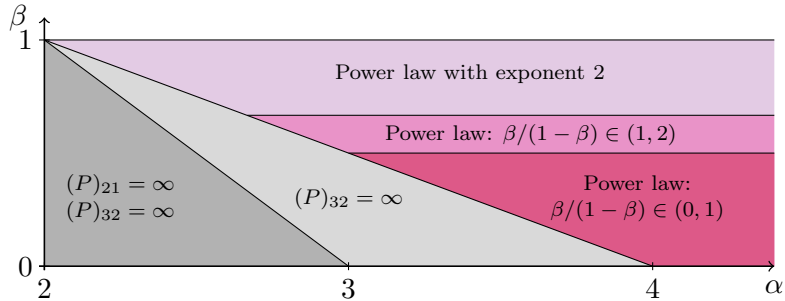

\centering
\PicTransitivityExponent
\caption{\label{fig:PowerLawTransitivityLocal} (Color online.) 
Power-law exponent of clustering spectrum as a function of layer size exponent $\alpha$ and layer strength exponent $\beta$. The assumptions of Theorem~\ref{the:PowerLawTransitivityLocal} do not hold in the grey areas where $(P)_{32} = \infty$.}
\end{figure}

\subsection{Existence of double phase transition}

For the power-law model \eqref{eq:PowerLaws}, the function in \eqref{eq:R0} can be computed as
$R_0(\theta) = \mu \sum_x R(x-1, \theta q(x)) x p(x)$.
By applying a classical giant component result for Bernoulli random graphs \cite[Theorem 5.4]{Janson_Luczak_Rucinski_2000}, one may verify that%
\footnote{The first implication in \eqref{eq:ERPhase} follows by noting that if $xy < 1$, then the proof of \cite[Theorem 5.4]{Janson_Luczak_Rucinski_2000} shows that
$\E \abs{C_{H_{xy}}(1)} = \sum_{t \ge 1} \pr( \abs{C_{H_{xy}}(1)} \ge t )
\le \sum_{t \ge 1} e^{-\frac12 (1-xy)^2 t} \le \int_0^\infty e^{-\frac12 (1-xy)^2 t} \le 2 (1-xy)^{-2}$, 
so that
$R(x-1,y) = \E \abs{C_{H_{xy}}(1)} - 1 \le 2 (1-xy)^{-2}$.}
\begin{equation}
\label{eq:ERPhase}
\begin{aligned}
 \limsup_{x \to \infty} \theta x q(x) \le 1-\epsilon \quad &\implies \quad \limsup_{x \to \infty} R(x-1, \theta q(x)) \le 2 \epsilon^{-2} \\
 \liminf_{x \to \infty} \theta x q(x) \ge 1+\epsilon   \quad &\implies \quad \liminf_{x \to \infty} \, x^{-1} R(x-1, \theta q(x)) > 0, 
\end{aligned}
\end{equation}
If $\alpha > 3$, then the limiting layer size distribution $p$ has a finite second moment 
and $R(x-1,y) \le x-1$ implies that $R_0(1) < \infty$. Hence $\theta_2 = 1$, and the second phase transition cannot occur. On the other hand, when $\alpha \in (2,3]$, the limiting layer size distribution has infinite second moment. In this case \eqref{eq:ERPhase} yields the following conclusions:
\begin{enumerate}
\item $\beta = 1$ with $b > 1$. Then $R_0(\theta) < \infty$ for $\theta < b^{-1}$, and $R_0(\theta) = \infty$ for $\theta > b^{-1}$. Hence $\theta_2 = b^{-1} \in (0,1)$. Assume in addition that the constant $a$ in \eqref{eq:PowerLaws} is large enough so that $\mu \theta (P)_{21} \ge 1$ for $\theta = \frac12 \theta_2$. Then $\hat f^+ \gest \hat f$ implies that $R_0(\theta) = \sum_t t \hat f^+(t) \ge \sum_t t \hat f(t) = \mu \theta (P)_{21} \ge 1$ for $\theta = \frac12 \theta_2$, and the continuity of $R_0(\theta)$ on $[0,\theta_2)$ implies that $\theta_1 \in (0, \frac12 \theta_2)$. There are hence two critical values $0 < \theta_1 < \theta_2 < 1$ in which the model displays two distinct phase transitions.

\item $\beta \in (1,\infty)$, or $\beta =1$ with $b<1$. Then $R_0(\theta) < \infty$ for all $\theta \in [0,1]$, so that $\theta_2 = 1$, and the second-type phase transition cannot occur.

\item $\beta \in [0,1)$. Then one can show that $R_0(\theta) = \infty$ for all $\theta \in (0,1]$, and hence $\theta_1 = \theta_2 = 0$, and there are no phase transitions of either type.
\end{enumerate} 

The above observations confirm the existence of a double phase transition in bond percolation, as postulated in \cite{ColomerDeSimon_Boguna_2014}, for a natural network model admitting tunable power-law exponents for both the degree distribution and the clustering spectrum. Together with Theorems~\ref{the:PowerLawDegree} and~\ref{the:PowerLawTransitivityLocal}, this opens up a flexible framework for studying the significance and interrelations of these power laws to bond and site percolation properties in clustered complex networks. The investigation of how these phase transitions are reflected in the core-periphery organisation of the network \cite{AngelesSerrano_Boguna_2006_II,ColomerDeSimon_Boguna_2014} remains an important topic for future research.

\section{Analysis of degree distributions}
\label{sec:DegreeAnalysis}

%

\subsection{Quantitative approximation for deterministic layer types}

The following quantitative estimate is valid for every scale.

\begin{proposition}
\label{the:DegreeApproximationQuantitative}
If the layer types are nonrandom and $(P_n)_{10} > 0$, then the model degree distribution $f^{(n)}$ defined by \eqref{eq:ModelDegreeDistribution} is approximated by a compound Poisson distribution $\CPoi(\lambda^{(n)}, g^{(n)}_{10})$ with rate parameter $\lambda^{(n)} = \frac{m}{n} (P_n)_{10}$ and increment distribution $g_{10}^{(n)} = \Bin_{10}(P_n)$ defined by \eqref{eq:MixedBin}
according to
\begin{equation}
 \label{eq:DegreeApproximationSimple}
 \dtv \Big( \, f^{(n)}, \, \CPoi(\lambda^{(n)}, g^{(n)}_{10}) \Big)
 \wle \left( 1 + \frac{m}{n} \right)^2 \supnorm{X}^4 n^{-1},
\end{equation}
where $\supnorm{X} = \max_{1\le k \le m} X_k$.
\end{proposition}
\begin{proof}
We approximate the degree $D_i = \deg_G(i)$ of node $i$ by a random integer
$
 L_i = \sum_{k=1}^m \deg_{G_k}(i).
$ 
Observe that $L_i \ne D_i$ if and only if there exists a node $j \ne i$ and some distinct layers $k < \ell$ such that $ij \in E(G_k)$ and $ij \in E(G_\ell)$. Hence by the union bound and the independence of $G_k$ and $G_\ell$,
\begin{align*}
 \pr( L_i \ne D_i )
 &\wle \sum_{j \ne i} \sum_{1 \le k < \ell \le m} \pr(ij \in E(G_k)) \, \pr(ij \in E(G_\ell)).
\end{align*}
Hence, noting that $\pr((ij \in E(G_k)) = \frac{(X_k)_2}{(n)_2} Y_k \le \frac{X_k^2}{n^2} Y_k \le n^{-2} \supnorm{X}^2$, it follows that
\begin{equation}
 \label{eq:DegreeApproximation1}
 \dtv(D_i, L_i)
 \wle (n-1) \binom{m}{2} ( n^{-2} \supnorm{X}^2 )^2
 \wle \frac{m^2}{n^3} \supnorm{X}^4.
\end{equation}

Now denote by $W_{xy} = \{k: (X_k, Y_k) = (x,y)\}$ the set of layers with size $x$ and strength $y$, and let $m_{xy} = \abs{W_{xy}}$. Also denote $S = \{ (x,y): m_{xy} > 0\}$. Then we see that
$
 L_i
 = \sum_{(x,y) \in S} \sum_{k \in W_{xy}} \deg_{G_k}(i).
$
Let us define a random variable
\begin{equation}
 \label{eq:RepresentationL}
 \hat L_i \weq \sum_{(x,y) \in S} \sum_{\ell=1}^{M_{xy}} A_{xy}(\ell),
\end{equation}
where $\law(M_{xy}) = \Bin(m_{xy}, \frac{x}{n})$, $\law(A_{xy}(\ell)) = \Bin(x-1,y)$, and all random variables on the right side are mutually independent. Then for any $(x,y) \in S$,
$
 \sum_{k \in W_{xy}} \deg_{G_k}(i)
 \eqst \sum_{\ell=1}^{M_{xy}} A_{xy}(\ell),
$
because the summands on the left are mutually independent, the number of layers $k \in W_{xy}$ containing node $i$ is $\Bin(m_{xy}, \frac{x}{n})$-distributed, and because $\law( \deg_{G_k}(i) \cond V(G_k) \ni i ) = \Bin(x-1,y)$
for each $k \in W_{xy}$. As a consequence, it follows that $L_i \eqst \hat L_i.$

Now denote $\lambda_{xy} = m_{xy} \frac{x}{n}$ and define a new random variable
\begin{equation}
 \label{eq:RepresentationLPrime}
 L_i'
 \weq \sum_{(x,y) \in S} \underbrace{\sum_{\ell=1}^{M_{xy}'} A_{xy}(\ell)}_{L_{xy}'},
\end{equation}
where $M_{xy}'$ are $\Poi( \lambda_{xy} )$-distributed, mutually independent, and independent of the random variables $A_{xy}(\ell)$. Because $\law(L_{xy}') = \CPoi(\lambda_{xy}, \Bin(x-1,y))$, Lemma~\ref{the:CPoiSum} implies that $\law(L_i') = \CPoi(\lambda^{(n)}, g^{(n)}_{10})$ with rate parameter
$
 \lambda^{(n)}
 = \sum_{(x,y) \in S} \lambda_{xy}
 = \frac{m}{n} (P_n)_{10}
$
and mixed binomial increment distribution
\[
 g^{(n)}_{10}
 \weq \sum_{(x,y) \in S} \dbin(x-1, y) \frac{\lambda_{xy}}{\lambda^{(n)}} 
 \weq \sum_{(x,y) \in S} \dbin(x-1, y) \frac{x P_n(\{(x,y)\})}{(P_n)_{10}}.
\]
As a consequence of Le Cam's inequality \cite{Steele_1994} it follows that $\dtv( M_{xy}, M_{xy}') \le m_{xy} \left(\frac{x}{n} \right)^2 \le n^{-2} \supnorm{X}^2 m_{xy}$, and hence
\begin{align*}
 \dtv(\hat L_i, L_i')
 \wle \sum_{(x,y) \in S} \dtv \Big( \sum_{\ell=1}^{M_{xy}} A_{xy}(\ell), \, \sum_{\ell=1}^{M_{xy}'} A_{xy}(\ell)\Big)
 \wle \sum_{(x,y) \in S} \dtv(M_{xy}, M'_{xy})
\end{align*}
implies that $\dtv( L_i, L_i' ) = \dtv( \hat L_i, L_i' ) \le n^{-2} m \supnorm{X}^2$. By combining this with \eqref{eq:DegreeApproximation1}, the claim follows.
\end{proof}

\subsection{Proof of Theorem~\ref{the:DegreeDistribution}}
We prove the claim in three stages: (i) under an extra assumption that the space of layer types is finite, (ii) under an extra assumption that the layer sizes are bounded, (iii) under no extra assumptions.  In what follows, $D_n = \deg_{G^{(n)}}(1)$ and we consider all models $n=1,2,\dots$ to be defined on a common probability space (see Section~\ref{sec:FormalModel} for formal details).

(i) Assume that the supports of $P_n$, $n \ge 1$, and $P$ are contained in a finite set $A \subset \Z_+ \times [0,1]$.
Denote by $P_{\theta_n} = \frac{1}{m} \sum_{k=1}^m \delta_{(X_{n,k},Y_{n,k})}$ the empirical layer type distribution of the $n$-th model, and denote by $\law(D_n \cond \theta_n)$ the conditional distribution of $D_n$ given layer types $\theta_n = ((X_{n,1},Y_{n,1}), \dots, (X_{n,m},Y_{n,m}))$.  Let us define $\lambda_{\theta_n} = \frac{m}{n} (P_{\theta_n})_{10}$, and
\[
 g_{\theta_n}(t)
 \weq
 \begin{cases}
 \Bin_{10}(P_{\theta_n}), &\quad (P_{\theta_n})_{10} > 0, \\
 \delta_0(t), &\quad \text{else},
 \end{cases}
\]
where $\Bin_{10}(P_{\theta_n})$ is defined by \eqref{eq:MixedBin} and $\delta_0$ is the Dirac measure at zero. Then by applying Proposition~\ref{the:DegreeApproximationQuantitative} and Lemma~\ref{the:CPoiPerturbation},
\begin{align*}
 &\dtv \Big( \law(D_n \cond \theta_n), \CPoi( \lambda, g) \Big) \\
 &\wle \dtv \Big( \law(D_n \cond \theta_n), \CPoi( \lambda_{\theta_n}, g_{\theta_n}) \Big) + 
    \dtv \Big( \CPoi( \lambda_{\theta_n}, g_{\theta_n}) , \CPoi( \lambda, g)  \Big) \\
 &\wle \left( 1 + \frac{m}{n} \right)^2 M^4 n^{-1} + \abs{\lambda_{\theta_n} - \lambda} + \lambda \dtv(g_{\theta_n}, g),\end{align*}
where the inequalities remain valid also on the event that $(P_{\theta_n})_{10} = 0$ because in this case all layers are empty and $\law(D_n \cond \theta_n) = \delta_0$.  On the event that $P_{\theta_n} \weakto P$, 
we see that $\lambda_{\theta_n} = \frac{m}{n} (P_{\theta_n})_{10} \to \mu (P)_{10} = \lambda$ and $g_{\theta_n} \weakto g$ (see Lemma~\ref{the:MixedBiasedERConvergence}), so that $\dtv \big( \law(D_n \cond \theta_n), \CPoi( \lambda, g) \big) \to 0$. Observe next that $\dtv(P_{\theta_n}, P) \prto 0$ by Lemma~\ref{the:EmpiricalDistributionConvergence}.   By applying Lemma~\ref{the:ConditionalConvergenceInProbability} with $\Phi_n(\theta_n, \xi_n) = \law(D_n \cond \theta_n)$, we conclude that $\dtv \big( \law(D_n \cond \theta_n), \CPoi( \lambda, g) \big) \prto 0$. Because $\dtv$ is a bounded metric, it follows that $\dtv \big( \law(D_n), \CPoi( \lambda, g) \big) \le \E \dtv \big( \law(D_n \cond \theta_n), \CPoi( \lambda, g) \big) \to 0$.

(ii) Assume now that the supports of $P_n$ and $P$ are all contained in $\{0,1,\dots,M\} \times [0,1]$. We will discretise the unit interval $[0,1]$ as in Section~\ref{sec:DiscreteLayerTypes}. Fix an integer $L \ge 1$, and denote by $G_n^{L-}$ (resp.\ $G_n^{L+}$) an overlay graph generated by a modified model where the layer strengths $Y_{n,k}$ are replaced by $\floor{Y_{n,k}}_L$ (resp.\ $\ceil{Y_{n,k}}_L$), defined by \eqref{eq:DiscreteLayerStrengths}.
Denote by $D_n^{L-}, D_n^{L+}$ the degrees of node 1 in $G_n^{L-}, G_n^{L+}$, respectively. Under a natural coupling of the Bernoulli variables describing the link indicators of the layers we have $G_n^{L-} \subset G_n \subset G_n^{L+}$ almost surely, and hence
\begin{equation}
 \label{the:DegreeLayerStrengthCoupling}
 \pr( D_n^{L-} \ge t) \wle \pr( D_n \ge t ) \wle \pr( D_n^{L+} \ge t).
\end{equation}
for all integers $t \ge 0$ and $L \ge 1$.

The averaged layer type distribution of $G_n^{L\pm}$ is given by $P_n \circ \sigma_{L\pm}^{-1}$, where $\sigma_{L-}(x,y) = (x, \floor{y}_L)$ and $\sigma_{L+}(x,y) = (x, \ceil{y}_L)$. By Lemma~\ref{the:DiscreteLayerTypes}, $P_n \circ \sigma_{L\pm}^{-1} \weakto P \circ \sigma_{L\pm}^{-1}$ and $(P_n \circ \sigma_{L\pm}^{-1})_{10} \to (P \circ \sigma_{L\pm}^{-1})_{10}$. Hence by part (i), it follows that $\pr( D_n^{L\pm} \ge t) \to \pr( D^{L\pm} \ge t)$, where $\law(D^{L\pm}) = \CPoi( \lambda, g_{L\pm})$ with $g_{L\pm}(t) = \Bin_{10}(P \circ \sigma_{L\pm}^{-1})$. Hence by \eqref{the:DegreeLayerStrengthCoupling},
\[
 \pr( D^{L-} \ge t)
 \wle \liminf_{n \to \infty} \pr( D_n \ge t )
 \wle \limsup_{n \to \infty} \pr( D_n \ge t )
 \wle \pr( D^{L+} \ge t).
\]
Lemma~\ref{the:DiscreteLayerTypes} also shows that $P \circ \sigma_{L\pm}^{-1} \weakto P$  and $(P \circ \sigma_{L\pm}^{-1})_{10} \to (P)_{10}$, so that (Lemma~\ref{the:MixedBiasedERConvergence}) $g_{L\pm} \weakto g$ and hence also (Lemma~\ref{the:CPoiPerturbation}) $\law(D^{L\pm}) \weakto \law(D)$ as $L \to \infty$, where $\law(D) =   \CPoi( \lambda, g)$.  The above inequalities then imply that $\pr( D_n \ge t ) \to \pr( D \ge t )$ for all $t$.  Hence $\law(D_n) \weakto \law(D)$.

(iii) Let us now prove Theorem~\ref{the:DegreeDistribution} without making any extra assumptions. Let $G^{M}_n$ be an overlay graph generated by truncated layers
\begin{equation}
 \label{eq:LayerTruncationNew}
 G^{M}_{n,k}
 \weq
 \begin{cases}
  G_{n,k} &\quad \text{if $\abs{V(G_{n,k})} \le M$}, \\
  \text{empty graph} &\quad \text{otherwise}.
 \end{cases}
\end{equation}
Denote by $D_n^{M}$ the degree of node~1 in $G_n^{M}$. Observe that $D_n \ne D_n^{M}$ implies that there exists a layer $G_{n,k}$ of size larger than $M$ which contains node 1, and this occurs with probability
\[
 \pr( V(G_{n,k}) \ni 1, \abs{V(G_{n,k})} > M)
 \weq \E \frac{X_{n,k}}{n} 1( X_{n,k} > M).
\]
Hence by the union bound,
\begin{equation}
 \label{eq:DegreeTruncation1}
 \dtv ( \law(D_n), \law(D_n^{M}) )
 \wle \sum_{k=1}^m \E \frac{X_{n,k}}{n} 1( X_{n,k} > M)
 \wle \frac{m}{n} h(M),
\end{equation}
where $h(M) = \sup_{n \ge 1} \int x 1(x > M) P_n(dx,dy)$.

Observe next that $G^{M}_n$ is an instance of the overlay model with layer types $(X_{n,k} 1(X_{n,k} \le M), Y_{n,k})$ and averaged layer type distribution $P_n \circ \sigma_M^{-1}$ where $\sigma_M(x,y) = (x 1(x \le M), y)$. By Lemma~\ref{the:DiscreteLayerTypes}, $P_n \circ \sigma_M^{-1} \weakto P \circ \sigma_M^{-1}$ together with $(P_n \circ \sigma_M^{-1})_{10} \to (P \circ \sigma_M^{-1})_{10}$. Hence by part (ii), it follows that
\[
 \dtv( \law(D_n^{M}), \CPoi( \lambda^{M}, g^{M} ) )
 \to 0,
\]
where $\lambda^{M} = \mu (P \circ \sigma_M^{-1})_{10}$ and $g^{M} = \Bin_{10}(P \circ \sigma_M^{-1})$. Now by \eqref{eq:DegreeTruncation1} and Lemma~\ref{the:CPoiPerturbation}, we find that
\begin{align*}
 \dtv (\law(D_n), \CPoi(\lambda, g) )
 &\wle \dtv( \law(D_n^{M}), \CPoi( \lambda^{M}, g^{M} ) ) \\
 &\qquad + \abs{\lambda^{M} - \lambda} + \lambda \dtv( g^{M} , g ) + \frac{m}{n} h(M),
\end{align*}
so that
\begin{equation}
 \label{eq:DegreeTruncation2}
 \limsup_{n \to \infty} \dtv (\law(D_n), \CPoi(\lambda, g) )
 \wle \abs{\lambda^M - \lambda} + \lambda \dtv( g^M , g ) + \mu h(M).
\end{equation}
Lemma~\ref{the:DiscreteLayerTypes} also implies that $h(M) \to 0$, and that
$P \circ \sigma_M^{-1} \weakto P$ together with $(P \circ \sigma_M^{-1})_{10} \to (P)_{10}$ as $M \to \infty$.  Hence $g^{M}  \weakto g$ by Lemma~\ref{the:MixedBiasedERConvergence}.
The claim of Theorem~\ref{the:DegreeDistribution} now follows because the right side of \eqref{eq:DegreeTruncation2} can be made arbitrarily small by choosing a large enough $M$.
\qed

\section{Analysis of clustering}
\label{sec:Clustering}

\subsection{General subgraph densities}
Subgraph frequencies in the overlay graph will be characterised using cross moments
\begin{equation}
 \label{eq:CrossMoments}
 (P_n)_{rs} = \int (x)_r y^s d P_n,
 \qquad
 (P_n)_{rs,tu} = \int (x)_r y^s \, (x)_t y^u d P_n
\end{equation}
of the averaged layer type distribution $P_n$ defined by \eqref{eq:AveragedLayerTypeDistribution}, and 
normalised cross moments defined by
\begin{equation}
 \label{eq:Mu}
 \mu^{(n)}_{rs} \weq \sum_{k=1}^m p^{(n)}_{rs}(k),
 \qquad
 \mu^{(n)}_{rs, tu} \weq \sum_{k=1}^m p^{(n)}_{rs}(k) \, p^{(n)}_{tu}(k),
\end{equation}
where $p^{(n)}_{rs}(k) \weq (n)_r^{-1} \E (X^{(n)}_k)_r (Y^{(n)}_k)^s$. These definitions are motivated by the following result, where $G_{k^*} = G^{(n)}_{k^*}$ represents a randomly chosen layer, and we recall the  the mixed binomial distribution $\Bin_{rs}(P_n)$ defined in \eqref{eq:MixedBin}.

\begin{lemma}
\label{the:SubgraphDensities}
Let $F_{rs}$ be a graph with node set in $\{1,\dots,n\}$ such that $\abs{V(F_{rs})} = r$ and $\abs{E(F_{rs})} = s$, and let $i$ be a node in $V(F_{rs})$ with $\deg_{F_{rs}}(i) = r-1$. Select $k^* \in \{1,\dots,m\}$ uniformly at random and independently of the layers. Then:
\begin{enumerate}[(i)]
\item $\pr( G_{k^*} \supset F_{rs} ) = m^{-1} \mu^{(n)}_{rs}$,
\item $\pr( \deg_{G_{k^*}}(i) = t \cond G_{k^*} \supset F_{rs} ) = \Bin_{rs}(P_n)(t-r+1)$ for all $t$.
\end{enumerate}
\end{lemma}
\begin{proof}
(i) Because $\pr( V(G_k) \supset V(F_{rs}) \cond X_k, Y_k ) = \frac{(X_k)_r}{(n)_r}$ for any $k$, we see that $\pr( G_k \supset F_{rs} ) = \E \frac{(X_k)_r}{(n)_r} Y_k^s = p^{(n)}_{rs}(k)$. The corresponding probability for a randomly selected $k^*$ equals $\pr( G_{k^*} \supset F_{rs} ) = \frac{1}{m} \sum_{k=1}^m p^{(n)}_{rs}(k) = (P_n)_{rs}$.

(ii) Denote $D_k = \deg_{G_k}(i)$. On the event that $G_k \supset F_{rs}$, we see that
$D_k = d + D_k'$ where $D_k' = \abs{N_{G_k}(i) \setminus V(F_{rs})}$ and $d = r-1$. Conditionally on $(X_k,Y_k) = (x,y)$ and $G_k \supset F_{rs}$, the random integer $D_k'$ is $\Bin(x-r, y)$-distributed. Hence
\[
 \pr( D_k = t, \, G_k \supset F_{rs} )
 \weq \E \left( \Bin(X_k-r, Y_k)(t-d) \frac{(X_k)_r}{(n)_r} Y_k^s \right).
\]
The corresponding probability for a randomly chosen $k^*$ is
\[
 \pr( D_{k^*} = t, \, G_{k^*} \supset F_{rs} )
 \weq \int \left( \Bin(x-r, y)(t-d) \frac{(x)_r}{(n)_r} y^s \right) P_n(dx,dy),
\]
so the claim follows by dividing both sides by $\pr(G_{k^*} \supset F_{rs}) = (n)_r^{-1} (P_n)_{rs}$.
\end{proof}

\subsection{Triangle densities}

The following quantitative bound is valid for every fixed $n$.

\begin{theorem}
\label{the:TriangleDensityDeg}
Let $K_3$ be a triangle with $i \in V(K_3) \subset [n]$.
Then:
\begin{enumerate}[(i)]
\item $\abs{ \pr( G \supset K_3 ) - \mu_{33} } \le 4 \mu_{21} \mu_{32} + \mu_{21}^3$.
\item $\pr( \deg_G(i) = t, \, G \supset K_3 ) = \mu_{33} \, f^{(n)} \conv g^{(n)}_{33}(t-2) + \epsilon(t)$,
where $f^{(n)}$ is the model degree distribution defined by \eqref{eq:ModelDegreeDistribution}, $g^{(n)}_{33} = \Bin_{33}(P_n)$ is defined by \eqref{eq:MixedBin}, and the approximation error is bounded by
\[
 \abs{\epsilon(t)}
 \wle (4+t) \mu_{21} \mu_{32} + \mu_{21}^3 + 2\mu_{10,33}.
\]
\end{enumerate}
\end{theorem}


\begin{proof}
Denote $\cK_3 = \{G \supset K_3\}$. 
Denote by $\cA_k = \{G_k \supset K_3\}$ the event that all node pairs of the triangle are linked by layer $k$. We also denote $D = \deg_G(i)$, $D_k = \deg_{G_k}(i)$, and $D_{-k} = \deg_{G_{-k}}(i)$ with $G_{-k} = \cup_{k' \ne k} G_k$.

(i) Denote
\[
 \epsilon_1(t)
 \weq \pr( D = t, \cK_3 ) - \pr( D = t, \cup_k \cA_k ),
\]
and observe that $0 \le \epsilon_1(t) \le \pr( D=t, \cE_{12} ) + \pr( D = t, \cE_{111} )$, where $\cE_{12}$ is the event that there exists one layer covering one link and a different layer covering two links of $K_3$, and $\cE_{111}$ is the event that three distinct layers cover the links of $K_3$.  We write $p(abc) = \pr(\cG^a_{12}, \cG^b_{13}, \cG^c_{23})$, where $\cG^a_{ij}$ the event that node pair $ij$ is linked in layer $a$.  We note that $p(abc) = p_{21}(a)p_{21}(b)p_{21}(c)$, $p(aab) = p_{32}(a)p_{21}(b)$, and $p(aaa) = p_{33}(a)$ for distinct layers $a,b,c$. Hence
\begin{align*}
 \pr(\cE_{12})
 \wle \sumd_{a,b} \Big(p(aab) + p(aba) + p(baa)\Big) 
 &\wle 3 \mu_{21} \mu_{32},
\end{align*}
and
$
 \pr(\cE_{111})
 \le \sum'_{a,b,c} p(abc)
 \le \mu_{21}^3.
$
Thus, $\sum_{t \ge 0} \abs{\epsilon_1(t)} \le 3 \mu_{21} \mu_{32} + \mu_{21}^3$.

Then denote
\[
 \epsilon_2(t)
 \weq \pr( D = t, \cup_k \cA_k ) - \sum_k \pr( D = t, \cA_k).
\]
Bonferroni's inequalities imply that $0 \le -\epsilon_2(t) \le \sum'_{k,k'} \pr( D = t, \cA_k, \cA_{k'})$, and hence, noting that $\mu_{33} \le \mu_{32} \le \mu_{21}$,
\[
 \sum_{t \ge 0} \abs{\epsilon_2(t)}
 \wle \sumd_{k,k'} \pr( \cA_k, \cA_{k'})
 \weq \sumd_{k,k'} p_{33}(k) p_{33}(k')
 \wle \mu_{33}^2
 \wle \mu_{21} \mu_{32}.
\]
By combining this with the bound for $\epsilon_1(t)$, we conclude that
\[
 \pr( D = t, \cK_3 )
 \weq \sum_k \pr( D = t, \cA_k) + \epsilon_1(t) + \epsilon_2(t),
\]
where $\sum_{t \ge 0} ( \abs{\epsilon_1(t)} + \abs{\epsilon_2(t)}) \le 4\mu_{21} \mu_{32} + \mu_{21}^3$. Hence claim (i) follows by summing the above equality over $t$, and noting that $\sum_k \pr(\cA_k) = \mu_{33}$.

(ii) We will next approximate
\begin{align}
 \sum_k \pr( D = t, \cA_k)
 \label{eq:TriangleDegreeNew3} &\wapprox \sum_k \pr( D_{-k} + D_k = t, \cA_k) \\
 \nonumber &\weq \sum_k \sum_{r+s=t} \pr( D_{-k} = r) \, \pr( D_k = s, \cA_k) \\
 \label{eq:TriangleDegreeNew4} &\wapprox \sum_k \sum_{r+s=t} \pr(D=r) \, \pr( D_k = s, \cA_k ).
\end{align}
Lemma~\ref{the:SubgraphDensities} shows that
$\sum_k \pr( D_k = s, \cA_k ) = \mu_{33} \Bin_{33}(P_n)(s-2)$. Hence the last term above
equals $\mu_{33} \, f^{(n)} \conv g^{(n)}_{33}(t-2)$, and to prove the claim it suffices to analyse the approximation errors in \eqref{eq:TriangleDegreeNew3}--\eqref{eq:TriangleDegreeNew4}.

The approximation error in \eqref{eq:TriangleDegreeNew3} equals $\epsilon_3(t) = \sum_k \epsilon_{3k}(t)$, where
\begin{align*}
 \epsilon_{3k}(t)
 \weq \pr( D = t, \cA_k) - \pr( D_{-k} + D_k = t, \cA_k).
\end{align*}
By applying Lemma~\ref{the:DegreeSplitLayers} with $A = \{k\}$, $B = [n] \setminus \{k\}$, $\cE_A = \{G_k \ni e_1, e_2, e_3\}$, and $\cE_B = \{\}$ being the sure event, we see that $\abs{\epsilon_{3k}(t)} \le c_B t \pr( D_k \le t, \cA_k ) \le c_B t \pr( \cA_k )$, where $c_B = \pr(G_{-k} \ni 12) \le \sum_{\ell \ne k} p_{21}(\ell) \le \mu_{21}$. Hence
\[
 \abs{\epsilon_3(t)}
 \wle t \mu_{21} \sum_k p_{33}(k)
 \weq t \mu_{21} \mu_{33}
 \wle t \mu_{21} \mu_{32}.
\]

The approximation error in \eqref{eq:TriangleDegreeNew4} equals $\epsilon_4(t) = \sum_k \epsilon_{4k}(t)$ where
\[
 \epsilon_{4k}(t)
 \weq \sum_{r+s=t} \Big( \pr(D=r) - \pr( D_{-k} = r) \Big) \pr(D_k = s, \cA_k).
\]
By Lemma~\ref{the:DegreeLessLayers}, $\sum_{t \ge 0} \abs{\epsilon_{4k}(t)} \le 2 \pr(D_k > 0) \pr(\cA_k)$.
Because $\pr(D_k > 0) \le p_{10}(k)$ and $\pr(\cA_k) = p_{33}(k)$, it follows that $\sum_{t \ge 0} \abs{\epsilon_{4}(t)} \le 2\mu_{10,33}$.  Claim (ii) follows by combining the above estimates for the total approximation error $\epsilon(t) = \epsilon_{1}(t) + \epsilon_{2}(t) + \epsilon_{3}(t) + \epsilon_{4}(t)$.
\end{proof}

\subsection{Two-star densities}

The following quantitative bound is valid for every fixed $n$.

\begin{theorem}
\label{the:TwoStarDensityDeg}
Consider a two-star $K_{12}$ with node set $V(K_{12}) \subset [n]$ and hub node $i$.
Then:
\begin{enumerate}[(i)]
\item $\abs{ \pr(G \supset K_{12} ) - (\mu_{32} + \mu_{21}^2)} \le 6 \mu_{21} \mu_{32} + 6 \mu_{21}^3 + \mu_{21}^4 + \mu_{21,21}$.
\item
$ \pr( \deg_G(i) = t, G \supset K_{12})
 =
 \mu_{32} \, f^{(n)} \conv g^{(n)}_{32}(t-2)
 + \mu_{21}^2 f^{(n)} \conv g^{(n)}_{21} \conv g^{(n)}_{21}(t-2)
 + \epsilon(t),
$
where $f^{(n)}$ is the degree distribution of $G$, and the approximation error is bounded by
\begin{align*}
 \abs{\epsilon(t)}
 \wle (6+2t) (\mu_{21} \mu_{32} + \mu_{21}^3) + \mu_{21}^4
 + 4 \mu_{10,32} + 4 \mu_{21} \mu_{10,21}
 + \mu_{21,21}.
\end{align*}
\end{enumerate}
\end{theorem}

\begin{proof}
We assume that $K_{12}$ is the two-star with node set $\{1,2,3\}$ and link set $\{12,13\}$, and denote the event under study by $\cK_{12} = \{G \supset K_{12}\}$. We denote by $\cG_{ij}^k$ the event that $ij \in E(G^k)$ and we set $\cA_{k\ell} = \cG^k_{12} \cap \cG^\ell_{13}$. We denote $G^{k\ell} = G^k \cup G^\ell$ and $G^{-k\ell} = \cup_{q \notin \{k,\ell\}} G^q$, and we set $D = \deg_{G}(1)$, $D_{k\ell} = \deg_{G^{k\ell}}(1)$ and $D_{-k\ell} = \deg_{G^{-k\ell}}(1)$.  We also denote $h_{k\ell}(s) = \pr( D_{k\ell} = s, \cA_{k\ell})$.

First we approximate
\begin{align}
 \pr( D = t, \cK_{12}) 
 \label{eq:TwoStarErr1} &\wapprox \sum_{k,\ell} \pr( D = t, \cA_{k\ell}) \\
 \label{eq:TwoStarErr2} &\wapprox \sum_{k,\ell} \pr( D_{k\ell} + D_{-k\ell} = t, \cA_{k\ell}) \\
 \nonumber &\weq \sum_{k,\ell} \sum_{r+s = t} \pr( D_{-k\ell} = r ) \, h_{k\ell}(s) \\
 \label{eq:TwoStarErr3} &\wapprox \sum_{k,\ell} \sum_{r+s = t} \pr( D = r ) \, h_{k\ell}(s),
\end{align}
so that
\begin{equation}
 \label{eq:TwoStarErr13}
 \pr( D = t, \cK_{12})
 \wapprox \sum_{r+s = t} f^{(n)}(r) \sum_{k} h_{kk}(s)
 \ + \sum_{r+s = t} f^{(n)}(r) \sumd_{k,\ell} h_{k\ell}(s).
\end{equation}

Then we note with the help of Lemma~\ref{the:SubgraphDensities} that $\sum_k h_{kk}(s) = \mu_{32} \, g^{(n)}_{32}(s-2)$.
Hence the first term on the right side of \eqref{eq:TwoStarErr13} equals
\begin{equation}
 \label{eq:TwoStarDegCommon}
 \sum_{r+s = t} f^{(n)}(r) \sum_{k} h_{kk}(s)
 \weq \mu_{32} f^{(n)} \conv g^{(n)}_{32}(t-2).
\end{equation}
Next we approximate, denoting $h_k(s) = \pr( D_{k} = s, \cG_{12}^k)$,
\begin{align}
 \sumd_{k,\ell} h_{k\ell}(s)
 \nonumber &\weq \sumd_{k,\ell} \pr( D_{k\ell} = s, \cG_{12}^k, \cG_{13}^\ell) \\
 \label{eq:TwoStarErr4} &\wapprox \sumd_{k,\ell} \pr( D_k + D_\ell = s, \cG_{12}^k, \cG_{13}^\ell) \\
 \nonumber &\weq \sumd_{k,\ell} \!\! \sum_{s_1+s_2 = s} h_k(s_1) h_\ell(s_2) \\
 \label{eq:TwoStarErr5} &\wapprox \, \sum_{k,\ell} \sum_{s_1+s_2 = s} h_k(s_1) h_\ell(s_2).
\end{align}
After noting (see Lemma~\ref{the:SubgraphDensities}) that $\sum_k h_k(s) = \mu_{21} g^{(n)}_{21}(s-1)$,
we conclude that
\[
 \sum_{k,\ell} \sum_{s_1+s_2 = s} h_k(s_1) h_\ell(s_2)
 \weq \mu_{21}^2  g^{(n)}_{21} \conv g^{(n)}_{21}(s-2),
\]
and hence the second term on the right side of \eqref{eq:TwoStarErr13} is approximately
\begin{equation}
 \label{eq:TwoStarDegDistinct}
 \sum_{r+s = t} f(r) \sumd_{k,\ell} h_{k\ell}(s)
 \wapprox \mu_{21}^2 f \conv g^{(n)}_{21} \conv g^{(n)}_{21}(t-2).
\end{equation}
By combining \eqref{eq:TwoStarErr13}, \eqref{eq:TwoStarDegCommon} and \eqref{eq:TwoStarDegDistinct}, we conclude that
\begin{equation}
 \label{eq:TwoStarDegCombined}
 \pr( D = t, \cK_{12})
 \wapprox \mu_{32} f^{(n)} \conv g^{(n)}_{32}(t-2) + \mu_{21}^2 f^{(n)} \conv g^{(n)}_{21} \conv g^{(n)}_{21}(t-2).
\end{equation}

The total approximation error in \eqref{eq:TwoStarDegCombined} can be written as
$\epsilon(t) = \epsilon_1(t) + \epsilon_2(t) + \epsilon_3(t) +\epsilon_4(t),$
where $\epsilon_1(t), \epsilon_2(t), \epsilon_3(t)$ are the approximation errors in \eqref{eq:TwoStarErr1}, \eqref{eq:TwoStarErr2}, \eqref{eq:TwoStarErr3}, respectively, and the 
%
approximation error in \eqref{eq:TwoStarDegDistinct} equals
\[
 \epsilon_4(t)
 \weq \sum_{r+s = t} f^{(n)}(r) \big(\epsilon_{41}(s) + \epsilon_{42}(s)\big),
\]
where $\epsilon_{41}(s), \epsilon_{42}(s)$ denote the errors made in \eqref{eq:TwoStarErr4}, \eqref{eq:TwoStarErr5}, respectively. We will next analyse the individual approximation errors one by one.

(i) The union bound shows that the approximation error $\epsilon_1(t)$ in \eqref{eq:TwoStarErr1}
is nonpositive for all $t$, and hence $\sum_{t \ge 0} \abs{\epsilon_1(t)} = \sum_{k,\ell} \pr( \cA_{k\ell}) - \pr( \bigcup_{k,\ell} \cA_{k\ell} )$. Bonferroni's inequalities imply that
\[
 \sum_{t \ge 0} \abs{\epsilon_1(t)}
 \wle \sumd_{(k_1,k_2), (\ell_1,\ell_2)} \pr(\cA_{k_1k_2}, \cA_{\ell_1\ell_2})
 \ =: \Delta.
\]
We split the right side above by
$
 \Delta
 = \Delta_2 + \Delta_3 + \Delta_4,
$
where $\Delta_i$, $i=2,3,4$, is the sum on the right side above over layer pairs $(k_1,k_2) \ne (\ell_1, \ell_2)$ such that the list $(k_1,k_2,\ell_1,\ell_2)$ contains precisely $i$ distinct elements. Denote
\[
 p(k_1k_2\ell_1\ell_2) \weq \pr( G_{k_1} \ni e_1, G_{k_2} \ni e_2, G_{\ell_1} \ni e_1, G_{\ell_2} \ni e_2).
\]
Then
\begin{align*}
 \Delta_2 &\weq \sumd_{a,b} \Big( p(aabb) + p(abba) + p(aaab) + p(aaba) + p(abaa) + p(baaa) \Big), \\
 \Delta_3 &\weq \sumd_{a,b,c} \Big( p(aabc) + p(abac) + p(abca) + p(baac) + p(baca) + p(bcaa) \Big).
\end{align*}
In the sum of $\Delta_2$, the terms $p(aabb)$ and $p(abba)$ equal $p_{32}(a) p_{32}(b)$ and the other terms equal $p_{32}(a) p_{21}(b)$. Because $p_{32}(b) \le p_{21}(b)$, it follows that
$
 \Delta_2
 \le 6 \sumd_{a,b} p_{21}(a) p_{32}(b)
 \le 6 \mu_{21} \mu_{32}.
$
In the sum of $\Delta_3$, the terms $p(abac)$ and $p(baca)$ equal  $p_{21}(a) p_{21}(b) p_{21}(c)$ and the other terms equal $p_{32}(a) p_{21}(b) p_{21}(c)$. Because $p_{32}(a) \le p_{21}(a)$, it follows that
$
 \Delta_3
 \le 6 \mu_{21}^3.
$
Furthermore, $\Delta_4 = \sum'_{a,b,c,d} p(abcd) \le \mu_{21}^4$.
As a conclusion, it follows that
\[
 \sum_{t \ge 0} \abs{\epsilon_1(t)}
 \wle 6 \mu_{21} \mu_{32} + 6 \mu_{21}^3 + \mu_{21}^4.
\]
Claim (i) now follows by combining the above bound with the equality
\[
 \sum_{k,\ell} \pr( \cA_{k\ell})
 \weq \sum_k p_{32}(k) + \sumd_{k,\ell} p_{21}(k) p_{21}(\ell)
 \weq \mu_{32} + \mu_{21}^2 - \mu_{21,21}.
\]

(ii) The approximation error in \eqref{eq:TwoStarErr2} equals $\epsilon_2(t) = \sum_{k,\ell} \epsilon_{2k\ell}(t)$ where
\[
 \epsilon_{2k\ell}(t)
 \weq \pr( D = t, \cA_{k\ell} ) - \pr( D_{k\ell} + D_{-k\ell} = t, \cA_{k\ell} ).
\]
By applying Lemma~\ref{the:DegreeSplitLayers} with $A = \{k,\ell\}$, $B = [m] \setminus \{k,\ell\}$, $\cE_A = \{G_k \ni e_1, G_\ell \ni e_2\}$, and $\cE_B = \{\}$ being the sure event, we see that
\[
 \abs{\epsilon_{2k\ell}(t)}
 \wle t c_B \pr( D_{k\ell} \le t, \cA_{k\ell} )
 \wle t c_B \pr( \cA_{k\ell} ),
\]
where $c_B \le \pr(G_{-k\ell} \ni 12) \le \pr(G \ni 12) \le \mu_{21}$. Hence
\[
 \abs{\epsilon_2(t)}
 \wle t \mu_{21} \sum_{k,\ell} \pr( \cA_{k\ell} )
 \wle t  ( \mu_{21} \mu_{32} + \mu_{21}^3 ).
\]

(iii) The approximation error in \eqref{eq:TwoStarErr3} equals $\epsilon_3(t) = \sum_{k,\ell} \epsilon_{3k\ell}(t)$ where
\[
 \epsilon_{3k\ell}(t)
 \weq \sum_{r+s = t} \Big( \pr(D = r) - \pr(D_{-k\ell} = r) \Big) h_{k\ell}(s).
\]
By applying Lemma~\ref{the:DegreeLessLayers} with $g(s) = \frac{h_{k\ell}(s)}{\pr(\cA_{k\ell})}$, it follows that $\sum_{t \ge 0} \abs{\epsilon_{3k\ell}(t)} \le 2 \pr(\cA_{k\ell}) \pr(D_{k\ell}>0)$. Observe now that
$\pr(D_{k\ell} > 0)
\le p_{10}(k) + p_{10}(\ell)$, 
%
Hence,
\begin{align*}
 \sum_{t \ge 0} \abs{\epsilon_3(t)}
 &\wle 2 \sum_{k,\ell} ( p_{10}(k) + p_{10}(\ell) ) \, \pr( \cA_{k\ell} ) \\
 &\weq 4 \sum_{k} p_{10}(k) p_{32}(k) + 4 \sumd_{k,\ell} p_{10}(k) p_{21}(k) p_{21}(\ell) \\
 &\wle 4 \mu_{10,32} + 4 \mu_{21} \mu_{10,21}.
\end{align*}


(iv) The approximation error in \eqref{eq:TwoStarErr4} equals $\epsilon_{41}(s) = \sum'_{k,\ell} \epsilon_{4k\ell}(s)$ where
\begin{align*}
 \epsilon_{4k\ell}(s)
 &\weq \pr( D_{k\ell} = s, \cA_{k\ell} ) - \pr( D_{k} + D_\ell = s, \cA_{k\ell} ).
\end{align*}
By applying Lemma~\ref{the:DegreeSplitLayers} with $A = \{k\}$ and $B = \{\ell\}$, together with $\cE_A = \{12 \in G_k\}$ and $\cE_B = \{13 \in G_\ell\}$, it follows that $\abs{\epsilon_{4k\ell}(s)} \le s p_{21}(k) p_{32}(\ell).$ By summing the above inequality with respect to $k,\ell$, it follows that $\abs{\epsilon_{41}(s)} \le s \mu_{21} \mu_{32}$.
The approximation error in \eqref{eq:TwoStarErr5} equals
\begin{align*}
 \abs{\epsilon_{42}(s)}
 \weq \sum_{k} \sum_{s_1+s_2 = s} \pr( D_{k} = s_1, \cG_{12}^k) \pr( D_{k} = s_2, \cG_{12}^k).
\end{align*}
Hence $\sum_{s \ge 0} \abs{\epsilon_{42}(s)} = \sum_k p_{21}(k)^2 = \mu_{21,21}$.
Hence,
\[
 \abs{\epsilon_4(t)}
 \wle \sum_{r+s = t} f^{(n)}(r) \big( \abs{\epsilon_{41}(s)} + \abs{\epsilon_{42}(s)} \big)
 \wle \max_{s \le t} \abs{\epsilon_{41}(s)} + \max_{s \le t} \abs{\epsilon_{42}(s)}
\]
shows that $\abs{\epsilon_4(t)} \le t \mu_{21} \mu_{32} + \mu_{21,21}$.

Claim (ii) follows by collecting all the bounds in (i)--(iv) together.
\end{proof}

\subsection{Lemma about cross moments}

\begin{lemma}
\label{the:SpecialUI}
Let $(X_1, Y_1), \dots, (X_m,Y_m)$ be random variables with values in $\{0,\dots,n\} \times [0,1]$ and averaged empirical distribution $P_n$ defined by \eqref{eq:AveragedLayerTypeDistribution}. If
$P_n \weakto P$ and $(P_n)_{rs} \to (P)_{rs} < \infty$, then the cross moments defined in \eqref{eq:CrossMoments}--\eqref{eq:Mu} satisfy $\mu^{(n)}_{10,rs} \ll m (n)_r^{-1}$ and $(P_n)_{10,rs} \ll n$.
\end{lemma}
\begin{proof}
Denote $A_k = X_k$ and $B_k = (X_k)_r Y_k^s$. Observe that $A_k \le a + A_k 1(A_k > a)$ and $B_k \le b + B_k 1(B_k > b)$ for any $a,b > 0$. Because $A_k \le n$, we find that
\begin{equation}
 \label{eq:UISpecial1}
 \begin{aligned}
 A_k \, \E B_k
 &\wle (a + A_k 1(A_k > a)) \E B_k \\
 &\wle a \E B_k + b n 1(A_k > a) + n \E B_k 1(B_k > b).
 \end{aligned}
\end{equation}
By taking expectations and averaging with respect to $k$, we find that
\begin{equation}
 \label{eq:UISpecial2}
 \frac{1}{m} \sum_{k=1}^m \E A_k \, \E B_k
 \wle a \E B_* + b n \pr(A_* > a) + n \E B_* 1(B_* > b),
\end{equation}
where $A_* = X_*$, $B_* = (X_*)_r Y_*^s$, and $(X_*,Y_*)$ is a generic $P_n$-distributed random variable.
Because the left side above equals $m^{-1} n(n)_r \mu^{(n)}_{10,rs}$, we conclude
\[
 m^{-1} (n)_r \mu^{(n)}_{10,rs}
 \wle \frac{a}{n} c + b \phi(a) + \psi(b),
\]
where $c = \sup_n (P_n)_{rs}$, $\phi(t) = \sup_n \int 1(x > t) dP_n$, and $\psi(t) = \sup_n \int (x)_r y^s 1( (x)_r y^s > t ) dP_n$. 
Then the tightness of $P_n$ implies that $\phi(a_n) \to 0$ for $a_n = n^{1/2}$.  Hence also $b_n \phi(a_n) \to 0$ where $b_n = \phi(a_n)^{-1/2} \to \infty$. The uniform $(x)_r y^s$-integrability of $P_n$ further implies that $\psi(b_n) \to 0$. Hence the right side above vanishes and first claim follows.

For the second claim, we may repeat the above reasoning to verify that \eqref{eq:UISpecial1} holds also with the $\E$-symbol removed. Therefore, \eqref{eq:UISpecial2} also holds when the left side is replaced by $(P_n)_{10,rs} = \frac{1}{m} \sum_{k=1}^m \E A_k B_k$. Hence the second claim follows by the same argument.
\end{proof}

\subsection{Proof of Theorem~\ref{the:TransitivityGlobal}}
\label{sec:ProofTransitivityOverall}

By Theorem~\ref{the:TriangleDensityDeg} and Theorem~\ref{the:TwoStarDensityDeg},
\begin{align*}
 \pr(\cK_3) &\weq \mu_{33} + O\big(\mu_{21}\mu_{32} + \mu_{21}^3 \big), \\
 \pr(\cK_{12}) &\weq \mu_{32} + \mu_{21}^2 + O\big( \mu_{21}\mu_{32} + \mu_{21}^3 + \mu_{21}^4 + \mu_{21,21} \big),
\end{align*}
where $\mu_{rs} = m (n)_r^{-1} (P_n)_{rs}$, and the associated cross moments are defined by \eqref{eq:CrossMoments}--\eqref{eq:Mu}. Because $(P_n)_{21} \lesim 1$ and $(P_n)_{32} \lesim 1$, it follows that $\mu_{21} \mu_{32} \lesim m^2 n^{-5} $, $\mu_{21}^3 \lesim m^3 n^{-6}$, and $\mu_{21}^4 \lesim m^4 n^{-8}$. Next, we note that $\mu_{21,21} \le m (n)_2^{-2} (P_n)_{21,21}$ by Jensen's inequality. Note also that $((x)_2y)^2 \le 2 x (x)_3 y^2$ for $x \ge 3$. Hence $((x)_2y)^2 \le 4 + 2 x (x)_3 y^2$, and $(P_n)_{21,21} \le 4 + 2 (P_n)_{10,32}$. Furthermore, Lemma~\ref{the:SpecialUI} implies that $(P_n)_{10,32} \ll n$. Hence $\mu_{21,21} \ll m n^{-3}$.

(i) Consider the case $\frac{m}{n} \to \mu \in [0,\infty)$. Then $\mu_{32} = \big( (P)_{32} + o(1) \big) m n^{-3}$
and
$
 \mu_{21}^2
 = \big( \mu (P)_{21}^2 + o(1) \big) m n^{-3}
$
imply that
\begin{align*}
 \pr(\cK_{12})
 &\weq (P)_{32} m n^{-3} + \mu (P)_{21}^2 m n^{-3} + o\big(m n^{-3} \big).
\end{align*}
Similarly, $\mu_{33} = \big( (P)_{33} + o(1) \big) n^{-3} m$ implies
\[
 \pr(\cK_3)
 \weq (P)_{33} m n^{-3} + o(m n^{-3}),
\]
and hence the first two claims of Theorem~\ref{the:TransitivityGlobal} follow.

(ii) Assume now that $n \ll m \ll n^2$. Then $m n^{-3}, m^2 n^{-5}, m^3 n^{-6} \ll m^3 n^{-4} $.  Hence $\pr(\cK_3) \ll m^2 n^{-4}$.
Furthermore, $m^4 n^{-8} \ll m^2 n^{-4}$, and we conclude that $\pr(\cK_{12}) = (P)_{21}^2  m^2 n^{-4} + o(m^2 n^{-4})$. Hence $\frac{\pr(\cK_3)}{\pr(\cK_{12})} \to 0$ implies the third claim of Theorem~\ref{the:TransitivityGlobal}.
\qed

\subsection{Proof of Theorem~\ref{the:TransitivityLocal}}
Let $K_{12}$ be the two-star on $\{1,2,3\}$ with links $\{12,13\}$. Let $K_{3}$ be the triangle on $\{1,2,3\}$.
Denote $\cK_3^{(n)} = \{G^{(n)} \supset K_3\}$ and $\cK_{12}^{(n)} = \{G^{(n)} \supset K_{12}\}$.
Let $D^{(n)} = \deg_{G^{(n)}}(1)$. By Theorem~\ref{the:TriangleDensityDeg} and Theorem~\ref{the:TwoStarDensityDeg},
\begin{align*}
 \pr( D^{(n)} = t, \, \cK_3^{(n)} )
 &\weq \mu_{33} \, f^{(n)} \conv g^{(n)}_{33}(t-2) + \epsilon_1^{(n)}(t), \\
 \pr( D^{(n)} = t, \, \cK_{12}^{(n)})
 &\weq \mu_{32} \, f^{(n)} \conv g^{(n)}_{32}(t-2) 
  + (\mu_{21})^2 f^{(n)} \conv g^{(n)}_{21} \conv g_{21}^{(n)}(t-2) + \epsilon_2^{(n)}(t).
\end{align*}
where the associated cross moments are defined by \eqref{eq:CrossMoments}--\eqref{eq:Mu}, the distributions 
$g^{(n)}_{rs} = \Bin_{rs}(P_n)$ are defined by \eqref{eq:MixedBin}, and 
\begin{align*}
 \abs{\epsilon_1^{(n)}(t)}
 &\wle (4+t) \mu_{21} \mu_{32} + (\mu_{21})^3 + 2\mu_{10,33}, \\
 \abs{\epsilon^{(n)}_2(t)}
 &\wle (6+2t) (\mu_{21} \mu_{32} + (\mu_{21})^3) + (\mu_{21})^4
 + 4 \mu_{10,32} + 4 \mu_{21} \mu_{10,21}
 + \mu_{21,21}.
\end{align*}
Now Lemma~\ref{the:SpecialUI} implies that $\mu_{10,21} \ll n^{-1}$ and $\mu_{10,33} \le \mu_{10,32} \ll n^{-2}$. Also, the argument in the proof of Theorem~\ref{the:TransitivityGlobal} (Section~\ref{sec:ProofTransitivityOverall}) implies that $\mu_{21,21} \ll n^{-2}$. Because $\mu_{rs} = m (n)_r^{-1} (P_n)_{rs}$ and $(P_n)_{21}, (P_n)_{32} \lesim 1$, it follows that $\mu_{21} \ll n^{-1}$ and $\mu_{21} \mu_{32} + \mu_{21}^3 + \mu_{21}^4 \ll n^{-2}$. Hence,
$
 \abs{\epsilon_1^{(n)}(t)} + \abs{\epsilon_2^{(n)}(t)}
 \ll (1+t) n^{-2}.
$
Note also that $\mu_{32} = (\mu+o(1))(P)_{32} n^{-2}$, $\mu_{33} = (\mu+o(1))(P)_{33} n^{-2}$, together with $\mu_{21}^2 = (1+o(1)) \mu (P)_{21}^2 n^{-2}$.  Moreover, by Theorem~\ref{the:DegreeDistribution}, $f^{(n)} \weakto f = \CPoi(\mu (P)_{10},g_{10})$. By Lemma~\ref{the:MixedBiasedERConvergence}, $g^{(n)}_{rs} \weakto g_{rs}$ for $rs = 21,32,33$. As a consequence,
\begin{align*}
 \mu_{33} \, f^{(n)} \conv g^{(n)}_{33}(t-2)
 &\weq (P)_{33} \mu n^{-2} f \conv g_{33}(t-2) + o(n^{-2}), \\
 \mu_{32} \, f^{(n)} \conv g^{(n)}_{32}(t-2)
 &\weq (P)_{32} \mu n^{-2} f \conv g_{32}(t-2) + o(n^{-2}), \\
 \mu_{21}^2 f^{(n)} \conv g^{(n)}_{21} \conv g_{21}^{(n)}(t-2) 
 &\weq (P)_{21}^2 \mu^2 n^{-2} f \conv g_{21} \conv g_{21}(t-2) + o(n^{-2}),
\end{align*}
and hence the claim follows.
\qed

%
%
%

\section{Analysis of connectivity}

The proof of Theorem~\ref{the:Giant} builds upon the approach developed in \cite{Bollobas_Janson_Riordan_2007} and extended to random intersection graphs in \cite{Bloznelis_2010_Largest}. We denote by $C_G(i)$ the component of node $i$, by $N_1(G) \ge N_2(G)$ the largest two component sizes, and by $B_t(G) = \{i: \abs{C_G(i)} > t\}$ be the set of nodes with component larger than $t$ in $G$. Here $\rho_t(f)$ denotes the probability that the total progeny of a Galton--Watson process with offspring distribution $f$ is larger than $t$, and $\rho(f) = \lim_{t \to \infty} \rho_t(f)$ is the long-term survival probability (see Appendix~\ref{sec:BranchingTrees}). We start by the case with deterministic layer types.


\subsection{Quantitative upper bound for deterministic layer types}

In this section we prove the following quantitative upper bound which is valid for any model instance with deterministic layer types, without taking limits. The upper bound is characterised by a distribution
\begin{equation}
 \label{eq:UpperOffspring}
 f_{\tau,n}
 \weq \law \Big( \sum_{k=1}^m B_k T_k \Big),
\end{equation}
where the random variables on the right are mutually independent and such that $\law(B_k) = \Ber( \frac{X_k}{n-\tau})$ and $\law(T_k) =  \Bin^+(X_k-1,Y_k)$. 

\begin{proposition}
\label{the:SingleUpperBound}
If the layer types are nonrandom with sizes bounded by $M$, then
for any $n \ge 3$ and $1 \le \tau \le n/2$, the probability of a node $i$ having a component larger than $\tau$ is bounded by 
$
 \pr(\abs{C_G(i)} > \tau)
 \le \rho_\tau(f_{\tau,n}) + c \tau^2 n^{-1} \log n,
$
where $c = e^{5M(1+m/n)}$.

\end{proposition}

\subsubsection{Restricted exploration process}
The proof of Proposition~\ref{the:SingleUpperBound} is based on a restricted component exploration process described in Algorithm~\ref{algo:UpperExploration}.  The algorithm explores each layer at most once, and always discovers a subset of $C_G(i)$. This subset may be strict (see Figure~\ref{fig:ModifiedExploration}).

\begin{algorithm}[H]
\small
\DontPrintSemicolon
\KwInput{Graph layers $G_1,\dots,G_m$, root node $i$
}
\KwOutput{A subset of the $G$-connected component of $i$}
~\\
Initialise: $\cQ \leftarrow \{i\}$, $\cM \leftarrow \emptyset$, $t \leftarrow 0$ \\
\While{$\cQ \ne \emptyset$}
{
$t \leftarrow t + 1$\\
Node selection: $v_t \leftarrow \min \cQ$, $\cQ \leftarrow \cQ \setminus \{v_t\}$ \\
\For{$k=1,\dots,m$}
{
\If{$V(G_k) \ni v_t$ and $k \notin \cM$}
{
Layer exploration: $\cZ \leftarrow N_{v_t}(\bar G_k)$\\
Queue update: $\cQ \leftarrow \cQ \cup \cZ$ \\
Update the set of explored layers: $\cM \leftarrow \cM \cup \{k\}$\\
}
}
}
Output node set $\{v_1,\dots, v_t\}$
\caption{Restricted exploration.}
\label{algo:UpperExploration}
\end{algorithm}

\begin{figure}[h]
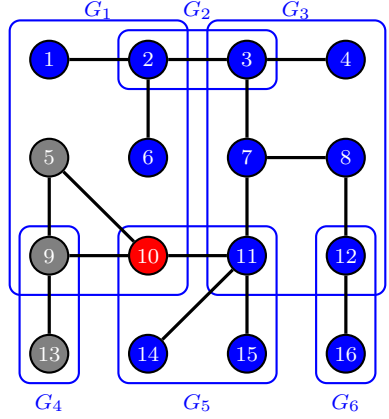

\centering
\scriptsize
\PicModifiedExplorationFourFour
\caption{\label{fig:ModifiedExploration} The component of node 1 equals $C_1 = \{1,\dots,16\}$,
but Algorithm~\ref{algo:UpperExploration} outputs $C_1 \setminus \{5,9,13\}$.  
Algorithm~\ref{algo:UpperExploration} discovers nothing while exploring node 10, because layer $G_1$ is already explored. A multi-overlap occurs while exploring node 11 when layer $G_5$ intersects the already explored layer $G_1$.
}
\end{figure}

\subsubsection{Properties of Algorithm~\ref{algo:UpperExploration}}
\label{sec:PropertiesUpperExploration}

We denote by $T_i$ the number of steps completed by Algorithm~\ref{algo:UpperExploration} started at root node $i$. For $t=1,\dots,T_i$, we denote by $\cW_t$ the set of layers which are explored during step $t$. We denote by $\cM^e_t = \cup_{s=1}^{t \wedge T_i} \cW_s$ the set of layers and by $\cN^e_{t} = \{v_1,\dots, v_{t \wedge T_i}\}$ the set of nodes explored up to time $t$.  We denote by $\cN^{d}_{t} = \{i\} \cup ( \cup_{k \in \cM^e_{t}} V(G_k))$ the set of nodes discovered up to time~$t$.

\begin{lemma}
\label{the:NumberLayersExplored}
The number of layers explored up to time $t$ is bounded by $\pr( \abs{\cM^e_t} > a t ) \le t e^{2M (n-t)^{-1} m - a}$ for all $a \ge 0$.
\end{lemma}
\begin{proof}
Consider an event $\cE_{t-1}^+ = \cE_{t-1}^+( A, B, C, v)$ that the exploration proceeds to step $t$, in the beginning of which the set of explored nodes equals $\cN^e_{t-1} = A$, the set of explored layers equals $\cM^e_{t-1} = B$, the set of discovered nodes equals $\cN^{d}_{t-1} = C$, and the currently explored node equals $v_t = v$, for some node sets $A \subset C$ with $v \in C \setminus A$ and some layer set $B$ such that the event $\cE_{t-1}^+$ has nonzero probability. The event $\cE_{t-1}^+$ is determined by the random graphs $\{G_k: k \in B\}$ and the indicator variables $\{1( V(G_k) \ni v ): v \in A, k \in [m]\}$. About the unexplored layers $G_k$, $k \in B^c$, the event $\cE_{t-1}^+$ reveals that $V(G_k) \subset A^c$, but nothing else.  Therefore, given $\cE_{t-1}^+$, the random graphs $\{G_k: k \in B^c\}$ are mutually independent and
\begin{equation}
 \label{eq:UniformPerStep}
 \text{$\law( V(G_k) \cond \cE_{t-1}^+)$ is uniform among the $X_k$-sets of $A^c$}.
\end{equation}
Given $\cE_{t-1}^+$, each unexplored layer $V(G_k)$ hence covers $v$ with probability $\frac{X_k}{n-(t-1)} \le \frac{M}{n-t}$, independently. Therefore, the number of layers explored during step $t$ satisfies
$\law( \abs{\cW_t} \cond \cE_{t-1}^+) \lest \Bin(m,\frac{M}{n-t})$, and a Chernoff inequality (Lemma~\ref{the:BinomialConcentration}) implies that $\pr(\abs{\cW_t} > a \cond \cE_{t-1}^+) \le e^{2 M (n-t)^{-1} m - a}$. Because the right side of the latter inequality does not depend on the choice of $A, B, C, v$, we conclude that
$
 \pr(\abs{\cW_t} > a \cond T_i \ge t)
 \le e^{2 M (n-t)^{-1}m - a}.
$
This implies the claim, because the inequality $\abs{\cM^e_t} \le t \max_{1 \le s \le t \wedge T_i} \abs{\cW_s} $ implies that
\[
 \pr( \abs{\cM^e_t} > a t )
 \wle \pr( \max_{1 \le s \le t \wedge T_i} \abs{\cW_s} > a )
 \wle \sum_{s=1}^t \pr( \abs{\cW_s} > a, \, T_i \ge s).
\]
\end{proof}


During an exploration step $t \le T_i$, a \emph{multi-overlap of type~1} occurs if one of the layers covering $v_t$ overlaps with previously explored layers in some other node besides $v_t$, and a \emph{multi-overlap of type~2} occurs if some of the layers covering $v_t$ overlap each other in more than one node. These events can be written as
\begin{align*}
 \cO_{1t} &\weq \{T_i \ge t\} \cap \{V'_k \cap \cN^{d}_{t-1} \ne \emptyset \ \text{for some $k \in \cW^+_t$}\}, \\
 \cO_{2t} &\weq \{T_i \ge t\} \cap \{V'_k \cap V'_\ell \ne \emptyset \ \text{for some distinct $k, \ell \in \cW^+_t$}\},
\end{align*}
where $V'_k = V(G_k) \setminus \{v_t\}$ and $\cW^+_t = \{k: V(G_k) \ni v_t\}$. We denote the occurrence of a multi-overlap by $\cO_t = \cO_{1t} \cup \cO_{2t}$, and we define $\cO_{\le t} = \cO_1 \cup \cdots \cup \cO_t$.

\begin{lemma}
\label{the:NoMultioverlap}
For any $n \ge 3$ and $1 \le \tau \le n/2$, the probability that a multi-overlap occurs during the first $\tau$ exploration steps is bounded by 
$
 \pr( \cO_{\le \tau} )
 \le c \tau^2 n^{-1} \log n,
$
where $c = e^{5M(1+m/n)}$.
\end{lemma}
\begin{proof}
Consider an event $\cE_{t-1}^+ =\cE_{t-1}^+(A,B,C,v)$ as in the proof of Lemma~\ref{the:NumberLayersExplored}. By \eqref{eq:UniformPerStep}, we know that given $\cE_{t-1}^+$, 
each layer $k \in B^c$ covers $v$ with probability $\frac{X_k}{n-(t-1)}$, and given $\cE_{t-1}^+ \cap \{V(G_k) \ni v\}$, the law of $V'_k = V(G_k) \setminus \{v\}$ with $k \in B^c$ is uniform among the $(X_k-1)$-sets of $(A \cup \{v_t\})^c$, and hence $V'_k$ overlaps $C$ with probability at most $\frac{(X_k-1)(\abs{C}-1)}{n-t}$. Because $\abs{B^c} \le n$ and $\abs{C} \le M \abs{B}$, the probability of a multi-overlap of type~1 is bounded by 
\[
 \pr(\cO_{1t} \cond \cE_{t-1}^+)
 \wle \sum_{k \in B^c} \frac{X_k}{n-(t-1)} \frac{(X_k-1)(\abs{C}-1)}{n-t}
 \wle n \frac{M^3 \abs{B}}{(n-t)^{2}}.
\]

Similarly, the $\cE_{t-1}^+$-conditional probability that two distinct layers $k,\ell \in B^c$ cover $v$ and overlap each other in some other node is bounded by $\frac{X_k}{n-(t-1)} \frac{X_\ell}{n-(t-1)} \frac{(X_k-1)(X_\ell-1)}{n-t} \le \frac{M^4}{(n-t)^3}$. Hence a multi-overlap of type~2 occurs with probability at most $\pr(\cO_{2t} \cond \cE_{t-1}^+) \le \binom{n}{2} \frac{M^4}{(n-t)^3}$. Hence for $t \le n/2$ and $\abs{B} \le at$ with $a,t \ge 1$,
\[
 \pr( \cO_t \cond \cE_{t-1}^+ )
 \wle 4 M^3 ( n^{-1} m + M n^{-2} m^2) a t n^{-1}.
\]
Because the right side above is valid whenever $\abs{B} \le at$, the above inequality also holds for $\cE_{t-1}^+$ replaced by the event that $\abs{\cM^e_{t-1}} \le at$ and $T_i \ge t$.  Lemma~\ref{the:NumberLayersExplored} now implies that
\begin{align*}
 \pr( \cO_t )
 &\weq \pr( \cO_t, \abs{\cM^e_{ t-1}} \le at, T_i \ge t ) + \pr( \cO_t, \abs{\cM^e_{ t-1}}> at, T_i \ge t ) \\
 &\wle \pr( \cO_t,  \abs{\cM^e_{ t-1}} \le at, T_i \ge t ) + \pr( \abs{\cM^e_{ t-1}} > at ) \\
 &\wle 4 (M^3 m/n + M^4 m^2/n^2) a t n^{-1} + e^{4 M m/n} t e^{- a}.
\end{align*}
Using $x \le 1+x \le e^x$ we find that $4 (M^3 m/n + M^4 m^2/n^2) = (2M)^2 (M m/n)(1+M m/n) \le e^{4M + 2M m/n}$. By plugging in $a = \log n$, it follows that
\begin{align*}
 \pr( \cO_t )
 \wle \left( e^{4M(1+m/n)} + e^{4 M (1+m/n)} \right) t n^{-1} \log n
 \wle e^{5M(1+m/n)} t n^{-1} \log n.
\end{align*}
Hence the claim follows by the union bound.
\end{proof}

\begin{lemma}
\label{the:RestrictedQueueUpper}
The probability that the restricted exploration discovers more than $\tau$ nodes is bounded by
$\pr( Q_\tau > 0 ) \le \rho_\tau(f_{\tau,n})$,
where the distribution $f_{\tau,n}$ is defined by \eqref{eq:UpperOffspring}.
\end{lemma}
\begin{proof}
The queue length process satisfies $Q_0 = 1$ and
\begin{equation}
 \label{eq:ExplorationQueueNew}
 Q_t \weq (Q_{t-1} - 1 + Z_t) 1(Q_{t-1}>0), \quad t =1,2,\dots,
\end{equation}
where $Z_t = \abs{\cZ_t}$ is the number of nodes added to the queue in step $t$.  Fix $1 \le t \le \tau$ and consider an event $\cE_{t-1}^+ = \cE_{t-1}^+(A,B,C,v)$ as in the proof of Lemma~\ref{the:NumberLayersExplored}. On this event,
\begin{equation}
 \label{eq:ExplorationQueueUpperNew}
 Z_t
 \wle \sum_{k \in B^c} 1( V(G_k) \ni v ) \deg_{\bar G_k}(v).
\end{equation}
By recalling \eqref{eq:UniformPerStep}, we know that conditionally on $\cE_{t-1}^+$, the random variables on the right side of \eqref{eq:ExplorationQueueUpperNew} are mutually independent, and such that $\law(1( V(G_k) \ni v ) \cond \cE_{t-1}^+) = \Ber(\frac{X_k}{n-(t-1)})$ and $\law(\deg_{\bar G_k}(v) \cond \cE_{t-1}^+) = \Bin^+(X_k-1,Y_k)$ for all $k \in B^c$. We conclude that
\[
 \law(Z_t \cond \cE_{t-1}^+)
 \wlest f_{t,n}
 \wlest f_{\tau,n}
 \quad \text{for all $t=1,\dots,\tau$}.
\] 
Because the above inequalities hold for all events $\cE_{t-1}^+ = \cE_{t-1}^+(A,B,C,v)$ of the above form, it also holds for the event $\{Q_{t-1} > 0\}$ that there is a node to explore at step $t$. Hence it follows that  $(Q_0, \dots, Q_\tau) \lest (Q'_0, \dots, Q'_\tau)$ where the right side is defined as in \eqref{eq:ExplorationQueueNew} but with $Z_1,Z_2,\dots$ replaced by independent $f_{\tau,n}$-distributed random integers $Z'_1, Z'_2, \dots$ The claim follows by noting that $\pr( Q'_\tau > 0 ) = \rho_\tau(f_{\tau,n})$ (see Appendix~\ref{sec:BranchingTrees}).
\end{proof}

\subsubsection{Proof of Proposition~\ref{the:SingleUpperBound}}

Let $Q_t = \abs{\cQ_t}$ be the exploration queue length in Algorithm~\ref{algo:UpperExploration} started at node $i$.  
We note that $Q_\tau > 0$ means that the restricted exploration discovers more than $t$ nodes of the component of $i$. Therefore $Q_\tau > 0$ implies $\abs{C_G(i)} > \tau$. The converse may not be true (see Figure~\ref{fig:ModifiedExploration}) because the restricted exploration may stop before discovering all nodes in the component of $i$. On the event $\cO_{\le \tau}^c$ that multi-overlaps do not occur up to time $\tau$, this cannot happen, and hence
$
 \{ Q_\tau > 0 \} \cap \cO_{\le \tau}^c
 = \{ \abs{C_G(i)} > \tau \} \cap \cO_{\le \tau}^c.
$
%
Therefore,
\begin{align*}
 \pr( \abs{C_G(i)} > \tau )
 &\weq \pr(Q_\tau > 0, \cO_{\le \tau}^c) + \pr(\abs{C_G(i)} > \tau, \cO_{\le \tau}) \\
 &\wle \pr(Q_\tau > 0)  + \pr( \cO_{\le \tau}).
\end{align*}
The claim follows by combining Lemma~\ref{the:NoMultioverlap} and Lemma~\ref{the:RestrictedQueueUpper}.
\qed

\subsection{Double upper bound for deterministic layer types}
To obtain an upper bound on the variance of the number of nodes contained in large components, 
we extend the analysis in Proposition~\ref{the:SingleUpperBound} to two restricted exploration processes run on the same graph instance. 

\begin{proposition}
\label{the:DoubleUpperBound}
For any $1 \le \tau \le n/4$, the components sizes of nodes $i \ne j$ are bounded by 
$ \pr( \abs{C_G(i)} > \tau, \abs{C_G(j)} > \tau)
 \le \rho_\tau(f_{2\tau,n})^2 + c \tau^2 n^{-1} \log n,
$
where $c = e^{9M(1+m/n)}$.
\end{proposition}

The proof of Proposition~\ref{the:DoubleUpperBound} requires Lemma~\ref{the:AvoidOne} and~\ref{the:DisjointExplorations}, outlined next.

%


\begin{lemma}
\label{the:AvoidOne}
The probability that Algorithm~\ref{algo:UpperExploration} started at $i$ discovers node $j \ne i$ during $\tau$ steps is bounded by
$
 \pr( \cN^{d}_{i,\tau} \ni j)
 \le 4 M^2 n^{-2} m \tau.
$
\end{lemma}
\begin{proof}
Fix $1 \le t \le \tau$ and consider an event $\cE_{t-1}^+ = \cE_{t-1}^+(A,B,C,v)$ as in the proof of Lemma~\ref{the:NumberLayersExplored}, for some node sets $A \subset C$ with $i \in C \setminus A$ and $C \not\ni j$, and some layer set $B$. By recalling \eqref{eq:UniformPerStep}, we know that for any $k \in B^c$, the $\cE_{t-1}^+$-conditional probability that $G_k$ covers $v$ and $j$ is bounded by $\frac{X_k}{n-(t-1)} \frac{X_k-1}{n-t} \le \frac{M^2}{(n-\tau)^2}$. Because $\abs{B^c} \le n$, the union bound implies
\[
 \pr( \cN^d_{i,t} \ni j \cond \cE_{t-1}^+)
 \wle m \frac{M^2}{(n-t)^2}
 \wle 4 M^2 n^{-2} m.
\]
Because the above inequality is valid whenever $C \not\ni j$, the above inequality also holds with $\cE_{t-1}^+$ replaced by the event  $\cE_{t-1}^{++} = \{\cN^{d}_{i,t-1} \not \ni j\} \cap \{T_i \ge t\}$. Hence the claim follows by noting that $\pr( \cN^{d}_{i,\tau} \ni j) = \sum_{t=1}^\tau \pr( \cN^{d}_{i,t} \ni j, \, \cE_{t-1}^{++}) \le \sum_{t=1}^\tau \pr( \cN^{d}_{i,t} \ni j \cond \cE_{t-1}^{++})$.
\end{proof}

\begin{lemma}
\label{the:DisjointExplorations}
For any $i\ne j$ and $1 \le \tau \le n/4$, the probability that explorations started at $i$ and $j$ overlap is bounded by 
$
 \pr( \cN^{d}_{i,\tau} \cap \cN^{d}_{j,\tau} \ne \emptyset)
 \le c \tau^2 n^{-1} \log n,
$
where $c = e^{7M(1+m/n)}$.
\end{lemma}
\begin{proof}
Consider an event $\cE_{i,\tau} = \cE_{i,\tau}(A_i, B_i, C_i)$ that after $\tau$ steps of exploration started from $i$, the set of explored nodes equals $\cN^e_{i,\tau} = A_i$, the set of explored layers equals $\cM^e_{i,\tau} = B_i$, and the set of discovered nodes equals $\cN^{d}_{i,\tau} = C_i$ for some node sets $A_i \subset C_i$ such that $C_i \not\ni j$ and $C_i \le M a \tau$, and some layer set $B_i$. Fix $1 \le t \le \tau$ and consider an event $\cE_{j,t-1}^+ = \cE_{j,t-1}^+(A_j, B_j, C_j, v)$ as in the proof of Lemma~\ref{the:NumberLayersExplored}, for some node sets $A_j \subset C_j$ with $v \in C_j \setminus A_j$ and $C_i \cap C_j = \emptyset$, and some layer set $B_j$. Then $\cN^{d}_{j,t-1}$ does not overlap $\cN^{d}_{i,\tau}$ on the event $\cF^+_{t-1} = \cE_{i,\tau} \cap \cE^+_{j,t-1}$.  We will next analyse the conditional probability that the same is true for $\cN^{d}_{j, t}$. We note that on the event $\cF^+_{t-1}$, the $j$-exploration at step $t$ only explores layers $k \in (B_i \cup B_j)^c$ because the layers in $B_i$ do not cover $v_j$, and the layers in $B_j$ have already been explored. We also observe that the event $\cF^+_{t-1}$ is determined by the random graphs $G_k$, $k \in B_i \cup B_j$, and the indicators $1(V(G_k) \ni a)$ for $k=1,\dots,n$ and $a \in A_i \cup A_j$. Hence given $\cF^+_{t-1}$, the layers $G_k$, $k \in (B_i \cup B_j)^c$ are mutually independent, and such that $V(G_k)$ is a uniformly random $X_k$-set in $(A_i \cup A_j)^c$. The probability that a layer $k \in (B_i \cup B_j)^c$ covers $v$ and overlaps with $C_i$, is at most
\[
 \frac{X_k}{n-\abs{A_i}-\abs{A_j}} \frac{(X_k-1)C_i}{n-\abs{A_i}-\abs{A_j}-1}
 \wle \frac{M^2 C_i}{(n - 2\tau)^{2}}.
\]
Hence, due to $\abs{(B_i \cup B_j)^c} \le n$, and $C_i \le M a\tau$, it follows that
\[
 \pr( \cN^{d}_{i,\tau} \cap \cN^{d}_{j,t} \ne \emptyset \cond \cF^+_{t-1} )
 \wle 4 M^3 a \tau n^{-2} m.
\]
Because the right side above does not depend on $A_i, B_i, C_i, A_j, B_j, C_j, v$, the above inequality remains valid also for 
$\cF^+_{t-1}$ replaced by $\cG^+_{t-1} = \{\cN^{d}_{i,\tau} \cap \cN^{d}_{j,t-1} = \emptyset\} \cap \{\abs{\cN^{d}_{i,\tau}} \le M a\tau \}$. Thus,
\begin{align*}
 \pr( \cN^{d}_{i,\tau} \cap \cN^{d}_{j,\tau} \ne \emptyset
    \, \cond \, \cN^{d}_{i,\tau} \cap \cN^{d}_{j, 0} = \emptyset, \abs{\cN^{d}_{i,  \tau}} \le M a \tau )
 &\weq \sum_{t=1}^\tau \pr( \cN^{d}_{i,\tau} \cap \cN^{d}_{j, t} \ne \emptyset
     \cond \cG^+_{t-1}) \\
 &\wle 4 M^3 a \tau^2 n^{-2} m.
\end{align*}

Hence noting that $\cN^{d}_{i,\tau} \cap \cN^{d}_{j,0} \ne \emptyset$ if and only if $\cN^{d}_{i,\tau} \not\ni j$, together with
$\abs{\cN^{d}_{i,\tau}} \le M \abs{\cM^e_{i,\tau}}$, applying Lemma~\ref{the:NumberLayersExplored} and Lemma~\ref{the:AvoidOne}, for $a, \tau \ge 1$,
\begin{align*}
 \pr( \cN^{d}_{i,\tau} \cap \cN^{d}_{j,\tau} \ne \emptyset )
 &\wle 4 M^3 a \tau^2 n^{-2} m + \pr( \cN^{d}_{i,\tau} \ni j ) + \pr( \abs{\cN^{d}_{i,\tau}} > M a \tau ) \\
 &\wle 4 M^3 a \tau^2 n^{-2} m + 4 M^2 n^{-2} m \tau + \tau e^{4M m/n - a}.
\end{align*}
Plugging in $a = \log \frac{n}{\tau}$ and noting that $a \le \log n$ and $\tau e^{-a} = \tau^2 n^{-1} \le \tau^2 n^{-1} \log n$,
 implies
\begin{align*}
 \pr( \cN^{d}_{i,\tau} \cap \cN^{d}_{j,\tau} \ne \emptyset )
 &\wle \left( 4 M^3 \frac{m}{n} + 4 M^2 \frac{m}{n} + e^{4M m/n} \right) \tau^2 n^{-1} \log n.
\end{align*}
The claim follows after noting that $4 M^3 \frac{m}{n} + 4 M^2 \frac{m}{n} \le 8 M^3 \frac{m}{n} = (2M)^3 \frac{m}{n} \le e^{6M(1+m/n)}$ implies the term on the right in parentheses is at most $e^{7M(1+m/n)}$.
\end{proof}

\begin{proof}[Proof of Proposition~\ref{the:DoubleUpperBound}]
Let $Q_{it}$ and $Q_{jt}$ be the exploration queue lengths of Algorithm~\ref{algo:UpperExploration} started at distinct nodes $i$ and $j$. We use the notations of Section~\ref{sec:PropertiesUpperExploration}.

Consider an event $\cE^+_{i\tau} = \{\cN^e_{i\tau} = A, \cM^e_{i\tau} = B, Q_{i\tau} > 0\}$ for some node set $A \not\ni j$ of size $\tau$, and some layer set $B$. 
Let $Q'_{jt}$ be the exploration queue of a modified exploration obtained by running Algorithm~\ref{algo:UpperExploration}
started from $j$ with a reduced set of input layers $\{G_k: k \in B^c\}$.
Then $(Q'_{j0},\dots,Q'_{j\tau}) = (Q_{0\tau},\dots,Q_{j\tau})$ on the event $\cM^e_{j\tau} \cap B = \emptyset$.
Hence
\begin{equation}
\label{eq:Decoupling1}
\begin{aligned}
 \pr( Q_{j,\tau} > 0, \, \cM^e_{i\tau} \cap \cM^e_{j, \tau} = \emptyset \cond \cE^+_{i\tau})
 &\weq \pr( Q'_{j\tau} > 0, \, \cM^e_{j\tau} \cap B = \emptyset \cond \cE^+_{i\tau}) \\
 &\wle \pr( Q'_{j\tau} > 0 \cond \cE^+_{i\tau}).
\end{aligned}
\end{equation}

The event $\cE^+_{i\tau}$ is determined by the random graphs $\{G_k: k \in B\}$ and the indicators $\{1(V(G_k) \ni a): k \in [m], a \in A\}$. Hence the $\cE^+_{i\tau}$-conditional distribution of $G_k$, $k \in B^c$, is such that these layers are mutually independent and $V(G_k)$ is a uniformly random $X_k$-set in $A^c$. Hence the $\cE^+_{i\tau}$-conditional law of the $Q'_{jt}$-exploration process is the same as the law of the exploration process $Q''_{jt}$ obtained by running Algorithm~\ref{algo:UpperExploration} started at $j$ for a model instance with node set $A^c$ and layer set $\{G_k: k \in B^c\}$. Hence
\begin{equation}
 \label{eq:Decoupling2}
 \pr( Q'_{j\tau} > 0 \cond \cE^+_{i\tau})
 \weq \pr( Q''_{j\tau} > 0).
\end{equation}

Let $Q'''_{jt}$ be an exploration queue of Algorithm~\ref{algo:UpperExploration} started at $j$ for a model instance with a full layer set $\{G_k: k \in [m]\}$ and node set $A^c \ni j$ of size $n - \tau$. Then $Q''_{jt} > 0$ implies $Q'''_{jt} > 0$ under a natural coupling, and we conclude with the help of \eqref{eq:Decoupling1} and \eqref{eq:Decoupling2} that
\[
 \pr( Q_{j,\tau} > 0, \, \cM^e_{i,\tau} \cap \cM^e_{j\tau} = \emptyset, \, \cE^+_{i\tau}) 
 \wle \pr( \cE^+_{i\tau}) \, \pr( Q'''_{j\tau} > 0).
\]
Because the probability on the right does not depend on the choice of $A,B$,
\[
 \pr( Q_{i\tau} > 0, \, Q_{j\tau} > 0, \, \cM^e_{i\tau} \cap \cM^e_{j\tau} = \emptyset ) 
 \wle \pr( Q_{i\tau} > 0) \, \pr(Q'''_{j\tau} > 0).
\]
By Lemma~\ref{the:RestrictedQueueUpper}, we see that $\pr( Q_{i\tau} > 0) \le \rho_\tau(f_{\tau,n})$. By applying the same lemma again for a model instance with the full layer set $\{G_k: k \in [m]\}$ and a node set of size $n-\tau$, we find that $\pr(Q'''_{j\tau} > 0) \le \rho_\tau(f'''_{\tau,n})$ where $f'''_{\tau,n}$ is defined as in \eqref{eq:UpperOffspring} but with $n$ replaced by $n-\tau$. Now we note that $f'''_{\tau,n} = f_{2\tau,n}$, and that $f_{\tau,n} \lest f_{2\tau,n}$ implies $\rho_\tau(f_{\tau,n}) \le \rho_\tau(f_{2\tau,n})$.
Hence
\begin{align*}
 \pr( Q_{i\tau} > 0, Q_{j\tau} > 0)
 &\wle \rho_\tau(f_{2\tau,n})^2 + \pr( \cM^e_{i\tau} \cap \cM^e_{j\tau} \ne \emptyset ).
\end{align*}
Because the indicators of $\{ \abs{C_G(i)} > \tau \}$ and $\{ Q_{i\tau} > 0 \}$ coincide on the event $\cO_{i, \le \tau}^c$, and the same is true for $j$, we find that
\begin{align*}
 \pr( \abs{C_G(i)} > \tau, \, \abs{C_G(j)} > \tau)
 &\wle \pr( Q_{i\tau} > 0, \, Q_{j\tau} > 0) +  \pr(\cO_{i,\le \tau}) + \pr(\cO_{j,\le \tau}) \\
 &\wle \rho_\tau(f_{2\tau,n})^2 + \pr( \cM^e_{i\tau} \cap \cM^e_{j\tau} \ne \emptyset ) + 2 \pr(\cO_{i,\le \tau}).
\end{align*}
The claim follows due to $\pr( \cM^e_{i\tau} \cap \cM^e_{j\tau} \ne \emptyset ) \le \pr( \cN^{d}_{i\tau} \cap \cN^{d}_{j\tau} \ne \emptyset)$ and Lemmas~\ref{the:NoMultioverlap} and~\ref{the:DisjointExplorations}.
\end{proof}

\subsection{Quantitative lower bound for deterministic layer types}

Proving a lower bound is more complicated than an upper bound, because we need to verify that the types of unexplored layers remain balanced during the exploration. We start by analysing the case with nonrandom layer types in a finite set in Proposition~\ref{the:QuantitativeLowerBound}. The proof is based on analysing a balanced exploration process in Algorithm~\ref{algo:LowerExploration} which uses a randomised selection of disjoint layers in Algorithm~\ref{algo:SetSelection} as a subroutine.

\begin{proposition}
\label{the:QuantitativeLowerBound}
Fix a finite set $A \subset \Z_+ \times [0,1]$, integers $1 \le M, \tau, \nu \le n$, and a number $\delta \in (0,1)$. Assume that $2 M^2 \abs{A} \frac{\nu \tau} {n} \le \delta$, $\tau \le \frac12 n$, and
$x \le M$ for all $(x,y) \in A$. Then
\begin{equation}
 \label{eq:QuantitativeLowerBound}
 \pr( \abs{C_G(i)} > \tau )
 \wge \rho_\tau(f_{\delta, \tau, \nu}) - \abs{A} \tau e^{4M m/n - \nu},
\end{equation}
where
\begin{equation}
 \label{eq:LowerOffspring}
 f_{\delta, \tau, \nu}
 \weq \law \Big( \sum_{(x,y) \in A} \sum_{k=1}^{m_{xy,\tau-1}} B_{xy}(k) T_{xy}(k) \Big),
\end{equation}
and the random variables on the right are mutually independent and such that $\law(B_{xy}(k)) = \Ber( (1-\delta) \frac{x}{n})$, $\law(T_{xy}(k)) = \Bin^+(x-1,y)$, and $m_{xy,\tau} = (m_{xy} - \tau\nu)_+$ where $m_{xy}$ is the number of layers of type $(x,y)$.
\end{proposition}



\begin{algorithm}[h]
\small
\DontPrintSemicolon
\KwInput{Layers $G_1,\dots,G_m$, root node $i$, parameters $\nu \in \Z_+$, $\delta \in (0,1)$
}
\KwOutput{Subset of $G$-component of $i$.}
~\\

State variables: \\
$\cQ_t$ = Set of nodes in the exploration queue after step $t$ \\
$\cN_{t}$ = Set of discovered nodes after step $t$ \\
$\cM_{t}$ = Set of available layers after step $t$ \\[1ex]

Initialise: Put node $i$ into the queue and declare $i$ discovered; declare all layers available; initialise state variables as: $\cQ_0 \leftarrow \{i\}$, $\cN_0 \leftarrow \{i\}$, $\cM_{0} \leftarrow \{1,\dots,m\}$; and set $t \leftarrow 0$. \\[1ex]

\While{$\cQ_t \ne \emptyset$}
{
Set $t \leftarrow t+1$ and select node $v_t \leftarrow \min \cQ_{t-1}$ for exploration \\[.5ex]

Declare the layers in $\cW^+_t \leftarrow \{k \in \cM_{t-1} : V(G_k) \ni v_t\}$ and the nodes in $\cZ_t^+ \leftarrow \cup_{k \in \cW^+_t} (V(G_k) \setminus \{v_t\})$ as discovered\\[.5ex]

Extract a disjoint subcollection of discovered layers $\cW^+_t$ by computing $\cW_t \leftarrow $ Output of Algorithm~\ref{algo:SetSelection} with ground set $\{v_1,\dots, v_{t-1}\}^c$, input sets $\{ V(G_k) \setminus \{v_t\} : k \in \cW^+_t \}$, taboo set $\cN_{t-1}$, parameter $\alpha_t = (1-\delta)(1-\frac{t-1}{n})$ \\[.5ex]

Explore the selected layers and determine the node set $\cZ_t \leftarrow \cup_{k \in \cW_t} N_{v_t}(\bar G_k)$,
where $\bar G_k$ is the transitive closure of $G_k$ \\[.5ex]

Update the exploration queue by $\cQ_t \leftarrow (\cQ_{t-1} \setminus \{v_t\}) \cup \cZ_t$ and the set of discovered nodes by $\cN_t \leftarrow \cN_{t-1} \cup \cZ_t^+$ \\[.5ex]

Layer balancing:
$\cM_t \leftarrow \cup_{(x,y) \in A} \cM_{xy,t}$ where $\cM_{xy,t}$ is a uniformly random subset of $\cW^u_{xy,t} = \{k \in \cM_{t-1} \setminus \cW^+_t: X_k=x, Y_k=y\}$ of size $\abs{\cW^u_{xy,t}} \wedge (m_{xy} - \nu t)_+$

}
Output: $\{v_1,\dots, v_t\}$

\caption{Balanced exploration.}
\label{algo:LowerExploration}
\end{algorithm}

%

\begin{algorithm}[H]
\small
\DontPrintSemicolon
\KwInput{List of subsets $(V_1,\dots,V_m)$ of a ground set $V$, taboo set $H_0 \subset V$,
parameter $\alpha \in (0,1)$}
\KwOutput{Random index set $K \subset \{1,\dots,m\}$}
~\\
Initialise $K \leftarrow \emptyset$ and $H \leftarrow H_0$ \\
\For{$k=1,\dots,m$}
{
$U_k \leftarrow $ uniform random number in $(0,1)$ \\ 
\If{$V_k \cap H = \emptyset$ and $U_k \le \alpha \binom{\abs{V}-\abs{H}}{\abs{V_k}}^{-1} \binom{\abs{V}}{\abs{V_k}}$}
{
Add the index $k$ to $K$ \\
Add the elements of $V_k$ to $H$
}
}
Output $K$
\caption{Extracting disjoint sets.}
\label{algo:SetSelection}
\end{algorithm}

\begin{lemma}
\label{the:SetSelection}
Let $H_0 \subset V$ be nonrandom sets. Let $V_1,\dots,V_m$ be independent uniformly random subsets of $V$ with nonrandom sizes $x_1,\dots, x_m$, and assume that $\abs{H_0} + \norm{x}_1 \le \abs{V}$ and $0 \le \alpha \le \left(1 - \frac{\abs{H_0}+\norm{x}_1}{\abs{V}} \right)^{\supnorm{x}}$ where $\norm{x}_1 = \sum_{k=1}^m x_k$ and $\norm{x}_\infty = \max_{1 \le k \le m} x_k$. Then the indicator variables $B_k = 1(k \in K)$ characterising the output of Algorithm~\ref{algo:SetSelection} are mutually independent and $\Ber(\alpha)$-distributed, and the sets $\{V_k: k \in K\}$ are mutually disjoint and disjoint from $H_0$ almost surely.
\end{lemma}
\begin{proof}
Denote by $H_k$ the state of $H$ after finishing round $k$ of Algorithm~\ref{algo:SetSelection}.
Then $H_k$ equals the union of $H_0$ and the sets $V_j$ admitted during rounds $j \le k$, and 
the if-statement guarantees that a set $V_k$ is admitted to $K$ only if it is disjoint from $H_{k-1}$. Hence the family $\{V_k: k \in K\} = \{V_k: B_k = 1\}$ is surely disjoint and disjoint from $H_0$. To investigate the joint distribution of $B_k = 1_K(k)$, $k=1,\dots,m$, denote by 
denote by $p(h,x) = \binom{\abs{V}-h}{x}\binom{\abs{V}}{x}^{-1}$ the probability that a random $x$-set in $V$ does not overlap a particular $h$-set of $V$.
%
%
Let $\cF_k$ be the sigma-algebra generated by $\{(U_j, V_j): j \le k\}$. Then $B_k,H_k$ are $\cF_k$-measurable, $V_k$ is independent of $\cF_{k-1}$, and $U_k$ is independent of $(\cF_{k-1}, V_k)$.
Hence,
\begin{equation}
 \label{eq:SetSelection}
\begin{aligned}
 \pr(V_k \cap H_{k-1} = \emptyset \cond \cF_{k-1})
 &\weq p(\abs{H_{k-1}}, x_k), \\
 \pr(U_k \le \tfrac{\alpha}{p(\abs{H_{k-1}}, \abs{V_k})} \cond \cF_{k-1})
 &\weq \tfrac{\alpha}{p(\abs{H_{k-1}}, \abs{V_k})} \wedge 1.
\end{aligned}
\end{equation}
A basic computation shows that $p(h, x) = \prod_{r=0}^{x-1} \left(1 - \frac{h}{\abs{V}-r} \right) \ge \left(1 - \frac{h}{\abs{V}-x} \right)^x$, so that $p(\abs{H_{k-1}}, x_k) \ge p(h_0 + \norm{x}_1 - x_k, x_k)$ with $h_0 = \abs{H_0}$ implies
\[
 p(\abs{H_{k-1}}, x_k)
 \wge \left(1 - \frac{h_0 + \norm{x}_1 - x_k}{\abs{V} - x_k} \right)^{x_k}
 \wge \left(1 - \frac{h_0+\norm{x}_1}{\abs{V}} \right)^{\supnorm{x}}
 \wge \alpha.
\]
Hence we may ignore the truncation by one in \eqref{eq:SetSelection}, and it follows that $\pr(B_k=1 \cond \cF_{k-1}) = \alpha$. This implies that $\pr(B_k=1) = \alpha$, and that $B_k$ is independent of $\cF_{k-1}$. Especially, $B_k$ is independent of $(B_1,\dots,B_{k-1})$, so we conclude that $B_1,\dots, B_m$ are mutually independent.
\end{proof}


\subsubsection{Proof of Proposition~\ref{the:QuantitativeLowerBound}}

It suffices to find a lower bound for the exploration queue length $Q_t = \abs{\cQ_t}$ in Algorithm~\ref{algo:LowerExploration}.  This is because Algorithm~\ref{algo:LowerExploration} started at node $i$ discovers a subset of $\abs{C_G(i)}$, and hence $\pr( \abs{C_G(i)} > \tau ) \ge \pr( Q_\tau > 0 )$. Denote by $\cM_{xy,t}$ the of available $xy$-layers, and recall that $m_{xy,t} = (m_{xy} - \nu t)_+$. Denote by $T_i$ the number of steps completed by the algorithm. The queue length obeys the recursion $Q_t = (Q_{t-1} - 1 + Z_t) 1(Q_{t-1}>0)$ where $Z_t = \abs{\cZ_t}$.  Algorithm~\ref{algo:SetSelection} guarantees that the node sets $V(G_k) \setminus \{v_t\}$ of the explored layers $k \in \cW_t$ are mutually disjoint and do not overlap any previously explored layers. Hence
\begin{equation}
 \label{eq:LowerBoundDiscovered}
 Z_t \weq \sum_{(x,y) \in A} \sum_{k \in \cM_{xy,t-1}} B_{1xyt}(k) B_{2xyt}(k) \, T_{xyt}(k),
\end{equation}
where $B_{1xyt}(k) = 1(k \in \cW^+_{xyt})$, $B_{2xyt}(k) = 1(k \in \cW_{xyt})$, and $T_{xyt}(k) = \abs{N_{v_t}(\bar G_k)}$ equals the number of neighbours of node $v_t$ in the transitive closure of $G_k$.

We will compare the queue length process to a random walk defined recursively by $Q'_0=1$ and $Q'_t = (Q'_{t-1}-1+ Z'_t) 1(Q'_{t-1}>0)$, where
\[
 Z'_t
 \weq \sum_{(x,y) \in A} \sum_{k=1}^{m_{xy,t-1}} B'_{1xyt}(k) B'_{2xyt}(k) T'_{xyt}(k),
\]
and where the random variables appearing on the right are mutually independent and such that
$\law(B'_{1xyt}(k)) = \Ber( \frac{x}{n-(t-1)} )$, $\law(B'_{2xyt}(k)) = \Ber( \alpha_t )$ with $\alpha_t = (1-\delta)(1-\frac{t-1}{n})$, and $\law(T'_{xyt}(k)) = \Bin^+(x-1,y)$. A key part of the proof is to show that
\begin{equation}
 \label{eq:LowerBoundKey}
 \pr( Q_t = r, \cA_{\le t} )
 \weq \pr( Q'_t = r, \cA'_{\le t} )
\end{equation}
for all $r > 0$ and $t \le \tau$, where $\cA_{\le t} = \cA_1 \cap \cdots \cap \cA_t$ and $\cA'_{\le t} = \cA'_1 \cap \cdots \cap \cA'_t$ are defined by
\[
 \cA_t
 = \left\{ T_i \ge t, \ \max_{(x,y) \in A} \abs{\cW^+_{xyt}} \le \nu \right\}
 \quad \text{and} \quad
 \cA'_t
 = \left\{ \max_{(x,y) \in A} \sum_{k=1}^{m_{xy,t}} B'_{1xyt}(k) \le \nu \right\}.
\]

To verify \eqref{eq:LowerBoundKey}, consider an event $\cE_{t-1}$ that $Q_{t-1}=q$, $\cA_{\le t-1}$ is valid, the set of previously explored nodes equals $\hat\cN^e_{\le t-1}$, the set of previously explored layers equals $\hat\cM^e_{\le t-1}$, the set of available $xy$-layers after $t-1$ steps is $\hat\cM_{xy,t-1}$, and node $v_t$ is explored on step $t$. This event is determined by the graphs $G_k$, $k \in \hat\cM^e_{\le t-1}$, the indicator variables $1(V(G_k) \ni v)$ for $k = [m]$ and $v \in \hat\cN^e_{\le t-1}$, and the random variables used in the randomised algorithm during steps $s \le t-1$. On the event $\cE_{t-1}$, the number of available $xy$-layers in the beginning of step $t$ equals $m_{xy,t-1}$, and the only thing known about the available layers is that they do not contain any of the explored nodes $\hat\cN^e_{\le t-1}$. Conditionally on $\cE_{t-1}$, the graphs $G_k$, $k \in \hat\cM_{t-1}$, are hence mutually independent and such that $V(G_k)$ is a uniform $X_k$-set in $[n] \setminus \hat\cN^e_{\le t-1}$.

Conditionally on the event $\cE_{t-1}$, each available layer $k \in \hat\cM_{xy,t-1}$ is discovered with probability $\frac{x}{n-(t-1)}$, independently of other available layers. Hence the indicators $B_{1xyt}(k)$ in \eqref{eq:LowerBoundDiscovered} are  independent and $\Ber(\frac{x}{n-(t-1)})$-distributed given $\cE_{t-1}$. Let $\cE_{t-1}^+ = \cE_{t-1} \cap \{\cW^+_{xyt} = \hat \cW^+_{xyt}, (x,y) \in A\}$ for some layer sets $\abs{\hat \cW^+_{xyt}} \le \nu$ such that $\cE_{t-1}^+$ has nonzero probability.  On the event $\cE_{t-1}^+$, the number of nodes discovered before step $t$ is bounded by $\abs{\cN_{t-1}} \le M \abs{A} \nu (t-1)$, and $\sum_{k \in \cW^+_t} \abs{V(G_k) \setminus \{v_t\} } \le M \abs{\cW^+_t} \le M \abs{A} \nu$, and it follows that
\begin{align*}
 \left( 1 - \frac{\abs{\cN_{t-1}} + \sum_{k \in \cW^+_t} \abs{V(G_k) \setminus \{v_t\} }} {n-(t-1)} \right)^M
 &\wge \left(1 - \frac{M \abs{A} \nu t} {n-(t-1)}\right)^M,
\end{align*}
where the right side is at least $\left(1 - 2 M \abs{A} \frac{\nu t}{n}\right)^M \ge 1 - 2 M^2 \abs{A} \frac{\nu t} {n} \ge \alpha_t$ due to $2 M^2 \abs{A} \frac{\nu t} {n} \le \delta$ and $\alpha_t \le 1-\delta$.
By Lemma~\ref{the:SetSelection}, we find that the indicators $B_{2xyt}(k)$ in \eqref{eq:LowerBoundDiscovered} are mutually independent and $\Ber(\alpha_t)$-distributed given $\cE_{t-1}^+$. Furthermore, also the random integers $T_{xyt}(k)$ in \eqref{eq:LowerBoundDiscovered} are mutually independent, independent of the indicators $B_{2xyt}(k)$, and such that $\law(T_{xyt}(k) \cond \cE_{t-1}^+) = \Bin^+(x-1,y)$. These observations allow us to conclude that
$
 \law( Z_t \cond \cE_{t-1}, \cA_t )
 = \law( Z'_t \cond \cA'_t )
$
and
$
 \pr( \cA_t \cond \cE_{t-1} )
 = \pr( \cA'_t ).
$
Hence for any $r > 0$, 
\begin{align*}
 \pr( Q_t = r, \cA_t \cond \cE_{t-1})
 &\weq \pr( \cA_t \cond \cE_{t-1}) \, \pr( q-1+Z_t = r \cond \cE_{t-1}, \cA_t ) \\
 &\weq \pr( \cA'_t ) \, \pr( q-1+Z'_t = r \cond \cA'_t) \\
 &\weq \pr( q-1+Z'_t = r, \cA'_t ) \\
 &\weq \pr( Q'_t = r, \cA'_t \cond Q'_{t-1} = q, \cA'_{\le t-1}).
\end{align*}
By multiplying both sides above by $\pr(\cE_{t-1})$ and summing over all $\cE_{t-1}$ which are subsets of the event $\{Q_{t-1}=q\} \cap \cA_{\le t-1}$, it follows that
\begin{align*}
 \pr( Q_t = r, \cA_t \cond Q_{t-1}=q, \cA_{\le t-1})
 &\weq \pr( Q'_t = r, \cA'_t \cond Q'_{t-1} = q, \cA'_{\le t-1}).
\end{align*}
Because the above equality holds for all $q, r > 0$, a simple induction argument, based on
\[
 \pr(Q_t = r, \cA_{\le t})
 \weq \sum_{q > 0} \pr(Q_{t-1} = q, \cA_{\le t-1}) \pr( Q_t = r, \cA_t \cond Q_{t-1}=q, \cA_{\le t-1}),
\]
confirms \eqref{eq:LowerBoundKey}. With the help of \eqref{eq:LowerBoundKey}, we now find that
\[
 \pr( Q_\tau > 0 )
 \wge \pr( Q_\tau > 0, \cA_{\le \tau} )
 \weq \pr( Q'_\tau > 0, \cA'_{\le \tau} ) 
 \wge \pr( Q'_\tau > 0 ) - \pr( (\cA'_{\le \tau})^c ).
\]
Denote $M'_{1xyt} = \sum_{k=1}^{m_{xy,t-1}} B'_{1xyt}(k)$ and observe that $m_{xy,t-1} \le m$ and $\law(B'_{1xyt}(k)) \lest \Ber(2 \frac{M}{n})$ imply that $\law(M'_{1xyt}) \lest \Bin(m, \frac{2 M}{n})$ for $t \le n/2$. The moment generating function of the latter distribution, evaluated at one, is bounded by $(1+\frac{2 M}{n}(e-1))^n \le e^{2M(e-1) m/n} \le e^{4M m/n}$. Therefore, Markov's inequality for $e^{M'_{1xyt}}$ implies $\pr( M'_{1xyt} > \nu ) \le e^{4M m/n - \nu}$, and
\[
 \pr( (\cA'_{\le \tau})^c )
 \wle \sum_{t=1}^\tau \pr( (\cA'_t)^c )
 \wle \abs{A} \tau e^{4M m/n - \nu}.
\]

Finally, observe that the distribution of $Z'_t$ coincides with $f_{\delta, t, \nu}$ defined by \eqref{eq:LowerOffspring}. Moreover, $f_{\delta, t, \nu} \gest f_{\delta, \tau, \nu}$ for all $t=1,\dots, \tau$.  Therefore, $\pr( Q'_\tau > 0 ) \ge \pr( Q''_\tau > 0 )$ where $(Q''_0, \dots, Q''_\tau)$ is defined as before, but with $Z'_1,\dots, Z'_\tau$ replaced by mutually independent $f_{\delta, \tau, \nu}$-distributed random integers $Z''_1,\dots,Z''_\tau$. The claim follows by noting that $\pr( Q''_\tau > 0 ) = \rho_\tau( f_{\delta, \tau, \nu} )$.

\qed

\subsection{Component analysis for a finite layer type space}


\begin{lemma}
\label{the:MBGiant1}
Under the assumptions and notations of Theorem~\ref{the:Giant}, together with the extra assumption that the supports of $P$ and $(P_n)_{n \ge 1}$ are all contained in a finite set $A \subset \{0,\dots,M\} \cap [0,1]$, the component size of any particular node $i$ satisfies
\begin{alignat}{2}
 \label{eq:MBGiant1Constant}
 \pr \big( \abs{C_{G^{(n)}}(i)} > \tau \big) &\wto \rho_\tau(f^+)
 \qquad &\text{for any constant $\tau \ge 1$}, \\
 \label{eq:MBGiant1}
 \pr \big( \abs{C_{G^{(n)}}(i)} > \omega \big) &\wto \rho(f^+)
 \qquad &\text{for $1 \ll \omega \ll n \log^{-1} n$},
\end{alignat}
the relative frequencies of nodes with large components satisfy
\begin{alignat}{2}
 \label{eq:EmpLargeComponentFrequencyConstant}
 n^{-1} \abs{B_\tau(G^{(n)})} &\wprto \rho_\tau(f^+)
 \qquad &\text{for any constant $\tau \ge 1$}, \\
 \label{eq:EmpLargeComponentFrequency}
 n^{-1} \abs{B_\omega(G^{(n)})} &\wprto \rho(f^+)
 \qquad &\text{for $1 \ll \omega \ll n \log^{-1} n$},
\end{alignat}
and the largest component size in $G^{(n)}$ satisfies
\begin{equation}
 \label{eq:GiantFiniteLayerTypes}
 n^{-1} N_1(G^{(n)}) \wprto \rho(f^+).
\end{equation}
\end{lemma}
\begin{proof}
We start by making an additional assumption that all layer types are nonrandom.  The extension to random layer types is treated in the end.

(i) Upper bound for \eqref{eq:MBGiant1Constant}. Fix $1 \le \tau \le n/2$. Then by Proposition~\ref{the:SingleUpperBound},
\begin{equation}
 \label{eq:FiniteUpper1Constant}
 \pr( \abs{C_{G^{(n)}}(i)} > \tau)
 \wle \rho_\tau(f_{\tau,n}) + c \tau^2 n^{-1} \log n,
\end{equation}
where $c = e^{5M(1+m/n)}$, and $f_{\tau,n}$ is the distribution defined by \eqref{eq:UpperOffspring}. A natural coupling implies that $\abs{\rho_{\tau}(f_{\tau,n}) - \rho_{\tau}(f^+)} \le \tau \dtv(f_{\tau,n}, f^+)$. Hence by \eqref{eq:FiniteUpper1Constant} it follows that
\begin{equation}
 \label{eq:FiniteUpper2Constant}
 \pr( \abs{C_{G^{(n)}}(i)} > \tau)
 \wle \rho_\tau(f^+) + c \tau^2 n^{-1} \log n + \tau \dtv(f_{\tau,n}, f^+).
\end{equation}

Define $\tilde f_{\tau,n}$ using the same formula~\eqref{eq:UpperOffspring}, but with the $\Ber(\frac{X_k}{n-\tau})$-distributed random variables $B_k$ replaced by $\Poi(\frac{X_k}{n-\tau})$-distributed random variables $\tilde B_k$.  Because $\dtv(\Ber(p), \Poi(p)) = p(1-e^{-p}) \le p^2$ for all $0 \le p \le 1$, a natural coupling implies that
\[
 \dtv(f_{\tau,n}, \tilde f_{\tau,n})
 \wle \sum_{k=1}^m \left( \frac{X_k}{n-\tau} \right)^2
 \wle 4 \frac{M^2}{n^2} m.
\]
Then we see by Lemma~\ref{the:CPoiSum} that $\tilde f_{\tau,n} = \CPoi(\frac{m}{n-\tau} (P_n)_{10}, g_n)$ where $g_n = \Bin^+_{10}(P_n)$.  Lemma~\ref{the:CPoiPerturbation} implies that
$
\dtv(\tilde f_{\tau,n}, f^+)
 \le \left| \frac{m}{n-\tau} (P_n)_{10} - \mu (P)_{10} \right| + \mu (P)_{10} \dtv(g_n, g),
$
and we conclude that
\[
 \dtv(f_{\tau,n}, f^+)
 \wle \left| \frac{m}{n-\tau} (P_n)_{10} - \mu (P)_{10} \right| + \mu (P)_{10} \dtv(g_n, g) + 4 \frac{M^2}{n^2} m.
\]
Because $g_n \weakto g$ by Lemma~\ref{the:MixedBiasedERConvergence}, it follows that $\dtv(f_{\tau,n}, f^+) \to 0$ as $n \to \infty$. Hence by \eqref{eq:FiniteUpper2Constant} it follows that $\limsup_{n \to \infty} \pr( \abs{C_{G^{(n)}}(i)} > \tau) \le \rho_\tau(f^+)$.

(ii) Upper bound for \eqref{eq:MBGiant1}. Fix $\epsilon > 0$ and select a large enough $t$ such that $\rho_t(f^+) \le \rho(f^+) + \epsilon$. Define $\tau_n = \floor{\omega_n \wedge n^{1/3}}$. Then $\tau_n \ge t$ for large values of $n$, and by \eqref{eq:FiniteUpper1Constant},
\begin{align*}
 \pr( \abs{C_{G^{(n)}}(i)} > \tau_n)
 \wle \rho_{\tau_n}(f_{\tau_n,n}) + c \tau_n^2 n^{-1} \log n
 &\wle \rho_t(f_{\tau_n,n}) + c n^{-1/3} \log n.
\end{align*}
A natural coupling implies that $\abs{\rho_t(f_{\tau_n,n}) - \rho_{t}(f^+)} \le t \dtv(f_{\tau_n,n},f^+)$. Hence it follows that
\[
 \pr( \abs{C_{G^{(n)}}(i)} > \omega_n)
 \wle \pr( \abs{C_{G^{(n)}}(i)} > \tau_n)
 \wle \rho(f^+) + c n^{-1/3} \log n + t \dtv(f_{\tau_n,n},f^+) + \epsilon.
\]
The upper bound analysis of \eqref{eq:MBGiant1Constant} shows that $\dtv(f_{\tau,n}, f^+) \to 0$ also for $\tau = \tau_n \gg 1$.  Hence we conclude that $\limsup_{n \to \infty} \pr( \abs{C_{G^{(n)}}(i)} > \omega_n) \le \rho(f^+)$.

(iii) Lower bound for \eqref{eq:MBGiant1Constant}. Fix $\epsilon > 0$. To avoid trivialities we assume that $(P)_{10} > 0$, in which case $(P_n)_{10} > 0$ for all large values of $n$. Define $f_\delta = \CPoi((1-\delta)\lambda, g)$ with $\lambda = \mu (P)_{10}$.  Lemma~\ref{the:CPoiPerturbation} then implies that $f_\delta \weakto f^+$ as $\delta \to 0$. Hence by Lemma~\ref{the:GWLarge} we may choose a small $\delta \in (0,1)$ such that $\rho_\tau(f_\delta) \ge \rho_\tau(f^+) - \epsilon$. Define $\nu_n = \ceil{2 \log n}$. Then $2 M^2 \abs{A} \frac{\tau \nu_n} {n} \le \delta$ for large values of $n$, and Lemma~\ref{the:QuantitativeLowerBound} implies, recalling that $\abs{\rho_\tau(f^{(n)}_{\delta, \tau, \nu_n}) - \rho_\tau(f_\delta)} \le \tau \dtv(f^{(n)}_{\delta, \tau, \nu_n}, f_\delta)$, 
\begin{equation}
 \label{eq:CPoiBinBoundKey}
 \begin{aligned}
 \pr( \abs{C_{G^{(n)}}(i)} > \tau )
 &\wge \rho_\tau(f^{(n)}_{\delta, \tau, \nu_n}) - \abs{A} e^{4M m/n} \tau n^{-2} \\
 &\wge \rho_\tau(f_\delta) - \tau \dtv(f^{(n)}_{\delta, \tau, \nu_n}, f_\delta) - \abs{A} e^{4M m/n} \tau n^{-2} \\
 &\wge \rho_\tau(f^+) - \epsilon - \tau \dtv(f^{(n)}_{\delta, \tau, \nu_n}, f_\delta) - \abs{A} e^{4M m/n} \tau n^{-2} \\
 \end{aligned}
\end{equation}
where $f^{(n)}_{\delta, \tau, \nu_n}$ is the distribution defined in \eqref{eq:LowerOffspring}. Hence it suffices to verify that $\dtv(f^{(n)}_{\delta, \tau, \nu_n}, f_\delta) \to 0$. To do this, define modifications of $f^{(n)}_{\delta, \tau, \nu_n}$ by
\[
 f^{(n)}_{\delta} = \law\Big( \sum_{(x,y) \in A} \sum_{k=1}^{m_{xy}} B_{xy}(k) T_{xy}(k) \Big),
 \quad
 \tilde f^{(n)}_{\delta} = \law\Big( \sum_{(x,y) \in A} \sum_{k=1}^{m_{xy}} \tilde B_{xy}(k) T_{xy}(k) \Big),
\]
where the random variables are mutually independent and such that $\law(B_{xy}(k)) = \Ber( (1-\delta) \frac{x}{n})$, $\law(\tilde B_{xy}(k)) = \Poi( (1-\delta) \frac{x}{n})$, and $\law(T_{xy}(k)) = \Bin^+(x-1,y)$. A natural coupling implies
\begin{equation}
 \label{eq:CPoiBinBound1}
 \dtv(f^{(n)}_{\delta, \tau, \nu_n}, f^{(n)}_{\delta})
 \wle \sum_{(x,y) \in A} (m_{xy} - m_{xy, \tau-1}) (1-\delta) \frac{x}{n}
 \wle M \abs{A} \frac{\tau \nu_n}{n}.
\end{equation}
Because $\dtv(\Ber(p), \Poi(p)) = p(1-e^{-p}) \le p^2$ for all $0 \le p \le 1$, and $\sum_{(x,y) \in A} \sum_{k=1}^{m_{xy}} \le m$, a natural coupling implies that
\begin{equation}
 \label{eq:CPoiBinBound2}
 \dtv(f^{(n)}_{\delta}, \tilde f^{(n)}_{\delta})
 \wle \sum_{(x,y) \in A} \sum_{k=1}^{m_{xy}} \left( (1-\delta) \frac{x}{n} \right)^2
 \wle \frac{M^2}{n^2} m.
\end{equation}
Now let us observe that $\law(\sum_{k=1}^{m_{xy}} \tilde B_{xy}(k) T_{xy}(k)) = \CPoi((1-\delta) m_{xy} \frac{x}{n}, \Bin^+(x-1,y))$, so by Lemma~\ref{the:CPoiSum} we see that $\tilde f^{(n)}_{\delta} = \CPoi((1-\delta) \lambda_n, g_n)$, where
$\lambda_n = \frac{m}{n} (P_n)_{10}$, and $g_n = \Bin^+_{10}(P_n)$. Then Lemma~\ref{the:CPoiPerturbation} implies
$
 \dtv(\tilde f^{(n)}_{\delta}, f_\delta)
 \le \abs{\lambda_n - \lambda} + \lambda \dtv(g_n, g),
$
and combining this with \eqref{eq:CPoiBinBound1}--\eqref{eq:CPoiBinBound2} shows that
\[
 \dtv(f^{(n)}_{\delta, \tau, \nu}, f_{\delta})
 \wle M \abs{A} \frac{\tau \nu_n}{n} + \frac{M^2}{n^2} m + \abs{\lambda_n - \lambda} + \lambda \dtv(g_n, g).
\]
Because $\lambda_n \to \lambda$ and $g_n \weakto g$ (Lemma~\ref{the:MixedBiasedERConvergence}), we see that $\dtv(f^{(n)}_{\delta, \tau, \nu}, f_{\delta}) \to 0$, and in light of \eqref{eq:CPoiBinBoundKey}, it follows that $\liminf_{n \to \infty} \pr( \abs{C_{G^{(n)}}(i)} > \tau ) \ge \rho(f^+)$.

(iv) Lower bound for \eqref{eq:MBGiant1}. Fix $\epsilon > 0$, define $\nu = \ceil{2 \log n}$, and let $\tau_n = \omega_n$. 
Again let us choose a small $\delta \in (0,1)$ such that $\rho(f_\delta) \ge \rho(f) - \epsilon$.
Recall that Lemma~\ref{the:QuantitativeLowerBound} implies
\[
 \pr( \abs{C_{G^{(n)}}(i)} > \tau )
 \wge \rho_\tau(f^{(n)}_{\delta, \tau, \nu_n}) - \abs{A} e^{4M m/n} \tau n^{-2}
 \wge \rho(f^{(n)}_{\delta, \tau, \nu_n}) - \abs{A} e^{4M m/n} \tau n^{-2}.
\]
Inspection of the previous part of the proof shows that $\dtv(f^{(n)}_{\delta, \tau, \nu}, f_{\delta}) \to 0$ also for $\tau = \tau_n$ with $1 \ll \tau_n \ll n \log^{-1}n$.
Hence also $\rho(f^{(n)}_{\delta, \tau, \nu}) \to \rho(f_{\delta})$ and
$
 \liminf_{n \to \infty} \pr( \abs{C_{G^{(n)}}(i)} > \omega )
 \ge \rho(f^+).
$

(iv) Proof of \eqref{eq:EmpLargeComponentFrequencyConstant}. Denote $p_i = \pr( \abs{C_{G^{(n)}}(i)} > \tau)$ and $p_{ij} = \pr( \abs{C_{G^{(n)}}(i)} > \tau, C_j(G^{(n)}) > \tau)$.  Symmetry implies that $\E \abs{B_\tau(G^{(n)})} = n p_1$ and $\Var \abs{B_\tau(G^{(n)})} = n p_1(1-p_1) + (n)_2 (p_{12} - p_1^2)$. Then \eqref{eq:MBGiant1Constant} implies
that $n^{-1} \E \abs{B_\tau(G^{(n)})} \to \rho_\tau$.  If $\rho_\tau = 0$, the claim follows by Markov's inequality. Assume next that $\rho_\tau > 0$.  Proposition~\ref{the:DoubleUpperBound} shows that
$
 p_{12}
 \le \rho_\tau(f_{2\tau,n})^2 + c \tau^2 n^{-1} \log n
$
where $c = e^{9M(1+m/n)}$ and $f_{2\tau,n}$ is defined by \eqref{eq:UpperOffspring}. The analysis of the upper bound for \eqref{eq:MBGiant1Constant} shows that $\rho_\tau(f_{2\tau,n}) \to \rho_\tau(f^+)$. Hence for any $\epsilon > 0$, we see that $p_{12} \le \rho_\tau(f)^2 + \epsilon$ for all sufficiently large $n$. Because $p_1 \to \rho_\tau(f^+)$ by \eqref{eq:MBGiant1Constant}, we conclude that $p_{12} - p_1^2 \le 2\epsilon$ for large $n$. Hence $\Var \abs{B_\omega(G^{(n)})} \le n p_1 + 2 n^2 \epsilon$ for large $n$, and
\[
 \frac{\Var \abs{B_\tau(G^{(n)})}}{(\E \abs{B_\tau(G^{(n)})})^2}
 \wle \frac{n p_1}{(n p_1)^2} + \frac{2 n^2 \epsilon}{(n p_1)^2}.
\]
Because $p_1 \asymp 1$, the ratio on the left vanishes and \eqref{eq:EmpLargeComponentFrequencyConstant} follows by Chebyshev's inequality.

(v) Proof of \eqref{eq:EmpLargeComponentFrequency} for $\omega \asymp \log n$. Now \eqref{eq:MBGiant1} implies that $p_1 = n^{-1} \E \abs{B_\omega(G^{(n)})} \to \rho$.  If $\rho = 0$, the claim follows by Markov's inequality. For $\rho > 0$, Proposition~\ref{the:DoubleUpperBound} shows that
$
 p_{12}
 \le \rho_\omega(f_{2\omega,n})^2 + c \omega^2 n^{-1} \log n.
$
By a similar argument as in (ii), we conclude $\frac{\Var \abs{B_\omega(G^{(n)})}}{(\E {B_\omega(G^{(n)})})^2} \ll 1$, so that Chebyshev's inequality now yields \eqref{eq:EmpLargeComponentFrequency} for $\omega \asymp \log n$.

(iv) Proof of \eqref{eq:EmpLargeComponentFrequency} for $1 \ll \omega \ll n \log^{-1} n$. Let $\omega' \asymp \log n$. Then $\abs{1( C_i > \omega) - 1( C_i > \omega')} = 1( C_i > \omega \wedge \omega') - 1(C_i > \omega \vee \omega')$
together with the triangle inequality shows that $\abs{ \, \abs{B_\omega} - \abs{B_{\omega'}} \, } \le \abs{ B_{\omega \wedge \omega'} } - \abs{B_{\omega \vee \omega'}}$. Taking expectations and Markov's inequality imply that
$\pr( \abs{ \, \abs{B_\omega} - \abs{B_{\omega'}} \, } > \epsilon n) \le \epsilon^{-1} n^{-1} \E \abs{ B_{\omega \wedge \omega'} } = \epsilon^{-1} \pr( C_i > \omega \wedge \omega' ) \ll 1$ by \eqref{eq:MBGiant1}.
Hence $\abs{B_\omega} - \abs{B_{\omega'}} = o_\pr(n)$. In the previous step we saw that $\abs{B_{\omega'}} = \rho n + o_\pr(m)$.  Hence $\abs{B_{\omega}} = \rho n + o_\pr(n)$, and \eqref{eq:EmpLargeComponentFrequency} holds also for general $1 \ll \omega \ll n \log^{-1} n$.

(v) Proof of an upper bound for \eqref{eq:GiantFiniteLayerTypes}. Fix $\epsilon > 0$, and let $\omega \asymp \log m$. Then $(\rho+\epsilon) n \ge \omega$ for large values of $n$. If $N_1(G^{(n)}) > (\rho+\epsilon) n$, then every node in a largest component has its component bigger than $\omega$, and hence $\abs{B_\omega} \ge N_1(G^{(n)}) \ge (\rho+\epsilon) n$.
Hence by \eqref{eq:MBGiant1},
\[
 \pr( n^{-1} N_1(G^{(n)}) > \rho+\epsilon)
 \wle \pr( n^{-1} \abs{B_\omega} > \rho+\epsilon)
 \wto 0.
\]

(vi) Proof of a lower bound for \eqref{eq:GiantFiniteLayerTypes}.
We assume that $(P)_{21} > 0$ because otherwise $g$ and $f^+ = \CPoi(\lambda, g)$ both degenerate to the Dirac measure at zero, and the lower bound is trivial. Fix $\epsilon > 0$. Fix $\delta \in (0,1)$ so small that $\rho_\delta(\CPoi((1-\delta)\lambda, g))$ satisfies $\rho_\delta \ge \rho - \epsilon/2$. Denote by $m_{xy} = m P_n(x,y)$ the number of $xy$-layers.
Let us partition the set of layers into two categories called \emph{red} and \emph{blue}, so that the number of red $xy$-layers equals $m^{(r)}_{xy} = \floor{ (1-\delta) m_{xy}}$ for each layer type $(x,y) \in A$, and denote by
$G^{(n,r)}$ the overlay graph on $[n]$ generated by the red layers.
%
Then $\frac{m^{(r)}_{xy}}{m} \to (1-\delta) P(x,y)$ implies that the total number of red layers equals $m^{(r)} \sim (1-\delta)m$ and the empirical layer type distribution of the red layers satisfies $P_n^{(r)} \weakto P$  with $(P_n^{(r)})_{10} \to (P)_{10}$.
By applying \eqref{eq:EmpLargeComponentFrequency} to the overlay graph $G^{(n,r)}$, it follows that the relative proportion of nodes with a large red component is approximated by
\begin{equation}
 \label{eq:RedLargeComponents}
 n^{-1} \abs{B_\omega(G^{(n,r)})} \wprto \rho_\delta
\end{equation}
for any $1 \ll \omega \ll n \log^{-1} n$. Furthermore, denoting $\cE_n = \left\{ \text{$B_\omega(G^{(n,r)})$ is $G^{(n)}$-connected}\right\}$,
\begin{align*}
 \pr( n^{-1} N_1(G^{(n)}) < \rho - \epsilon )
 &\wle \pr( n^{-1} N_1(G^{(n)})  < \rho - \epsilon, \, \cE_n) + \pr(\cE_n^c ) \\
 &\wle \pr( n^{-1} \abs{B_\omega(G^{(n,r)})} < \rho - \epsilon) + \pr(\cE_n^c ) \\
 &\wle \pr( n^{-1} \abs{B_\omega(G^{(n,r)})} < \rho_\delta - \epsilon/2 ) + \pr(\cE_n^c).
\end{align*}
In light of \eqref{eq:RedLargeComponents}, it suffices to show that $\cE_n$
occurs with high probability. 

On the complement of $\cE_n$, there exists a pair of distinct $G^{(n,r)}$-components $C',C'' \subset B_\omega(G^{(n,r)})$ such that there are no $G^{(n)}$-links between $C',C''$. Especially, there are no links between $C',C''$ generated by the blue layers. Denote by $p_{xy}$ the conditional probability that a particular blue layer of type $(x,y)$ connects $C'$ and $C''$ by a link, given the red layers and the event that $C',C''$ are distinct $G^{(n,r)}$-components both larger than $\omega$. Then by applying Lemma~\ref{the:IntersectBoth} and noting that $(x)_2 \le M^2 1(x \ge 2)$, it follows that
\begin{equation}
 \label{eq:BlueLink}
 p_{xy}
 \wge \frac{2 \abs{C'} \abs{C''}}{(n)_2} 1(x \ge 2) y
 \wge M^{-2} \left( \frac{\omega}{n} \right)^2 (x)_2 y.
\end{equation}
Denote by $M_b$ the number of blue layers generating at least one link between $C'$ and $C''$. Then using $1-x \le e^{-x}$,
\begin{align*}
 \E e^{-s M_b}
 \weq \prod_{(x,y)\in A} \left( (1-p_{xy} + p_{xy}e^{-s} \right)^{m^{(b)}_{xy}}
 &\wle e^{-(1-e^{-s}) \sum_{(x,y)\in A} m^{(b)}_{xy} p_{xy}}.
\end{align*}
By noting that $m^{(b)}_{xy} \ge \delta m P_n(x,y)$ and applying \eqref{eq:BlueLink}, we see that for $\omega = n^{2/3}$,
$
 \sum_{(x,y)\in A} m^{(b)}_{xy} p_{xy}
 \ge \delta M^{-2} m n^{-2/3} (P_n)_{21}
 \ge c_1 n^{1/3}
$
for large value of $n$, where $c_1 = \frac12 \delta \mu M^{-2} (P)_{21}$. Markov's inequality implies that for any $a,s > 0$,
\begin{equation}
 \label{eq:BlueLink2}
 \pr( M_b < a)
 \wle e^{sa} \E e^{-sM_b}
 \wle e^{sa - (1-e^{-s}) c_1 n^{1/3}}.
\end{equation}


By noting that $1 - p_{xy} \le e^{- p_{xy}}$
it follows that the conditional probability that there are no blue links between $C'$ and $C''$ is bounded by
\begin{align*}
 \prod_{(x,y) \in A} \left( 1 - p_{xy} \right)^{m^{(b)}_{xy}}
 \wle e^{- \sum_{(x,y) \in A} M^{-2} (\frac{\omega}{n})^2 (x)_2 y m^{(b)}_{xy}}
 \weq e^{- M^{-2} (\frac{\omega}{n})^2 m^{(b)} (P^{(b)}_n)_{21}}.
\end{align*}
Note that there are at most $\frac{n}{\omega} = n^{1/3}$ distinct $G^{(n,r)}$-components inside $B_\omega(G^{(n,r)})$. The number of such component pairs is hence at most $\frac12 n^{2/3}$, and the union bound together with \eqref{eq:BlueLink2} with $a=s=1$ then confirms that
\begin{equation}
 \label{eq:RedLargeConnected}
 \pr(\cE_n^c)
 \wle \frac12 n^{2/3} e^{1 - (1-e^{-1}) c_1 n^{1/3}}
 \wto 0.
\end{equation}
This fact together with \eqref{eq:RedLargeComponents} implies that $n^{-1} N_1(G^{(n)}) \ge \rho - \epsilon$ with high probability.


(vii) Finally, let us extend the proofs to random layer types.  Denote by $\pr_{\theta_n}$ the regular conditional distribution of the $n$-th model given layer types $\theta_n = ( (X_{n,1}, Y_{n,1}), \dots, (X_{n,m}, Y_{n,m}) )$, see Section~\ref{sec:FormalModel} for formal details. In this case the earlier analysis of \eqref{eq:GiantFiniteLayerTypes}
confirms that $\pr_{\theta_n}( \abs{ n^{-1} N_1(G^{(n)}) - \rho(f^+)} > \epsilon ) \to 0$ for any realisation of $(\theta_1, \theta_2,\dots)$ for which the empirical layer type distributions converge according to $\dtv(P_{\theta_n}, P) \to 0$. Because $P_n \weakto P$  it follows by Lemma~\ref{the:EmpiricalDistributionConvergence} that $\dtv(P_{\theta_n}, P) \prto 0$. Now by applying Lemma~\ref{the:ConditionalConvergenceInProbability}
with $\Phi_n(\theta_n, \xi_n) = n^{-1} N_1(G_n)$ and $G_n = G_n(\xi_n)$, we find that $\Phi_n \prto \rho(f^+)$, and hence
\eqref{eq:GiantFiniteLayerTypes} also holds for random layer types.  The same argument also confirms \eqref{eq:MBGiant1Constant}--\eqref{eq:EmpLargeComponentFrequency} for random layer types.
\end{proof}

\begin{lemma}
\label{the:IntersectBoth}
Let $C_1,C_2$ be disjoint subsets of $[n]$ of sizes $c_1,c_2$. Let $V$ be a uniformly random $x$-set in $[n]$ with $x \ge 2$. Then the probability that $V$ intersects both $C_1$ and $C_2$ is at least $ \frac{2 c_1 c_2}{n(n-1)}$.
\end{lemma}
\begin{proof}
Denote $p_x = \pr( V \cap C_1 \ne \emptyset, V \cap C_2 \ne \emptyset )$ for $V$ being a uniformly random $x$-set in $[n]$. Define a random set $V'$ so that the conditional distribution of $V'$ given $V$ is uniformly random among the 2-subsets of $V$. Then $V' \subset V$ with probability one, and the unconditional distribution of $V'$ is uniform among the 2-subsets of $[n]$. Hence it follows that
\[
 p_x
 \wge \pr( V' \cap C_1 \ne \emptyset, V' \cap C_2 \ne \emptyset )
 \weq p_2
 \weq \frac{ c_1 c_2 }{\binom{n}{2}}.
\]
\end{proof}

\subsection{Discretising layer types}
\label{sec:DiscreteLayerTypes}

Layer sizes are compactified using the function $\sigma_M: (x,y) \mapsto (x1(x \le M), y)$ which simply sets the layer size to zero. In the proofs we also need to discretise layer strengths.  Some care is needed to avoid possible atoms of the limiting layer type distribution. Given a probability measure $P$ on $\Z_+ \times [0,1]$, for every integer $L \ge 1$ we define functions $\sigma_{L-}, \sigma_{L+}: \Z_+ \times [0,1] \to \Z_+ \times [0,1]$ as follows.  First, let $B_P$ be the set of points $y \in (0,1)$ for which $P(\Z_+ \times \{y\}) > 0$. Because $B_P$ is countable, for every integer $L \ge 1$ we may select a set of points $0 = s_{0} < s_1 < \cdots < s_L = 1$ such that $\{s_1,\dots, s_{L-1}\} \cap B_P = \emptyset$ and $\abs{s_i - s_{i-1}} \le 2 L^{-1}$ for all $i=1,\dots,L$.  Then we define
\begin{equation}
\label{eq:DiscreteLayerStrengths}
 \begin{aligned}
  \floor{y}_L &\weq \sum_{i=1}^L s_{i-1} 1(s_{i-1} \le y < s_i) + s_{L-1} 1(y=L), \\
  \ceil{y}_L &\weq s_1 1(y=0) + \sum_{i=1}^L s_i 1(s_{i-1} < y \le s_i),
\end{aligned}
\end{equation}
and set $\sigma_{L-}(x,y) = (x, \floor{y}_L)$ and $\sigma_{L+}(x,y) = (x, \ceil{y}_L)$.

\iflongversion\else\color{\iftextcolor}
The following result (proof omitted) can be readily verified by standard dominated convergence and Skorohod's coupling arguments \cite[Proposition 4.30]{Kallenberg_2002}.
\color{black}\fi

\begin{lemma}
\label{the:DiscreteLayerTypes}
Consider probability measures on $\Z_+ \times [0,1]$ such that $P_n \weakto P$ and $(P_n)_{10} \to (P)_{10} \in [0,\infty)$. Then
(i) $P \circ \sigma_M^{-1} \weakto P$  and $(P \circ \sigma_M^{-1})_{10} \to (P)_{10}$ as $M \to \infty$;
(ii) $P_n \circ \sigma_M^{-1} \weakto P \circ \sigma_M^{-1}$ and $(P_n \circ \sigma_M^{-1})_{10} \to (P \circ \sigma_M^{-1})_{10}$ as $n \to \infty$; 
(iii) $P \circ \sigma_{L\pm}^{-1} \weakto P$  and $(P \circ \sigma_{L\pm}^{-1})_{10} \to (P)_{10}$ as $L \to \infty$;
(iv) $P_n \circ \sigma_{L\pm}^{-1} \weakto P \circ \sigma_{L\pm}^{-1}$ and $(P_n \circ \sigma_{L\pm}^{-1})_{10} \to (P \circ \sigma_{L\pm}^{-1} )_{10}$ as $n \to \infty$;
and (v) $h(M) = \sup_{n \ge 1} \int x 1(x > M) P_n(dx,dy) \to 0$.
\end{lemma}
\iflongversion\color{\iftextcolor}
\begin{proof}
(i) Let $f$ be bounded and continuous. Then $f \circ \sigma_M \to f$ pointwise as $M \to \infty$ and $\abs{f \circ \sigma_M} \le \abs{f}$ pointwise, so that by Lebesgue's dominated convergence, $P \circ \sigma_M^{-1} (f) = P( f \circ \sigma_M) \to P(f)$. The same argument applied to $f(x,y) = x$ shows that $(P \circ \sigma_M^{-1})_{10} = P( f \circ \sigma_M) \to P(f) = (P)_{10}$.

(ii) Because $f \circ \sigma_M$ is bounded and continuous whenever $f$ is, it follows that
$P_n \circ \sigma_M^{-1} (f) = P_n( f \circ \sigma_M) \to P( f \circ \sigma_M) = P \circ \sigma_M^{-1} (f)$. For $f(x,y) = x$, we find that $(P_n \circ \sigma_M^{-1})_{10} = P_n( f \circ \sigma_M) \to P( f \circ \sigma_M) = (P \circ \sigma_M^{-1})_{10}$.

(iii) The construction in \eqref{eq:DiscreteLayerStrengths} guarantees that $\floor{y}_L \to y$ and $\ceil{y}_L \to y$ as $L \to \infty$ for every $y \in [0,1]$. Therefore the functions $\sigma_{L-}, \sigma_{L+}$ converge pointwise to the identity map on $\Z_+ \times [0,1]$ as $L \to \infty$. Hence by Lebesgue's dominated convergence, $P \circ \sigma_{L\pm}^{-1} (f) = P( f \circ \sigma_{L\pm}) \to P(f)$ for any bounded continuous $f$. The same argument applied to $f(x,y) = x$ implies that $(P \circ \sigma_{L\pm}^{-1})_{10} = P( f \circ \sigma_{L\pm}) \to P(f) = (P)_{10}$.

(iv) By Skorohod's coupling \cite[Proposition 4.30]{Kallenberg_2002}, there exist random variables $(X_n,Y_n)$ and $(X,Y)$ such that $\law(X_n,Y_n) = P_n$, $\law(X,Y) = P$, and $(X_n,Y_n) \to (X,Y)$ almost surely.
Hence $\sigma_{L-}(X_n,Y_n) = (X_n, \floor{Y_n}_L) \to (X, \floor{Y}_L)$ whenever $Y \in [0,1] \setminus \{s_1,\dots, s_{L-1}\}$. Now $\pr(Y = s_i) = P( \Z_+ \times \{s_i\}) = 0$ by construction, so we conclude that $\sigma_{L-}(X_n,Y_n) \to \sigma_{L-}(X,Y)$ almost surely. Hence $P_n \circ \sigma_{L-}^{-1} \weakto P \circ \sigma_{L-}^{-1}$.  The same argument also works for $P_n \circ \sigma_{L+}^{-1}$, and $(P_n \circ \sigma_{L\pm}^{-1})_{10} \to (P \circ \sigma_{L\pm}^{-1} )_{10}$ then follows by dominated convergence.

(v) Let $(X_n,Y_n)$ and $(X,Y)$ be random variables distributed according to $P_n$ and $P$, respectively.  Then $X_n \to X$ in distribution and $\E X_n = (P_n)_{10} \to (P)_{10} = \E X < \infty$. Hence $(X_n)_{n \ge 1}$ is uniformly integrable, and $h(M) = \sup_{n \ge 1} \E X_n 1(X_n > M) \to 0$.
\end{proof}
\color{black}\fi

\subsection{Proof of Theorem~\ref{the:Giant}}

Denote $P^{M} = P \circ \sigma_M^{-1}$ where $\sigma_M: (x,y) \mapsto (x1(x \le M), y)$. Let $f^+ = \CPoi( \mu (P)_{10}, g)$ and $f^{M} = \CPoi( \mu (P^{M})_{10}, g^{M})$, where $g = \Bin^+_{10}(P)$ and $g^{M} = \Bin^+_{10}(P^{M})$ are defined by~\eqref{eq:MixedBinPlus}.  Then by Lemma~\ref{the:DiscreteLayerTypes}, $P^{M} \weakto P$ together with $(P^{M})_{10} \to (P)_{10}$.  Lemma~\ref{the:MixedBiasedERConvergence} implies that $g^{M} \weakto g$. Hence $f^{M} \weakto f^+$ (Lemma~\ref{the:CPoiPerturbation}), implying that $\rho_t(f^{M}) \to \rho_t(f^+)$ for all $t$ and $\rho(f^{M}) \to \rho(f^+)$ (Lemma~\ref{the:GWLarge}).

(i) Lower bound.  Fix $\epsilon > 0$. Fix a large enough $M$ such that $\rho(f^{M}) \ge \rho(f^+) - \epsilon$. Then apply the layer strength discretisation procedure \eqref{eq:DiscreteLayerStrengths} to $P^{M}$, and define $\sigma_{L-}$ accordingly. Define $P^{ML-} = P^{M} \circ \sigma_{L-}^{-1}$. Lemma~\ref{the:DiscreteLayerTypes} then implies that $P^{ML-} \weakto P^{M}$ and $(P^{ML-})_{10} \to (P^{M})_{10}$ as $L \to \infty$. The same argument as above then implies that $f^{ML-} \weakto f^{M}$ and $\rho(f^{ML-}) \to \rho(f^{M})$ as $L \to \infty$, where $f^{ML-} = \CPoi( \mu (P^{ML-})_{10}, g^{ML-})$ and $g^{ML-} = \Bin^+_{10}(P^{ML-})$.  Hence we may fix a large $L$ so that $\rho(f^{ML-}) \ge \rho(f^{M}) - \epsilon$.

Now for each $n$, consider a modification $G^{(nML-)}$ of $G^{(n)}$ where each layer of type $(x,y)$ is replaced by a layer of type $(x1(x \le M), \floor{y}_L)$.   Under a natural coupling, $N_1(G^{nML-}) \le N_1(G^{n})$ almost surely for every $n$, and
\begin{align*}
 \frac{N_1(G^{(n)})}{n}
 \wge \frac{N_1(G^{nML-})}{n}
 \wge \rho(f^+) - 2\epsilon + \left( \frac{N_1(G^{nML-})}{n} - \rho(f^{ML-}) \right).
\end{align*}
The averaged layer type distribution of $G^{(nML-)}$ equals $P_n^{ML-} = P_n \circ \sigma_{M}^{-1} \circ \sigma_{L-}^{-1} $. In light of Lemma~\ref{the:DiscreteLayerTypes}, we see that $P_n^{ML-} \weakto P^{ML-}$ and $(P_n^{ML-})_{10} \to (P^{ML-})_{10}$ as $n \to \infty$. A suitable lower bound follows from the above inequality, because $\frac{N_1(G^{nML-})}{m} \prto \rho(f^{ML-})$ due to 
Lemma~\ref{the:MBGiant1}.

(ii) Upper bound. Given $\delta, \epsilon > 0$, choose a large enough $t$ so that $\rho_{t}(f^+) \le \rho(f^+) + \epsilon/5$. Then choose a large enough $M$ so that $\rho_{t}( f^{M} ) \le \rho_{t}(f^+) + \epsilon/5$ and $h(M) \le \frac{\delta \epsilon}{40\mu t}$ where $h(M) = \sup_n \int x 1(x>M) d P_n$ (see Lemma~\ref{the:DiscreteLayerTypes}).  By similar arguments as in the proof of the lower bound, we find that
$P^{ML+} \weakto P^{M}$, $g^{ML+} \weakto g^{M}$, and $f^{ML+} \weakto f^{M}$ as $L \to \infty$, where 
$f^{ML+} = \CPoi(\mu (P^{ML+})_{10}, g^{ML+})$ with 
$P^{ML+} = P^{M} \circ \sigma_{L+}^{-1}$ and $g^{ML+} = \Bin^+_{10}(P^{ML+})$.
Hence we may choose (Lemma~\ref{the:GWLarge}) a large enough $L$ so that $\rho_{t}( f^{ML+}) \le \rho_{t}( f^{M}) + \epsilon/5$. Hence $\rho_{t}( f^{ML+}) \le \rho(f^+) + \frac{3}{5}\epsilon$.

Let $G^{(n,M)}$ and $G^{(n,ML+)}$ be modified overlay graphs in which each layer of type $(x,y)$ is replaced by a layer of type $(x 1(x \le M), y)$ and $(x 1(x \le M), \ceil{y}_L)$, respectively. We fix a natural coupling under which $G^{(n,M)} \subset G^{(n,ML+)}$ almost surely.  Then by Lemma~\ref{the:ComponentOverlayTruncation},
\begin{align*}
 \frac{N_1(G^{(n)})}{n}
 \wle \frac{\abs{B_{t}(G^{(n,ML+)})}}{n} + \frac{t}{n} (Z_{n,M} + 1),
\end{align*}
where $Z_{n,M}$ is the number of nodes covered by layers larger than $M$ in the nontruncated model $G^{(n)}$.
By Lemma~\ref{the:MBGiant1}, we may choose an integer $n_0$ such that $\frac{1}{m} \le h(M)$, $\frac{m}{n} \le 2\mu$, and
\begin{equation}
 \label{eq:UpperBoundN1}
 \pr\left(  \frac{\abs{ B_t(G^{(n,ML+)})}}{n} > \rho_t(f^{(ML+)}) + \frac{\epsilon}{5} \right)
 \wle \frac{\delta}{2}
\end{equation}
for all $n \ge n_0$. Then we note that
\[
 \E Z_{n,M}
 \wle \E \sum_{k=1}^m X^{(n)}_k 1(X^{(n)}_k > M)
 \wle m h(M),
\]
so that $\E \frac{t}{n} (Z_{n,M} + 1) \le \frac{t}{n} (m h(M) + 1) \le 4 \mu t h(M) \le \frac{\delta \epsilon}{10}$, and by Markov's inequality, $\pr( \frac{t}{n} (Z_{n,M} + 1) > \frac{\epsilon}{5} ) \le \frac{\delta}{2}$. Hence for all $n \ge n_0$,
\begin{align*}
 \frac{N_1(G^{(n)})}{n}
 \wle \rho_t(f^{(ML+)}) + \frac{2}{5} \epsilon
 \wle \rho(f^+) + \epsilon
\end{align*}
with probability at least $1-\delta$.

(iii) Upper bound on the second largest component.
Fix $\delta, \epsilon, t, M, L$ as in part (ii) and define $G^{(n,M)}$ and $G^{(n,ML+)}$ in the same way.  By Lemma~\ref{the:BigComponents} and Lemma~\ref{the:ComponentOverlayTruncation},
\[
 N_1(G^{(n)}) + N_2(G^{(n)})
 \wle \abs{B_{t}(G^{(n)})} + 2t
 \wle \abs{B_{t}(G^{(n,M)})} + 2t + t Z_{n,M}.
\]
Under a natural coupling, $\abs{B_t(G^{(n,M)})} \le \abs{B_t(G^{(n,ML+)})}$, so that
\[
 \frac{N_2(G^{(n)})}{n}
 \wle \frac{\abs{B_{t}(G^{(n,ML+)})}}{n} - \frac{N_1(G^{(n)})}{n} + \frac{t}{n} (Z_{n,M} +2).
\]
By part (i), $\frac{N_1(G^{(n)})}{n} \ge \rho(f^+) - \epsilon/5$ with probability at least $1-\delta/2$, whereas part (ii) implies that
$\pr( \frac{t}{n} (Z_{n,M} + 1) > \frac{\epsilon}{5} ) \le \frac{\delta}{2}$. Together with \eqref{eq:UpperBoundN1} it follows that 
$\frac{N_2(G^{(n)})}{n} \le \frac{6}{5}\epsilon$ with probability at least $1-\frac{3}{2} \delta$, whenever $n$ is large enough.
Hence $\frac{N_2(G^{(n)})}{n} \prto 0$.

\qed

\subsection{Proofs for percolation models}

\subsubsection{Site percolation}

Proof of Theorem~\ref{the:SitePercolation}. 
The site-percolated graph $\check G^{(n)}$ is an instance of the overlay graph model \eqref{eq:OverlayGraph} with $\check n = \abs{S_n}$ nodes and $m$ layers $\check G_{1},\dots, \check G_{m}$ where $\check G_k$ is the subgraph of $G_k$ induced by $S_n$, and $G_1,\dots,G_m$ are the original layers generating the graph $G$.  The layer types $(\check X_k, \check Y_k)$ in the site-percolated model are mutually independent, and 
$\law(\check X_k \cond X_k=x_k)$ is hypergeometric with probability mass function
\[
 \Hyp(n,\check n, x_k)(t) \weq \frac{\binom{\check n}{t} \binom{n-\check n}{x_k-t} }{ \binom{n}{x_k} }.
\]
The site-percolated graph is hence an instance of the overlay model with $\check n$ nodes, $m$ layers, and averaged layer type distribution 
\[
 \check P_n( A )
 \weq \int (\Hyp(n,\check n, x) \times \, \delta_{y})(A) \, P_n(dx,dy).
\]

Define probability kernels $K_n,K$ on $\Z_+ \times [0,1]$ by formulas $K_n((x,y),A) = (\Hyp(n,\check n, x) \times \delta_y)(A)$ and $K((x,y),A) = (\Bin(x,\theta) \times \delta_y)(A)$. By \cite[Theorem 4]{Diaconis_Freedman_1980}, $\dtv( \Hyp(n,\check n, x), \Bin(x,\frac{\check n}{n}) ) \le 2 \frac{x}{n}$. A basic coupling of coin flips implies that $\dtv(\Bin(x,\frac{\check n}{n}), \Bin(x,\theta)) \le \abs{ \frac{\check n}{n} - \theta } x$. Then for any bounded continuous function $\phi$ on $\Z_+ \times [0,1]$,
\begin{align*}
 \abs{K_n \phi(x,y) - K \phi(x,y)}
 \wle 2 \supnorm{\phi} \dtv( \Hyp(n,\check n, x), \Bin(x,\theta) )
 &\wle 2 \supnorm{\phi} \left( \frac{2}{n} + \abs{\frac{\check n}{n} - \theta} \right) x
\end{align*}
for all $x,y$. Hence $K_n \phi \to K\phi$ uniformly on the compact subsets of $\Z_+ \times [0,1]$. Because $(P_n)_{n \ge 1}$ is tight, it follows that $\check P_n(\phi) = P_n (K_n \phi) \to P (K \phi) = \check P(\phi)$ for any bounded continuous $\phi$.
Hence $\check P_n \weakto \check P$. Direct computations show that $(\check P_n)_{10} = \frac{\check n}{n}(P_n)_{10} \to \theta (P)_{10} = (\check P)_{10}$, and $\frac{m}{\check n} \to \check \mu = \theta^{-1} \mu$. Theorem~\ref{the:SitePercolation}:\eqref{ite:SitePercolation1}--\eqref{ite:SitePercolation2} now follow by applying Theorems~\ref{the:DegreeDistribution} and~\ref{the:Giant} to $\check G^{(n)}$ and noting that $\check \mu (\check P)_{10} = \mu ( P)_{10}$.

Assume next that $(P_n)_{rs} \to (P)_{rs} \in (0,\infty)$ for $rs=21,32,33$. A direct computation using the binomial distribution shows that $(\check P)_{rs} = \theta^r (P)_{rs}$. Theorem~\ref{the:SitePercolation}:\eqref{ite:SitePercolation3} now follows by 
applying Theorem~\ref{the:TransitivityGlobal} to conclude that the  clustering coefficient of $\check G^{(n)}$ converges to $\check \tau = \frac{ (\check P)_{33} }{(\check P)_{32} + \check \mu (\check P)_{21}^2} = \frac{ (P)_{33} }{(P)_{32} + \mu (P)_{21}^2} = \tau$.  Theorem~\ref{the:SitePercolation}:\eqref{ite:SitePercolation4} follows similarly from Theorem~\ref{the:TransitivityLocal}.


\subsubsection{Layerwise bond percolation}

Proof of Theorem~\ref{the:BondPercolation} for the layerwise bond-percolated graph $\tilde G^{(n)}$. The graph $\tilde  G_n$ is an instance of the overlay model with $n$ nodes and $m$ layers $\tilde G_{1},\dots, \tilde G_{m}$ where $\tilde G_k$ has size $X_k$ and strength $\theta Y_k$. The layers $(\tilde G_k, X_k, \theta Y_k)$ are mutually independent, with averaged layer type distribution 
\[
 \tilde P_n( A )
 \weq \int (\delta_x \times \, \delta_{\theta y})(A) \, P_n(dx,dy)
\]
converging according to $\tilde P_n \weakto \hat P$ and $(\tilde P)_{10} \to (\hat P)_{10}$. Furthermore, a direct computation shows that $(\hat P)_{rs} = \theta^s (P)_{rs}$.  Statements \eqref{ite:BondPercolation1}--\eqref{ite:BondPercolation2} of Theorem~\ref{the:BondPercolation} now follow by Theorems~\ref{the:DegreeDistribution} and~\ref{the:Giant}, and noting that $(\hat P)_{10} = (P)_{10}$. Statements \eqref{ite:BondPercolation3}--\eqref{ite:BondPercolation4} follow analogously by Theorems~\ref{the:TransitivityGlobal} and~\ref{the:TransitivityLocal}, and noting that $\hat \tau = \frac{ (\hat P)_{33} }{(\hat P)_{32} + \hat (\hat P)_{21}^2} = \theta \frac{ (P)_{33} }{(P)_{32} + \mu (P)_{21}^2} = \theta\tau$.

\subsubsection{Bond percolation coupling}
\label{sec:BondLayerCoupling}
We will utilise the fact that the overlay bond-percolated graph does not differ much from the layerwise bond-percolated graph $\tilde G^{(n)}$, for which the theorem has already been proved. The conditional distribution of $\hat G^{(n)}$ given the layers $(G_k, X_k, Y_k)$ is an inhomogeneous Bernoulli graph on $\{1,\dots,n\}$ where each node pair $ij$ is linked with probability $\hat p_{ij} = \theta (M_{ij} \wedge 1)$ where $M_{ij} = \sum_k 1( E(G_k) \ni ij )$ is the number of layers linking a node pair $ij$.  The corresponding conditional distribution of $\tilde G^{(n)}$ is a similar inhomogeneous Bernoulli graph with link probabilities $\tilde p_{ij} = 1- (1-\theta)^{M_{ij}}$.  Because $\hat p_{ij} \le \tilde p_{ij}$, this suggest the following coupling construction:
\begin{enumerate}[(i)]
\item Sample the layers $(G_k, X_k, Y_k)$, $k=1,\dots,m$.
\item Sample independent inhomogeneous Bernoulli graphs $\tilde H$ and $H^*$ with link probabilities $\tilde p_{ij}$ and $p_{ij}^* = \frac{\hat p_{ij}}{\tilde p_{ij}}$ with the convention $\frac{0}{0} = 0$.
\item Define $\hat G = G \cap \hat H$ and $\tilde G = G \cap \tilde H$ with $G$ defined by \eqref{eq:OverlayGraph} and $\hat H = \tilde H \cap H^*$.
\end{enumerate}
Then $(\hat G, \tilde G, G)$ constitutes a coupling of the overlay bond-percolated, layerwise bond-percolated, and nonpercolated graphs such that $\hat G \subset \tilde G \subset G$ almost surely.


%

\subsubsection{Proof of Theorem~\ref{the:BondPercolation}:\eqref{ite:BondPercolation1} for overlay bond percolation}

Let us denote by $\hat D_n = \deg_{\hat G^{(n)}}(i)$ and $\tilde D_n = \deg_{\tilde G^{(n)}}(i)$ the degrees of node~$i$ in the overlay bond-percolated and layerwise bond-percolated graph, respectively. Using the coupling of Section~\ref{sec:BondLayerCoupling}, we observe that $\hat D_n = \tilde D_n$ on the event $M_i \le 1$, where $M_i = \max_{j \ne i} M_{ij}$. Hence $\dtv( \law(\hat D_n), \law(\tilde D_n)) \le \pr( M_i > 1)$. The union bound implies that
\[
 \pr( M_{ij} > 1 )
 \wle \sumd_{k,\ell} \pr( E(G_k) \ni ij ) \, \pr( E(G_\ell) \ni ij)
 \wle \left( \sum_k \pr( E(G_k) \ni ij ) \right)^2.
\]
By noting that $\pr( E(G_k) \ni ij ) = \E \frac{(X_k)_2}{(n)_2} Y_k$, we conclude that
\begin{equation}
 \label{eq:CommonLayers}
 \pr( M_{ij} > 1 ) \wle \left( m (n)_2^{-1} (P_n)_{21} \right)^2,
\end{equation}
Another union bound shows that $\pr( M_i > 1 ) \le \sum_{j \ne i} \pr( M_{ij} > 1 )$ and hence 
\begin{equation}
 \label{eq:BondLayerDegrees}
 \dtv( \law(\hat D_n), \law(\tilde D_n))
 \wle (m/n)^2 (P_n)_{21}^2 (n-1)^{-1}.
\end{equation}

Because $\law(\tilde D_n) \weakto \CPoi( \mu (P)_{10}, \Bin_{10}(\hat P) )$, the same result for the bond-percolation graph follows from \eqref{eq:BondLayerDegrees} in case of bounded layer sizes. In the general case, 
we truncate layers as in \eqref{eq:LayerTruncationNew}, and 
denote by $\hat D_n^{M}$ (resp.\ $\tilde D_n^{M}$) the degree of node~$i$ in $\hat G_n^{M}$ (resp.\ $\tilde G_n^{M}$).  
Then \eqref{eq:BondLayerDegrees} implies that $\dtv ( \law(\hat D^M_n), \law(\tilde D_n^{M}) ) \le c M^4 n^{-1}$ for all large values of $n$, with $c = 2\mu^2$. The reasoning in \eqref{eq:DegreeTruncation1} works also for bond-percolated models, and hence $\dtv ( \law(\hat D_n), \law(\hat D_n^{M}) ) \le h(M)$ and 
$\dtv ( \law(\tilde D_n), \law(\tilde D_n^{M}) ) \le h(M)$ 
where $h(M) = \sup_{n \ge 1} \int x 1(x > M) P_n(dx,dy)$.
We conclude that
\[
 \dtv ( \law(\hat D_n), \law(\tilde D_n) )
 \wle c M^4 n^{-1} + 2 h(M)
\]
for all $M$. By choosing $M \asymp n^{1/5}$, we see that $\dtv ( \law(\hat D_n), \law(\tilde D_n) ) \to 0$, and 
Theorem~\ref{the:BondPercolation}:\eqref{ite:BondPercolation1} follows for $\hat G^{(n)}$.

\subsubsection{Proof of Theorem~\ref{the:BondPercolation}:\eqref{ite:BondPercolation3} for overlay bond percolation}

For any distinct nodes $i,j,k$, we see that $\pr(\hat G^{(n)}_{ij}, \hat G^{(n)}_{ik}, \hat G^{(n)}_{jk}) = \theta^3 \pr(G^{(n)}_{ij}, G^{(n)}_{ik}, G^{(n)}_{jk})$ and $\pr(\hat G^{(n)}_{ij}, \hat G^{(n)}_{ik}) = \theta^2 \pr(G^{(n)}_{ij}, G^{(n)}_{ik})$. Hence $\hat \tau^{(n)} = \theta \tau^{(n)}$ for every $n$, and the claim follows by applying Theorem~\ref{the:TransitivityGlobal} to the nonpercolated model.

\subsubsection{Proof of Theorem~\ref{the:BondPercolation}:\eqref{ite:BondPercolation4} for overlay bond percolation}
Fix any distinct nodes $i,j,k$, and note that the clustering spectrum of $\hat G^{(n)}$ can be written as
$\hat \sigma^{(n)}(t) = \pr(\hat \cA_{n,t})/\pr( \hat \cB_{n,t} )$ where
\begin{align*}
 \hat \cA_{n,t} &\weq \{ \deg_{\hat G^{(n)}}(i) = t, \, \hat G^{(n)}_{ij}, \hat G^{(n)}_{ik}, \hat G^{(n)}_{jk}\}, \\
 \hat \cB_{n,t} &\weq \{ \deg_{\hat G^{(n)}}(i) = t, \, \hat G^{(n)}_{ij}, \hat G^{(n)}_{ik}\}.
\end{align*}
A similar formula also holds for the clustering spectrum $\tilde \sigma^{(n)}(t)$ of the layerwise bond-percolated graph, with $\tilde \cA_{n,t}, \tilde \cB_{n,t}$ defined analogously.  Observe that $1( \hat \cA_{n,t} ) = 1 ( \tilde \cA_{n,t} )$ and $1( \hat \cB_{n,t} ) = 1 ( \tilde \cB_{n,t} )$ on the event that $M_i = \max_{j \ne i} M_{ij} \le 1$ and $M_{jk} \le 1$, where $M_{ij}$ refers to the number of layers linking node pair $ij$ in the coupling construction of Section~\ref{sec:BondLayerCoupling}.
By exchangeability, the union bound, estimate \eqref{eq:CommonLayers}, and $(P_n)_{21} \lesim 1$, it follows that
\[
 \pr(M_i > 1 \ \text{or} \ M_{jk} > 1) 
 \wle n \pr( M_{ij} > 1 )
 \wle n \left( m (n)_2^{-1} (P_n)_{21} \right)^2
 \wlesim n^{-1}.
\]
Hence $\pr(\hat \cA_{n,t}) = \pr(\tilde \cA_{n,t}) + O(n^{-1})$ and $\pr(\hat \cB_{n,t}) = \pr(\tilde \cB_{n,t}) + O(n^{-1})$.
Hence $\hat \tau^{(n)}(t) = (1+o(1)) \tilde \tau^{(n)}(t)$, and the claim follows from the corresponding result for the layerwise bond-percolated model.

\subsubsection{Proof of Theorem~\ref{the:BondPercolation}:\eqref{ite:BondPercolation2} for overlay bond percolation}
The coupling construction in Section~\ref{sec:BondLayerCoupling} shows that all components in $\hat G^{(n)}$ are stochastically smaller than their counterparts in $\tilde G^{(n)}$. Hence the upper bounds concerning component sizes in $\hat G^{(n)}$ follow directly from the result of Theorem~\ref{the:BondPercolation}:\eqref{ite:BondPercolation2} for $\tilde G^{(n)}$. Therefore, we only need to prove that with high probability $\hat G^{(n)}$ contains a component of size $(1+o_\pr(1)) \rho(\hat f^+) n$.

Let us investigate how Lemma~\ref{the:MBGiant1} behaves when $G^{(n)}$ is replaced by $\hat G^{(n)} = G^{(n)} \cap H$ where $H$ is a homogeneous Bernoulli graph on $\{1,\dots,n\}$ with link probability $\theta$.
Define a modification of Algorithm~\ref{algo:LowerExploration} where the layer exploration step is replaced by
$\cZ_t \leftarrow \cup_{k \in \cW_t} N_{v_t}(\hat G'_k)$
where $\hat G_k'$ is the transitive closure of $\hat G_k = G_k \cap H$. By construction, the modified version of Algorithm~\ref{algo:LowerExploration} discovers a subset of the $\hat G^{(n)}$-component of the root. Furthermore, the algorithm avoids multi-overlaps, and therefore the output of Algorithm~\ref{algo:LowerExploration} is the same as if it were run for the layerwise bond-percolated model $\tilde G^{(n)}$ with mutually independent layers $\tilde G_k = G_k \cap H_k$ as in \eqref{eq:LayerPercolation}. Hence Lemma~\ref{the:QuantitativeLowerBound} is valid for the overlay bond-percolated model, with the same lower bound as for the layerwise bond-percolated model. Hence the statements in \eqref{eq:MBGiant1Constant}--\eqref{eq:EmpLargeComponentFrequency} of Lemma~\ref{the:MBGiant1} are valid just the same as for the layer-percolated model.

To finish extending Lemma~\ref{the:MBGiant1} to the overlay bond-percolated graph, we still need to verify the sprinkling argument in the proof of the lower bound for \eqref{eq:GiantFiniteLayerTypes}. To do this, we modify the earlier argument slightly using a modified coupling.  As in the earlier proof for the nonpercolated model, fix a small $\delta \in (0,1)$, partition the set of layers into red layers and blue layers, and denote by $G^{(r)}$ and $G^{(b)}$ the overlay graphs generated by the red and blue layers. Let $\theta^{(b)} = \delta$ and define $\theta^{(r)} = 1 - \frac{1-\theta}{1-\delta}$. Let $H^{(r)}, H^{(b)}$ be mutually independent homogeneous Bernoulli graphs on $[n]$ with link probabilities $\theta^{(b)}$ and $\theta^{(r)}$, respectively, sampled independently of the layers. 
Then $\tilde G = G \cap H$ with $G = G^{(r)} \cup G^{(b)}$ and $H = H^{(r)} \cup H^{(b)}$ is an instance of the bond-percolated overlay graph. For a lower bound, we note that $\tilde G \supset \tilde G^{(r)} \cup \tilde G^{(b)}$
where $\tilde G^{(r)} = G^{(r)} \cap H^{(r)}$ and $\tilde G^{(b)} = G^{(b)} \cap H^{(b)}$. Note that $\theta - \delta/2 \le \theta^{(r)} \le \theta$ for $0 < \delta \le \frac12$.

Let $B = B_\omega( \tilde G^{(r)} )$ be the set of nodes having $\tilde G^{(r)}$-component larger than $\omega = n^{2/3}$. 
Then by \eqref{eq:MBGiant1Constant}--\eqref{eq:EmpLargeComponentFrequency} of Lemma~\ref{the:MBGiant1}, it follows that $B \ge (\rho(\hat f^+)-\epsilon) n$ with high probability, where $\epsilon>0$ becomes arbitrarily small after choosing a small enough $\delta > 0$. We claim that $B$ is $\tilde G$-connected with high probability for $\omega = n^{2/3}$.  If $B$ is not $\tilde G$-connected, then there exist disjoint $\tilde G^{(r)}$-components $C',C''$ both of size at least $\omega$, between which there are no $\tilde G$-links and hence no $\tilde G^{(b)}$-links. Let us condition on the red layers and $H^{(r)}$.  Given these, the blue layers and $H^{(b)}$ behave independently. Denote by $M_b$ the number of blue layers containing at least one link between $C'$ and $C''$. Denote by $L_b = \abs{E(G^{(b)}, C',C'')}$ (resp.\ $\tilde L_b = \abs{E(\tilde G^{(b)}, C',C'')}$)  the number of $G^{(b)}$-links (resp.\ $\tilde G^{(b)}$-links) between $C'$ and $C''$. Let $s = \delta^{-1} \log n$ and $t = 3 \delta^{-1} \log n$, and observe that
\begin{align*}
 \pr( \tilde L_b = 0 \cond L_b \ge s)
 \wle (1-\delta)^{s}
 \wle e^{-\delta s}
 \weq n^{-1}.
\end{align*}
Given $M_b \ge t$, we know that $L_b \gest N_{t}$ where $N_{t}$ is the number of distinct coupon types obtained after collecting $t$ random coupons from a collection of $n_0 = \abs{C' \times C''}$ coupon types. By Lemma~\ref{the:CouponCollection}, for large enough $n$ such that $1 + s \le \frac12 t$ and $t \le n_0^{1/4}$,
\[
 \pr( L_b < s \cond M_b \ge t )
 \wle \pr( N_t < s )
 \wle n_0^{-1}
 \wle \omega^{-2}.
\]
By applying \eqref{eq:BlueLink2} and noting that $t \le (\frac12 - e^{-1}) c_1 n^{1/3}$ for large $n$, we see that $\pr( M_b < t) \le e^{t - (1-e^{-1}) c_1 n^{1/3}} \le e^{- c_2 n^{1/3}}$ with $c_1 = \frac12 \delta \mu M^{-2} (P)_{21}$ and $c_2 = \frac14 \delta \mu M^{-2} (P)_{21}$.
Now
\begin{align*}
 \pr( \tilde L_b = 0 )
 &\wle \pr( \tilde L_b = 0 \cond L_b \ge s) + \pr(L_b < s) \\
 &\wle \pr( \tilde L_b = 0 \cond L_b \ge s) + \pr(L_b < s \cond M_b \ge t) + \pr( M_b < t) \\
 &\wle n^{-1} +  \omega^{-2} + e^{- c_2 n^{1/3}} \\
 &\wle 3n^{-1}.
\end{align*}
Now there are at most $\frac{n}{\omega} = n^{1/3}$ such components $C',C''$, and hence at most $\frac12 n^{2/3}$ such component pairs. Hence the probability that there exists a component pair $C',C''$ with no $\tilde G^{(b)}$-links in between, is at most $\frac{3}{2} n^{-1/3}$. We conclude that $B$ is $\tilde G$-connected with high probability.  This confirms that the lower bound for \eqref{eq:GiantFiniteLayerTypes} in Lemma~\ref{the:MBGiant1} extends to the overlay bond-percolated setting.

All the rest in the proof of Theorem~\ref{the:Giant} extends to the overlay bond-percolated setting in a straightforward manner. This concludes the proof of Theorem~\ref{the:BondPercolation}:\eqref{ite:BondPercolation2}.

\qed

\section{Analysis of power-law models}
\label{sec:PowerLawAnalysis}

\subsection{Mixed binomial power laws}

When the limiting layer type distribution factorises according to \eqref{eq:LayerTypeFactorized}--\eqref{eq:PowerLaws} and $\alpha + s \beta > r+1$, we find that the mixed binomial distribution in \eqref{eq:MixedBin} can be written as
\[
 \Bin_{rs}(P)(t)
 \weq \sum_{x=1}^\infty \Bin(x-r, q(x))(t) \, \tilde p_{rs}(x),
\]
where $\tilde p_{rs}(x) = \frac{(x)_r q(x)^s p(x)}{(P)_{rs}}$ is a biased layer size distribution. Assumptions \eqref{eq:PowerLaws} imply that the biased layer size distribution follows a power law
$\tilde p_{rs}(x) \sim \frac{a b^s}{(P)_{rs}} x^{-(\alpha+s\beta-r)}.$ 
If $\beta > 0$ or $b < 1$, then Lemma~\ref{L1} shows that also the mixed binomial distribution follows a power law
\begin{equation}
 \label{eq:MixedBinPowerLaw}
 \Bin_{rs}(P)(t) \wsim d_{rs} t^{-\delta_{rs}}
\end{equation}
with parameters
\begin{equation}
 \label{eq:MixedBinPowerLawParameters}
 \delta_{rs} = 1 + \frac{\alpha+s\beta-r-1}{1-\beta}
 \qquad \text{and} \qquad
 d_{rs} = \frac{ab^s}{(P)_{rs}} \frac{b^{\delta_{rs}-1}}{1-\beta}.
\end{equation}

\subsection{Proof of Theorem~\ref{the:PowerLawDegree}}
The limiting degree distribution given by Theorem~\ref{the:DegreeDistribution} equals $f = \CPoi(\mu (P)_{10}, g_{10})$ with  $g_{10} = \Bin_{10}(P)$. 

(i) Assume first that $0 \le \beta < 1$ and that either $\beta > 0$ or $b < 1$. By \eqref{eq:MixedBinPowerLaw}, we find that $g_{10}(t) \sim d_{10} t^{-\delta_{10}}$. The above formula implies that $g_{10}$ is subexponential \cite[Theorem 4.14]{Foss_Korshunov_Zachary_2013} and it follows that \cite[Theorem 4.30]{Foss_Korshunov_Zachary_2013} $f(t) \sim \mu (P)_{10} g_{10}(t) \sim \mu (P)_{10} d_{10} t^{-\delta_{10}}$.

(ii) Consider the case with $\beta = 0$ and $b=1$, and assume that $q(x)=1$ for all but finitely many $x$.
Then $\Bin(x-1,q(x)) = \delta_{x-1}$ for large values of $x$, and it follows that $g_{10}(t) = \tilde p_{10}(t+1)$ for all large $t$. Hence $g_{10}(t) \sim \tilde p_{10}(t)$, and the claim follows as in (i).

(iii) If $\beta \ge 1$, then $M = \sup_{x \ge 1} (x-1)q(x) < \infty$. The generating function of the limiting degree distribution equals $\sum_{t \ge 0} z^t f(t) = e^{\lambda(\hat g_{10}(z)-1)}$, where
\[
 \hat g_{10}(z)
 \weq \sum_{x \ge 1} (1-q(x) + q(x) z)^{x-1} \tilde p_{10}(x).
\]
Because $1-y + y z \le e^{y(z-1)}$ for all real numbers $z$, it follows that $\hat g_{10}(z) \le e^{M(z-1)}$
and hence $\sum_{t \ge 0} z^t f(t)$ is finite for all $z>0$.
\qed

\subsection{Proof of Theorem~\ref{the:PowerLawTransitivityLocal}}

The limiting clustering spectrum $\sigma(t)$ in Theorem~\ref{the:TransitivityLocal} is represented using convolutions of the limiting degree distribution $f = \CPoi(\mu (P)_{10}, g_{10})$ and distributions $g_{rs} = \Bin_{rs}(P)$ defined by \eqref{eq:MixedBin}.  Theory of discrete subexponential densities \cite[Lemmas 4.9 and 4.14]{Foss_Korshunov_Zachary_2013} implies that
$
 (f_1 \conv f_2)(t) \sim f_1(t) + f_2(t)
$ 
for all probability densities on the positive integers such that $f_i(t) \sim a_i t^{-\alpha_i}$ with $a_i > 0$ and $\alpha_i > 1$. By Theorem~\ref{the:PowerLawDegree}, we know that $f(t) \sim \mu(P)_{10} d_{10} t^{-\delta_{10}}$, and 
by \eqref{eq:MixedBinPowerLaw}, we find that $g_{rs}(t) \sim d_{rs} t^{-\delta_{rs}}$ with parameters given by \eqref{eq:MixedBinPowerLawParameters}. Because $\delta_{32} < \delta_{21} < \delta_{10}$, it follows that
\[
 (f \conv g_{32})(t)
 \wsim f(t) + g_{32}(t)
 \wsim g_{32}(t)
\]
and
\[
 (f \conv g_{21} \conv g_{21})(t)
 \wsim f(t) + g_{21}(t) + g_{21}(t)
 \wll g_{32}(t).
\]
Hence by formula~\eqref{eq:TransitivityLocal},
\[
 \sigma(t)
 \wsim \frac{(P)_{33}}{(P)_{32}} \frac{f(t) + g_{33}(t)}{g_{32}(t)}
 \wsim \frac{(P)_{33}}{(P)_{32}} \frac{\mu (P)_{10} d_{10} t^{-\delta_{10}}  + d_{33} t^{-\delta_{33}}}{ d_{32} t^{-\delta_{32}}}.
\]
Because $\delta_{33} - \delta_{10} = \frac{3\beta-2}{1-\beta}$, we see that $\sigma(t)$ follows a power law with density exponent $\delta_{33}-\delta_{32} = \frac{\beta}{1-\beta}$ for $\beta \le \frac{2}{3}$, and density exponent $\delta_{10}-\delta_{32} = 2$ for $\beta \ge \frac{2}{3}$. The constant term of the power law is determined by \eqref{eq:MixedBinPowerLawParameters}. \qed

\appendix

\section{Supplementary results}
\label{sec:Supplement}


\subsection{Formal model definition}
\label{sec:FormalModel}
Fix integers $n, m \ge 1$.  Let $p_{n,1}, \dots, p_{n,m}$ be probability measures on $\Z_+ \times [0,1]$, and let $q_n$ be a probability kernel from $\Z_+ \times [0,1]$ into $\cG_n$ defined by
$
 q_n((x,y), g)
 = \binom{n}{x}^{-1} (1-y)^{\binom{n}{2} - \abs{E(g)}} \, y^{\abs{E(g)}}.
$
The space of possible layer type configurations $\theta_n = ( (x_1,y_1), \dots, (x_m, y_m))$ is denoted by $\Omega_{1,n} = (\Z_+ \times [0,1])^m$, and the space of possible layer configurations $\xi_n = (g_1,\dots, g_m)$ by $\Omega_{2,n} = \cG_n^m$. Define a probability measure $\bar p_n$ on $\Omega_{1,n}$ and a probability kernel $\bar q_n$ from $\Omega_{1,n}$ to $\Omega_{2,n}$ by
\begin{align*}
 \bar p_n( d\theta_n )
 \weq \prod_{k=1}^m p_{n,k}(dx_k, dy_k),
 \quad
 \bar q_n( \theta_n, \xi_n )
 \weq \prod_{k=1}^m q_n((x_k,y_k), g_k).
\end{align*}
The joint probability distribution of layers and their types is a probability measure $\pr_n = \bar p_n \otimes \bar q_n$ on $\Omega_{n} = \Omega_{1,n} \times \Omega_{2,n}$. We denote by $\pr_{\theta_n}(A) = \bar q_n( \theta_n, A)$ the regular conditional distribution of the layers given layer types $\theta_n$.  The empirical layer type distribution is defined by $P_{\theta_n} = \frac{1}{m} \sum_{k=1}^m \delta_{x_k,y_k}$. The averaged layer type distribution is denoted by $P_n = \frac{1}{m} \sum_{k=1}^m p_{n,k}$.

By defining $\pr$ as the product measure on $\Omega = \Omega_1 \times \Omega_2 \times \cdots$ we may consider all models on all scales simultaneously on a common probability space.  Then $\theta_n$, $\xi_n$, and the graphs $G_n = G_n(\theta_n,\xi_n)$ can be viewed as random variables on $\Omega$ defined using canonical coordinate projections $(\theta, \xi) \to \theta_n$, $(\theta, \xi) \to \xi_n$, and the deterministic map $\xi_n \mapsto G_n(\xi_n) = (\{1,\dots,n\}, \cup_{k=1}^m E(g_{k}))$.

\begin{lemma}
\label{the:ConditionalConvergenceInProbability}
Let $\Phi_n: \Omega_n \to \R$ measurable functions such that
$\pr_{\theta_n}( \{ \xi_n: \abs{ \Phi_n( \theta_n, \xi_n) - c } > \epsilon \}) \to 0$ for all $\epsilon > 0$ and 
for all $(\theta_1,\theta_2,\dots)$ such that $\dtv(P_{\theta_n} ,P ) \to 0$. Assume that $\dtv( P_{\theta_n}, P ) \prto 0$. Then $\Phi_n \prto c$.
\end{lemma}
\begin{proof}
We will apply the result \cite[Lemma 4.2]{Kallenberg_2002} that $X_n \prto X$ if and only if for any subsequence of $\N$ there exists a further subsequence along which the convergence takes place $\pr$-almost surely. Fix a subsequence $\N' \subset \N$. Because $\dtv( P_{\theta_n}, P ) \prto 0$ as $n \to \infty$ along $\N'$, there exists a further subsequence $\N'' \subset \N$ such that $\dtv( P_{\theta_n}, P ) \to 0$ $\pr$-almost surely along $\N''$. Then for any $\epsilon > 0$, the random variables $Z_{\epsilon, n} = \pr_{\theta_n}( \{ \xi_n: \abs{ \Phi_n( \theta_n, \xi_n) - c } > \epsilon \})$ satisfy $Z_{\epsilon, n} \to 0$ $\pr$-almost surely along $\N''$. Dominated convergence then implies that $\pr( \abs{\Phi_n-c} > \epsilon) = \E Z_{\epsilon, n} \to 0$ along $\N''$. Then there exists a further subsequence $\N'''$ such that $\Phi_n \to c$ $\pr$-almost surely along $\N'''$.
\end{proof}

\begin{lemma}
\label{the:EmpiricalDistributionConvergence}
Assume that $P_n$, $n \ge 1$, and $P$ are supported on a finite set $A \subset \Z_+ \times [0,1]$, and that $P_n \weakto P$. Then $\dtv(P_{\theta_n}, P) \prto 0$.
\end{lemma}
\begin{proof}
Now $\E P_{\theta_n}(x,y) = P_n(x,y) \to P(x,y)$ for all $(x,y) \in A$.  Because the layer types are independent, $\Var P_{\theta_n}(x,y) = \frac{1}{m^2} \sum_{k=1}^m \Var 1(X_{n,k} =x, Y_{n,k} =y) \le m^{-1}$. Hence by Chebyshev's inequality, $P_{\theta_n}(x,y) \prto P(x,y)$ for all $(x,y) \in A$, and the claim follows.
\end{proof}

\subsection{Elementary analysis}

\iflongversion\else\color{\iftextcolor}
The following elementary result (proof omitted) follows by standard Riemann integral approximations.
\color{black}\fi

\begin{lemma}
\label{the:UnimodularSum}
Fix integers $a < b$ and let $f: [a,b] \to [0,\infty)$ be unimodular in the sense that there exists $s^* \in [a,b]$ such that $f$ is nondecreasing on $[a,s^*]$ and nonincreasing on $[s^*,b]$.
Then $\left| \sum_{k=a}^b f(k) - \int_a^b f(s) \, ds \, \right| \le \supnorm{f}$.
\end{lemma}
\iflongversion\color{\iftextcolor}
\begin{proof}
Let us abbreviate $\int_a^b f = \int_a^b f(s) \, ds$ and $\sum_a^b f = \sum_{a \le k \le b} f(k)$. Denote $r_1 = \floor{s^*}$ and $r_2 = \ceil{s^*}$. Then by writing
\[
 \sum_{k=a}^{r_1-1} f(k)
 \weq \int_{a}^{r_1} f(\floor{s}) ds
 \qquad \text{and} \qquad
 \sum_{k=r_2+1}^{b} f(k)
 \weq \int_{r_2}^{b} f(\ceil{s}) ds
\]
we find that $\sum_{a}^{r_1-1} f \le \int_{a}^{r_1} f$ and $\sum_{r_2+1}^{b} f \le \int_{r_2}^{b} f$. If $r_1 = r_2 = s^*$, then $f(r_1)=f(r_2)=\supnorm{f}$, and we see that $\sum_{a}^{b} f \le \int_a^b f + \supnorm{f}$. If $r_1 = r_2-1$, then $f(r_1) \wedge f(r_2) \le f(s)$ for $s \in [r_1,r_2]$ implies that
\[
 f(r_1)+f(r_2)
 \weq f(r_1) \wedge f(r_2) + f(r_1) \vee f(r_2)
 \wle \int_{r_1}^{r_2} f + \supnorm{f},
\]
and hence $\sum_{a}^{b} f \le \int_a^b f + \supnorm{f}$ also in this case.

To obtain a lower bound, a similar reasoning shows that $\int_{a}^{r_1} f \le \sum_{a+1}^{r_1} f $ and $\int_{r_2}^{b} f \le \sum_{r_2}^{b-1} f$. Together with the fact that $\int_{r_1}^{r_2} f \le \supnorm{f} 1(r_1<r_2)$, it follows that
$\int_a^b f \le \sum_{a+1}^{r_1} f + \sum_{r_2}^{b-1} f + \supnorm{f} 1(r_1<r_2)$.  Now, because $\sum_{a+1}^{r_1} f + \sum_{r_2}^{b-1} f = \sum_{a+1}^{b-1} f + \supnorm{f} 1(r_1=r_2) \le \sum_{a}^{b} f + \supnorm{f} 1(r_1=r_2)$, it follows that
$\int_a^b f \le \sum_{a}^{b} f + \supnorm{f}$.
\end{proof}
\color{black}\fi


\subsection{Power laws}

The following result characterises conditions under which a mixed binomial distribution follows a power law.

\begin{lemma}
\label{L1}
Consider a mixed binomial distribution $g(r) = \sum_{k \ge 1} p_k f_k(r)$ where $f_k = \Bin(x_k, y_k)$
and $(p_k)$ is a probability distribution on $\{1,2,\dots\}$.  Assume that
\[
 x_k = \big(a + O\big(k^{-\alpha/2}\big)\big) k^{\alpha},
 \quad
 y_k = \big(b + O\big(k^{-\alpha/2}\big)\big) k^{-\beta},
 \quad
 p_k = (c+o(1)) k^{-\gamma},
\]
for some $0 \le \beta < \alpha < \beta + 2$ and $\gamma>1$, and some $a,b,c>0$ such that $\beta>0$ or $b < 1$.
Then
\[
 g(r) = (d+o(1)) r^{-\delta}
\]
where $\delta = 1 + \frac{\gamma-1}{\alpha-\beta}$ and $d = (ab)^{\delta-1}c/(\alpha-\beta)$.
\end{lemma}

\begin{proof}
Denote the mean and variance of $f_k$ by $\mu_k = x_k y_k$ and $\sigma_k^2 = x_k y_k(1-y_k)$. 
Denote $x_k = (1+\epsilon_{1,k}) a k^{\alpha}$, $y_k = (1+\epsilon_{2,k}) b k^{-\beta}$, and define $\epsilon_k$ by the formula $1+\epsilon_k = (1+\epsilon_{1,k})(1+\epsilon_{2,k})$. Then $\epsilon_k = O(k^{-\alpha/2})$, and we may fix constants $k_0,M > 0$ such that $\abs{\epsilon_k} \le M k^{-\alpha/2} \le \frac14$ for all $k \ge k_0$.
Then
\[
 \mu_k = (1+\epsilon_k) abk^\rho
\]
where $\rho = \alpha-\beta$. Define
\[
 A_r = \{k \in \N: \abs{ abk^\rho - r } \le \Delta_r\}
\]
where $\Delta_r = r^{1/2} \log r$. Let us choose $r_0$ large enough so that $\max_{k < k_0} x_k < r_0$ and $4 M (\frac{5a}{4})^{1/2} r^{1/2} \le \Delta_r \le \frac12 r$ for all $r \ge r_0$.

(i) We will first verify that for all $r \ge r_0$,
\begin{equation}
 \label{eq:MixingLocal}
 \sum_{k \notin A_r} f_k(r) p_k
 \weq \sum_{k: k\ge k_0: x_k \ge r, k \notin A_r} f_k(r) p_k
 \wle e^{-\frac{\Delta_r^2}{10 r}}.
\end{equation}
Because $f_k(r) = 0$ for $x_k < r$, we observe that only indices $k$ with $k \ge k_0$  and $x_k \ge r$ appear in the sum 
$g(r) = \sum_{k: x_k \ge r} p_k f_k(r)$ when $r \ge r_0$. This confirms the equality in \eqref{eq:MixingLocal}. For such $k$, $r \le x_k$ and $x_k \le (1+\frac14) a k^\alpha$ imply $k \ge (\frac{4}{5a})^{1/\alpha} r^{1/\alpha}$, and this further shows that $\abs{\epsilon_k} \le M k^{-\alpha/2} \le M (\frac{5a}{4})^{1/2} r^{-1/2}$, so that $\abs{\epsilon_k} r \le \frac14 \Delta_r$. Then by writing
\[
 \mu_k-r
 \weq (1+\epsilon_k)(abk^\rho - r) + \epsilon_k r,
\]
we find that when $r \ge r_0$,
$
 \abs{\mu_k-r}
 \wge (1 - \abs{\epsilon_k}) \Delta_r - \abs{\epsilon_k} r
 \wge \frac12 \Delta_r
$
for all $k$ such that $x_k \ge r$ and $k \notin A_r$. For such values of $k$, Chernoff inequalities for the binomial distribution (Lemma~\ref{the:BinomialConcentration}) imply (using $\Delta_r \le \frac12 r$) that
\[
 f_k(r)
 \wle e^{-\frac{\Delta_r^2}{8(r+\frac12 \Delta_r)}} 
 \wle e^{-\frac{\Delta_r^2}{10 r}}.
\]

(ii) For $r \ge r_0$ and for values $k \in A_r$, we have $\frac12 r \le abk^\rho \le 2r$ due to $\Delta_r \le \frac12 r$, and hence $c_0 r^{1/\rho} \le k \le c_0' r^{1/\rho}$, where $c_0 = (2ab)^{-1/\rho}$ and $c_0' = (ab/2)^{-1/\rho}$. Then let
\[
 \epsilon'_r
 \weq \max_{k \ge c_0 r^{1/\rho}} \abs{\epsilon_k}.
\]
Then $\epsilon'_r$ is decreasing and nonnegative. Now $\abs{\epsilon_k} \le M k^{-\alpha/2} \le
c_0^{-\alpha/2} M r^{-\alpha/(2\rho)}$ for $k \ge k_0$ and $k \ge c_0 r^{1/\rho}$. Hence $\epsilon'_r = O(r^{-\alpha/(2\rho)})$. Now it follows that the mean of $f_k$ is approximated by
\[
 \mu_k \weq (1 + O(r^{-1}\Delta_r) + O(\epsilon'_r)) r
\]
uniformly for $k \in A_r$.
Next, we note that
$y_k = \Theta( r^{-\beta/\rho})$ for $\beta > 0$, 
and $y_k = b + O(k^{-\alpha/2}) = b + O(r^{-\alpha/(2\rho)})$ for $\beta = 0$, uniformly for $k \in A_r$.
It follows that, denoting $\beta' = \beta$ for $\beta > 0$ and $\beta' = \alpha/2$ for $\beta = 0$,
\[
 \sigma_k^2
 \weq \Big(1 + O(r^{-1}\Delta_r) + O( r^{-\beta'/\rho}) + O(\epsilon'_r)\Big) \sigma_0^2 r
\]
where $\sigma_0^2 = 1-b$ for $\beta = 0$ and $\sigma_0^2 = 1$ for $\beta >0$.
Also,
\[
 k^{-\gamma}
 \weq (ab)^{\gamma/\rho} (abk^\rho)^{-\gamma/\rho}
 \weq (1+O(r^{-1}\Delta_r)) (ab)^{\gamma/\rho} r^{-\gamma/\rho}.
\]
Hence,
\begin{equation}
 \label{eq:MixingDensityLocal}
 p_k
 \weq (1+o(1)) c_1 r^{-\gamma/\rho}.
\end{equation}
for $c_1 = (ab)^{\gamma/\rho}c$, uniformly for $k \in A_r$.

(iii) We will next approximate the binomial density $f_k$ by a normal density with the same mean and variance.
By a local limit theorem \cite[Lemma 5]{Zolotukhin_Nagaev_Chebotarev_2018} (see also  \cite{Bloznelis_2019,Leskela_Stenlund_2011}),
\[
 \left| f_k(r) - \frac{1}{\sigma_k} \phi \bigg( \frac{r-\mu_k}{\sigma_k} \bigg) \right|
 \wle 0.516 \, \sigma_k^{-2},
\]
for all $0 \le r \le k-1$ and all $k \ge 2$, where $\phi(s) = (2\pi)^{-1/2} e^{-s^2/2}$ is the standard normal density. Hence
\begin{equation}
 \label{eq:LLT}
 f_k(r)
 \weq \frac{1}{\sigma_k} \phi \bigg( \frac{r-\mu_k}{\sigma_k} \bigg) + O(r^{-1})
\end{equation}
uniformly for $k \in A_r$.

(iv) We will approximate the parameters of the normal density in \eqref{eq:LLT} by $\mu_k \approx abk^\rho$ and $\sigma_k \approx \sigma_0 r^{1/2}$. To see that these approximations hold uniformly, denote $s_{k,r} = \frac{\mu_k-r}{\sigma_k}$ and $t_{k,r} = \frac{abk^\rho-r}{\sigma_0 r^{1/2}}$. Note that
\[
 s_{k,r}
 \weq \sigma_k^{-1} (1+O(\epsilon'_r)) ( abk^\rho - r)
\]
and
\begin{equation}
 \label{eq:MixedVarAppr}
 \sigma_k^{-1}
 \weq \big(1 + O(r^{-1}\Delta_r) + O( r^{-\beta'/\rho}) + O(\epsilon'_r) \big) \sigma_0^{-1} r^{-1/2}.
\end{equation}
Hence
\[
 s_{k,r}
 \weq \big(1 + O(r^{-1}\Delta_r) + O( r^{-\beta'/\rho}) + O(\epsilon'_r) \big) t_{k,r}.
\]
Note that $s^2 - t^2 = (2+u)ut^2$ for $s=(1+u)t$. By applying this formula with $u$ being the above approximation error, using $\abs{t_{k,r}} = O(r^{-1/2} \Delta_r)$, we find that
\begin{align*}
 s_{k,r}^2 - t_{k,r}^2
 &\weq \big( O(r^{-1}\Delta_r) + O( r^{-\beta'/\rho}) + O(\epsilon'_r) \big) O(t_{k,r}^2) \\
 &\weq O(r^{-2}\Delta_r^3) + O( r^{-1-\beta'/\rho} \Delta_r^2) + O(\epsilon'_r r^{-1} \Delta_r^2),
\end{align*}
uniformly for $k \in A_r$. Our choice of $\Delta_r = r^{1/2} \log r$ implies that $s_{k,r}^2 - t_{k,r}^2 = o(1)$ uniformly with respect to $k \in A_r$.
%
%
Then $\abs{e^t-1} \le e{\abs{t}}$ for $\abs{t} \le 1$ implies
\[
 \frac{\phi(s_{k,r})}{\phi(t_{k,r})}
 \weq e^{\frac12( t_{k,r}^2 - s_{k,r}^2 )}
 \weq 1 + O( \abs{t_{k,r}^2 - s_{k,r}^2} )
 \weq 1 + o(1),
\]
and
\begin{align*}
 \phi \bigg( \frac{\mu_k-r}{\sigma_k} \bigg)
 &\weq (1+o(1)) \phi \bigg( \frac{abk^\rho-r}{\sigma_0 r^{1/2}} \bigg)
\end{align*}
uniformly for $k \in A_r$. Together with \eqref{eq:LLT} and \eqref{eq:MixedVarAppr}, it follows that
\begin{equation}
 \label{eq:LLTNice}
 f_k(r)
 \weq (1+o(1)) \frac{1}{\sigma_0 r^{1/2}} \phi \bigg( \frac{abk^\rho-r}{\sigma_0 r^{1/2}} \bigg) + O(r^{-1})
\end{equation}
uniformly for $k \in A_r$.

(v) By Lemma~\ref{the:UnimodularSum}, it follows that
\begin{align*}
 \sum_{k \in A_r} \frac{1}{\sigma_0 r^{1/2}} \phi \bigg( \frac{ab k^\rho-r}{\sigma_0 r^{1/2}} \bigg)
 &\weq \int_{A_r} \frac{1}{\sigma_0 r^{1/2}} \phi \bigg( \frac{abs^\rho-r}{\sigma_0 r^{1/2}} \bigg) ds + O(r^{-1/2}).
\end{align*}
By a change of variables $s = \nu(t)$ with $\nu(t) = (t/ab)^{1/\rho}$, we find that
\begin{align*}
 \int_{A_r} \frac{1}{\sigma_0 r^{1/2}} \phi \bigg( \frac{abs^\rho-r}{\sigma_0 r^{1/2}} \bigg) ds
 &\weq \int_{r-\Delta_r}^{r+\Delta_r} \frac{1}{\sigma_0 r^{1/2}} \phi \bigg( \frac{t-r}{\sigma_0 r^{1/2}} \bigg) \nu'(t) \, dt \\
 &\weq \E \left( \nu'(r+ \sigma_0 r^{1/2}Z) \, 1(\sigma_0 r^{1/2} \abs{Z} \le  \Delta_r) \right),
\end{align*}
where $\law(Z)$ is standard normal.
Because $\nu'(r) = c_2 r^{1/\rho-1}$ with $c_2 = \rho^{-1} (ab)^{-1/\rho}$, we see that
$\nu'(r+\sigma_0 r^{1/2}z) = (1+o(1)) \nu'(r)$ uniformly for $\abs{z} \le \sigma_0^{-1} r^{-1/2} \Delta_r$. Hence it follows by Lebesgue's dominated convergence that
\begin{align*}
 \int_{A_r} \frac{1}{\sigma_0 r^{1/2}} \phi \bigg( \frac{ab s^\rho-r}{\sigma_0 r^{1/2}} \bigg) ds
 \weq (1+o(1)) \nu'(r)
 \weq (c_2+o(1)) r^{1/\rho-1}.
\end{align*}
Because $r^{-1/2} \ll r^{1/\rho-1}$ due to $\rho < 2$, it follows that
\begin{equation}
 \label{eq:NormalSum}
 \sum_{k \in A_r} \frac{1}{\sigma_0 r^{1/2}} \phi \bigg( \frac{ab k^\rho-r}{\sigma_0 r^{1/2}} \bigg)
 \wsim c_2 r^{1/\rho-1}.
\end{equation}
A similar computation also shows that
\begin{equation}
 \label{eq:ConcentratedVolume}
 \abs{A_r}
 \weq \int_{r-\Delta_r}^{r+\Delta_r} \nu'(t) \, dt 
 \wsim 2 \Delta_r r^{1/\rho-1}.
\end{equation}

(vi) By combining \eqref{eq:MixingDensityLocal}, \eqref{eq:LLTNice}, \eqref{eq:NormalSum}, and \eqref{eq:ConcentratedVolume} we now conclude that
\begin{align*}
 \sum_{k \in A_r} f_k(r) p_k
 &\wsim c_1 r^{-\gamma/\rho} \sum_{k \in A_r} f_k(r) \\
 &\wsim c_1 r^{-\gamma/\rho}
 \sum_{k \in A_r} \left( \frac{1}{\sigma_0 r^{1/2}} \phi \bigg( \frac{a bk^\gamma-r}{\sigma_0 r^{1/2}} \bigg) + O(r^{-1}) \right) \\
 &\wsim c_1 r^{-\gamma/\rho} c_2 r^{1/\rho-1}.
\end{align*}
Together with \eqref{eq:MixingLocal}, this now implies the claim, because $e^{-\frac{\Delta_r^2}{10 r}} \ll r^{-\delta}$ for
$\delta = 1 + \frac{\gamma-1}{\alpha-\beta}$.
\end{proof}

\subsection{Compound Poisson and binomial distributions}

Recall that $\CPoi(\lambda, f)$ denotes the compound Poisson distribution with rate parameter $\lambda$ and increment distribution $f$.
\iflongversion
\color{\iftextcolor}
The following three elementary results are included for ease of reference, although they are rather immediately available in the literature (e.g.\ \cite{Barbour_Holst_Janson_1992,Janson_Luczak_Rucinski_2000}).
\color{black}
\else\color{\iftextcolor}
The following three elementary results (proofs omitted) are included for ease of reference. They are either well known or easily verified (e.g.\ \cite{Barbour_Holst_Janson_1992,Janson_Luczak_Rucinski_2000}).
\color{black}\fi

\begin{lemma}
\label{the:CPoiSum}
Let $X = \sum_i X_i$ be a sum of independent random variables such that $\law(X_i) = \CPoi(\lambda_i, g_i)$ with $0 < \sum_i \lambda_i < \infty$. Then $\law(X) = \CPoi(\lambda, g)$ with $\lambda = \sum_i \lambda_i$ and $g = \sum_i \frac{\lambda_i}{\lambda} g_i$.
\end{lemma}
\iflongversion\color{\iftextcolor}
\begin{proof}
The probability generating function of a compound Poisson distribution $\CPoi(\lambda_i, g_i)$ equals $\exp( \lambda_i (G_{g_i}(z)-1)$. Hence the probability generating function of $\sum_i X_i$ equals
\[
 G_{X}(z)
 \weq \prod_i G_{X_i}(z)
 \weq \exp \Big( \sum_i \lambda_i ( G_{g_i}(z)-1) \Big)
 \weq \exp \Big( \lambda ( G_{g}(z)-1) \Big),
\]
where $G_g(z)$ is the probability generating function of $g = \sum_i \frac{\lambda_i}{\lambda} g_i$.
\end{proof}
\color{black}\fi

\begin{lemma}
\label{the:CPoiPerturbation}
For any $\lambda, \lambda' \ge 0$ and any probability measures $f,f'$ on $\R$,
\[
 \dtv\bigg( \CPoi(\lambda, f), \, \CPoi(\lambda', f') \bigg)
 \wle \min\{\lambda, \lambda'\} \, \dtv(f, f') + \abs{\lambda-\lambda'}.
\]
\end{lemma}
\iflongversion\color{\iftextcolor}
\begin{proof}
By symmetry, we may assume that $\lambda \le \lambda'$. Denote $g = \CPoi(\lambda, f)$, $g' = \CPoi(\lambda, f')$, and $g'' = \CPoi(\lambda', f')$. By triangle inequality, it suffices to verify that $\dtv(g,g') \le \lambda \, \dtv(f, f')$ and $\dtv(g',g'') \le \lambda'-\lambda$.

(i) Let $(X,X')$ a coupling of $f$ and $f'$ which is optimal in the sense that $\pr(X \ne X') = \dtv(f, f')$. Define a coupling of $g$ and $g'$ by
\[
 Y \weq \sum_{j=1}^{\Lambda} X_j
 \qquad \text{and} \qquad
 Y' \weq \sum_{j=1}^{\Lambda} X_j',
\]
where $\Lambda, (X_1,X_1'), (X_2,X_2'), \dots$ are mutually independent random variables such that $\law(\Lambda) = \Poi(\lambda)$ and $\law(X_j,X_j') = \law(X,X')$ for all $j$. Then by the union bound, we see that
\[
 \pr( Y \ne Y' \cond \Lambda = \ell)
 \weq \pr\left( \sum_{j=1}^{\ell} X_j \ne \sum_{j=1}^{\ell} X_j' \right)
 \wle \ell \, \pr(X \ne X').
\]
By summing both sides weighted by $\pr(\Lambda = \ell)$, it follows that
$\pr(Y \ne Y') \le \E( \Lambda) \pr( X \ne X' )$ and hence $\dtv(g,g') \le \lambda \dtv(f,f')$.

(ii) Let $Y'$ and $\Delta$ be independent random numbers such that $\law(Y') = \CPoi(\lambda, f')$ and $\law(\Delta) = \CPoi(\delta, f')$ with $\delta = \lambda'-\lambda$. Define $Y'' = Y' + \Delta$ and note by Lemma~\ref{the:CPoiSum} that $\law(Y'') = \CPoi(\lambda', f')$. Hence
\[
 \dtv(g', g'')
 \wle \pr(Y' \ne Y'')
 \weq \pr(\Delta \ne 0) 
 \wle 1-e^{-\delta}
 \wle \delta
 \weq \lambda'-\lambda.
\]
\end{proof}
\color{black}\fi

\begin{lemma}
\label{the:BinomialConcentration}
If $X$ is $\Bin(n,p)$-distributed with mean $\mu = np$, then (i) $\pr( X > a ) \le e^{2\mu-a}$ for all $a \ge 0$, 
(ii) $\pr( X \le a ) \le e^{-\mu/8}$ for any $a \le \frac12 \mu$, and (iii) $\pr(X=r) \le e^{-\frac{s^2}{2(r+s)}}$ for any $s > 0$ and for all integers $r$ such that $\abs{r-\mu} \ge s$.
\end{lemma}
\iflongversion\color{\iftextcolor}
\begin{proof}
(i) Because $\E e^{X} = (1 + p(e-1))^n \le e^{(e-1) \mu} \le e^{2\mu}$, Markov's inequality implies that $\pr( X > a ) = \pr( e^X > e^a) \le e^{-a} \E e^{X} \le e^{2\mu - a}$.

(ii) Because $(\mu-a)^2 \ge \frac{1}{4} \mu^2$, it follows by \cite[Theorem 2.1]{Janson_Luczak_Rucinski_2000} that
$\pr( X \le a ) \le e^{-(\mu-a)^2/(2\mu)} \le e^{-\mu/8}$.

(iii) The approximation $\pr( X = r ) \le \min\{ \pr( X \le r), \pr( X \ge r ) \}$ combined with suitable Chernoff bounds \cite[Theorem 2.1]{Janson_Luczak_Rucinski_2000} will do the job, as shown below. Fix an integer $r \ge 0$ and consider the following two cases:
\begin{enumerate}[(a)]
\item If $r \le \mu - s$. Then the bound $\pr( X \le \mu - t ) \le e^{- \frac{t^2}{2\mu} }$ for $t = \mu - r$, together with the fact that $t \mapsto \frac{(t-r)^2}{2t}$ is increasing on $(r,\infty)$, implies that
\[
 \pr(X \le r)
 \weq \pr( X \le \mu - (\mu-r) )
 \wle \exp\left( - \frac{(\mu-r)^2}{2\mu} \right)
 \wle \exp\left( - \frac{s^2}{2(r+s)} \right).
\] 
\item If $r \ge \mu + s$. Then the bound $\pr( X \ge \mu + t ) \le e^{- \frac{t^2}{2(\mu + t/3)} }$ for $t = s$, and the fact that $\mu + s/3 \le r \le r+s$ imply that
\[
 \pr(X \ge r)
 \weq \pr( X \ge \mu+s )
 \wle \exp\left( - \frac{s^2}{2(\mu+s/3)} \right)
 \wle \exp\left( - \frac{s^2}{2(r+s)} \right).
\]
\end{enumerate}
\end{proof}
\color{black}\fi

\subsection{Biased and truncated probability measures}

Below $P(\psi) = \int \psi(x) P(dx)$ is used as a shorthand for integrals.  When $P(\psi) \in (0,\infty)$, we denote by $P^\psi = \frac{\psi(x) P(dx)}{P(\psi)} = \frac{\psi dP}{P(\psi)}$ the $\psi$-biased probability measure $P^\psi(A) = \frac{\int_A \psi(x) P(dx)}{P(\psi)}$. For a probability measure $P$ and a probability kernel $K$ we denote by $PK$  the probability measure $PK(A) = \int K(x,A) P(dx)$.  For a function $\phi$, we define a function $K\phi$ by $K\phi(x) = \int \phi(y) K(x,dy)$.
\iflongversion
\color{\iftextcolor}
The following three results are proved for ease of reference, although they are rather immediate consequences of standard Wasserstein-type estimates of probability kernels (e.g.\ \cite{Leskela_2010,Leskela_Vihola_2017}).
\color{black}
\else\color{\iftextcolor}
The following three results (proofs omitted) follow by standard dominated convergence and Skorohod's coupling arguments \cite{Kallenberg_2002} or Wasserstein-type estimates of probability kernels (e.g.\ \cite{Leskela_2010,Leskela_Vihola_2017}), by noting that the kernels $K( (x,y), t ) = \Bin(x-r, y)(t)$ and $K^+( (x,y), t ) = \Bin^+(x-r, y)(t)$ are continuous in $y$ (being polynomials of finite order). 
\color{black}\fi

\begin{lemma}
\label{the:BiasedWeakConvergence}
Let $P_n, P$ be probability measures on a separable metric space such that $P_n \weakto P$ and $P_n(\psi) \to P(\psi) \in (0,\infty)$ for some continuous function $\psi \ge 0$. Then $P_n^\psi \weakto P^\psi$.
\end{lemma}
\iflongversion\color{\iftextcolor}
\begin{proof}
By Skorohod coupling \cite[Proposition 4.30]{Kallenberg_2002} there exist random variables $X_n, X$ such that $\law(X_n) = P_n$, $\law(X) = P$, and $X_n \to X$ almost surely. Let $\phi$ be a bounded and continuous. Then $Y_n = \phi(X_n) \psi(X_n)$ converges almost surely to $Y = \phi(X) \psi(X)$, and $\abs{Y_n} \le \norm{\phi}_\infty \psi(X_n)$ almost surely for all $n$. Because $\E \phi(X_n) \to \E \phi(X) < \infty$, Lebesgue's dominated convergence theorem (as stated in \cite[Theorem 1.21]{Kallenberg_2002}) implies that $P_n( \phi \psi) = \E Y_n \to \E Y = P(\phi \psi)$. Hence $P_n^\psi(\phi) = \frac{P_n( \phi \psi)}{P_n(\psi)} \to \frac{P( \phi \psi)}{P(\psi)} = P^\psi(\phi)$.
\end{proof}
\color{black}\fi



\begin{lemma}
\label{the:KernelConvergenceSpecial}
Let $P_n,P$ be probability measures on $\Z_+ \times [0,1]$, and let $K$ be a probability kernel from $\Z_+ \times [0,1]$ into $\Z_+$ such that $y \mapsto K((x,y),t)$ is continuous for every $x,t \in \Z_+$.
If $P_n \weakto P$, then $P_n K \weakto PK$.
\end{lemma}
\iflongversion\color{\iftextcolor}
\begin{proof}
Let $\phi: \Z_+ \to \R$ be bounded. Assume that $(x_n, y_n) \to (x,y)$.  Then the probability measures on $\Z_+$ defined by $Q_n(A) = K((x_n,y_n), A)$ and $Q(A) = K((x,y), A)$ converge according to $Q_n(\{t\}) \to Q(\{t\})$ for all $t \in \Z_+$, and hence also weakly. Hence the function $K\phi$ defined by $K\phi(x,y) = \sum_t K((x,y),t) \phi(t)$ is bounded and continuous. Now $P_n \weakto P$ implies that $P_nK(\phi) = P_n(K\phi) \to P(K \phi) = PK(\phi)$. Hence $P_nK \weakto PK$.
\end{proof}
\color{black}\fi

\begin{lemma}
\label{the:MixedBiasedERConvergence}
If $P_n \weakto P$ and $(P_n)_{rs} \to (P)_{rs} \in (0,\infty)$, then the laws in \eqref{eq:MixedBin}--\eqref{eq:MixedBinPlus} satisfy
$\Bin_{rs}(P_n) \weakto \Bin_{rs}(P)$ and $\Bin^+_{rs}(P_n) \weakto \Bin^+_{rs}(P)$. 
\end{lemma}
\iflongversion\color{\iftextcolor}
\begin{proof}
Define $\psi$-biased probability measures $P_n^\psi, P^\psi$ using $\psi(x,y) = (x)_r y^s$.  Then $P_n^* \weakto P^*$ by Lemma~\ref{the:BiasedWeakConvergence}. Observe next that
the kernels $K( (x,y), t ) = \Bin(x-r, y)(t)$ and $K^+( (x,y), t ) = \Bin^+(x-r, y)(t)$ are continuous in $y$ (being polynomials of finite order). The claims now follow by Lemma~\ref{the:KernelConvergenceSpecial}
because $\Bin_{rs}(P_n) = P^\psi_n K$ and $\Bin^+_{rs}(P_n) = P^\psi_n K^+$.
\end{proof}
\color{black}\fi

\subsection{Graph components}

Denote by $N_1(G) \ge N_2(G)$ the largest two component sizes in $G$ (with $N_2(G) = 0$ if $G$ is connected.) Let $B_t(G) = \{i \in V(G): \abs{C_i(G)} > t\}$ be the set of nodes with component larger than $t$.

\begin{lemma}
\label{the:BigComponents}
For all $t \ge 0$:
(i) $N_1(G) \le \max\{\abs{B_t(G)}, \, t\}$ and (ii) $N_1(G) + N_2(G) \le \abs{B_t(G)} + 2t$.
\end{lemma}
\begin{proof}
(i) Let $C_1$ be a component of $G$ of size $\abs{C_1} = N_1(G)$. If $\abs{C_1} \le t$, there is nothing to prove. If $\abs{C_1} > t$, then every node in $C_1$ has component larger than $t$, and hence $C_1 \subset B_t(G)$ implies $\abs{N_1(G)} \le \abs{B_t(G)}$.

(ii) If $N_2(G) \le N_1(G) \le t$, the claim is clear. If $N_2(G) \le t < N_1(G)$, the claim follows from (i). Assume now that $t < N_2(G) \le N_1(G)$, and let $C_1,C_2$ be components of $G$ with sizes $\abs{C_1} = N_1(G)$ and $\abs{C_2} = N_2(G)$. Then every node in $C_1 \cup C_2$ has component larger than $t$, and the claim follows from $N_1(G) + N_2(G) = \abs{C_1 \cup C_2} \le \abs{B_t(G)}$.
\end{proof}

\subsection{Graph superpositions}

Let $G_1,\dots,G_m$ be graphs such that $V(G_k) \subset V$ for all $k$. For $A \subset [m]$ we denote by $G_A$ the overlay graph with $V(G_A)=V$ and $E(G_A) = \cup_{a \in A} E(G_a)$.

\begin{lemma}
\label{the:ComponentOverlayTruncation}
For any $A,B \subset [m]$ and $t \ge 0$,
\begin{align*}
 \abs{B_t(G_{A \cup B})} &\wle \abs{B_t(G_A)} + t \abs{U_B}, \\
 N_1(G_{A \cup B}) &\wle \max\{ \abs{B_t(G_A)} + t \abs{U_B}, \, t \},
\end{align*}
where $U_B = \cup_{k \in B} V(G_k)$.
\end{lemma}

\begin{proof}
Assume that $i \in B_t(G_{A \cup B}) \setminus B_t(G_A)$. Then $\abs{C_i(G_{A \cup B})} > t$ but $\abs{C_i(G_A)} \le t$, and we see that $C_i(G_A)$ must contain some node $j \in U_B$. Then $i \in C_j(G_A)$ and $\abs{C_j(G_A)} \le t$. We conclude that
\[
 B_t(G_{A \cup B}) \setminus B_t(G_A)
 \ \subset \nhquad \bigcup_{j \in U_B: \abs{C_j(G_A)} \le t} C_j(G_A).
\]
Hence
\[
 \abs{B_t(G_{A \cup B})}
 \wle \abs{B_t(G_A)} + \abs{B_t(G_{A \cup B}) \setminus B_t(G_A)}
 \wle \abs{B_t(G_A)} + t \abs{U_B}.
\]
The second inequality follows because $N_1(G_{A \cup B}) \le \max\{\abs{B_t(G_{A \cup B})}, \, t\}$ by Lemma~\ref{the:BigComponents}.
\end{proof}


In the following two results, we denote by $N_A$ the set of neighbours of node $i$ in $G_A$, and we set $D_A = \abs{N_A}$ to denote the degree of $i$ in $G_A$.

\begin{lemma}
\label{the:DegreeLessLayers}
Let $g$ be an arbitrary probability density on $\Z_+$. Let
$
 \epsilon(t)
 = \sum_{r+s=t} \Big( \pr( D_{A \cup B} = r ) - \pr( D_A=r ) \Big) g(s).
$
Then $\sum_{t \ge 0} \abs{\epsilon(t)} \le 2 \pr(D_B > 0)$.
\end{lemma}
\begin{proof}
Denote the densities of the degrees by $f_{A \cup B} = \law(D_{A \cup B})$ and $f_{A} = \law(D_{A})$. Then $\sum_{t \ge 0} \abs{\epsilon(t)} = \norm{f_{A \cup B} \conv g - f_{A} \conv g}_1 = 2 \dtv( f_{A \cup B} \conv g, f_{A} \conv g) \le 2 \dtv( f_{A \cup B}, f_{A})$. Further, $\dtv( f_{A \cup B}, f_{A}) \le \pr(D_{A \cup B} \ne D_{A}) \le \pr(D_B > 0)$.
\end{proof}

\begin{lemma}
\label{the:DegreeSplitLayers}
Assume that $G_1,\dots,G_m$ are mutually independent, let $A,B \subset [m]$ be disjoint, and let $\cE_A,\cE_B$ be events determined by $(G_a)_{a \in A}$ and $(G_b)_{b \in B}$, respectively. Then
\[
 \pr( D_{A \cup B} = t, \cE_A, \cE_B ) 
 \weq \pr( D_A + D_B =t, \cE_A, \cE_B ) + \epsilon(t),
\]
where the error term is bounded by $\abs{\epsilon(t)} \le c_B t \pr( D_A \le t, \cE_A )$, and where $c_B = \max_{j \ne i} \pr( ij \in E(G_B), \cE_B )$.

\end{lemma}
\begin{proof}
Because $D_{A \cup B} = D_A + D_B$ outside the event $\cF = \{\abs{N_A \cap N_B} > 0\}$, we see that
\[
 \epsilon(t)
 \weq \pr( D_{A \cup B} = t, \cE_A, \cE_B, \cF ) - \pr( D_A + D_B = t, \cE_A, \cE_B, \cF ).
\]
Hence it follows that $\abs{\epsilon(t)} \le \pr( D_A \le t, \cE_A, \cE_B, \cF )$, where the upper bound can be expressed as
\begin{align*}
 \pr( D_A \le t, \cE_A, \cE_B, \cF )
 &\weq \sum_{U: \abs{U} \le t, i \notin U} \pr( N_A = U, \cE_A ) \, \pr( \abs{U \cap N_B} > 0, \cE_B ).
\end{align*}
Because $\pr( \abs{U \cap N_B} > 0, \, \cE_B ) \le \sum_{j \in U} \pr(ij \in E(G_B), \, \cE_B) \le c_B t$ whenever $\abs{U} \le t$, the claim follows.
\end{proof}

\subsection{Galton--Watson processes}
\label{sec:BranchingTrees}

Let $f$ be a probability measure on $\Z_+$ and consider a Galton--Watson branching process with offspring distribution $f$. The exploration queue length of the corresponding tree \cite[Section 3.3]{VanDerHofstad_2017} satisfies the recursion $Q_0=1$ and $Q_t = 1(Q_{t-1}>0) (Q_{t-1}-1+Z_t)$ where $Z_1,Z_2,\dots$ are independent $f$-distributed random integers. The total progeny equals $T = \inf\{t \ge 1: Q_t = 0\} \in [0,\infty]$. We denote $\rho_t(f) = \pr(T > t)$ and $\rho(f) = \pr(T=\infty)$. We also note that $\pr(T > t) = \pr(Q_t > 0)$.
\iflongversion\else\color{\iftextcolor}
The following result (proof omitted) can be verified by straightforward extensions of the arguments in \cite{Bollobas_Janson_Riordan_2007} and \cite[Lemma 2.6]{Leskela_Ngo_2017}.
\color{black}\fi

\begin{lemma}
\label{the:GWLarge}
If $f_n \weakto f$, then (i) $\rho_t(f_n) \to \rho_t(f)$ for all $t \ge 0$.  If $f_n \weakto f$ and $f(0) > 0$, then (ii)
$\rho(f_n) \to \rho(f)$, and (iii) $\rho_{\omega_n}(f_n) \to \rho(f)$ for all sequences $\omega_n \to \infty$.
\end{lemma}
\iflongversion\color{\iftextcolor}
\begin{proof}
A natural coupling of exploration processes implies that $\abs{\rho_t(f_n) - \rho_t(f)} \le t \dtv(f_n, f)$ for all $t$. Hence (i) follows by noting that weak convergence and total variation convergence are equivalent for probability measures $f_n, f$ on the countable space $\Z_+$.  Claim (ii) follows by \cite[Lemma 2.6]{Leskela_Ngo_2017}.   For (iii), we first note that $\rho_t(f) \to \rho(f)$ as $t \to \infty$. Hence given any $\epsilon > 0$, we may choose $t$ so that $\rho(f) \le \rho_t(f) \le \rho(f) + \epsilon$. Then, we see that
\[
 \rho(f_n) \wle \rho_{\omega_n}(f_n) \wle \rho_{t}(f_n) + \epsilon
\]
for all sufficient large values of $n$ such that $\omega_n \ge t$. Now (iii) follows by noting that $\rho(f_n) \to \rho(f)$ by (i), and 
$\rho_{t}(f_n) \to \rho_{t}(f)$ by (ii).
\end{proof}
\color{black}\fi

\subsection{Coupon collection}
The classical coupon collector's problem involves a collector who at each round receives a coupon with type selected uniformly at random among a set of $n$ types, independently of previous rounds. We denote by $N_t$ the number of distinct coupon types obtained after collecting $t$ coupons.

\begin{lemma}
\label{the:CouponCollection}
Fix integers $k, t, n \ge 1$ such that $\frac{1}{k} \ge \frac{1}{t} + \frac{1}{n}$. Then the probability that the number of distinct coupon types obtained after collecting $t$ coupons is less than $k$ is at most
\begin{equation}
 \label{eq:Coupon1}
 \pr( N_t < k)
 \wle \left(\frac{t}{k}\right)^k \left(\frac{n}{k} - \frac{n}{t} \right)^{-(t-k)}.
\end{equation}
Especially, $\pr( N_t < k) \le n^{-\alpha}$ whenever $\alpha + k \le (1-\beta)t$ and $t \le n^{\beta/2}$ for some $\alpha > 0$ and $\beta \in (0,1)$
\end{lemma}
\begin{proof}
Fix $s = \log(\frac{n}{k} - \frac{n}{t})$. Then $s \ge 0$ and $\frac{k}{n} e^s = 1 - \frac{k}{t} < 1$. Denote by $T_k$ the number of coupons needed to obtain $k$ distinct coupon types. Then $T_{j+1}-T_j$ is geometrically distributed with moment generating function $\E e^{s(T_{j+1}-T_j)} = \frac{(1-j/n)e^s}{1-(j/n)e^s}$. Hence
\[
 \E e^{s T_{k}}
 \weq \prod_{j=0}^{k-1} \frac{(1-j/n)e^s}{1-(j/n)e^s}
 \wle \left( \frac{e^s}{1 - \frac{k}{n} e^s} \right)^k.
\]
Markov's inequality applied to $e^{s T_{k+1}}$ hence shows that
\[
 \pr( N_t < k)
 \weq \pr( T_{k} > t )
 \wle e^{-st} \E e^{s T_{k+1}}
 \weq \frac{e^{s(k-t)}}{\left(  1 - \frac{k}{n} e^s \right)^k}
 \weq \frac{(\frac{n}{k} - \frac{n}{t})^{k-t}}{\left( \frac{k}{t} \right)^k}.
\]

Observe next that $\frac{1}{k} - \frac{1}{t} \ge t^{-2}$ implies that the right side of \eqref{eq:Coupon1} is at most
$\left(\frac{t}{k}\right)^k \left( \frac{t^2}{n} \right)^{t-k} \le t^{2t} n^{-(t-k)} \le n^{\beta t - (t-k)}$ for $t \le n^{\beta/2}$. Hence
$\pr( N_t < k) \le n^{-\alpha}$ when we also assume that $\alpha + k \le (1-\beta)t$.
\end{proof}

\ifllstyle
\iflongversion
\bibliographystyle{siamplain}
\else
\bibliographystyle{abbrv}
\fi
\ifmbstyle
\bibliographystyle{siamplain}
\fi
\ifaap
\bibliographystyle{imsart-number}
\fi
\bibliography{lslReferences}

\newcommand{\SortNoop}[1]{}\def\cprime{$'$}
\begin{thebibliography}{10}

\bibitem{Abbe_2018_JMLR}
{\sc E.~Abbe}, {\em Community detection and stochastic block models: {Recent}
  developments}, Journal of Machine Learning Research, 18 (2018), pp.~1--86.

\bibitem{Andersson_Britton_2000}
{\sc H.~Andersson and T.~Britton}, {\em Stochastic Epidemic Models and Their
  Statistical Analysis}, Springer, 2000.

\bibitem{AngelesSerrano_Boguna_2006_I}
{\sc M.~\'Angeles~Serrano and M.~Bogu\~n\'a}, {\em Clustering in complex
  networks. {I}. {General} formalism}, Phys. Rev. E, 74 (2006), p.~056114,
  \url{https://doi.org/10.1103/PhysRevE.74.056114}.

\bibitem{AngelesSerrano_Boguna_2006_II}
{\sc M.~\'Angeles~Serrano and M.~{Bogu\~n\'a}}, {\em Clustering in complex
  networks. {II}. {Percolation} properties}, Phys. Rev. E, 74 (2006),
  p.~056115, \url{https://doi.org/10.1103/PhysRevE.74.056115}.

\bibitem{Ball_Sirl_Trapman_2014}
{\sc F.~G. Ball, D.~J. Sirl, and P.~Trapman}, {\em Epidemics on random
  intersection graphs}, Ann. Appl. Probab., 24 (2014), pp.~1081--1128,
  \url{https://doi.org/10.1214/13-AAP942}.

\bibitem{Barbour_Holst_Janson_1992}
{\sc A.~D. Barbour, L.~Holst, and S.~Janson}, {\em Poisson Approximation},
  Oxford University Press, 1992.

\bibitem{Benson_Liu_Yin_2020}
{\sc A.~R. Benson, P.~Liu, and H.~Yin}, {\em A simple bipartite graph
  projection model for clustering in networks}, 2020,
  \url{https://arxiv.org/abs/2007.00761},
  \url{https://arxiv.org/abs/2007.00761}.

\bibitem{Bloznelis_2010_Largest}
{\sc M.~Bloznelis}, {\em The largest component in an inhomogeneous random
  intersection graph with clustering}, Electron. J. Combin., 17 (2010).

\bibitem{Bloznelis_2013}
{\sc M.~Bloznelis}, {\em Degree and clustering coefficient in sparse random
  intersection graphs}, Ann. Appl. Probab., 23 (2013), pp.~1254--1289,
  \url{https://doi.org/10.1214/12-AAP874}.

\bibitem{Bloznelis_2019}
{\sc M.~Bloznelis}, {\em Local probabilities of randomly stopped sums of power
  law lattice random variables}, Lithuanian Mathematical Journal, 59 (2019),
  pp.~437--468.

\bibitem{Bloznelis_Godehardt_Jaworski_Kurauskas_Rybarczyk_2015}
{\sc M.~Bloznelis, E.~Godehardt, J.~Jaworski, V.~Kurauskas, and K.~Rybarczyk},
  {\em Recent Progress in Complex Network Analysis: {P}roperties of Random
  Intersection Graphs}, Springer, 2015, pp.~79--88,
  \url{https://doi.org/10.1007/978-3-662-44983-7_7},
  \url{http://dx.doi.org/10.1007/978-3-662-44983-7_7}.

\bibitem{Bloznelis_Petuchovas_2017}
{\sc M.~Bloznelis and J.~Petuchovas}, {\em Correlation between clustering and
  degree in affiliation networks}, in Algorithms and Models for the Web Graph,
  A.~Bonato, F.~Chung~Graham, and P.~Pra{\l}at, eds., Cham, 2017, Springer
  International Publishing, pp.~90--104.

\bibitem{Bode_Fountoulakis_Muller_2015}
{\sc M.~Bode, N.~Fountoulakis, and T.~M{\"u}ller}, {\em On the largest
  component of a hyperbolic model of complex networks}, Electron. J. Combin.,
  22 (2015).

\bibitem{Bollobas_Janson_Riordan_2007}
{\sc B.~Bollob{\'a}s, S.~Janson, and O.~Riordan}, {\em The phase transition in
  inhomogeneous random graphs}, Random Struct. Algor., 31 (2007), pp.~3--122,
  \url{https://doi.org/10.1002/rsa.20168}.

\bibitem{Bradonjic_Hagberg_Hengartner_Percus_2010}
{\sc M.~Bradonji{\'{c}}, A.~Hagberg, N.~W. Hengartner, and A.~G. Percus}, {\em
  Component evolution in general random intersection graphs}, in Algorithms and
  Models for the Web Graph, R.~Kumar and D.~Sivakumar, eds., 2010, pp.~36--49.

\bibitem{Breiger_1974}
{\sc R.~L. Breiger}, {\em The duality of persons and groups}, Social Forces, 53
  (1974), pp.~181--190, \url{https://doi.org/10.1093/sf/53.2.181}.

\bibitem{Britton_Deijfen_Lageras_Lindholm_2008}
{\sc T.~Britton, M.~Deijfen, A.~N. Lager{\aa}s, and M.~Lindholm}, {\em
  Epidemics on random graphs with tunable clustering}, J. Appl. Probab., 45
  (2008), pp.~743--756, \url{https://doi.org/10.1239/jap/1222441827}.

\bibitem{ColomerDeSimon_Boguna_2014}
{\sc P.~{Colomer-de-Sim\'on} and M.~Bogu\~n\'a}, {\em Double percolation phase
  transition in clustered complex networks}, Phys. Rev. X, 4 (2014), p.~041020,
  \url{https://doi.org/10.1103/PhysRevX.4.041020},
  \url{https://link.aps.org/doi/10.1103/PhysRevX.4.041020}.

\bibitem{Diaconis_Freedman_1980}
{\sc P.~Diaconis and D.~Freedman}, {\em Finite exchangeable sequences}, Ann.
  Probab., 8 (1980), pp.~745--764,
  \url{https://doi.org/10.1214/aop/1176994663},
  \url{https://doi.org/10.1214/aop/1176994663}.

\bibitem{Epasto_Lattanzi_PaesLeme_2017}
{\sc A.~Epasto, S.~Lattanzi, and R.~Paes~Leme}, {\em Ego-splitting framework:
  {F}rom non-overlapping to overlapping clusters}, in Proceedings of the 23rd
  ACM SIGKDD Conference on Knowledge Discovery and Data Mining, 2017.

\bibitem{Feld_1981}
{\sc S.~L. Feld}, {\em The focused organization of social ties}, American
  Journal of Sociology, 86 (1981), pp.~1015--1035,
  \url{http://www.jstor.org/stable/2778746}.

\bibitem{Foss_Korshunov_Zachary_2013}
{\sc S.~Foss, D.~Korshunov, and S.~Zachary}, {\em An Introduction to
  Heavy-Tailed and Subexponential Distributions}, Springer, 2013.

\bibitem{Fountoulakis_VanDerHoorn_Muller_Schepers_2020}
{\sc N.~Fountoulakis, P.~van~der Hoorn, T.~M{\"u}ller, and M.~Schepers}, {\em
  Clustering in a hyperbolic model of complex networks}, 2020,
  \url{https://arxiv.org/abs/2003.05525},
  \url{https://arxiv.org/abs/2003.05525}.
\newblock arXiv:2003.05525.

\bibitem{Frieze_Karonski_2016}
{\sc A.~Frieze and M.~Karo{\'n}ski}, {\em Introduction to Random Graphs},
  Cambridge University Press, 2016.

\bibitem{Godehardt_Jaworski_2001}
{\sc E.~Godehardt and J.~Jaworski}, {\em Two models of random intersection
  graphs and their applications}, Electronic Notes in Discrete Mathematics, 10
  (2001), pp.~129--132.

\bibitem{Iskhakov_etal_2020}
{\sc L.~{Iskhakov}, B.~{Kami{\'n}ski}, M.~{Mironov}, P.~{Pra{\l}at},
  L.~{Prokhorenkova}, and D.~{Higham}}, {\em Local clustering coefficient of
  spatial preferential attachment model}, Journal of Complex Networks, 8
  (2020), pp.~1--32.

\bibitem{Jacob_Morters_2015}
{\sc E.~Jacob and P.~M{\"o}rters}, {\em Spatial preferential attachment
  networks: Power laws and clustering coefficients}, Ann. Appl. Probab., 25
  (2015), pp.~632--662, \url{https://doi.org/10.1214/14-AAP1006}.

\bibitem{Jacob_Morters_2017}
{\sc E.~Jacob and P.~M{\"o}rters}, {\em Robustness of scale-free spatial
  networks}, Ann. Probab., 45 (2017), pp.~1680--1722,
  \url{https://doi.org/10.1214/16-AOP1098}.

\bibitem{Janson_Luczak_Rucinski_2000}
{\sc S.~Janson, T.~{\L}uczak, and A.~Ruci\'{n}ski}, {\em Random Graphs}, Wiley,
  2000, \url{https://doi.org/10.1002/9781118032718},
  \url{http://dx.doi.org/10.1002/9781118032718}.

\bibitem{Kallenberg_2002}
{\sc O.~Kallenberg}, {\em Foundations of Modern Probability}, Springer,
  second~ed., 2002.

\bibitem{Karjalainen_Leskela_2017}
{\sc J.~Karjalainen and L.~Leskel\"a}, {\em Moment-based parameter estimation
  in binomial random intersection graph models}, in 14th Workshop on Algorithms
  and Models for the Web Graph (WAW), 2017.

\bibitem{Karjalainen_VanLeeuwaarden_Leskela_2018}
{\sc J.~Karjalainen, J.~S.~H. van Leeuwaarden, and L.~Leskel\"a}, {\em
  Parameter estimators of sparse random intersection graphs with thinned
  communities}, in 15th Workshop on Algorithms and Models for the Web Graph
  (WAW), 2018.

\bibitem{Karonski_Scheinerman_Singer-Cohen_1999}
{\sc M.~Karo{\'n}ski, E.~R. Scheinerman, and K.~B. Singer-Cohen}, {\em On
  random intersection graphs: {T}he subgraph problem}, Combin. Probab. Comput.,
  8 (1999), pp.~131--159, \url{https://doi.org/10.1017/S0963548398003459}.

\bibitem{Kiwi_Mitsche_2019}
{\sc M.~Kiwi and D.~Mitsche}, {\em On the second largest component of random
  hyperbolic graphs}, SIAM Journal on Discrete Mathematics, 33 (2019),
  pp.~2200--2217, \url{https://doi.org/10.1137/18M121201X}.

\bibitem{Krioukov_Papadopoulos_Kitsak_Vahdat_Boguna_2010}
{\sc D.~Krioukov, F.~Papadopoulos, M.~Kitsak, A.~Vahdat, and
  M.~Bogu{\~n}{\'a}}, {\em Hyperbolic geometry of complex networks}, Physical
  Review E, 82 (2010), p.~036106.

\bibitem{Krot_OstroumovaProkhorenkova_2015}
{\sc A.~Krot and L.~Ostroumova~Prokhorenkova}, {\em Local clustering
  coefficient in generalized preferential attachment models}, in Algorithms and
  Models for the Web Graph (WAW), D.~F. Gleich, J.~Komj{\'a}thy, and N.~Litvak,
  eds., 2015, pp.~15--28.

\bibitem{Lageras_Lindholm_2008}
{\sc A.~N. Lager{\aa}s and M.~Lindholm}, {\em A note on the component structure
  in random intersection graphs with tunable clustering}, Electron. J. Combin.,
  15 (2008),
  \url{http://www.combinatorics.org/Volume_15/Abstracts/v15i1n10.html}.

\bibitem{Leskela_2010}
{\sc L.~Leskel{\"a}}, {\em Stochastic relations of random variables and
  processes}, J. Theor. Probab., 23 (2010), pp.~523--546,
  \url{https://doi.org/10.1007/s10959-009-0216-8},
  \url{http://dx.doi.org/10.1007/s10959-009-0216-8}.

\bibitem{Leskela_Ngo_2017}
{\sc L.~Leskel\"a and H.~Ngo}, {\em The impact of degree variability on
  connectivity properties of large networks}, Internet Mathematics, 1 (2017),
  pp.~1--24, \url{https://doi.org/10.24166/im.07.2017}.

\bibitem{Leskela_Stenlund_2011}
{\sc L.~Leskel{\"a} and M.~Stenlund}, {\em A local limit theorem for a
  transient chaotic walk in a frozen environment}, Stoch. Proc. Appl., 121
  (2011), pp.~2818--2838.

\bibitem{Leskela_Vihola_2017}
{\sc L.~Leskel\"a and M.~Vihola}, {\em Conditional convex orders and measurable
  martingale couplings}, Bernoulli, 23 (2017), pp.~2784--2807,
  \url{https://doi.org/10.3150/16-BEJ827}.

\bibitem{Newman_2003_Structure}
{\sc M.~E.~J. Newman}, {\em The structure and function of complex networks},
  SIAM Review, 45 (2003), pp.~167--256,
  \url{https://doi.org/10.1137/S003614450342480}.

\bibitem{Petti_Vempala_2018-02}
{\sc S.~Petti and S.~Vempala}, {\em Approximating sparse graphs: {T}he random
  overlapping communities model}, 2018, \url{https://arxiv.org/abs/1802.03652}.
\newblock arXiv:1802.03652.

\bibitem{Spirakis_Nikoletseas_Raptopoulos_2013}
{\sc P.~G. Spirakis, S.~Nikoletseas, and C.~Raptopoulos}, {\em A guided tour in
  random intersection graphs}, in Automata, Languages, and Programming, F.~V.
  Fomin, R.~Freivalds, M.~Kwiatkowska, and D.~Peleg, eds., 2013, pp.~29--35.

\bibitem{Steele_1994}
{\sc J.~M. Steele}, {\em Le {Cam}'s inequality and {Poisson} approximations},
  American Mathematical Monthly, 101 (1994), pp.~48--54,
  \url{http://www.jstor.org/stable/2325124}.

\bibitem{Stegehuis_VanDerHofstad_Janssen_VanLeeuwaarden_2017}
{\sc C.~Stegehuis, R.~van~der Hofstad, A.~J. E.~M. Janssen, and J.~S.~H. van
  Leeuwaarden}, {\em Clustering spectrum of scale-free networks}, Phys. Rev. E,
  96 (2017), p.~042309, \url{https://doi.org/10.1103/PhysRevE.96.042309},
  \url{https://link.aps.org/doi/10.1103/PhysRevE.96.042309}.

\bibitem{Vadon_2020}
{\sc V.~Vadon}, {\em Local and global structure of networks with communities},
  PhD thesis, Technische Universiteit Eindhoven, 2020.

\bibitem{Vadon_Komjathy_VanDerHofstad_2019}
{\sc V.~Vadon, J.~Komj{\'a}thy, and R.~van~der Hofstad}, {\em A new model for
  overlapping communities with arbitrary internal structure}, Applied Network
  Science, 4 (2019), p.~42, \url{https://doi.org/10.1007/s41109-019-0149-9}.

\bibitem{VanDerHofstad_2017}
{\sc R.~van~der Hofstad}, {\em Random Graphs and Complex Networks - {V}ol.
  {I}}, Cambridge University Press, 2017,
  \url{http://www.win.tue.nl/~rhofstad/NotesRGCN.html}.

\bibitem{Vazquez_Pastor-Satorras_Vespignani_2002}
{\sc A.~V\'azquez, R.~Pastor-Satorras, and A.~Vespignani}, {\em Large-scale
  topological and dynamical properties of the internet}, Phys. Rev. E, 65
  (2002), p.~066130, \url{https://doi.org/10.1103/PhysRevE.65.066130}.

\bibitem{Yang_Leskovec_2012}
{\sc J.~Yang and J.~Leskovec}, {\em Community-affiliation graph model for
  overlapping network community detection}, in 2012 IEEE 12th International
  Conference on Data Mining, Dec 2012, pp.~1170--1175,
  \url{https://doi.org/10.1109/ICDM.2012.139}.

\bibitem{Yang_Leskovec_2014}
{\sc J.~Yang and J.~Leskovec}, {\em Structure and overlaps of ground-truth
  communities in networks}, ACM Trans. Intell. Syst. Technol., 5 (2014),
  \url{https://doi.org/10.1145/2594454}.

\bibitem{Zolotukhin_Nagaev_Chebotarev_2018}
{\sc A.~Zolotukhin, S.~Nagaev, and V.~Chebotarev}, {\em On a bound of the
  absolute constant in the {B}erry-{E}sseen inequality for i.i.d. {B}ernoulli
  random variables}, Mod. Stoch. Theory Appl., 5 (2018), pp.~385--410,
  \url{https://doi.org/10.15559/18-vmsta113},
  \url{https://doi.org/10.15559/18-vmsta113}.

\end{thebibliography}
\end{document}

\clearpage

\section{LEFTOVERS START HERE}

\section{Analysis of connected components}

\subsection{Double branching process upper bound --- Newer}

\begin{lemma}
\label{the:UniformRandomSet}
Let $U$ be uniformly distributed among the $x$-sized subsets of a set $V$ of size $m$. Then for any $A \subset V$ and $v \in V \setminus A$,
\[
 \pr( U \ni v, \, U \cap A \ne \emptyset)
 \wle \abs{A} \left( \frac{x}{m} \right)^2.
\]
\end{lemma}
\begin{proof}
By the union bound,
\[
 \pr( U \ni v, \, U \cap A \ne \emptyset)
 \wle \sum_{a \in A} \pr( U \supset \{a,v\} )
 \weq \abs{A} \frac{(x)_2}{(m)_2},
\]
so the claim follows because $(x-1)/(m-1) \le x/m$ for $x \le m$.
\end{proof}

Note:  \rnote{We could do this with the true component explorations, or the restricted explorations, perhaps the latter is simpler to work with.}

(i) Denote by $V_{i,t}$ the set of explored nodes and by $W_{i,t}$ the set of explored layers during the first $t$ steps of the exploration from node $i$. Let $\bar V_{i,t} = \cup_{k \in W_{i,t}} V(G_k)$ the set of nodes covered by the discovered layers. Let $i$ and $j$ be distinct.
Let
\[
 \cA_{ij}
 \weq \set{ \bar V_{i,t} \cap V_{j,t} = \emptyset, \, W_{i,t} \cap W_{j,t} = \emptyset }.
\]
Fix some node sets $A_i \ni i$ and $C_i \not\ni j$ such that $A_i \subset C_i$ and $\abs{A_i}=t$, and consider the event
\begin{equation}
 \label{eq:EventI}
 \cE_i
 \weq \{ V_{i,t} = A_i, \, W_{i,t} = B_i, \, \bar V_{i,t} = C_i \}.
\end{equation}
Note that $\abs{L_j} \ge t$ if and only if $\abs{V_{j,t}} = t$, and therefore
\begin{align*}
 \pr( \abs{L_j} \ge t, \, \cA_{ij} \cond \cE_i)
 &\weq \sum_{A_j} \sum_{B_j} \pr( V_{j,t} = A_j, \, W_{j,t} = B_j \cond \cE_i)
\end{align*}
where the first sum is over all node sets $A_j \ni j$ of size $t$ which do not overlap with $C_i$, and the second sum over all layer sets $B_j$ not overlapping with $B_i$. Note that on the event $\cE_i$, no layer in $B_i$ contains a node in $A_j$. Therefore, on the event $\cE_i$, the indicator of whether or not the event $\{V_{j,t} = A_j, \, W_{j,t} = B_j\}$ occurs can be expressed as a deterministic function of the random integers $1(V(G_k) \ni a)$, $a \in A_j$, $k \in B_i^c$, and the random graphs $G_k$, $k \in B_j \subset B_i^c$. Conditionally on $\cE_i$, the layers $G_k$, $k \in B_i^c$, are mutually independent and such that each $V(G_k)$ is uniformly distributed among the $X_k$-sized subsets of $A_i^c$. Hence it follows that
\[
 \pr( V_{j,t} = A_j, \, W_{j,t} = B_j \cond \cE_i)
 \weq \pr(  V'_{j,t} = A_j, \,  W'_{j,t} = B_j ),
\]
where $V'_{j,t}, W'_{j,t}$ are the sets of explored nodes and layers after $t$ steps in an exploration starting from node $j$ of a modified overlay graph model $G'$ defined with node set $A_i^c$ and layer types $(X_k,Y_k)$, $k \in B_i^c$.  By summing over $A_j$ and $B_j$ as described above, 
it follows that
\[
 \pr( \abs{L_j} \ge t, \cA_{ij} \cond \cE_i)
 \weq \pr( \abs{L'_j} \ge t, \, V'_{j,t} \cap C_i = \emptyset )
 \wle \pr( \abs{L'_j} \ge t ).
\]
Furthermore, 
$
 \pr( \abs{L'_j} \ge t )
 \le \pr( \abs{L''_j} \ge t ),
$
where $L''_j$ refers to an exploration starting from an arbitrary node $j$ of a modified overlay graph model $G''$ defined with $m-t$ nodes and $n$ layers with types $(X_k,Y_k)$, $k \in [n]$. Hence
\begin{align*}
 \pr( \cE_i, \, \abs{L_j} \ge t, \, \cA_{ij} ) \wle \pr( \cE_i ) \, \pr( \abs{L''_j} \ge t ),
\end{align*}
and by summing both sides over all events $\cE_i$ of the form in \eqref{eq:EventI}, it follows that
\[
 \pr( \abs{L_i} \ge t, \, \abs{L_j} \ge t, \, \bar V_{i,t} \not\ni j, \, \cA_{ij} )
 \wle \pr( \abs{L_i} \ge t ) \, \pr( \abs{L_j''} \ge t ).
\]
Now \mnote{Could be a lemma}
\begin{align*}
 \pr( \abs{L_i} \ge t, \, \abs{L_j} \ge t )
 \wle \pr( \abs{L_i} \ge t ) \, \pr( \abs{L_j''} \ge t )
 + \pr( \cA_{ij}^c )
 + \pr(\bar V_{i,t} \ni j).
\end{align*}

\begin{rcomm}
For $\omega = \Theta(m^{1/3})$ and $n = \Theta(m)$, the upper bound is small for $t \ll m^{1/3}$.
\end{rcomm}

(ii) Let us now verify that $\pr( \bar V_{i,t} \cap \bar V_{j,t} \ne \emptyset)$ is small. Note first that
\[
 \pr( \bar V_{i,t} \cap \bar V_{j,t} \ne \emptyset)
 \wle \sum_{s=1}^t \pr(\cC_s),
\]
where $\cC_s = \{ \bar V_{i,t} \cap \bar V_{j, s-1} = \emptyset, \, \bar V_{i,t} \cap \bar V_{j,s} \ne \emptyset\}$.
On the event $\cC_s$ there exists a node $v \in \bar V_{j, s-1}$ which is explored during step $s$, and a layer $k$ which covers $v$ and overlaps with $\bar V_{i,t}$. Let us now consider the event 
\[
 \cE_i 
 \weq \set{ V_{i,t} = A_i, \, W_{i,t} = B_i, \, \bar V_{i,t} = C_{i} }
\]
for some $i \in A_i \subset C_i$ such that $C_i \not\ni j$ and some layer set $B_i$. Consider also the event
\[
 \cE_j
 \weq \set{ V_{j, s-1} = A_j, \, W_{j, s-1} = B_j, \, \bar V_{j, s-1} = C_{j} }
\]
for some $j \in A_j \subset C_j$ such that $C_i \cap C_j = \emptyset$ and some layer set $B_j$ such that $B_i \cap B_j = \emptyset$. Conditionally on $\cE_i$ and $\cE_j$, the layers $k \in (B_i \cup B_j)^c$ are mutually independent and such that $V(G_k)$ is uniformly distributed among the $X_k$-sized subsets of $(A_i \cup A_j)^c$. On the event $\cC_s \cap \cE_i \cap \cE_j$ there exists a node $v \in C_j \setminus A_j$ which is explored during the $s$-th step of the $j$-exploration, and a layer $k \in (B_i \cup B_j)^c$ such that $V(G_k) \ni v$ and $V(G_k) \cap C_i \ne \emptyset$. Therefore,
\[
 \pr( \cC_s \cap \cE_i \cap \cE_j )
 \wle \sum_{v \in C_j \setminus A_j} \sum_{k \in (B_i \cup B_j)^c}
 \pr( V(G_k) \ni v, \, V(G_k) \cap C_i \ne \emptyset, \, \cE_i, \, \cE_j )
\]
Now 
\begin{align*}
 \pr( V(G_k) \ni v, \, V(G_k) \cap C_i \ne \emptyset \cond \cE_i, \, \cE_j )
 &\weq \pr( V(G_k) \ni v, \, V(G_k) \cap (C_i \setminus A_i) \ne \emptyset \cond \cE_i, \, \cE_j ) \\
 &\wle \abs{C_i \setminus A_i} \frac{X_k^2}{\abs{(A_i \cup A_j)^c}^2} \\
 &\wle \omega M^2 (m-2t)^{-2} \\
 &\wle 4 \omega M^2 m^{-2}
\end{align*}
for $C_i \le \omega$, $X_k \le M$, and $\abs{A_i}, \abs{A_j} \le t \le \frac14 m$. Therefore,
\[
 \pr( \cC_s \cap \cE_i \cap \cE_j )
 \wle 4 \omega^2 M^2 m^{-2} n \pr( \cE_i, \, \cE_j )
\]
for $C_j \le \omega$. By summing over $\cE_i, \cE_j$ of the proper form, we find that
\[
 \pr( \cC_s \cond \bar V_{i,t} \not\ni j, \, \abs{\bar V_{i,t}} \le \omega, \, \abs{\bar V_{j,t}} \le \omega )
 \wle 4 \omega^2 M^2 m^{-2} n.
\]
Hence, by the union bound,
\[
 \pr( \bar V_{i,t} \cap \bar V_{j,t} \ne \emptyset
   \cond \bar V_{i,t} \not\ni j, \, \abs{\bar V_{i,t}} \le \omega, \, \abs{\bar V_{j,t}} \le \omega )
 \wle 4 t \omega^2 M^2 m^{-2} n,
\]
for $t \le \frac14 m$, and it follows that (see below) for all $t \le \frac14 m$,
\begin{align*}
 \pr( \bar V_{i,t} \cap \bar V_{j,t} \ne \emptyset)
 &\wle 4 t \omega^2 M^2 m^{-2} n + \pr( \bar V_{i,t} \ni j ) + 2 \pr(  \abs{\bar V_{i,t}} > \omega ) \\
 &\wle 4 t \omega^2 M^2 m^{-2} n + \omega m^{-1} + 3 \pr(  \abs{\bar V_{i,t}} > \omega ).
\end{align*}
By \eqref{eq:NumberCovered} and Markov's inequality,
\[
 \pr(  \abs{\bar V_{i,t}} > \omega )
 \wle 2 M^2 t \omega^{-1} m^{-1} n
\]
for $t \le \frac12 m$. Hence
\begin{align*}
 \pr( \bar V_{i,t} \cap \bar V_{j,t} \ne \emptyset)
 \wle 4 t \omega^2 M^2 m^{-2} n + \omega m^{-1} + 6 M^2 t \omega^{-1} m^{-1} n.
\end{align*}

\begin{rcomm}
For $\omega = \Theta(m^{1/3})$ and $n = \Theta(m)$, the upper bound is small for $t \ll m^{1/3}$.
\end{rcomm}

\begin{rcomm}
By a more detailed analysis, we can improve the upper bound by replacing $\omega^2$ with $\omega$, and $M^2$ by $(P_n)_{2,0}$.
\end{rcomm}

(iii) Note that (think of $j$ being independent of the $i$ exploration) $\pr(\bar V_{i,t} \ni j \, \cond \, \abs{\bar V_{i,t}} = s ) = \frac{s}{m}$. Hence,
\[
 \pr(\bar V_{i,t} \ni j)
 \wle \pr( \abs{\bar V_{i,t}} > \omega ) + \frac{\omega}{m}.
\]

(iv) We will verify that for all $t$,
\begin{equation}
 \label{eq:NumberCovered}
 \E \abs{\bar V_{i,t}}
 \wle \frac{t}{m-(t-1)} \sum_{k=1}^n X_k^2.
\end{equation}
To do this, let us first analyse the probability of the event $\cD_{i,k,s}$ that layer $k$ is discovered during step $s$ of the exploration process started at $i$. Let $\cA_{i, s, v}$ be the event that there is a node to explore during step $s$ and this node equals $v$, for an exploration process starting at $i$. Fix a node set $A_i$ of size $s-1$ and some layer set $B_i$ such that the event $\cA_{i, s, v} \cap \cE_i$ has nonzero probability, where $\cE_i = \{ V_{i,s-1}=A_i, \, W_{i,s-1}=B_i \}$.
Then
\[
 \pr( \cD_{i,k,s} \cond \cA_{i, s, v}, \cE_i )
 \weq \pr( V(G_k) \ni v \cond \cA_{i, s, v}, \cE_i )
 \weq 1(k \in B_i^c) \frac{X_k}{\abs{A_i^c}}
 \wle \frac{X_k}{m-(t-1)}.
\]
Because the above upper bound holds for all $\cE_i$ of the aforementioned type and for all nodes $v$, 
it follows that the probability of discovering layer $k$ during the $s$-th step is bounded by
\[
 \pr( \cD_{i,k,s} )
 \weq \sum_{v=1}^m \sum_{\cE_i} \pr( \cD_{i,k,s} \cond \cA_{i, s, v}, \cE_i ) \, \pr ( \cA_{i, s, v}, \cE_i )
 \wle \frac{X_k}{m-(t-1)}.
\]
Now, due to $\abs{\bar V_{i,t}} \le \sum_{k \in W_{i,t}} X_k$, we conclude that
\[
 \E \abs{\bar V_{i,t}}
 \wle \sum_{k=1}^n X_k \pr( W_{i,t} \ni k )
 \weq \sum_{k=1}^n X_k \sum_{s=1}^t \pr( \cD_{i,k,s} ).
\]
Hence \eqref{eq:NumberCovered} follows by combining the last two inequalities.

\subsection{Compound Poisson approximation of the upper bound}

\begin{lemma}
\label{the:CompoundPoissonUpperBound}
For any $1 \le \omega \le \frac12 m$, the distribution of the random integer $Z^+_\omega$ defined in \eqref{eq:OffspringUpper} is approximated by
\[
 \dtv( \law(Z^+_\omega), \CPoi(\lambda_n, \bar g_n )
 \wle 4 m^{-2} \sum_{k=1}^n X_k^2 + 2 \omega m^{-2} \sum_{k=1}^n X_k,
\]
where $\lambda_n = \frac{n}{m} (P_n)_{10}$ and $\bar g_n(t) = \int \Bin^+(x-1,y)(t) \frac{x P_n(dx,dy)}{(P_n)_{10}}$.
\end{lemma}
\begin{proof}
Denote $n_{xy}$ the number of layers of type $(x,y)$, and let $K$ be the set of layer types for which $n_{xy} > 0$. Define 
\[
 Z^+
 = \sum_{(x,y) \in K} \sum_{j=1}^{N^+_{xy}} D_{xy}(j),
 \quad
 Y^+
 = \sum_{(x,y) \in K} \sum_{j=1}^{M^+_{xy}} D_{xy}(j),
 \quad
 Y
 = \sum_{(x,y) \in K} \sum_{j=1}^{M_{xy}} D_{xy}(j),
\]
where the random variables appearing in the sums are mutually independent and distributed according to
\[
 \law(N^+_{xy}) = \Bin(n_{xy}, \frac{x}{m-\omega}),
 \quad
 \law(M^+_{xy}) = \Poi( n_{xy} \frac{x}{m-\omega} ),
 \quad
 \law(M_{xy}) = \Poi( n_{xy} \frac{x}{m} ),
\]
together with $\law(D_{xy}(j)) = \Bin^+(x-1,y)$.

Then we observe that $\law(Z^+) = \law(Z^+_\omega)$. Moreover, we observe that $Y = \sum_{(x,y) \in K} Y_{xy}$ where the summands are mutually independent with $\law(Y_{xy}) = \CPoi(\lambda_{xy}, \Bin^+(x-1,y) )$ with $\lambda_{xy} = n_{xy} \frac{x}{m}$. Hence by Lemma~\ref{the:CPoiSum}, the distribution of $Y$ is compound Poisson $\CPoi(\lambda, \bar g)$ with rate parameter
\[
 \lambda
 \weq \sum_{(x,y) \in K} \lambda_{xy} \sum_{(x,y) \in K} n_{xy} \frac{x}{m}
 \weq \frac{n}{m} (P_n)_{10}
\]
and increment distribution
\[
 \bar g
 \weq \sum_{(x,y) \in K} \frac{\lambda_{xy}}{\lambda} \Bin^+(x-1,y)
 \weq \sum_{(x,y) \in K} \frac{x n_{xy}}{n (P_n)_{10}} \Bin^+(x-1,y)
 \weq \int \Bin^+(x-1,y) \frac{x P_n(dx,dy)}{(P_n)_{10}}.
\]

By Lemma~\ref{the:dtvBinPoi}, $\dtv( \law( N^+_{xy} ), \law( M^+_{xy} ) ) \le n_{xy} (\frac{x}{m-\omega})^2$. Hence a basic coupling implies that
\[
 \dtv( \law( Z^+ ), \law( Y^+ ) )
 \wle \sum_{(x,y) \in K}  n_{xy} (\frac{x}{m-\omega})^2
 \weq \frac{1}{(m-\omega)^2} \sum_{k=1}^n X_k^2.
\]
Furthermore, Lemma~\ref{the:dtvPoiPoi} implies that
\[
 \dtv( \law( M^+_{xy} ), \law( M_{xy} ) )
 \wle n_{xy} (\frac{x}{m-\omega} - \frac{x}{m})
 \weq x n_{xy} \frac{\omega}{m(m-\omega)}.
\] 
Hence another basic coupling implies that
\[
 \dtv( \law( Y^+ ), \law( Y ) )
 \wle \sum_{(x,y) \in K} x n_{xy} \frac{\omega}{m(m-\omega)}
 \weq \frac{\omega}{m(m-\omega)} \sum_{k=1}^n X_k.
\]
Hence for $\omega \le \frac12 m$ it follows that
\begin{align*}
 \dtv( \law( Z^+_\omega ), \CPoi(\lambda, \bar g ) )
 &\weq \dtv( \law( Z^+ ), \law( Y ) ) \\
 &\weq \dtv( \law( Z^+ ), \law( Y^+ ) ) + \dtv( \law( Y^+ ), \law( Y ) ) \\
 &\wle \frac{1}{(m-\omega)^2} \sum_{k=1}^n X_k^2 + \frac{\omega}{m(m-\omega)} \sum_{k=1}^n X_k \\
 &\wle 4 m^{-2} \sum_{k=1}^n X_k^2 + 2 \omega m^{-2} \sum_{k=1}^n X_k.
\end{align*}
\end{proof}

\subsection{Large-scale analysis of the upper bound}

Consider a sequence of models $(G^{(n)}: n \ge 1)$ such that the $n$-th model has deterministic layer sizes $(X_1^{(n)}, \dots, X^{(n)}_n)$ and deterministic layer strengths $(Y_1^{(n)}, \dots, Y^{(n)}_n)$, where the empirical distribution of the layer types $(X^{(n)}_k, Y^{(n)}_k)$ in the $n$-th model is $P_n$.  For a technical compactification argument, we assume that
\begin{equation}
 \label{eq:NoRedundantLayers}
 \inf_{n \ge 1} \min_{1 \le k \le n} X^{(n)}_k \wge 2,
\end{equation}
and
\begin{equation}
 \label{eq:NoWeakSmallLayers}
 \inf_{n \ge 1} \min_{k: X_k^{(n)} \le M} Y^{(n)}_k  > 0
 \quad \text{for all $M$}.
\end{equation}
Condition \eqref{eq:NoRedundantLayers} is natural because layers of size less than two do not contribute and links to the overlay graph and may therefore be ignored as redundant. Condition \eqref{eq:NoWeakSmallLayers} means that there are \emph{no weak small layers}, and this is natural for real-world models in which layer strength decreases with respect to layer size.

\begin{lemma}
\label{the:ProbabilityLargeComponent}
Assume that $\frac{m}{n} \to \beta \in (0,\infty)$ and $P_n \weakto P$  with $(P_n)_{10} \to (P)_{10} \in (0,\infty)$. Assume also that there exist some constants $M < \infty$ and $\epsilon > 0$ such that
\[
 2 \le X^{(n)}_k \le M
 \quad \text{and} \quad
 \epsilon \le Y^{(n)}_k \le 1
\]
for all $n \ge 1$ and all $1 \le k \le n$. Then
\[
 \pr( \abs{C_i^{(n)}} \ge \omega_n)
 \wle \zeta(\CPoi(\lambda, \bar g)) + o(1)
\]
for any $\omega_n \gg 1$, where $\zeta$ is the smallest nonnegative fixed point of the generating function of $\CPoi(\lambda, \bar g)$ with $\lambda = \mu (P)_{10}$ and $\bar g(t) = \int \Bin^+(x-1,y)(t) \frac{x P(dx,dy)}{(P)_{10}}$.
\end{lemma}
\begin{proof}
Denote $M_n := \max_k X^{(n)}_k$ and $\hat{q}_n := \min_{k} \left( 1 - \Bin^+(X^{(n)}_k-1,Y^{(n)}_k})(0) \right)$. Then $M_n \le M$ and $\hat q_n \ge \epsilon$ due to the fact that $1-\Bin^+(x-1,y)(0) \ge y$ for all $x \ge 2$. Choose an arbitrary sequence $1 \ll \omega_n' \ll m^{1/2}$ such that $\omega'_n \le \omega_n$. Then $\pr( \abs{C_i^{(n)}} \ge \omega_n) \le \pr( \abs{C_i^{(n)}} \ge \omega_n')$ implies that it is sufficient to prove the claim for $\omega_n'$ instead of $\omega_n$.  Therefore, without loss of generality, we shall from now on assume that $1 \ll \omega_n \ll m^{1/2}$.

By Lemma~\ref{the:GiantUpperFinite}, we find that
\[
 \pr( \abs{C_i^{(n)}} \ge \omega)
 \wle \pr \big( \gw(Z^+_{n,\omega}) \ge \omega \big) + 4 \frac{M_n^2}{\hat q_n^2} m^{-1}\omega^2 + 2 \omega^{-1}.
\]
By Lemma~\ref{the:CompoundPoissonUpperBound}, the offspring distribution above is approximated by
\[
 \dtv( Z^+_{n,\omega},  Y_n )
 \wle 4 m^{-2} n \int x^2 P_n(dx,dy) + 2 \omega m^{-2} n \int x P_n(dx,dy)
 \wle 4 (M_n + \omega) m^{-2} n  (P_n)_{10},
\]
where $Y_n \eqst \CPoi(\lambda_n, \bar g_n)$ with $\lambda_n = \frac{n}{m} (P_n)_{10}$ and $\bar g_n(t) = \int \Bin^+(x-1,y)(t) \frac{x P_n(dx,dy)}{(P_n)_{10}}$. Hence by Lemma~\ref{the:GWHittingTimes},
\[
 \left| \pr( \gw(Z^+_{n,\omega}) \ge \omega ) - \pr( \gw(Y_n) \ge \omega ) \right|
 \wle \omega \, \dtv( Z^+_{n,\omega},  Y_n ).
\]
Hence
\[
 \pr( \abs{C_i^{(n)}} \ge \omega)
 \wle \pr \big( \gw(Y_n) \ge \omega \big) + 4 \omega (M_n + \omega) m^{-2} n (P_n)_{10}
 + 4 \frac{M_n^2}{\hat q_n^2} m^{-1}\omega^2 + 2 \omega^{-1}.
\]
When $M_n, \hat q_n, (P_n)_{10} \asymp 1$, $1 \ll \omega \ll m^{1/2}$ with $m \asymp n$, we obtain
\[
 \pr( \abs{C_i^{(n)}} \ge \omega)
 \wle \pr \big( \gw(Y_n) \ge \omega \big) + o(1).
\]
When $\frac{m}{n} \to \beta \in (0,\infty)$ and $P_n \weakto P$  with $(P_n)_{10} \to (P)_{10} \in (0,\infty)$, then \rnote{verify} 
$\lambda_n \to \lambda = \mu (P)_{10}$ and $\bar g_n \to \bar g(t) = \int \Bin^+(x-1,y)(t) \frac{x P(dx,dy)}{(P)_{10}}$. Hence $\law(Y_n) \to \law(Y) = \CPoi(\lambda, \bar g)$ weakly \rnote{verify}.
With Lemma~\ref{the:GWLarge} we may now conclude that $\pr \big( \gw(Y_n) \ge \omega \big) \to \pr \big( \gw(Y) = \infty \big) = \zeta$, and the claim follows.
\end{proof}

\subsection{Upper bound for the number of nodes in big components (asymptotic, compact setting)}

\begin{lemma}
Under the same assumptions as in Lemma~\ref{the:ProbabilityLargeComponent}, the largest component size in the model is bounded by
\[
 \cmax(G^{(n)}) \wle \zeta m + o_\pr(m)
\] 
for any $1 \ll \omega_n \le \zeta m$, where $\zeta$ is the smallest nonnegative fixed point of the generating function of $\CPoi(\lambda, \bar g)$ with $\lambda = \mu (P)_{10}$ and $\bar g(t) = \int \Bin^+(x-1,y)(t) \frac{x P(dx,dy)}{(P)_{10}}$.
\end{lemma} 
\begin{proof}[Proof sketch]
Fix $\epsilon > 0$.  Denote by $\cmax(G)$ the size of the largest component in graph $G$, and by $B_t(G)$ the set nodes with $G$-component of size at least $t$. Note that $\cmax(G) \le \max\{ \abs{B_t(G)}, t \}$ for all $t \ge 0$. Therefore,
\begin{equation}
 \label{eq:LargeComponentUpper}
 \pr( \cmax(G^{(n)}) \ge t )
 \wle \pr( \abs{B_t(G^{(n)})} \ge t )
 \wle \pr( \abs{B_\omega(G^{(n)})} \ge t )
\end{equation}
for all $\omega \le t$. Especially,
\[
 \pr \left( \cmax(G^{(n)}) \ge (\zeta+\epsilon) m \right)
 \wle \pr \left( \abs{B_\omega(G^{(n)})} \ge (\zeta+\epsilon) m \right)
\]
for all $\omega \le (\zeta+\epsilon) m$. Fix now an arbitrary $\omega_n \gg 1$ such that $\omega_n \le (\zeta+\epsilon) m_n$, and denote $b_n = \abs{B_{\omega_n}(G^{(n)})}$.
Observe that for all $i$,
\[
 \E b_n
 \weq m \pr( \abs{\abs{C_{G^{(n)}}(i)}} \ge \omega),
\]
and so that by Lemma~\ref{the:ProbabilityLargeComponent} (recalling that all layer types are in a compact set), 
\[
 \E \left( \frac{b_n}{m_n} \right)
 \wle \zeta + o(1).
\]
Hence $\E \left( \frac{b_n}{m_n} \right) \le \zeta + \frac12 \epsilon$ for all large enough $n$, and it follows by Chebyshev's inequality that
\begin{align*}
 \pr \left( \cmax(G^{(n)}) \ge (\zeta+\epsilon) m \right)
 &\wle \pr \left( b_n \ge (\zeta+\epsilon) m \right) \\
 &\wle \pr \left( b_n - \E(b_n) \ge \frac12 \epsilon m \right) \\
 &\wle 4 \epsilon^{-2} m^{-2} \Var( b_n).
\end{align*}
The claim follows after verifying that $\Var(b_n) \ll m^2$. \rnote{To do this, we need double exploration process analysis.}
\end{proof}

\subsection{Compactification argument}

Fix an integer $M \ge 2$, and 
define a \emph{small-layer graph} $G^{(n,M)}$ by
\[
 V(G^{(n,M)}) \weq \{1,\dots,m\}
 \qquad\text{and}\qquad
 E(G^{(n,M)}) \ = \bigcup_{k: X_k^{(n)} \le M} E(G_k^{(n)}).
\] 
Then $G^{(n,M)}$ is a subgraph of $G^{(n)}$ generated by aggregating all layers of size at most $M$. Structurally, the graph 
$G^{(n,M)}$ is also an instance of the overlay graph model with $m$ nodes, $n_M = \sum_{k=1}^n 1(X_k^{(n)} \le M)$ layers, and layer type distribution
\[
 P_{n,M}(A) \weq P_{n}(A \cond K_M), 
\]
where $K_M = \{2,\dots,M\} \times [0,1]$. If we define a similar truncated version of the limiting layer type distribution $P$ by
\[
 P_{\infty,M}(A) \weq P(A \cond K_M),
\]
then under mild regularity, it follows that $P_{n,M} \to P_{\infty, M}$ and $(P_{n,M})_{10} \to (P_{\infty, M})_{10}$ as $n \to \infty$, together with $P_{n,M} \to P_n$ and $P_{\infty,M} \to P$ as $M \to \infty$.  We aim to approximate the size of the largest component $C_1(G^{(n)})$ in $G^{(n)}$ according to
\[
 \frac{\abs{C_1(G^{(n)})}}{m}
 \wapprox \frac{\abs{C_1(G^{(n,M)})}}{m}
 \wapprox \zeta( \CPoi( \lambda_{\infty,M}, \bar g_{\infty,M} ) )
 \wapprox \zeta( \CPoi( \lambda, \bar g ) ),
\]
where $\zeta(f)$ indicates the Galton--Watson survival probability for offspring distribution $f$, 
$\lambda_{\infty,M} = \mu (P_{\infty,M})_{10}$, and $\bar g_{\infty,M}(t) = \int \Bin^+(x-1,y)(t) \frac{x P_{\infty,M}(dx,dy)}{(P_{\infty,M})_{10}} $.

Assume now that the layer types of the original model satisfy conditions \eqref{eq:NoRedundantLayers} and \eqref{eq:NoWeakSmallLayers}.
Because $1 - \Bin^+(x-1,y)(0) = 1 - (1-y)^{x-1} \ge y$ for $x \ge 2$, we see that the layer types of $G^{(n,M)}$ satisfy
\[
 \hat{q}_n
 \weq \min_{k: X^{(n)}_k \le M} \left( 1 - \Bin^+(X^{(n)}_k-1,Y^{(n)}_k)(0) \right)
 \wasymp 1.
\]

By Lemma~\ref{the:ComponentOverlayTruncation}, we see that
\[
 \abs{C_1(G^{(n)})}
 \wle \max\{ \abs{B'_t} + t \abs{V''_{n,M}}, \, t \},
\]
where $B'_t$ is the of nodes with $G^{(n,M)}$-component of size at least $t$, and $V''_{n,M} = \cup_{k: X^{(n)}_k > M} V(G_k^{(n)})$ is the set of nodes covered by one or more big layers.


Observe now that
\[
 n^{-1} \abs{V''_{n,M}}
 \wle  n^{-1} \sum_{k: X_k^{(n)} > M} X_k^{(n)}
 \weq \int x \, 1_{K_M^c}(x,y) \, P_n(dx,dy)
 \wle h(M),
\]
where $h(M) = \sup_{n \ge 1} \int x \, 1_{K_M^c}(x,y) \, P_n(dx,dy)$. Now $P_n \weakto P$  and $(P_n)_{10} \to (P)_{10} < \infty$ imply uniform integrability in the sense that $h(M) \to 0$ as $M \to \infty$. Now choose $1 \ll t_M \ll h(M)^{-1}$. Then
\begin{align*}
 m^{-1} \max_{i} \abs{\abs{C_{G^{(n)}}(i)}}
 &\wle \max\{ m^{-1} \abs{B'_{t_M}} + m^{-1} t_M \abs{V''_{n,M}}, \, m^{-1} t_M \} \\
 &\wle \max\{ m^{-1} \abs{B'_{t_M}} + m^{-1} n t_M h(M), \, m^{-1} t_M \} \\
\end{align*}

\subsection{Component sizes in truncated overlay graphs}

The following result shows how the component sizes of an overlay graph $G$ can be approximated by the component sizes of a truncated overlay graph $G'$ generated by a constrained set of layers.

\begin{lemma}
\label{the:ComponentOverlayTruncationOld}
Let $G = G' \cup G''$ be a union of graphs $G',G''$  with a common node set $V$ and link sets $E(G') = \cup_{k \in W'} E(G_k)$ and $E(G'') = \cup_{k \in W''} E(G_k)$, where $W'$ and $W''$ are disjoint.
Then
\[
 \max_{i \in V} \abs{C_i(G')}
 \wle \max_{i \in V} \abs{C_i(G)}
 \wle \max\{ \abs{B'_t} + t \abs{V''}, \, t \},
\]
for all integers $t \ge 0$, where $V'' = \cup_{k \in W''} V(G_k)$ and $B'_t = \{i \in V: \abs{C_i(G')} \ge t\}$.
\end{lemma}
\begin{proof}
The first inequality is trivial because $G'$ is a subgraph of $G$. To verify the second inequality, let us abbreviate $C_i = C_i(G)$ and $C'_i = C_i(G')$, and denote by $B_t = \{i \in V: \abs{C_i(G)} \ge t\}$ the set of nodes with $G$-component of size at least $t$.  If $i \in B_t \setminus B'_t$, then $C_i \ge t$ but $\abs{C_i'} < t$, and we see that $C'_i$ must intersect  $V''$. Hence $C'_i = C'_j$ for some $j \in V''$ such that $\abs{C'_j} < t$. We conclude that
\[
 B_t \setminus B'_t
 \ \subset \nhquad \bigcup_{j \in V'': \abs{C'_j} < t} C'_j.
\]
Hence we conclude that
\[
 \abs{B_t}
 \wle \abs{B'_t} + \abs{B_t \setminus B'_t}
 \wle \abs{B'_t} + t \abs{V''}.
\]
The second inequality now follows by observing that $\max_i C_i \le \max\{ \abs{B_t}, t \}$. 
\end{proof}

\subsection{Old generic text}
The giant component analysis requires care, because the random 2-section induces dependencies.  The Galton--Watson branching process with the limiting degree distribution being the offspring distribution might not produce the correct analysis for the existence of giants (compare with the distribution in \cite[Eq. (3.2)]{Britton_Deijfen_Lageras_Lindholm_2008}). The key difference to standard RIGs is that here the set of nodes in the connected component of $i$ which share a layer with $i$ is a larger set than the 1-neighbourhood of $i$. This is why a Galton--Watson process with the limiting degree distribution as the offspring distribution should give a lower bound for the size of component of $i$.

A correct answer should come using a similar branching process as in \cite{Britton_Deijfen_Lageras_Lindholm_2008}. To heuristically derive the correct offspring distribution, let us proceed as follows. Denote by $\bar G_k$ the transitive closure of $G_k$ (the graph where $ij$ is linked when there is a path in $G_k$ from $i$ to $j$). The transitive neighbourhood of node $i$ in layer $G_k$ equals $N_{\bar G_k}(i) = C_{G_k}(i) \setminus \{i\}$ and has cardinality $\deg_{\bar G_k}(i)$, where $C_{G_k}(i)$ denotes the connected component of node $i$. In a sparse setting a large proportion of the sets $\{N_{\bar G_k}(i): 1 \le k \le n\} $ are w.h.p.\ disjoint, and on this event the number of nodes connected to node $i$ by a single-layer path equals
\[
 R_i
 \ := \ \sum_{k} 1_{V(G_k)}(i) \deg_{\bar G_k}(i).
\]
This leads to the following conjecture, where we denote by $f_{x_k,q_k}$ the distribution of $\deg_{\bar G^k}(i)$.

\subsection{Lower bound exploration (not super old)}

An exploration process for a lower bound of the component $C_w$ of node $w$ is defined as follows. Here $\cK$ is a finite set  of layer types $(x,q)$ with $2 \le x \le M$ and $0 < q \le 1$, with cardinality $K = \abs{\cK}$. We assume that $\pimin = \min_{(x,q) \in \cK} \frac{n_{x,q}}{n} > 0$, where $n_{x,q}$ is the number of layers of type $(x,q)$. Fix a small $0 < \epsilon < 1$, a maximum step count $1 \le t_1 \le m$, and a maximum mark count $1 \le n_1 \le n$.

\paragraph{Initialisation.} Use your favourite orderings to create an ordered list of nodes $(v_1, \dots, v_m)$ and an ordered list of layers $(F_1,\dots,F_n)$. In what follows, \emph{first}, \emph{last}, etc.\ refer to these orderings. Create one more list called the \emph{exploration queue} which initially only contains node $w$.

\paragraph{Exploration step $t \ge 1$.} 

If the exploration queue is empty or $t > t_1$, stop. Otherwise proceed as follows.

\begin{enumerate}[(a)]
\item  \emph{Thinning}. Traverse the list of layer types and for each layer type, declare the first $n_{x,q}(t) = n_{x,q} - (t-1) n_1$ previously unmarked layers of type $(x,q)$ as \emph{available} for step $t$.

\item \emph{Marking}. Select the oldest node $u_t$ from the exploration queue. Traverse the list of layers declared available in step (a), and declare an available layer $F$ \emph{marked} if $V(F)$ contains node $u_t$.
Verify that the number of layers of type $(x,q)$ that were marked is at most $n_1$ for each layer type. If not, stop the exploration.

\item \emph{Admission}.
Traverse the list of layers marked in (b) and for each marked layer $F$, denote by $H(F)$ the union of $\{w\}$ and the set of nodes covered by the layers that were marked before $F$ during the current or previous steps.  Declare a marked layer $F$ \emph{regular} if $V(F) \cap H(F) = \{u_t\}$. In this case flip a coin and declare $F$ \emph{admitted} with probability
\[
 p_{\rm coin}(x, h, t)
 \weq \frac{(1-\epsilon)\frac{m-(t-1)}{m}}{\phi(x-1, h-t, m-t)},
 \qquad
 \phi(x, h, m)
 \weq \frac{\binom{m-h}{x}}{\binom{m}{x}},
 \]
where $x = \abs{V(F)}$ and $h = \abs{H(F)}$. \rnote{Comment $\phi(x-1, h-t, m-t) \ge 1 - \frac{xh}{m}$}
\begin{rcomm}
Each yet unmarked layer is marked in step $t$ with probability $\frac{x}{m-(t-1)}$.\\
Each layer that was marked in step $t$ gets admitted with probability $(1-\epsilon)\frac{m-(t-1)}{m}$.\\
Hence the probability that a layer gets marked and admitted equals $(1-\epsilon) \frac{x}{m}$.
\end{rcomm}

\item \emph{layer exploration}. Traverse the list of layers, and whenever encountering a layer $F$ which was admitted in (b), append all nodes in the $F$-component of $u_t$ to the exploration queue in the order of node labels. Then remove $u_t$ from the exploration queue.
\end{enumerate}

\subsubsection{Number of marked layers}

Given the past information and assuming that the exploration reaches step $t$, the number of layers of type $(x,q)$ marked in step $t$ is $\Bin( n_{x,q} - (t-1) n_1, \frac{x}{m-(t-1)})$. Note that $\E e^{\Bin( n, p )} = \left( 1-p + p e \right)^n \le e^{(e-1) np} \le e^{2 np}$, and hence for $n_{x,q} \le n \le m$, and $t \le \pimin \epsilon \omega^{-1} n \le \pimin \epsilon n \le \frac12 n$ , we have
\[
 \frac{n x}{n-t}
 \wle \frac{n x}{n-\pimin \epsilon n}
 \weq \frac{x}{1-\pimin \epsilon}
 \wle 2 x.
\]
Hence the conditional probability of marking more than $\omega$ layers of type $(x,q)$ during step $t$ is bounded by
\begin{align*}
 &\pr \Bigg( \Bin\Big( n_{x,q} - (t-1) \omega, \frac{x}{m-(t-1)}\Big) > \omega \Bigg) \\
 &\wle \pr \Bigg( \Bin\Big( n_{x,q}, \frac{x}{m-t}\Big) > \omega \Bigg) \\
 &\wle e^{-\omega} e^{\frac{2 n x}{m-t}}
 \wle e^{-\omega} e^{4 x},
\end{align*}
and
\[
 \pr\left( \bigcup_{x,q} \bigcup_{1 \le t \le \pimin \epsilon \omega^{-1} n} \Bin\left( n_{x,q}, \frac{x}{m-(t-1)}\right) > \omega \right)
 \wle \abs{K}  \pimin \epsilon \omega^{-1} n e^{-\omega} e^{4 \xmax}.
\]
The right side tends to zero when $K$ and $\xmax$ do not depend on scale, and when
$\omega e^\omega \gg n$, for example $\omega \sim 1.1 \log n$. Hence with high probability, the exploration does not stop due to marking too many layers.


\subsection{Lower bound exploration (original)}

An exploration process for a lower bound of the component $C_w$ of node $w$ is defined as follows. Here $\cK$ is a finite set  of layer types $(x,q)$ with $2 \le x \le M$ and $0 < q \le 1$, with cardinality $K = \abs{\cK}$. We assume that $\pimin = \min_{(x,q) \in \cK} \frac{n_{x,q}}{n} > 0$, where $n_{x,q}$ is the number of layers of type $(x,q)$. Fix a small $0 < \epsilon < 1$, a maximum step count $1 \le t_1 \le m$, and a maximum mark count $1 \le n_1 \le n$.

\paragraph{Initialisation.} Use your favourite orderings to create an ordered list of nodes $(v_1, \dots, v_m)$ and an ordered list of layers $(F_1,\dots,F_n)$. In what follows, \emph{first}, \emph{last}, etc.\ refer to these orderings. Create one more list called the \emph{exploration queue} which initially only contains node $w$.

\paragraph{Exploration step $t \ge 1$.} 

If the exploration queue is empty or $t > t_1$, stop. Otherwise proceed as follows.

\begin{enumerate}[(a)]
\item  \emph{Thinning}. Traverse the list of layer types and for each layer type, declare the first $n_{x,q}(t) = n_{x,q} - (t-1) n_1$ previously unmarked layers of type $(x,q)$ as \emph{available} for step $t$.

\item \emph{Marking and admission}. Select the oldest node $u_t$ from the exploration queue. Traverse the list of layers declared available in step (a), and declare an available layer $F$ \emph{marked} if $V(F)$ contains node $u_t$.
For each such layer, denote by $H(F)$ the union of $\{w\}$ and the set of nodes covered by the layers that were marked before $F$ during the current or previous steps.  Declare a marked layer $F$ \emph{regular} if $V(F) \cap H(F) = \{u_t\}$. In this case flip a coin and declare $F$ \emph{admitted} with probability
\[
 p_{\rm coin}(x, h, t)
 \weq \frac{(1-\epsilon)\frac{x}{m}}{p_{\rm reg}(x, h, t)},
 \qquad
 p_{\rm reg}(x, h, t)
 \weq \frac{\binom{m-h}{x-1}}{\binom{m-(t-1)}{x}},
 \]
where $x = \abs{V(F)}$ and $h = \abs{H(F)}$. \rnote{Comment $\le 1$}

\item \emph{Mark counting}.  Verify that the number of layers of type $(x,q)$ marked in (b) is at most $n_1$ for each layer type. If not, stop the exploration.

\item \emph{layer exploration}. Traverse the list of layers, and whenever encountering a layer $F$ which was admitted in (b), append all nodes in the $F$-component of $u_t$ to the exploration queue in the order of node labels. Then remove $u_t$ from the exploration queue.

\end{enumerate}

\subsection{Analysis of lower bound exploration}

Fix some integers $t$ and $h$, some distinct nodes $w=u_1, u_2,\dots, u_t$, and some sets of layer indices $A_{x,q}(t) \subset [n]$ of sizes $n_{x,q}(t) = n_{x,q} - (t-1) n_1$, etc.

Denote by $\cE$ the event that the exploration is not stopped before (b) of step $t$, and that for each layer type, the set of layer indices of type $(x,q)$ declared available for step $t$ equals $A_{x,q}(t)$, etc. Any layer that is declared available in (a) of step $t$ has not been marked during the steps $s \le t-1$. Hence conditionally on the event $\cE$, the available layers are mutually independent, and the distribution of $V(F_k)$ for $k \in A_{x,q}(t)$ is uniform on the $x$-sets of the set of $m-(t-1)$ yet unexplored nodes. Hence, the probability that $F_k$ becomes marked in (b) of step $t$ equals
\[
 \pr( V(F_k) \ni u_t \cond \cE )
 \weq \frac{x}{m-(t-1)},
\]
and the $\cE$-conditional distribution of the number of marked $(x,q)$-layers is $\Bin( n_{x,q}(t), \frac{x}{m-(t-1)})$, and so on.

Now fix a layer index $k \in A_{x,q}(t)$, and denote by $\cE'$ the event that $\cE$ occurs, and $\abs{H(F_k)} = h$. Conditionally on $\cE'$ our only information about $F_k$ is that $V(F_k)$ does not contain $u_1,\dots, u_{t-1}$. Conditional on $\cE'$, layer $F_k$ is hence declared marked with probability
\[
 \pr( V(F_k) \ni u_t \cond \cE' )
 \weq \frac{x}{m-(t-1)},
\]
regular with probability
\[
 \pr \Big( V(F_k) \ni u_t, \, V(F_k) \cap H(F_k) = \{u_t\} \, \cond \, \cE' \Big)
 \weq p_{\rm reg}(x, h, t),
\]
and admitted with probability
\[
 p_{\rm reg}(x, h, t) \, p_{\rm coin}(x, h, t) 
 \weq (1-\epsilon)\frac{x}{m}
\]
Let $B_{t ,k}$ be the indicator of the event that layer $F_k$ is admitted during step $t$.
The above formula shows that, conditionally on $\cE$ (knowing that the exploration step $t$ is started and the available layers are known), the admission indicators are mutually independent, and $B_{t,k} = 1$ with probability $(1-\epsilon)\frac{x_k}{m}$. Now the $\cE'$-conditional probability of $F_k$ being admitted, given it was marked, equals
\[
 \frac{(1-\epsilon)\frac{x}{m}}{\frac{x}{m-(t-1)}}
 \weq (1-\epsilon) \frac{m-(t-1)}{m}
 \weq (1-\epsilon) \left( 1 - \frac{t-1}{m} \right).
\]

\rnote{Conditionally on not marking too much, we might get bias.}


\subsubsection{Adaptive thinning lemma}
Text.

\begin{bcomm}
\begin{algorithm}[H]
\DontPrintSemicolon
\KwInput{Finite set $S$. List of subsets $(X_1,\dots,X_n)$ of $S$. List of numbers $(U_1,\dots,U_n)$ in $(0,1)$. Threshold parameter $\epsilon \in [0,1]$.
}
\KwOutput{Index set $J \subset [n]$.}
~\\
Initialise: Set $m \leftarrow \abs{S}$. Set $H \leftarrow \emptyset$.
Define $\phi(x,h) = \binom{m-h}{x} \binom{m}{x}^{-1}$\\

\For{$t=1,\dots,n$}
{
 If $X_t \cap H = \emptyset$ and $U_t \le \frac{1-\epsilon}{\phi(\abs{X_t}, \abs{H})}$, set $B_t \leftarrow 1$ and $H \leftarrow H \cup X_t$.
}
$J \leftarrow \{t \in [n]: B_t=1\}$
\caption{Randomized set thinning.}
\label{algo:SetThinning}
\end{algorithm}
\end{bcomm}

\begin{lemma}
\label{the:AdaptiveThinningNew}
Let $X_1,\dots,X_n$ be (random or nonrandom) subsets of a finite set $S$ with sizes $x_1, \dots, x_n$ such that $\norm{x}_1 \norm{x}_\infty \le \epsilon \abs{S}$ for some $0 \le \epsilon \le 1$, where $\norm{x}_1 = \sum_i x_i$ and $\norm{x}_\infty = \max_i x_i$. Let $U_1,\dots,U_n$ be independent uniform random numbers in $(0,1)$. Then the output $J \subset [n]$ of Algorithm~\ref{algo:SetThinning} 
satisfies:
\begin{enumerate}[(i)]
\item \label{ite:AdaptiveThinningNew1}
The sets $\{X_j: j \in J\}$ are disjoint almost surely,
\item \label{ite:AdaptiveThinningNew2}
The indicator variables $B_i = 1(J \ni i)$, $1 \le i \le n$, are mutually independent and $\Ber(1-\epsilon)$-distributed.
\end{enumerate}
\end{lemma}

\begin{lemma}
\label{the:AdaptiveThinning}
Let $X_1,\dots,X_n$ be independent uniformly random sets in $[m]$ with sizes $x_1, \dots, x_n$ such that $\norm{x}_1 \norm{x}_\infty \le \epsilon m$ for some $0 \le \epsilon \le 1$, where $\norm{x}_1 = \sum_i x_i$ and $\norm{x}_\infty = \max_i x_i$. Then there exists a random set $J \subset [n]$ with the properties:
\begin{enumerate}[(i)]
\item \label{ite:AdaptiveThinning1}
 The random sets $\{X_j: j \in J\}$ are disjoint almost surely,
\item \label{ite:AdaptiveThinning2}
 The indicator variables $B_i = 1(J \ni i)$, $1 \le i \le n$, are mutually independent and $\Ber(1-\epsilon)$-distributed.
\end{enumerate}
\end{lemma}
\begin{proof}
Fix uniformly distributed random numbers $U_1,\dots,U_n$ in the continuous interval $(0,1)$ which are mutually independent, and independent of the random sets $X_1,\dots,X_n$. We define $B_1 = 1(U_1 \le 1-\epsilon)$,  and define recursively for
$1 \le t \le n-1$,
\begin{equation}
 \label{eq:AdaptiveThinningRecursion}
 B_{t+1}
 \weq 1 \Big( X_{t+1} \cap H_t = \emptyset \Big) \, 1\Big( U_{t+1} \le \frac{1-\epsilon}{\phi(x_{t+1}, \abs{H_t})} \Big),
\end{equation}
where $H_t = \cup_{s \le t: B_s=1} X_s$ and $\phi(x,h) = \frac{\binom{m-h}{x}}{\binom{m}{x}}$. Then we let
$J = \{ i: B_i = 1 \}$.

Equation \eqref{eq:AdaptiveThinningRecursion} guarantees that $t+1$ is included in $J$ only if $X_{t+1}$ is disjoint from all the previously included sets. Hence \eqref{ite:AdaptiveThinning1} is valid.

To verify \eqref{ite:AdaptiveThinning2}, observe first that $\phi(x,h)$ equals the probability that a uniformly random $x$-set in $[m]$ does not overlap a given $h$-set in $[m]$.
By the union bound, we find that
\begin{equation}
 \label{eq:AdaptiveThinningBound}
 \phi(x,h) \wge 1 - x \frac{h}{m}.
\end{equation}
Let $\cF_t$ be the sigma-algebra generated by $X_1,\dots,X_t$ and $U_1,\dots, U_t$. Because $H_t$ is measurable with respect to $\cF_t$, and because $X_{t+1}$ and $U_{t+1}$ are independent of each other and independent of $\cF_t$, it follows that
\begin{align*}
 \pr_{\cF_t}( B_{t+1}=1 )
 &\weq \pr_{\cF_t} \Big( X_{t+1} \cap H_t = \emptyset, \ U_{t+1} \le \frac{1-\epsilon}{\phi(x_{t+1}, \abs{H_t})} \Big) \\
 &\weq \pr_{\cF_t} \Big( X_{t+1} \cap H_t = \emptyset \Big) \,
    \pr_{\cF_t} \Big( U_{t+1} \le \frac{1-\epsilon}{\phi(x_{t+1}, \abs{H_t})} \Big) \\
 &\weq \phi(x_{t+1}, \abs{H_t} ) \, \min \Big( \frac{1-\epsilon}{\phi(x_{t+1}, \abs{H_t})}, \, 1 \Big).
\end{align*}
Observe next that by \eqref{eq:AdaptiveThinningBound},
\[
 \phi(x_{t+1}, \abs{H_t})
 \wge 1 - \frac{\abs{H_t} x_{t+1} }{m}
 \wge 1 - \frac{\norm{x}_1 \norm{x}_\infty}{m}
 \wge 1 - \epsilon
\]
almost surely. Therefore,
\begin{align*}
 \pr_{\cF_t}( B_{t+1}=1 )
 \weq 1-\epsilon
\end{align*}
almost surely. The above formula implies that $B_{t+1}$ is $\Ber(1-\epsilon)$-distributed and independent of $\cF_t$. Because $B_1,\dots,B_t$ are $\cF_t$-measurable, we conclude that $B_{t+1}$ is independent of $(B_1,\dots,B_t)$. Because this is true for all $1 \le t \le n-1$, we conclude that \eqref{ite:AdaptiveThinning2} is valid.
\end{proof}

\begin{rcomm}
If a layer $F$ of size $x$ is declared available in step $t$, then it has not been marked during the steps $s \le t-1$. Conditionally on this information, $V(F)$ is a uniformly distributed $x$-set in the set of $m-(t-1)$ yet unexplored nodes. Hence, conditionally on this information, the probability that $F$ is marked equals $p_{\rm mark}(x, t) = \frac{x}{m-(t-1)}$ and the probability that $F$ is regular equals $p_{\rm reg}(x, h, t)$. Note that it is harder for layers to be regular that appear later in the list. This is balanced by the randomized admission rule. Namely, now each layer $F$ of type $(x,q)$ is admitted with probability
\[
 p_{\rm adm}(x, h, t)
 \weq p_{\rm reg}(x, h, t) \, p_{\rm coin}(x, h, t) 
 \weq (1-\epsilon)\frac{x}{m}
\]
as long as $h = \abs{\cH_F}$ satisfies $p_{\rm reg}(x, h, t) \ge (1-\epsilon)\frac{x}{m}$.  In this case each marked layer of type $(x,q)$ is admitted with probability
\[
 \frac{p_{\rm adm}(x, h, t)}{p_{\rm mark}(x, t) }
 \weq \frac{(1-\epsilon)\frac{x}{m}}{\frac{x}{m-(t-1)}}
 \weq (1-\epsilon) \frac{m-(t-1)}{m}.
\]

Given the past information, the number $N_{x,q}(t)$ of layers of type $(x,q)$ marked in step $t$ is $\Bin( n'_{x,q} - (t-1) \omega, \frac{x}{m-(t-1)})$.

The number of new nodes discovered in step $t$, see (72) in MB notes, can be represented as
\[
 \sum_{(x,q) \in K} \sum_{\ell=1}^{X^{(t)}_{x,q}} B^{(t)}_{x,q}(\ell) D^*_{x,q}(\ell),
\]
where $B^{(t)}_{x,q}(t,\ell)$ are Bernoulli-distributed with success probability $(1-\epsilon) \frac{m-(t-1)}{m}$, $X^{(t)}_{x,q}$ represents a $\Bin( n'_{x,q} - (t-1) \omega, \frac{x}{m-(t-1)})$-distribution conditioned to be at most $\omega$, and $D^*_{x,q}(\ell)$ represents the transitive degree distribution of an $(x,q)$-layer.
\end{rcomm}

\subsection{Alternative lower bound exploration}

An exploration process for a lower bound of the component $C_w$ of node $w$ is defined as follows. Here $\cK$ is a finite set  of layer types $(x,q)$ with $2 \le x \le M$ and $0 < q \le 1$, with cardinality $K = \abs{\cK}$. We assume that $\pimin = \min_{(x,q) \in \cK} \frac{n_{x,q}}{n} > 0$, where $n_{x,q}$ is the number of layers of type $(x,q)$. Fix a small $0 < \epsilon < 1$, a maximum step count $1 \le t_1 \le m$, and a maximum mark count $1 \le n_1 \le n$.

\paragraph{Initialisation.} Create three ordered lists by ordering in an arbitrary fashion the sets of nodes, layers, and layer types, respectively. Create one more list called the \emph{exploration queue} which initially only contains node $w$.

\paragraph{Exploration step $t \ge 1$.} 

If the exploration queue is empty or
$t > t_1$, stop. Otherwise proceed as follows.

\begin{enumerate}[(a)]
\item  \emph{Thinning}. Traverse the list of layer types and for each layer type, declare the first $n_{x,q}(t) = n_{x,q} - (t-1) n_1$ previously unmarked layers of type $(x,q)$ as \emph{available} for step $t$. 

\item \emph{Marking}. Select the oldest node $u_t$ from the exploration queue. Traverse the list of layers declared available in step (a), and declare an available layer $F$ \emph{marked} if $V(F)$ contains node $u_t$.

\item \emph{Mark counting}.  Verify that the number of layers of type $(x,q)$ marked in (b) is at most $n_1$ for each layer type. If not, stop the exploration.

\item \emph{Admission}. Traverse the list of layers marked in (b). Denote by $\cH_t$ the union of $\{w\}$ and the set of nodes covered by the layers that were marked in (b) during the current or previous steps.  Declare a marked layer $F$ \emph{regular} if $V(F) \cap \cH_t = \{u_t\}$. In this case flip a coin and declare $F$ \emph{admitted} with probability
\[
 p_{\rm coin}(x, h, t)
 \weq \frac{(1-\epsilon)\frac{x}{m}}{p_{\rm reg}(x, h, t)} \wedge 1,
 \qquad
 p_{\rm reg}(x, h, t)
 \weq \frac{\binom{m-h}{x-1}}{\binom{m-(t-1)}{x}},
 \]
where $x = \abs{V(F)}$ and $h = \abs{\cH_t}$.

\item \emph{layer exploration}. Traverse the list of layers admitted in (d). For any admitted layer $F$, append all nodes in the $F$-component of $u_t$ to the exploration queue. Then remove $u_t$ from the exploration queue.
\end{enumerate}

\subsection{Analysis of alternative lower bound exploration}

Fix some integers $t$ and $h$, some distinct nodes $w=u_1, u_2,\dots, u_t$, and some sets of layers $A_{x,q}(t)$ of sizes $n_{x,q}(t) = n_{x,q} - (t-1) n_1$, etc.

Denote by $\cE$ the event that the exploration is not stopped before phase (a) of step $t$, and that for each layer type, the set of layers of type $(x,q)$ declared available for step $t$ equals $A_{x,q}(t)$, etc. Any layer that is declared available in (a) of step $t$ has not been marked during the steps $s \le t-1$. Hence conditionally on the event $\cE$, the available layers are mutually independent random sets, and the distribution of each $V(F)$ for $F \in A_{x,q}(t)$ is uniform on the $x$-sets of the set of $m-(t-1)$ yet unexplored nodes. Hence, the probability that $F \in A_{x,q}(t)$ becomes marked in (a) of step $t$ equals
\[
 \pr( V(F) \ni u_t \cond \cE )
 \weq \frac{x}{m-(t-1)},
\]
and the $\cE$-conditional distribution of the number of marked $(x,q)$-layers is $\Bin( n_{x,q}(t), \frac{x}{m-(t-1)})$, and so on.

Denote by $\cE'$ the event that $\cE$ occurs and the exploration is not stopped in (c) of step $t$. Conditional on $\cE'$, the number of marked $(x,q)$-layers is $\CondBin( n_{x,q}(t), \frac{x}{m-(t-1)} ,n_1)$-distributed, and these counts are mutually independent.

Denote by $\cE''$ the event that $\cE'$ occurs and that $\abs{\cH_t} = h$. Conditional on $\cE''$, \rnote{given this, are the regularity events independent and good?}

Moreover, the probability that $F$ is regular equals $p_{\rm reg}(x, h, t)$ where $h = \abs{\cH_t}$. We want to restrict to regular layers in the exploration, but we also want the number of explored layers not to be random. This is balanced by the randomized admission rule. Namely, now each layer $F$ of type $(x,q)$ is admitted with probability
\[
 p_{\rm adm}(x, h, t)
 \weq p_{\rm reg}(x, h, t) \, p_{\rm coin}(x, h, t) 
 \weq (1-\epsilon)\frac{x}{m}
\]
as long as $h = \abs{\cH_t}$ satisfies $p_{\rm reg}(x, h, t) \ge (1-\epsilon)\frac{x}{m}$.  In this case each marked layer of type $(x,q)$ is admitted with probability
\[
 \frac{p_{\rm adm}(x, h, t)}{p_{\rm mark}(x, t) }
 \weq \frac{(1-\epsilon)\frac{x}{m}}{\frac{x}{m-(t-1)}}
 \weq (1-\epsilon) \frac{m-(t-1)}{m}.
\]

By Lemma~\ref{the:IntersectionCombinatorics}, we see that $p_{\rm reg}(x, h, t) \ge (1 - \epsilon) \frac{x}{m}$ whenever $\frac{x h}{m} \le \epsilon$. We need an upper bound for $h = \abs{\cH_t}$. Assume that we are performing step $t \ge 1$. If the previous $t-1$ steps have been successful and also (b) and (c) of step $t$, then in each of (b) steps so far, we have marked at most $\abs{K} \omega$ layers. Hence so far we have marked at most $t \abs{K} \omega$ layers, and the number of nodes covered by these layers is at most $h := \abs{\cH_t} \le t \abs{K} \omega \xmax$. Now
\[
 xh
 \wle t \abs{K} \omega \xmax^2,
\]
so that $xh \le \epsilon m$ for 

Hence before starting step $t$, the number of nodes covered by the layers marked so far is at most
\[
 (t-1) \abs{K} \, \omega \xmax
 \wle t \abs{K} \, \omega \xmax
 \wle \pimin \epsilon \omega^{-1} n \abs{K} \, \omega \xmax
 \wle \pimin \epsilon n \abs{K} \, \xmax
\]

\subsection{Intersection probabilities}
\begin{lemma}
\label{the:IntersectionCombinatorics}
Let $A \subset S$ be finite nonempty sets with cardinalities $m_1 \le m_2$, and let $X$ be a uniformly random $x$-element subset of $S$ with $1 \le x \le m_2$. Then
\begin{equation}
 \label{eq:IntersectionCombinatorics1}
 \pr( X \cap A = \emptyset)
 \weq \frac{\binom{m_2-m_1}{x}}{\binom{m_2}{x}}
 \wge 1 - x \frac{m_1}{m_2},
\end{equation}
and for any $a \in A$,
\begin{equation}
 \label{eq:IntersectionCombinatorics2}
 \pr( X \cap A = \{a\})
 \weq \frac{\binom{m_2-m_1}{x-1}}{\binom{m_2}{x}}
 \wge \frac{x}{m_2} \left( 1 - (x-1) \frac{m_1-1}{m_2-1} \right).
\end{equation}
Moreover, for any integers $0 \le t-1 \le h \le m$, and any $1 \le x \le m-h$,
\begin{equation}
 \label{eq:IntersectionCombinatorics3}
 p_{\rm reg}(x, h, t)
 \ := \
 \frac{\binom{m-h}{x-1}}{\binom{m-(t-1)}{x}}
 \wge \frac{x}{m} \left( 1 - \frac{x h}{m} \right).
\end{equation}
\end{lemma}
\begin{proof}
The equality in \eqref{eq:IntersectionCombinatorics1} follows by noting that the number of $x$-sets of $S$ which do not intersect $A$ equals $\binom{m_2-m_1}{x}$. The inequality in \eqref{eq:IntersectionCombinatorics1} follows by noting that
\[
 \pr( X \cap A \ne \emptyset)
 \weq \pr( \cup_{a \in A} \{ X \ni a \})
 \wle \sum_{a \in A} \pr(  \{ X \ni a \})
 \weq m_1 \frac{x}{m_2}.
\]
The equality in \eqref{eq:IntersectionCombinatorics2} follows by noting that the number of $x$-sets of $S$ which intersect $A$ precisely at $a$ equals $\binom{m_2-m_1}{x-1}$. To verify the inequality in \eqref{eq:IntersectionCombinatorics2}, note that
\[
 \pr( X \cap A = \{a\} \cond X \ni a)
 \weq \pr( \tilde X \cap \tilde A = \emptyset)
\]
where $\tilde X$ is a uniformly random $(x-1)$-element subset of $\tilde S = S \setminus \{a\}$, and $\tilde A = A \setminus \{a\}$.
Because $\pr( X \ni a ) = \frac{x}{m_2}$, the second inequality follows by applying the first inequality to $\tilde A$, $\tilde S$, and $\tilde X$.

To verify \eqref{eq:IntersectionCombinatorics3}, note that by applying \eqref{eq:IntersectionCombinatorics2} with $m_1 = h-(t-1)$ and $m_2 = m-(t-1)$ we see that
\[
 \frac{\binom{m-h}{x-1}}{\binom{m-(t-1)}{x}}
 \wge \frac{x}{m-(t-1)} \left( 1 - (x-1) \frac{h-t}{m-t} \right).
\]
This implies \eqref{eq:IntersectionCombinatorics3} because $\frac{h-t}{m-t} \le \frac{h}{m}$ for $h \le m$.
\end{proof}

\subsection{Comments on proving the upper bound}
MB 2019-08-13 notes contain proofs of the conjecture where layer strengths are deterministic functions of layer sizes. To generalise the MB approach, assume first that the empirical distributions $\pi^{(\nu)}$ and the limiting distribution $\pi$ all have a common finite support $K \subset \Z_+ \times (0,1)$.
We denote
\[
 \hat{q}
 \ := \ \inf_{\nu \ge 1} \inf_{k \le m^{(\nu)}} ( 1 - f_{x^{(\nu)}_k, q^{(\nu)}_k}(0) )
\]
and note that (because the transitive degree distribution is stochastically larger than the ordinary degree distribution)
\[
 \hat{q}
 \weq  \min_{(x,q) \in K} ( 1 - f_{x,q}(0) )
 \wge \min_{(x,q) \in K} ( 1 - (1-q)^{x-1} )
 \ > \ 0.
\]
due to the finite support assumption. We also denote $M = \max_{(x,q) \in K} x < \infty$.

As in (48), let us represent the key compound Poisson distribution as
\[
 Y
 \weq \sum_{(x,q) \in K} \sum_{s=1}^{\Lambda_{x,q}} T^*_{x,q}(s),
\]
where $\Lambda_{x,q}$ is $\Poi(\lambda_{x,q})$-distributed with rate $\lambda_{x,q} = \mu \pi(x,q)$, and where $T^*_{x,q}(s)$ are distributed according to the transitive degree distribution $f_{x,q}$. We define $Y^+_\epsilon$ and $Y^-_\epsilon$ similarly, replacing $\lambda_{x,q}$ by $\lambda^+_{x,q} = \lambda_{x,q} + \epsilon$ and $\lambda^-_{x,q} = (1-\epsilon) \lambda_{x,q}$, respectively. For the upper bound, we fix $1 < \omega' < n$ and define a compound binomial random integer
\[
 Z^+
 \weq \sum_{(x,q) \in K} \sum_{s=1}^{N_{x,q}^+} T^*_{x,q}(s),
\]
where $N_{x,q}^+$ is $\Bin( n_{x,q}, \frac{x}{m-\omega'})$-distributed and $n_{x,q} = n \pi_{x,q}$ denotes the number of layers of type $(x,q)$.

\subsection{Computing transitive degree distributions}

Let $K_{n,p}$ be the random graph on node set $\{1,\dots,n\}$ where each node pair is linked with probability $p$, independently of other node pairs.
Denote by $f_{n,p}$ the probability distribution of the degree of node $i$ in the transitive closure of $K_{n,p}$. Obviously, $f_{0,p} = f_{1,p} = \delta_0$ and $f_{2,p} = \Ber(p)$. For $n \ge 3$ there is no simple expression available, but the following recursive formulas can be used in numerical computations to compute $f_{n,p}$ and the generating function $\hat f_{n,p}(s) = \sum_k s^k f_{n,p}(k)$. A related recursive system of equations in a continuous-time epidemic model has been derived in \cite[Theorem 2.2]{Andersson_Britton_2000}. Alternatively, for a $f_{n, p}$-distributed random integer $T$, the generating function of $n-1-T$ can be expressed as
\[
 \sum_{r=0}^{n-1} (n-1)_r (1-p)^{r(n-r)} \gon_{p, r}(z)
\]
where $\gon_{p, r}(z) = \gon_r( z \, | \, (1-p)^0, (1-p)^1, \dots)$ is the Gontcharoff polynomial of degree $r$ generated by the sequence $( (1-p)^0, (1-p)^1, \dots )$, see \cite[Section 3.1]{Lefevre_Picard_1990} for details. 

\begin{lemma}
For any $n \ge 1$, $p \in [0,1]$, and $k \ge 0$,
\begin{equation}
 \label{eq:TransitiveRecursion}
 f_{n,p}(k)
 \weq (1 - q^k) f_{n-1,p}(k-1) + q^{k+1} f_{n-1,p}(k).
\end{equation}
where $q=1-p$. Moreover, the generating functions satisfy\rnote{wrong, see proof}
\begin{equation}
 \label{eq:TransitiveRecursionGen}
 \hat f_{n,p}(s)
 \weq s \hat f_{n-1,p}(s) + q(1-s) \hat f_{n-1,p}(qs).
\end{equation}
\end{lemma}
\begin{proof}
Denote by $C_G(i)$ the connected component of a node $i$ in $G = K_{n,p}$, and let $c_G(i) = \abs{C_G(i)}$. Choose another node $j \ne i$ from $G$, and let $F$ be the subgraph of $G$ induced by $V(G) \setminus \{j\}$. Consider the event that $C_F(i) = A$ for some $A \subset V(F)$ of size $\abs{A}=k$. On this event, $C_G(i) = A$ if there are no links from $A$ to $j$, and $C_G(i) = A \cup \{j\}$ otherwise \rnote{wrong, it is possible that $C_G(i) = A \cup B$ where $B$ is the component of $j$ in $G[A^c]$}. Because the link indicators of node pairs between $A$ and $j$ are independent of the event $C_F(i) = A$, it follows that
\[
 \pr( c_G(i) = k \cond C_F(i) = A )
 \weq q^{k}
\]
and
\[
 \pr( c_G(i) = k+1 \cond C_F(i) = A )
 \weq 1 - q^{k}.
\]
Because the above equations are valid for all $\abs{A}=k$, it follows that
\[
 \pr( c_G(i) = k \cond c_F(i) = k )
 \weq q^{k}
\]
and
\[
 \pr( c_G(i) = k+1 \cond c_F(i) = k )
 \weq 1 - q^{k},
\]
from which we conclude that
\[
 \pr( c_G(i) = k )
 \weq (1 - q^{k-1}) \pr( c_F(i) = k-1 ) + q^k \pr( c_F(i) = k ).
\]
The above equality implies \eqref{eq:TransitiveRecursion} after observing that $\pr( c_G(i) = k ) = f_{n,p}(k+1)$, and that $\pr( c_F(i) = k ) = f_{n-1,p}(k+1)$ because $F$ has the same distribution as $K_{n-1,p}$.

Equation \eqref{eq:TransitiveRecursionGen} follows from  by multiplying both sides of \eqref{eq:TransitiveRecursion} by $s^k$ and summing both the resulting equation, keeping in mind that $f_{n,p}(k)$ is nonzero only for $0 \le k \le n-1$.
\end{proof}

The generating function of $\CPoi(\lambda, g)$ equals $e^{\lambda (\hat g(z)-1)}$, where $\hat g(z)$ is the generating function of the increment distribution. Now
\[
 \hat g(z)
 \weq \sum_{s =0}^\infty z^s g(s)
 \weq \int \hat f_{x,q}(z) \ \frac{x \, \pi(dx, dq)}{(\pi)_{10}}.
\]
For a model where the layer sizes $q_k = q(x_k)$ are deterministic functions of the layer strength,
\[
 \hat g(z)
 \weq \sum_{s =0}^\infty z^s g(s)
 \weq \sum_{x=0}^\infty \hat f_{x,q(x)}(z) \ \frac{x \, \pi(x)}{(\pi)_{1}}
\]
with $\pi(x)$ being the (limiting) layer size distribution and $(\pi)_1 = \sum_x x \pi(x)$.

\rnote{When is the mean of $f$ above greater than one?.  We need the mean of $g$, but this seems complicated. These are probably related to Gontcharoff polynomials, see \cite{Ball_Sirl_Trapman_2014}.}

\rnote{What about thinning the RIG using bond percolation with rate $p$? We might get a similar model but with $q_k$ replaced by $pq_k$.}

\section{Summary}

\paragraph{Notations.}
\begin{center}
\scriptsize
\begin{itemize}
\item $\ang{f} = \sum_{k=1}^n f(k)$ for real functions on $[n]$
\item $\pi = \frac{1}{n} \sum_{k=1}^n \delta_{(x_k, q_k)}$ denotes the empirical joint distribution of layer sizes and layer strengths
\item $p_{ab}(k) = \frac{(x_k)_a}{(m)_a} q_k^b$, so that $\ang{p_{ab}}$ equals the expected number of layers covering any particular subgraph with $a$ nodes and $b$ links
\item $(\pi)_{ab} = \frac{1}{n} \sum_{k=1}^n (x_k)_a q_k^b$, so that $\ang{p_{ab}} = \frac{n}{(m)_a} (\pi)_{ab}$
\end{itemize}
\end{center}

\paragraph{Assumption A (Regular layers)} $(\pi)_{10}, (\pi)_{21}, (\pi)_{32}, (\pi)_{33} \asymp 1$, and $1 \le \norm{x}_\infty \ll m$. We assume that $m \ge 2$ always to avoid trivialities, and that $n \ll m^2$ which guarantees that the model is sparse.

Then the link density is $\mu_{21} \asymp m^{-2}n$ and the mean degree is $\asymp m^{-1} n$. Moreover, $\mu_{32}, \mu_{33} \asymp m^{-3}n$, and we have the upper bound
\[
 \ang{p_{21} p_{32}}
 \wle \left( \frac{\norm{x}_\infty}{m} \right)^2 \mu_{33}
 \wasymp \left( \frac{\norm{x}_\infty}{m} \right)^2 m^{-3} n,
\]
and by \eqref{eq:p21squared},
\[
 \ang{p_{21}^2}
 \wle 16 m^{-4} n + 2 \frac{\norm{x}_\infty}{m} \ang{ p_{32}}
 \wasymp \supnorm{x} m^{-4}n 
\]
\begin{rcomm}
A simpler but weaker bound is $\ang{p_{21}^2} \le \left( \frac{\norm{x}_\infty}{m} \right)^2 \ang{ p_{21}} \asymp \supnorm{x}^2 m^{-4} n$.
\end{rcomm}

\rnote{If we also assume that $n \asymp m$, 
then $\mu_{21} \asymp m^{-1} $ and $\mu_{32}, \mu_{33} \asymp m^{-2}$,
and then $\mu_{21}^2 \asymp m^{-2}$ implies that $\ang{p_{21}^2}, \ang{p_{21} p_{32}} \ll \mu_{21}^2$.}

\begin{remark}[Doubly stochastic model]
Fix sequences $(m_\nu)_{\nu \ge 1}$ and $(n_\nu)_{\nu \ge 1}$ of positive integers indexed by a scale parameter $\nu=1,2,\dots$ Fix a sequence $(P_\nu)_{\nu \ge 1}$ of probability measures on $\Z_+ \times [0,1]$ such that $P_\nu( [m_\nu] \times [0,1] ) = 1$, and such that $P_\nu \to P_\infty$ weakly together with $( P_\nu )_{ab} \to ( P_\infty )_{ab}$ for $(a,b) \in \{(1,0),(2,0),(2,1), (3,2), (3,3)\}$, where $( P )_{ab} = \int (x)_a q^b P(dx,dq)$. For each $\nu$, let $(x^{(\nu)}_1,q^{(\nu)}_1), \dots, (x^{(\nu)}_{n_\nu}, q^{(\nu)}_{n_\nu})$ be independent samples from $P_\nu$. Assume that $m_\nu, n_\nu \to \infty$ and that $n_\nu \ll m_\nu^{4/3}$. Then by Lemma~\ref{the:QWLLN_New}, it follows that the sampled sequence satisfies the regular layers assumption with high probability (with respect to the $P_\nu$-sequence).
\end{remark}

\paragraph{Results.}

Under assumption A, by Theorem~\ref{the:DegreeApproximationQuantitative}, for $\lambda = \frac{n}{m} (\pi)_{10}$,
\[
 \dtv \Big( \law(\deg_G(i)), \CPoi(\lambda, g_{10}) \Big)
 \wle \frac{n^2}{2 m (m)_2} (\pi)_{21}^2 + \frac{2n \norm{x}_\infty}{m^2} (\pi)_{21}
 \wasymp m^{-3} n^2 + \norm{x}_\infty m^{-2}n,
\]
so the error vanishes when $n \ll m^{3/2}$ and $\supnorm{x} \ll m^2 n^{-1}$. For the error to vanish, it suffices to assume that $n \lesim m$ and $\supnorm{x} \ll m$. Alternatively, it suffices that $n \ll m^{3/2}$ and $\supnorm{x} \lesim m^{1/2}$. So far have not discussed whether or not $g_{10}$ converges to a limiting distribution.

The two-star density is approximated (Theorem \ref{the:TwoStarDensityDeg}) by
\[
 \pr(\cK_{12})
 \wsim \mu_{32} + \mu_{21}^2
\]
when $n \ll m^2$ and $\supnorm{x} \ll m$. The triangle density is approximated (Theorem~\ref{the:TriangleDensity}) by
\[
 \pr(\cK_3)
 \wsim \mu_{33}
\]
when $n \ll m^{3/2}$ and $\supnorm{x} \ll m$.  The hub-degree dependent two-star density is approximated by (Theorem~\ref{the:TwoStarDensityDeg})
\begin{equation*}
 \begin{aligned}
 &\pr(D = t, \cK_{12})
 \weq \mu_{32} \, f \conv g_{32}(t-2)
 \ + \ \mu_{21}^2 f \conv g_{21} \conv g_{21}(t-2)
 \ + \ \epsilon(t),
 \end{aligned}
\end{equation*}
and the degree-dependent triangle density is approximated by (Theorem~\ref{the:TriangleProbability})
\[
 \pr(D = t, \cK_3)
 \weq \mu_{33} \, f \conv g_{33}(t-2) + \epsilon(t),
\]
where in both cases $\sum_{t \ge 0} \abs{\epsilon(t)} \ll \pr(\cK_{12})$ when $n \ll m^{4/3}$ and $\supnorm{x} \ll m n^{-1/2} \wedge m^{1/2}$. A sufficient condition for this is that $n \lesim m$ and $\supnorm{x} \ll m^{1/2}$. Hence we get the following approximation for the model clustering profile
\[
 \frac{\pr(D = t, \cK_3)}{\pr(D = t, \cK_{12})} 
 \wsim \frac{\mu_{33} \, f \conv g_{33}(t-2)}
 {\mu_{32} \, f \conv g_{32}(t-2) + \mu_{21}^2 f \conv g_{21} \conv g_{21}(t-2)} 
\]
which is accurate when $n \ll m^{4/3}$ and $\supnorm{x} \ll m n^{-1/2} \wedge m^{1/2}$.
\rnote{Question: How does this behave when $m \ll n \ll m^{4/3}$? Then things diverge in the nominator and denominator. Perhaps a Gaussian local limit theorem could help.}

\rnote{Allowing $n \sim m \log m$ might give us connected graphs and consistent statistical estimators.}


\section{Doubly stochastic model --- Randomly generated layer types}

Fix integers $m,n \ge 1$ and a probability measure $P$ on $\Z_+ \times [0,1]$ with mass supported on $[m] \times [0,1]$, consider the random graph obtained by first generating an iid sample $(X_1,Q_1), \dots, (X_n,Q_n)$ from $P$, and conditionally on this sample, letting $G$ be the random graph distributed as earlier, with layer sizes $X_k$ and layer strengths $Q_k$. We call this the \new{doubly stochastic model} with parameter triple $(m,n,P)$.

For each $\nu \ge 1$, consider a triplet $(m_\nu, n_\nu, P_\nu)$ where $P_\nu$ is a probability measure on $\Z_+ \times [0,1]$ supported on $[m_\nu] \times [0,1]$. Consider a model $G^{(\nu)}$ with parameters $(m^{(\nu)}, n^{(\nu)}, x^{(\nu)}, q^{(\nu)})$.

\begin{rcomm}
Let us apply Lemma~\ref{the:PowerBound} with $d=2$ and $c=1$ and $u \ge 1$. Then for any $x \in \Z_+$ and $q \in [0,1]$ such that $0 < x^b q^a \le u$,
\[
 \sup_{(x,q) \in S} \frac{x^2 q^1}{x^b q^a}
 \weq
 \begin{cases}
  1, &\quad 2 \le b, \ 1 \ge a, \\
  u^{2/b-1}, &\quad 2 > b, \ 1/2 \ge a/b, \\
 \infty, &\quad \text{else}.
 \end{cases}
\]
For $u= \delta n$ and $b = 4/3$ and $a = 2/3$, this gives 
\[
 \sup_{(x,q) \in S} \frac{x^2 q^1}{x^b q^a}
 \weq
  (\delta n)^{1/2}.
\]
Noting that $(x)_2 q^1 \le x^2 q^1$, with the other lemma, with $t(u) = u^{1/2}$ and $s = \epsilon n^{1/2}$, and $\phi(x,q) = x^b q^a$,
\[
 \pr \left( (\pi)_{21} > \epsilon n^{1/2} \right) 
 \wle \frac{(\delta n)^{1/2}}{ \epsilon n^{1/2}} \E X^b Q^a + \frac{n}{\delta n} h_\phi(\delta n)
 \wle \frac{\delta^{1/2}}{\epsilon} \E X^{4/3} Q^{2/3} + \frac{n}{\delta n} h_\phi(\delta n).
\]
Hence it seems that for $(\pi)_{21} = o_\pr(n^{1/2})$ it suffices to assume that $X_\nu^{4/3} Q_\nu^{2/3} = ( X_\nu^2 Q_\nu^1 )^{2/3}$ is uniformly integrable.
\end{rcomm}

\begin{rcomm}
Let $c, d > 0$. Let $\alpha \ge 0$. Then we choose $b = d/(\alpha+1)$ and $a = c/(\alpha+1)$. These choices imply that $a,b > 0$, $b \le d$ and $c/a = d/b$. Lemma~\ref{the:PowerBound} then implies that
\[
 \sup_{(x,q) \in S_u} \frac{x^d q^c}{x^b q^a}
 \weq \sup_{(x,q) \in S_u} \frac{x^d q^c}{(x^d q^c)^{1/(\alpha+1)}}
 \weq u^\alpha,
\]
where $S_u = \{(x,q) \in \R^2: x \ge 1, \, 0 < q \le 1, \, x^a q^b \le u\}$, that is, 
$S_u = \{(x,q) \in \R^2: x \ge 1, \, 0 < q \le 1, \, x^d q^c \le u^{\alpha+1}\}$. Then by Lemma~\ref{the:UIBoundGen} we find that, noting that $(x)_d q^c \le x^d q^c$,
\begin{align*}
 \pr \left( (\pi)_{dc} > \epsilon n^{\alpha} \right) 
 &\wle \frac{(\delta n)^{\alpha}}{ \epsilon n^{\alpha}} \E (X^d Q^c)^{1/(\alpha+1)} + \frac{n}{\delta n} h_\phi(\delta n) \\
 &\wle \frac{\delta^{\alpha}}{\epsilon} \E (X^d Q^c)^{1/(\alpha+1)} + \frac{1}{\delta} h_\phi(\delta n).
\end{align*}
Hence it seems that for $(\pi)_{dc} = o_\pr(n^{\alpha})$ it suffices to assume that $(X^d Q^c)^{1/(\alpha+1)}$ is uniformly integrable. For example, for $(d, c)=(2,1)$ and $\alpha = 1/2$, we require $X^{4/3} Q^{2/3}$ to be UI.\mnote{$c=0$ should follow, too} For example, for $(d, c)=(3,1)$ and $\alpha = 1$, we require $X^{3/2} Q^{1/2}$ to be UI.
\end{rcomm}

\begin{lemma}
\label{the:PowerBound}
Let $a, b > 0$ and $u \ge 1$. Then for any real numbers $c$ and $d$,
\[
 \sup_{(x,q) \in S_u} \frac{x^d q^c}{x^b q^a}
 \weq
 \begin{cases}
  1, &\quad d < b, \ c \ge a, \\
  u^{d/b-1}, &\quad d \ge b, \ c/a \ge d/b, \\
 \infty, &\quad \text{else},
 \end{cases}
\]
where  $S_u = \{(x,q) \in \R^2: x \ge 1, \, 0 < q \le 1, \, x^a q^b \le u\}$.
\end{lemma}
\begin{proof}
(i) Assume that $d \ge b$. Then for any $0 < q \le 1$, the function $x \mapsto \frac{x^d q^c}{x^b q^a}$ subject to $x \ge 1$ and $x^b q^a \le u$ is maximised at $x = u^{1/b} q^{-a/b}$, for which
\[
 \frac{x^d q^c}{x^b q^a}
 \weq \big( u^{1/b} q^{-a/b} \big)^{d-b} q^{c-a}
 \weq u^{d/b-1} q^{c-ad/b}.
\]
Therefore,
\[
 \sup_{(x,q) \in S_u} \frac{x^d q^c}{x^b q^a}
 \weq \sup_{0 < q \le 1} u^{d/b-1} q^{c-ad/b}
 \weq \begin{cases}
  u^{d/b-1}, &\quad c \ge ad/b, \\
 \infty, &\quad c < ad/b.
 \end{cases} 
\]

(ii) Assume next that $d < b$. Then for any $0 < q \le 1$, the function $x \mapsto \frac{x^d q^c}{x^b q^a}$ subject to $x \ge 1$ and $x^b q^a \le u$ is maximised at $x = 1$, and therefore
\[
 \sup_{(x,q) \in S_u} \frac{x^d q^c}{x^b q^a}
 \weq
 \sup_{0 < q \le 1} q^{c-a}
 \weq
 \begin{cases}
  1, &\quad c \ge a, \\
 \infty, &\quad c < a.
 \end{cases} 
\]

\end{proof}

\begin{lemma}[Simple version of Lemma~\ref{the:UIBoundGen}]
\label{the:UIBoundSimple}
Let $X,X_1,\dots,X_n \ge 0$ be identically distributed (not necessarily independent) random numbers.\mnote{independence not used} Then for any $\alpha, \delta, \epsilon > 0$,
\[
 \pr \left( \frac{1}{n}\sum_{k=1}^{n} X_k^{1+\alpha} > \epsilon s \right) 
 \wle \frac{\delta^\alpha n^\alpha}{\epsilon s} \E X + \frac{1}{\delta} \E X 1(X > \delta n).
\]
\end{lemma}
\begin{proof}
Let $M = \max_k X_k$. Note that $\frac{1}{n}\sum_{k=1}^{n} X_k^{1+\alpha} \le u^\alpha \frac{1}{n}\sum_{k=1}^{n} X_k$ on the event $M \le u$. Hence for any $u, s > 0$,
\begin{align*}
 \pr \left( \frac{1}{n}\sum_{k=1}^{n} X_k^{1+\alpha} > s \right) 
 &\wle \pr \left( \frac{1}{n}\sum_{k=1}^{n} X_k^{1+\alpha} > s, \ M \le u \right) + \pr(M > u) \\
 &\wle \pr \left( \frac{1}{n}\sum_{k=1}^{n} X_k > s u^{-\alpha}\right) + n \pr(X > u) \\
 &\wle \frac{u^\alpha}{s} \E X + \frac{n}{u} \E X 1(X > u).
\end{align*}
The claim follows by substituting $u = \delta n$ and $s = \epsilon n^\alpha$ above.\mnote{fixme}
\end{proof}

\begin{lemma}
\label{the:UIBoundGen}
Let $X,X_1,\dots,X_n$ be identically distributed (not necessarily independent\mnote{independence not used})  random variables in some space $S$. Let $\phi,\psi$ be nonnegative functions on $S$ such that $\phi(x) = 0 \implies \psi(x) = 0$. Then for any $\epsilon, s, u > 0$,
\[
 \pr \left( \frac{1}{n}\sum_{k=1}^{n} \psi(X_k) > s \right) 
 \wle \frac{t(u)}{s} \E \phi(X) + \frac{n}{u} h_\phi(u).
\]
where $h_\phi(u) = \E \phi(X) 1(\phi(X) > u)$ and $t(u) = \sup_{x: 0 < \phi(x) \le u} \frac{\psi(x)}{\phi(x)}$.
\end{lemma}
\begin{proof}
We denote $M = \max_{1 \le k \le n} \phi(X_k)$. Note that for any $u>0$,
\[
 \pr(M > u)
 \wle n \pr( \phi(X) > u )
 \wle \frac{n}{u} h_\phi(u).
\]
Note also that on the event $M \le u$,
\[
 \sum_k \psi(X_k)
 \weq \sum_{k: \phi(X_k) > 0} \psi(X_k)
 \wle t(u) \sum_k \phi(X_k).
\]
Hence by Markov's inequality,
\begin{align*}
 \pr \left(  \frac{1}{n} \sum_k \psi(X_k) > s, \, M \le u \right)
 \wle \pr \left( \frac{t(u)}{n} \sum_k \phi(X_k) > s \right)
 &\wle \frac{t(u)}{s} \E \phi(X).
\end{align*}
It hence follows that
\begin{align*}
 \pr \left( \frac{1}{n} \sum_k \psi(X_k) > s \right) 
 &\wle \pr \left( \frac{1}{n} \sum_k \psi(X_k) > s, \, M \le t \right)  + \pr(M >t) \\
 &\wle \frac{t(u)}{s} \E \phi(X) + \frac{n}{u} h_\phi(u).
\end{align*}
\end{proof}

The following result shows for example that if $(X_\nu^{4/3})_{\nu \ge 1}$ is uniformly integrable, then $\frac{1}{n_\nu} \sum_{k=1}^{n_\nu} X_{\nu, k}^2 = o_\pr(n_\nu^{1/2})$, or in other words $(\pi)_2 \ll n^{1/2}$ stochastically.

\begin{lemma}[{Extends \cite[Eq.~(4.5)]{Bloznelis_2013}}]
For each integer $\nu \ge 1$, let $X_\nu, X_{\nu, 1}, X_{\nu, 2}, \dots$ be identically distributed (possibly dependent) random variables in $\R_+$. Assume that $(X_\nu^\alpha)_{\nu \ge 1}$ is uniformly integrable for some $\alpha > 0$, and that $n_\nu \to \infty$ as $\nu \to \infty$. Then for any $\beta > \alpha$,
\[
 \sum_{k=1}^{n_\nu} X_{\nu, k}^\beta
 \weq o_\pr( n_\nu^{\beta/\alpha} )
 \qquad \text{as $\nu \to \infty$.}
\]
\end{lemma}
\begin{proof}
Denote $h_\nu(u) = \E X_\nu^\alpha 1(X_\nu^\alpha > u)$, and let $h(u) = \sup_\nu h_\nu(u)$. Note first that for all $u \ge 0$,
\[
 \E X_\nu^\alpha
 \weq \E X_\nu^\alpha 1(X_\nu^\alpha \le u) + \E X_\nu^\alpha 1(X_\nu^\alpha > u)
 \wle u + h(u).
\]
Uniform integrability implies that $\lim_{u \to \infty} h(u) = 0$. Hence we may fix a number $c$ such that $h(c) \le 1$. Then $\E X_\nu^\alpha \le c +1$ for all $\nu$.

Let us apply Lemma \ref{the:UIBoundGen} with $\psi(x) = x^\beta$ and $\phi(x) = x^\alpha$. Then we get $t(u) = \sup_{x: 0 < \phi(x) \le u} \frac{\psi(x)}{\phi(x)} = u^{\beta/\alpha-1}$, and
\begin{align*}
 \pr \left( \frac{1}{n_\nu} \sum_{k=1}^{n_\nu} X_{\nu, k}^\beta > s \right) 
 &\wle \frac{u^{\beta/\alpha-1}}{s} \E X_\nu^\alpha + \frac{n}{u}\E X_\nu^\alpha 1(X_\nu^\alpha > u) \\
 &\wle \frac{u^{\beta/\alpha-1}}{s} (c+1) + \frac{n}{u} h(u).
\end{align*}
For $s = \epsilon n_\nu^{\beta/\alpha-1}$ and $u = \delta n_\nu$, this implies
\[
 \pr \left( \sum_{k=1}^{n_\nu} X_{\nu, k}^\beta > \epsilon n_\nu^{\beta/\alpha} \right) 
 \wle \frac{\delta^{\beta/\alpha-1}}{\epsilon}(c+1) + \frac{1}{\delta}h(\delta n_\nu).
\]
Hence
\[
 \limsup_{\nu \to \infty} \pr \left( \sum_{k=1}^{n_\nu} X_{\nu, k}^\beta > \epsilon n_\nu^{\beta/\alpha} \right) 
 \wle \frac{\delta^{\beta/\alpha-1}}{\epsilon}(c+1).
\]
Because the above inequality is valid for all $\delta > 0$, the claim follows.
\end{proof}

\appendix

\section{Vector norms et cetera}

\begin{lemma}
\label{the:VectorNorms}
For any $x_1,\dots,x_n \ge 0$ and any $p \ge 1$,
\[
 n^{1-p} \left( \sum_{k=1}^n x_k \right)^p
 \wle \sum_{k=1}^n x_k^p
 \wle \left( \sum_{k=1}^n x_k \right)^p.
\]
Moreover, the $p$-th power norm of $x$ satisfies $\norm{x}_p \le \norm{x}_1 \le n^{1-1/p} \norm{x}_p$ for any $x \in \R^n$.
\end{lemma}
\begin{proof}
Without loss of generality we may and will assume that $\sum_{k=1}^n x_k > 0$.
Because $\frac{{x_k}}{\norm{x}_1} \le 1$ for all $k$, it follows that
\[
 \sum_{k=1}^n \left( \frac{{x_k}}{\norm{x}_1} \right)^p
 \wle \sum_{k=1}^n \left( \frac{{x_k}}{\norm{x}_1} \right)
 \weq 1,
\]
from which we obtain the inequality on the right. To verify the other inequality, note that by applying Jensen's inequality to the convex function $x \mapsto x^p$,
\[
 \left( \frac{1}{n} \sum_{k=1}^n x_k \right)^p
 \wle \frac{1}{n} \sum_{k=1}^n x_k^p.
\]
The statement related to $x \in \R^n$ follows by taking absolute values.
\end{proof}

\begin{lemma}
\label{the:EmpiricalUIPower}
For any $x_1,\dots,x_n \ge 0$, any $t \ge 0$, and any integer $r \ge 1$,
\[
 \sum_{k=1}^n x_k^r
 \wle t^r n + \phi(t)^r n^r,
\]
where $\phi(t) = \frac{1}{n} \sum_{k=1}^n x_k 1(x_k > t)$.
\end{lemma}
\begin{proof}
The claim follows by noting that
\[
 \sum_{k=1}^n x_k^r 1( x_k \le t)
 \wle t^r n
\]
and
\[
 \sum_{k=1}^n x_k^r 1( x_k > t)
 \weq \sum_{k=1}^n \Big( x_k 1( x_k > t) \Big)^r
 \wle \left( \sum_{k=1}^n x_k 1( x_k > t) \right)^r
 \weq \Big( \phi(t) n \Big)^r.
\]
\end{proof}

\begin{lemma}
\label{the:MinSum}
For any $c_1,c_2 >0$ and $\alpha \ge 1$ and $\beta > 0$,
\[
 \min_{t > 0} \left( c_1 t^\alpha + c_2 t^{-\beta} \right)
 \weq \left\{ \left(\frac{\beta}{\alpha}\right)^{\frac{\alpha}{\alpha+\beta}}
 + \left(\frac{\alpha}{\beta}\right)^{\frac{\beta}{\alpha+\beta}} \right\}
 c_1^{\frac{\beta}{\alpha+\beta}} c_2^{\frac{\alpha}{\alpha+\beta}},
\]
and $t = \left(\frac{c_2 \beta}{c_1 \alpha}\right)^{1/(\alpha+\beta)}$ equals the unique point at which the minimum is attained. Especially,
\[
 \min_{t > 0} \left( c_1 t + c_2 t^{-1} \right)
 \weq (c_1 c_2)^{1/2},
\]
and
\[
 \min_{t > 0} \left( c_1 t^2 + c_2 t^{-1} \right)
 \weq \frac{3}{2^{2/3}} \, c_1^{\frac{1}{3}} c_2^{\frac{2}{3}}.
\]
\end{lemma}
\begin{proof}
Denote $f(t) = c_1 t^\alpha + c_2 t^{-\beta}$.
Differentiation shows that
\[
 f'(t)
 \weq c_1 \alpha t^{\alpha-1} - c_2 \beta t^{-\beta-1}
\]
and
\[
 f''(t)
 \weq c_1 \alpha (\alpha-1) t^{\alpha-2} + c_2 \beta (\beta+1) t^{-\beta-2}
 \ > \ 0.
\]
Hence the global minimum of $f$ is attained at $t =  \left( \frac{c_2 \beta}{c_1 \alpha} \right)^{\frac{1}{\alpha + \beta}}$, which is the unique point at which $f'(t) = 0$. The claim follows by substituting this value.
\end{proof}

\section{Component size distributions in ER graphs (DEPRECATED)}

In many computations we need to known the distribution (or generating function) of the size of a connected component of a node in an \Erdos--\Renyi graph with $x$ nodes and linking probability $q$. This is related to Gontcharoff polynomials.   Ball, Sirl, and Trapman in \cite{Ball_Sirl_Trapman_2014} apply a general result of Ball and O'Neill \cite{Ball_ONeill_1999} which generalises Lefevre and Picard \cite{Lefevre_Picard_1990}, and they also cite \cite{Ball_Mollison_Scalia-Tomba_1997} and \cite{Ball_Sirl_Trapman_2010}. (See~\cite{Daniels_1967} but this is for a bit different model.) What we need is given for the generating function by \cite[Corollary 3.3]{Lefevre_Picard_1990}.

Ludwig \cite{Ludwig_1975} wrote down a recursive formula for the direct probabilities.

Let $G$ be a graph. We denote by $N_G(U)$ the set of nodes in $V(G) \setminus U$ which are adjacent to one or more nodes in a set $U \subset V(G)$. We also denote by $\bar G$ the transitive closure of $G$. Then $N_{\bar G}(U)$ equals the set of nodes in $V(G) \setminus U$ which are reachable from one or more nodes in $U$ by a path in $G$, and the connected component of a node $i \in V(G)$ equals $\{i\} \cup N_{\bar G}(\{i\})$.

Observe also that for $U_1 = N_G(U_0)$,
\[
 N_{\bar G}(U_0)
 \weq U_1 \cup N_{\bar G}(N_G(U))
\]
You can think $U_0$ as the initial infectives, and $U_1$ as the nodes which become infective at the first stage of an epidemic.

Let $G$ be a graph. For $U \subset V(G)$ we recursively define $B_G(U,0) = U$, and $B_G(U,r+1) = B_G(U,r) \cup N_G( B_G(U,r) )$, so that $B_G(U,r)$ equals the set of nodes which can be reached from $U$ by a path of length at most $r$. We denote by $B_G^+(U) = \cup_{r \ge 1} B_G(U,r)$ the set of nodes in $V(G) \setminus U$ which are reachable from $U$ by path in $G$.

Then it follows that we can write $D_G(U)$ as a disjoint union
\[
 B_G^+(U)
 \weq U' \cup B^+_{G'}( U' ),
\]
where $U' = B_G(U,1)$ and $G' = G[V(G) \setminus U]$. Assume now that $G$ is an ER graph with $k+\ell$ nodes, so that $k$ are initially infective ($\abs{U}=k$) and the rest are susceptible. Then any susceptible node not in $U$ remains susceptible with probability $(1-p)^k$, independently. It follows that $\abs{U'}$ is $\Bin(\ell, 1-(1-p)^k)$-distributed.

\section{Sampling with and without replacement}
\label{sec:Sampling}

Let $V$ be a subset of $[m]$ of size $x \le m$ selected uniformly at random. Fix distinct $i,j \in [m]$. Let $f = \law( 1_{V_k}(i), 1_{V_k}(j) )$ and let $g = \law( 1_{V_k}(i)) \times \law(1_{V_k}(j) )$. Note that
\begin{align*}
 f(0,0)
 &\weq \Big( 1 - \frac{x}{m} \Big) \Big( 1 - \frac{x}{m-1} \Big), \\
 f(1,0)
 &\weq \frac{x}{m} \Big( 1 - \frac{x-1}{m-1} \Big), \\
 f(0,1)
 &\weq \Big( 1 - \frac{x}{m} \Big) \frac{x}{m-1}, \\
 f(1,1)
 &\weq \frac{x}{m} \frac{x-1}{m-1},
\end{align*}
which also shows that $f(0,1) = f(1,0)$. A direct computation shows that
\[
 f(x,y) - g(x,y)
 \weq
 \begin{cases}
  \frac{x}{m} \frac{m-x}{m(m-1)}, &\quad (x,y) = (0,1), (1,0), \\
  - \frac{x}{m} \frac{m-x}{m(m-1)}, &\quad (x,y) = (0,0), (1,1),
 \end{cases}
\]
%
and we conclude the following result.

\begin{lemma}
\label{the:Sampling}
\[
 \dtv(f,g)
 \weq 2 \frac{x}{m} \frac{m-x}{m(m-1)}
 \weq 2 \left( \frac{1}{m-1} \right) \frac{x}{m} \left( 1 - \frac{x}{m} \right).
\]
\end{lemma}

Now let $V_k$ be subsets of $[m]$ of sizes $x_k \le m$. Let $f = \prod_{k=1}^n \law( 1_{V_k}(i), 1_{V_k}(j) )$, and let $g = \prod_{k=1}^n \law( 1_{V_k}(i) ) \times \law( 1_{V_k}(j) )$. Then the above equality implies that
\[
 \dtv(f,g)
 \wle 2 \sum_{k=1}^m \left( \frac{1}{m-1} \right) \frac{x_k}{m} \left( 1 - \frac{x_k}{m} \right)
 \wle \frac{2}{(m)_2} \sum_{k=1}^m x_k
 \weq \frac{2n}{(m)_2} (\pi)_{10}.
\]


\section{Total variation and Wasserstein metrics}
\label{sec:TotalVariation}

Recall the total variation distance of probability measures $f,g$ on a countable space $S$. See for example \cite[Section 4.1]{Levin_Peres_Wilmer_2008}. We know that
\[
 \dtv(f,g)
 \weq \frac12 \sum_{x \in S} \abs{f(x) - g(x)}
 \weq \sum_{x \in S_+} \big(f(x) - g(x) \big)
 \weq f(S_+) - g(S_+),
\]
where $S_+ = \{x: f(x) > g(x)\}$. Note also that
\[
 \sum_{x \in S_+^c} \big(f(x) - g(x) \big)
 \weq f(S_+^c) - g(S_+^c)
 \weq g(S_+) - f(S_+).
\]

\begin{lemma}
\label{the:dtvMarginals}
Let $f$ and $g$ be probability distributions on a product space $S_1 \times \cdots S_n$, such that $f$ has marginal distributions $f_i$ and $g$ has marginal distributions $g_i$, $i=1,\dots,n$. Then
\[
 \dtv(f, g)
 \wle \dtv(f_1, g_1) + \cdots + \dtv(f_n, g_n).
\]
\end{lemma}
\begin{proof}
We know that the exists (require Polish?) an optimal coupling, a pair of random variables $X = (X_1,\dots,X_n)$ and $Y=(Y_1,\dots,Y_n)$ such that
\[
 \dtv(f,g)
 \weq \pr( X \ne Y).
\]
By the union bound, and the well-known fact (require Polish?), it follows that
\[
 \pr(X \ne Y)
 \wle \sum_{i=1}^n \pr(X_i \ne Y_i)
 \wle \sum_{i=1}^n \dtv( \law(X_i), \law(Y_i))
 \wle \sum_{i=1}^n \dtv( f_i, g_i).
\]
\end{proof}

\begin{lemma}
\label{the:dtvBerPoi}
The total variation distance between a Bernoulli and Poisson distribution with the same mean satisfies $\dtv(\Ber(p), \Poi(p)) = p (1-e^{-p}) \le p^2$ for any $0 \le p \le 1$.
\end{lemma}
\begin{proof}
Note that $\abs{f_0-g_0} = e^{-p} - (1-p)$, $\abs{f_1-g_1} = p(1-e^{-p})$, and that
\begin{align*}
 \sum_{x \ge 2} \abs{f_x-g_x}
 \weq \sum_{x \ge 2} g_x
 \weq 1 - e^{-p} - pe^{-p}.
\end{align*}
The first claim follows by summing up the above equations and dividing the outcome by two. The second follows by noting that $1-p \le e^{-p}$ for all $p$.
\end{proof}

\begin{lemma}
\label{the:dtvBinPoi}
The total variation distance between a binomial and a Poisson distribution with the same mean satisfies $\dtv(\Bin(n,p), \Poi(np)) \le n p^2$ for any $n \ge 1$ and $0 \le p \le 1$.
\end{lemma}
\begin{proof}
Let $M = \sum_{i=1}^n A_i$ and $N = \sum_{i=1}^n B_i$, where the summands are mutually independent and such that $\law(A_i) = \Ber(p)$ and $\law(B_i) = \Poi(p)$. Then $(M,N)$ is a coupling of $\Bin(n,p)$ and $\Poi(np)$, and by the union bound, $\pr( M \ne N ) \le \sum_{i=1}^n \pr(A_i \ne B_i) \le n \dtv(\Ber(p), \Poi(p))$. By Lemma~\ref{the:dtvBerPoi}, the claim follows.
\end{proof}

\begin{lemma}
\label{the:dtvPoiPoi}
The total variation distance between two Poisson distributions is bounded by $\dtv(\Poi(\lambda), \Poi(\mu)) \le 1-e^{-\abs{\lambda-\mu}} \le \abs{\lambda-\mu}$.
\end{lemma}
\begin{proof}
Assume that $0 \le \lambda \le \mu$ without loss of generality. Let $X,D$ be independent Poisson distributions with means $\lambda, \mu-\lambda$, respectively. Define $Y = X+D$. Then $(X,Y)$ is a coupling of $\Poi(\lambda)$ and $\Poi(\mu)$, and $\pr( X \ne Y ) = \pr(D \ne 0) = 1 - e^{-(\mu-\lambda)}$. Hence the first claim follows. The second follows by noting that $1-t \le e^{-t}$ for all $t$.

\end{proof}

\begin{lemma}
\label{the:dtvSum}
For any collection of independent nonnegative random numbers $A_i, B_i, B_i'$, $i \in I$, \mnote{no more needed?}
\[
 \dtv\left( \sum_i A_i B_i, \, \sum_i A_i B_i' \right)
 \wle \sum_i \pr(A_i \ne 0) \, \dtv(B_i, \, B_i').
\]
\end{lemma}
\begin{proof}
For any $i \in I$ there exists a coupling $(\tilde B_i, \tilde B_i')$ of $B_i$ and $B_i'$ such that $\pr(\tilde B_i \ne \tilde B_i') = \dtv(B_i, \, B_i')$. Hence there exist a collection of independent random variables $\{\tilde A_i, (\tilde B_i, \tilde B_i'), i \in I\}$ such that $\left( \sum_i \tilde A_i \tilde B_i, \, \sum_i \tilde A_i \tilde B_i' \right)$ is a coupling of $\sum_i A_i B_i$ and $\sum_i A_i B_i'$. This coupling satisfies
\begin{align*}
 \pr \left( \sum_i \tilde A_i \tilde B_i \ne \sum_i \tilde A_i \tilde B_i' \right)
 &\wle \pr \left( \bigcup_i \{ \tilde A_i \tilde B_i \ne \tilde A_i \tilde B_i' \} \right) \\
 &\wle \sum_i \pr \left( \tilde A_i \tilde B_i \ne \tilde A_i \tilde B_i' \right) \\
 &\weq \sum_i \pr \left( \tilde A_i \ne 0, \tilde B_i \ne \tilde B_i' \right) \\
 &\weq \sum_i \pr(\tilde A_i \ne 0) \pr(\tilde B_i \ne \tilde B_i').
\end{align*}
Because $\tilde A_i \eqst A_i$ and $\pr(\tilde B_i \ne \tilde B_i') = \dtv(B_i, \, B_i')$, the claim follows.
\end{proof}

\section{Empirical distributions}

\begin{lemma}[Empirical distribution of dependent variables]
\label{the:EmpiricalDistributionMean}
Let $X_1,\dots,X_m$ be random variables on a countable space $S$. Let $\pi = \frac{1}{m} \sum_{i=1}^m \delta_{X_i}$ be the empirical distribution, and define $\bar\pi(A) = \E \pi(A)$. Then for any $\epsilon > 0$,
\[
 \pr\left( \sup_{s \in S} \abs{ \pi(s) - \bar\pi(s) } > \epsilon \right)
 \wle \frac{2}{\epsilon^2} \left( \frac{1}{m} + \frac{1}{m^2} \sumd_{i,j} d_{ij} \right),
\]
where $d_{ij} = \dtv( \law(X_i, X_j), \law(X_i) \times \law(X_j))$.
\end{lemma}
\begin{proof}
Because $\E \pi(s) = \bar\pi(s)$ for all $s \in S$, we see by applying the union bound and Chebyshev's inequality that
\begin{align*}
 \pr\left( \sup_{s} \abs{ \pi(s) - \bar\pi(s) } > \epsilon \right)
 &\wle \sum_s \pr\left( \abs{ \pi(s) - \bar\pi(s) } > \epsilon \right)
 &\le \ \epsilon^{-2} \sum_s \Var(\pi(s)),
\end{align*}
where
\[
 \Var(\pi(s))
 \weq \Cov( \pi(s), \pi(s) )
 \weq \frac{1}{m^2} \sum_{i,j} \Cov\Big( 1(X_i=s), \, 1(X_j=s) \Big).
\]
Next, we note that
\begin{align*}
 &\sum_s \Cov\Big( 1(X_i=s), \, 1(X_j=s) \Big) \\
 &\weq \sum_s \Big( \pr ( X_i = s, X_j = s ) - \pr ( X_i = s) \pr( X_j = s ) \Big) \\
 &\wle \sum_s \Big| \pr ( X_i = s, X_j = s ) - \pr ( X_i = s) \pr( X_j = s ) \Big| \\
 &\wle \sum_s \sum_t \Big| \pr ( X_i = s, X_j = t ) - \pr ( X_i = s) \pr( X_j = t ) \Big| \\
 &\weq 2 \dtv \Big(  \law(X_i, X_j), \ \law(X_i) \times \law(X_j) \Big).
\end{align*}
Now, it follows that
\begin{align*}
 \sum_s \Var( \pi(s) )
 &\weq \frac{1}{m^2} \sum_{i,j} \sum_s \Cov\Big( 1(X_i=s), \, 1(X_j=s) \Big)
 \wle \frac{2}{m^2} \sum_{i,j} d_{ij}.
\end{align*}
Now the claim follows by applying the generic bound $d_{ii} \le 1$ to the diagonal elements of $(d_{ij})$.
\end{proof}

\section{Weak convergence and uniform integrability}
\label{sec:WeakConvergence}

\begin{lemma}[Maxima of random numbers]
\label{the:IIDMaxima}
Let $M_n = \max(X_1^{(n)}, \dots, X_n^{(n)})$ be a maximum of independent nonnegative random numbers with a common distribution $P_n$. If $(P_n)_{n \ge 1}$ is $\alpha$-uniformly integrable for some $\alpha > 0$, then $M_n \ll_{\pr} n^{1/\alpha}$.
\end{lemma}
\begin{proof}
Denote $\phi(t) = \sup_n \int x^\alpha 1(x>t) P_n(dx)$, and note that for any $t > 0$ and any $P_n$-distributed random variable $X$,
\begin{align*}
 \pr( X > t )
 \wle t^{-\alpha} \E X^\alpha 1( X > t )
 \wle t^{-\alpha} \phi(t).
\end{align*}
Hence, the union bound implies that
\begin{align*}
 \pr( M_n > t )
 \wle n \pr( X_1^{(n)} > t )
 \wle n t^{-\alpha} \phi(t).
\end{align*}
By substituting $t = \epsilon n^{1/\alpha}$, we find that $\pr( M_n > \epsilon n^{1/\alpha} ) \to 0$ for any $\epsilon > 0$.
\end{proof}

\begin{rcomm}
Application of Lemma~\ref{the:QWLLN_New}. Assume that $P^\nu \to P$ weakly and $(P^\nu)_{sr} \to (P)_{sr}$. Let $(X^\nu_1,Q^\nu_1), \dots, (X^\nu_n,Q^\nu_n)$ be independent and $P^\nu$-distributed. Let $\hat P^\nu_n$ be the empirical distribution. Then for every bounded continuous function $\phi$, $\pr_\nu( \abs{P^\nu(\phi) - P(\phi)} > \epsilon ) \to 0$, and this is also true for $\phi(x,q) = (x)_s q^r$. Hence $P^\nu \weakto P$  in probability.
\end{rcomm}
\begin{lemma}[Quantitative weak law of large numbers]\mnote{In degree proof $h(M)=0$}
\label{the:QWLLN_New}
Let $X_1,X_2,\dots,X_n$ be independent random numbers with a finite mean, and denote $\mu = \frac{1}{n} \sum_{k=1}^n \E X_k$. Then
\[
 \pr \left( \Big\lvert \frac{1}{n} \sum_{k=1}^n X_k - \mu \Big\rvert > \epsilon \right)
 \wle 9 \frac{M^2}{\epsilon^2 n} + 6 \frac{h(M)}{\epsilon}.
\]
for all $\epsilon, M > 0$, where $h(M) = \max_{1 \le k \le n} \E \abs{X_k}1(\abs{X_k} > M)$.
\end{lemma}
\begin{proof}
Fix $\epsilon, M > 0$. We only need to do the proof when $h(M) \le \epsilon/6$, because otherwise the upper bound is bigger than one and claim is trivial. Define a truncation map of the real line into $[-M,M]$ by $t_M(x) = (-M) \vee (x \wedge M)$. Define truncated random numbers $X'_k = t_M(X_k)$ and $X' = t_M(X)$. Then
\[
 \underbrace{\frac{1}{n} \sum_{k=1}^n X_k - \mu}_{\Delta}
 \weq \underbrace{\frac{1}{n} \sum_{k=1}^n (X_k' - \E X'_k)}_{\Delta_1}
 + \underbrace{\frac{1}{n} \sum_{k=1}^n (X_k-X_k')}_{\Delta_2}
 + \underbrace{\frac{1}{n} \sum_{k=1}^n \E (X'_k - X_k)}_{\Delta_3}.
\]
Note that by independence and Chebyshev's inequality,
\[
 \pr( \abs{\Delta_1} > \epsilon/3 )
 \wle (\epsilon/3)^{-2} \Var(\Delta_1)
 \weq (\epsilon/3)^{-2} \frac{1}{n^2} \sum_k \Var(X'_k)
 \wle (\epsilon/3)^{-2} \frac{1}{n} M^2.
\]
Furthermore,
\begin{align*}
 \E \abs{X'_k - X_k}
 &\weq \E \abs{X_k'-X_k} 1(\abs{X_k} > M) \\
 &\wle \E (\abs{X_k'}+\abs{X_k})1(\abs{X_k} > M) \\
 &\wle 2 h(M).
\end{align*}
Hence by Markov's inequality,
\begin{align*}
 \pr( \abs{\Delta_2} > \epsilon/3 )
 \wle (\epsilon/3)^{-1} \frac{1}{n} \sum_k \E \abs{X_k-X'_k}
 \wle 2 (\epsilon/3)^{-1} h(M).
\end{align*}
The above upper bound on $\E\abs{ X'_k - X_k}$ also shows that $\pr( \abs{\Delta_3} > \epsilon/3 ) = 0$ 
when $h(M) \le \epsilon/6$. The claim now follows by the union bound.
\end{proof}

\begin{lemma}[Limits of integrals unbounded functions]
\label{the:UnboundedIntegral}
Assume that $\pi_n \to \pi$ in distribution, \mnote{no more needed?} and that a continuous function $\phi$ is bounded by $\abs{\phi(x)} \le \psi(x)$ for some continuous function $\psi$ such that $\pi_n(\psi) \to \pi(\psi) < \infty$. Then $\pi_n(\phi) \to \pi(\phi)$.
\end{lemma}
\begin{proof}
By Skorohod coupling there exists random variables $X_n, X$ defined on a common probability space such
that $\law(X_n) = \pi_n$, $\law(X) = \pi$, and $X_n \to X$ almost surely. Let $Y_n = \phi(X_n), Y = \phi(Y)$ and $Z_n = \psi(X_n), Z = \psi(X)$. Then $\abs{Y_n} \le \abs{Z_n}$, $Y_n \to Y$, and $Z_n \to Z$ almost surely. Also $\E Z_n \to \E Z < \infty$. Hence by Lebesgue's dominated convergence theorem (see the formulation in \cite[Theorem 1.21]{Kallenberg_2002}) it follows that $\pi_n(\phi) = \E Y_n \to \E Y = \pi(\phi)$.
\end{proof}

\section{Cross-factorial moments}
\label{sec:CrossFactorialMoments}

\begin{lemma}
\label{the:CrossFactorialMoments}
For any $k \ge 1$,
\[
 (\pi)_{k,k-1}
 \wle (2k)^k + 2 (\pi)_{1,0}^{1/k} (\pi)_{k+1,k}^{1-1/k}.
\]
\end{lemma}
\begin{proof}
Let $(X,Q)$ be a $\pi$-distributed random pair. Observe first that $x - j \le x \le 2(x-k)$ for all $x \ge 2k$ and all $j \ge 1$, and therefore
\[
 \frac{(x)_k}{(x-k)^{k-1}}
 \weq x \prod_{j=1}^{k-1} \frac{x-j}{x-k}
 \wle 2^{k-1} x
 \wle 2^k x
\]
for all $x \ge 2k$. Hence by applying $(x)_k = \frac{(x)_{k+1}}{x-k}$, we find that
\[
 (x)_k
 \weq (x)_{k+1}^{1-1/k} \frac{(x)_k^{1/k}}{(x-k)^{1-1/k}}
 \weq (x)_{k+1}^{1-1/k} \left( \frac{(x)_k}{(x-k)^{k-1}} \right)^{1/k}
 \wle 2 (x)_{k+1}^{1-1/k} x^{1/k}
\]
for $x \ge 2k$. Together with the fact that $(x)_k \le (2k)^k$ for $x \le 2k$, it follows that
\begin{align*}
 (\pi)_{k,k-1}
 &\weq \E (X)_k Q^{k-1} 1(X \le 2k) + \E (X)_k Q^{k-1} 1(X > 2k) \\
 &\wle (2k)^k + 2 \E (X)_{k+1}^{1-1/k} Q^{k-1} X^{1/k}.
\end{align*}
Now Hölder's inequality with conjugate exponents $p=k/(k-1)$ and $p'=k$ implies that
\[
  \E (X)_{k+1}^{1-1/k} Q^{k-1} X^{1/k}
  \wle \Big( \E (X)_{k+1} Q^{k} \Big)^{1-1/k} \Big( \E X \Big)^{1/k},
\]
so we conclude that
\[
 (\pi)_{k,k-1}
 \wle (2k)^k + 2 (\pi)_{k+1,k}^{1-1/k} (\pi)_{1,0}^{1/k}.
\]
\end{proof}

\section{Binomial kernel}
\label{sec:BinomialKernel}

\begin{lemma}
\label{the:ChernoffThirdMoment}
Let $X$ be $\Bin(n,p)$-distributed. Then
\[
 (\E X)^3 \pr(X \le s)
 \wle
 \begin{cases}
  689 &\quad \text{for} \ s \le \frac12 np, \\
  8 s^3 &\quad \text{for} \ s \ge \frac12 np.
 \end{cases}
\]
\end{lemma}
\begin{proof}
If $s \le \frac12 np$, then by a Chernoff bound \cite[Theorem 2.1]{Janson_Luczak_Rucinski_2000} we find that
\[
 \pr( X \le s )
 \wle \pr( X \le \frac12 \E X)
 \wle e^{-np/8}.
\]
Note also that $t^3 e^{-t/8} \le 24^3 e^{-3} \le 689$ for all $t$, so that in this case
\[
 (\E X)^3 \pr( X \le s )
 \wle (np)^3 e^{-np/8}
 \wle 689.
\]

On the other hand, if $s \ge \frac12 np$, then the trivial upper bound $\pr(X \le s) \le 1$ implies that $(\E X)^3 \pr(X \le s) \le (np)^3 \le 8 s^3$.
\end{proof}

The binomial kernel is denoted by
\[
 \Bin(n, p, A)
 \weq \sum_{s \in A} \binom{n}{s} (1-p)^{n-s} p^s, \quad n \in \Z_+, \ p \in [0,1], \ A \subset \Z_+.
\]
For fixed $n,p$, the map $A \mapsto \Bin(n,p,A)$ is the binomial distribution with $n$ trials and success probability $p$.

\begin{lemma}[Upper tail of the binomial kernel]
\label{the:BinKernelChernoff}
Fix integers $a \ge 1$ and $0 \le b \le a$. Then for any $M \ge 1$ and $M' \ge 7 a! M$,
\begin{equation}
 \label{eq:BinKernelChernoff}
 \Bin(x,q,[M',\infty)) \wle e^{-M'}
\end{equation}
for 
all $(x,q)$ such that $(x)_a q^b \le M$.
\end{lemma}
\begin{proof}
Assume first that $x \ge a$. Then by applying the inequality $\frac{(x)_a}{x^a} \ge \frac{1}{a!}$ it follows that
\[
 M
 \wge 
 (x)_a q^b
 \weq \frac{(x)_a}{x^a} (xq)^b x^{a-b}
 \wge \frac{1}{a!} (xq)^b x^{a-b}.
\]
If $b=0$, this implies $x \le (a! M)^{1/a} \le a! M < 7 a! M \le M'$, and therefore the left side of \eqref{eq:BinKernelChernoff} is zero. If $1 \le b \le a$, the above inequality implies 
\[
 (xq)^b \le \frac{a!M}{x^{a-b}} \le a! M,
\]
so that $7xq \le 7(a!M)^{1/b} \le 7 a! M \le M'$, and \eqref{eq:BinKernelChernoff} follows by applying a Chernoff bound \cite[Corollary 2.4]{Janson_Luczak_Rucinski_2000}. 

Assume next that $x < a$. Then the fact that $a \le 7 a! M \le M'$ immediately confirms that the left side of \eqref{eq:BinKernelChernoff} is zero.
\end{proof}


Let $\pi$ be a probability measure on $\Z_+ \times [0,1]$. Then a $\pi$-mixed binomial distribution is the probability distribution on $\Z_+$ with density function (note that we define $\binom{0}{0} = 0$ and $\binom{n}{k} = 0$ for $k > n$)
\[
 k \mapsto \int_{\Z_+ \times [0,1]} \binom{n}{k} (1-q)^{n-k} q^k \, \pi(dn, dq).
\]
We can represent this distribution as the law of
\[
 M
 \weq \sum_{n=1}^N B_n
 \quad\text{with}\quad B_n = 1(U_n \le Q),
\]
where $(N,Q), U_1, U_2, \dots$ are independent, with $(N,Q)$ being $\pi$-distributed and $U_1,U_2,\dots$ uniform on the unit interval.

Observe that $\E_Q B_{n_1} \cdots B_{n_r} = Q^j$ where $j$ is the number of distinct entries among $n=(n_1,\dots, n_r)$. Therefore,
\begin{align*}
 \E_{N,Q} M^r
 \weq \sum_{n \in [N]^r} \E_Q B_{n_1} \cdots B_{n_r}
 \weq \sum_{j=1}^r c_{j, r} (N)_j q^j,
\end{align*}
where $c_{j, r}$ is the number of labeled partitions $(A_1,\dots,A_j)$ of $[r]$ into $j$ nonempty disjoint sets. By taking expectations above we find that the moments of the mixed binomial distribution are given by
\begin{align*}
 \E M^r
 \weq \sum_{j=1}^r c_{j, r} \E (N)_j q^j.
\end{align*}
Because $c_{1,1} = 1$, we obtain
\[
 \E M \weq \E N Q.
\]
Because $c_{1,2} = 1$ and $c_{2,2} = 1$, we obtain
\[
 \E M^2 \weq \E N Q + \E (N)_2 Q^2
\]
Because $c_{1,3} = 1$, $c_{2,3} = 3$, and $c_{3,3} = 1$, we obtain
\[
 \E M^3 \weq \E N Q + 3 \E (N)_2 Q^2 + \E (N)_3 Q^3.
\]

\section{Zipf distribution}

The Zipf distribution with density exponent $\alpha \in \R$, min value $a > 0$ and max value $b > a$ is a probability measure on $\{1,2,3,\dots\} \cap [a,b]$ with density
\[
 f(x)
 \weq c x^{-\alpha}
\]
and normalising constant $c = \sum_{a \le x \le b} x^{-\alpha}$. The unbounded Zipf distribution with $b = \infty$ is defined for $\alpha > 1$. A unit-scale Zipf distribution with density exponent $\alpha > 1$ is the one with $a=1$ and $b=\infty$, and has density
\[
 f(x) \weq \zeta(\alpha)^{-1} x^{-\alpha},
\]
which is normalised by Riemann's zeta function $\zeta(\alpha)^{-1} = \sum_{x \ge 1} x^{-\alpha}$.

An $x^k$-biasing of a unit-scale Zipf distribution with density exponent $\alpha > k+1$ is a unit-scale Zipf distribution with density exponent $\alpha-k$.

Python's numpy computes samples from the unit-scale Zipf distribution as follows. Let $U \in (0,1)$ be a uniform random number. Then $U^{-1/(\alpha-1)}$ is Pareto distributed with unit scale, tail exponent $\alpha-1$, and density exponent $\alpha$.
Let $Y = \floor{U^{-1/(\alpha-1)}}$. Then $Y$ is a down-truncated Pareto random integer with density
\[
 g(x)
 \weq {x}^{-(\alpha-1)} - ({x}+1)^{-(\alpha-1)}, \quad x=1,2,\dots
\]
We use this as a proposal distribution for rejection sampling. Observe that 
\begin{align*}
 \frac{f(x)}{g(x)}
 &\weq \zeta(\alpha)^{-1} \frac{x^{-\alpha}}{{x}^{-(\alpha-1)} - ({x}+1)^{-(\alpha-1)}} \\
 &\weq \zeta(\alpha)^{-1} \frac{1}{ {x} - (x+1) ( x / (x+1) )^\alpha  } \\
 &\weq \zeta(\alpha)^{-1} \frac{1}{ {x} - (x+1) ( 1 + 1/x )^{-\alpha}  } \\
\end{align*}

 and $T = (1+1/X)^\alpha$ for $U \in (0,1)$ being a uniform random number. Define $b = 2^\alpha$ and let $V \in (0,1)$ be another uniform random number. If $V X (T-1)/(b-1) \le T/b$, return $X$.

\section{Compound Poisson distributions}
\label{sec:CompoundPoisson}

\subsection{Definition}
The compound Poisson distribution with rate parameter $\lambda \ge 0$ and increment distribution $f$ is the probability measure on $\R$ defined by 
\[
 \CPoi(\lambda, f)
 \weq \sum_{k=0}^\infty e^{-\lambda} \frac{\lambda^k}{k!} f^{\ast k},
\]
where $f^{\ast k}$ denotes the $k$-fold convolution of a probability distribution $f$ on $\R$. This is the law of the random number
\begin{equation}
 \label{eq:CPoiRepresentation}
 Y \weq \sum_{k=1}^\Lambda X_k,
\end{equation}
where $\Lambda$ is Poisson distributed with mean $\lambda$, the summands $X_1,X_2,\dots$ are $f$-distributed, and all random variables appearing on the right are independent.

\subsection{Generating function, moments, and cumulants}

Let $Y$ be a compound Poisson distributed random integer with rate parameter $\lambda$ and increment distribution $f$. Denote the generating function and the cumulant generating function of the compound Poisson distribution by
\begin{align*}
 G_Y(s) &\weq \E s^Y, \\
 K_Y(t)
 &\weq \log G_Y(e^t).
\end{align*}

By noting that for independent $f$-distributed random variables $X, X_1, X_2, \dots$,
\[
 \E s^{\sum_{i=1}^n X_i}
 \weq \prod_{i=1}^n \E s^{X_i}
 \weq G_X(s)^n
\]
we see that
\begin{equation}
 \label{eq:GfCPoi}
 G_Y(s)
 \weq e^{\lambda( G_X(s) - 1)},
\end{equation}
with $G_X(s) = \E s^X$. By taking logarithms and substituting $s = e^t$ above, 
we see that the cumulant generating function equals
\[
 K_Y(t)
 \weq \lambda (M_X(t) - 1).
\]
where $M_X(t) = \E e^{t X}$ is the moment generating function of $X$. Recalling that the moments of $X$ are given by $\E X^r = M^{(r)}_X(0)$, we find that the cumulants of the compound Poisson distribution are given by
\[
 K^{(r)}_Y(0)
 \weq \lambda M^{(r)}_X(0)
 \weq \lambda \E X^r, \quad r \ge 1.
\]
In the special case where $f = \delta_1$ we obtain the standard Poisson distribution, and the above formula shows that all cumulants of the standard Poisson distribution are equals to $\lambda$. Recall also that
\[
 K^{(1)}_Y(0)
 \weq \E(Y)
\]
and
\[
 K^{(2)}_Y(0)
 \weq \Var(Y)
 \weq \E(Y^2) - (\E Y)^2.
\]
Hence the mean and the variance of the compound Poisson distribution equal $\E(Y) = \lambda \E(X)$ and $\Var(Y) = \lambda \E(X^2)$. By applying the formulas
\begin{align*}
 \E(Y) &\weq  K^{(1)}_Y(0), \\
 \E(Y^2) &\weq K^{(2)}_Y(0) + ( K^{(1)}_Y(0) )^2, \\
 \E(Y^3) &\weq K^{(3)}_Y(0) + 3 K^{(2)}_Y(0)  K^{(1)}_Y(0) + ( K^{(1)}_Y(0) )^3,
\end{align*}
we find that the first three moments of the compound Poisson distribution are
\begin{equation}
\label{eq:CompoundPoissonMoments}
\begin{aligned}
 \E(Y) &\weq \lambda \E(X), \\
 \E(Y^2) &\weq \lambda \E(X^2) + \lambda^2 (\E X )^2, \\
 \E(Y^3) &\weq \lambda \E(X^3) + 3 \lambda^2 \E X^2 \E X + \lambda^3 (\E X)^3.
\end{aligned}
\end{equation}

The following lemma implies, by conditioning on $\Lambda$, that
\[
 \E(\Lambda) \E(X^r)
 \wle \E(Y^r)
 \wle \E(\Lambda^r) \E(X^r)
\]
for all integers $r \ge 1$. Especially, the $r$-th moment of the compound Poisson distribution is finite if and only if the $r$-th moment of the increment distribution is finite.

\begin{lemma}
\label{the:CompoundPoissonMoments}
If $Y = \sum_{j=1}^n X_j$ is a sum of independent nonnegative random numbers, then
for any integer $r \ge 1$,
\[
 n \E(X^r)
 \wle \E(Y^r)
 \wle n^r \E(X^r).
\]
\end{lemma}
\begin{proof}
The lower bound follows by noting that $Y^r \ge \sum_{j=1}^n X_j^r$
and taking expectations. For the upper bound, a generalized Hölder's inequality implies
\[
 \E X_{j_1} \cdots X_{j_r}
 \wle (\E X_{j_1}^r)^{1/r} \cdots (\E X_{j_r}^r)^{1/r}
 \weq \E(X^r)
\]
for all $j_1, \dots, j_r$, so that
\[
 \E(Y^r)
 \weq \E \sum_{j \in [n]^r} X_{j_1} \cdots X_{j_r}
 \wle n^r \E(X^r).
\]
\end{proof}

For the standard Poisson distribution, let $N_t$ be a unit rate Poisson process. Then
\[
 \E N_t^r
 \weq t^r \E \left( t^{-1} \sum_{s=1}^t (N_s-N_{s-1}) \right)^r
 \wle t^r \E \left( t^{-1} \sum_{s=1}^t (N_s-N_{s-1})^r \right)
 \weq t^r \E N_1^r.
\]
Hence for a rate-$\lambda$ Poisson distributed random integer $\Lambda$, $\E \Lambda \le \lambda^r \E N_1^r$.
We may conclude that
\[
 \lambda \E(X^r)
 \wle \E(Y^r)
 \wle c_r \lambda^r \E(X^r)
\]
where $c_r = e^{-1} \sum_{k=0}^\infty \frac{k^r}{k!}$ is the $r$-th moment of a unit-rate Poisson distribution.

%

\section{Older stuff}

\subsection{Analysis of triangle covering probabilities}
To prove the above claims rigorously, we need to study the probabilities
\begin{align*}
 p_1^{(k)} &\weq \pr\bigg( V_k \supset \{1,2,3\}, \ \cA^{(k)}, \ \deg_G(1) = d\bigg), \\
 \tilde p_1^{(k)} &\weq \pr\bigg( V_k \supset \{1,2,3\}, \ \cA^{(k)}, \ d_1 + d_{*}' = d\bigg), \\
 \tilde p_{1,s}^{(k)} &\weq \pr\bigg( V_k \supset \{1,2,3\}, \ \cA^{(k)}, \ d_1^{(k)} = s+2 \bigg).
\end{align*}
Observe that for a layer $k$ of size $x_k$,
\begin{align*}
 \tilde p_{1,r}^{(k)}
 &\weq \pr\bigg( V_k \supset \{1,2,3\}, \ \cA^{(k)}, \ d_1^{(k)} = r+2 \bigg) \\
 &\weq \frac{(x_k)_3}{(m)_3} q_{k}^3 \pr( \Bin(x_k-3, q_{k}) = r) \\
\end{align*}

\subsection{Degree distribution with forbidden layers}

Here a lemma that is needed in the sequel. Let $K \subset [n]$ be a set of forbidden layers. Denote by $G^{(-K)} = \cup_{k \in [n] \setminus K} G^{(k)}$ the graph with layers $K$ excluded. Note that
\[
 0
 \wle \deg_G(i) - \deg_{G^{(-K)}}(i)
 \wle \sum_{k \in K} \deg_{G^{(k)}}(i),
\]
and
\[
 \E \deg_{G^{(k)}}(i)
 \weq \sum_{j \ne i} \frac{(x_k)_2}{(m)_2} q_k
 \weq m^{-1} (x_k)_2 q_k.
\]
Hence
\[
 \pr( \deg_G(i) \ne \deg_{G^{(-K)}}(i) )
 \wle m^{-1} \sum_{k \in K} (x_k)_2 q_k.
\]

Specializing to a singleton $K=\{k\}$, we find that
\[
 \pr( \deg_G(i) \ne \deg_{G^{(-k)}}(i) )
 \wle m^{-1} (x_k)_2 q_k,
\]
and
\[
 \pr( \deg_G(i) \ne \deg_{G^{(-k)}}(i) \ \text{for some $k$})
 \wle \sum_k m^{-1} (x_k)_2 q_k
 \weq (m/n) (\pi_n)_{2,1}.
\]

\section{Connectivity and isolated nodes}

What is the probability that a node is isolated? Let $G_k$ be the graph on node set $[n]$ with links $\{i,j\}$ such that $B_{ik}B_{jk} C_{ijk} = 1$. Node $i$ is isolated if and only if it is isolated in every layer. Given the layer size $X_k$, node $i$ belongs to $V_k$ with probability $\frac{X_k}{n}$, and on this event the number of other nodes in $V_k$ equals $X_k-1$, so that the probability that node $i$ does not connect to any other node in $G_k$ has probability
$
 (1-Q_k)^{X_k-1}.
$
Hence the probability that node $i$ is isolated in $G_k$, given $X_k$ and $Q_k$, equals
\[
 \pr_{X_k, Q_k}( \deg_{G_k}(i) = 0)
 \weq 1-\frac{X_k}{n} + \frac{X_k}{n} (1-Q_k)^{X_k-1}
\]
Hence the probability that node $i$ is isolated, given the layer sizes and strengths equals
\[
 \pr_{X, Q}( \deg_{G}(i) = 0)
 \weq \prod_{k=1}^m \left( 1-\frac{X_k}{n} + \frac{X_k}{n} (1-Q_k)^{X_k-1} \right).
\]
Because different layers have independent and identically distributed characteristics, it follows that
\[
 \pr( \deg_{G}(i) = 0)
 \weq \left( 1- \E  \frac{X_1}{n} + \E  \frac{X_1}{n} (1-Q_1)^{X_1-1} \right)^m.
\]
Hence the expected number of isolated nodes equals
\[
 \E Y
 \weq n\left( 1 - \left( 1- \E  \frac{X_1}{n} + \E  \frac{X_1}{n} (1-Q_1)^{X_1-1} \right)^m \right).
\]
When we write $a_n = \E  {X_1}(1- (1-Q_1)^{X_1-1})$, this equals
\[
 \E Y
 \weq n\left( 1 - \left( 1- \frac{a_n}{n} \right)^m \right)
 \wapprox n \left( 1 - \left( e^{-a_n} \right)^{m/n} \right)
 \wapprox m a_n,
\]
when $a_n \ll n$ and so on\dots

\subsubsection{Lower bound on the number of isolated nodes}

For node $i$ to have neighbors in $G_k$ it is necessary that $i \in V_k$ which happens with probability $\frac{X_k}{n}$ given the size $X_k$ of layer $k$. Given this event, there are $X_k-1$ other nodes in $V_k$, and any of them is linked in $G_k$ to $i$ with probability $Q_k$, independently. Hence the expected degree of node $i$ in graph $G_k$ given $(X_k,Q_k)$ equals
\[
 \frac{X_k}{n} (X_k-1)Q_k,
\]
and by Markov's inequality,
\[
 \pr( \deg_{G_k}(i) > 0 \cond X_k, Q_k)
 \wle \frac{X_k}{n} (X_k-1)Q_k.
\]
Now by the union bound,
\[
 \pr_{X,Q}( \deg_{G}(i) > 0)
 \wle \frac{1}{n} \sum_{k=1}^m X_k(X_k-1)Q_k,
\]
where $\pr_{X,Q}$ refers to the conditional distribution given the layer sizes and strengths. Let $Y$ be the number of isolated nodes. Then there are $n-Y$ nonisolated nodes, and by the above bound,
\[
 \E_{X,Q}(n-Y)
 \wle n\frac{1}{n} \sum_{k=1}^m X_k(X_k-1)Q_k
\]
Hence
\[
 \E_{X,Q}(Y)
 \wge n -  \sum_{k=1}^m X_k(X_k-1)Q_k.
\]

\section{Discussion about other related models}
We will discuss the mixed-membership stochastic block model in \cite{Airoldi_Blei_Fienberg_Xing_2008}. We will discuss the directed version of the model as in the paper. A model with $m$ nodes and $n$ channel types (layers) is parametrized by $\alpha \in (0,\infty)^n$ and an $n$-by-$n$ matrix $K \in [0,1]^{n \times n}$. Here $\alpha_k$ represents the overall attractiveness of layer $k$ and $K_{k,\ell}$ represents the transmission strength for a channel with transmission type $i$ and receival type $j$. The model is defined by first sampling $m$ independent random variables $U_1, \dots, U_m$ in the probability simplex $\cS_1 = \{x \in \R_+^n: \norm{x}_1 = 1\}$ using a Dirichlet distribution with parameter $\alpha$. Then for $(i,j) \in [m]^2_{\ne}$ we sample $Z^+_{ij}, Z^-_{ij} \in [n]$ independently using distributions $U_i$ and $U_j$. Here $Z^+_{ij}$ (resp.\ $Z^-_{ij}$) describes the type of transmission (resp.\ receival) from node $i$ to $j$. Then the probability that $i$ is linked to $j$ equals $K(Z^+_{ij}, Z^-_{ij})$. A full sample from the model is $(U, Z^+, Z^-, X) \in \cS_1^m \times [n]^{m(m-1)} \times [n]^{m(m-1)} \times \{0,1\}^{n(n-1)}$ where $X$ is the adjacency matrix of the graph on $[m]$. Note that conditionally on $U$, node $i$ is linked to $j$ with probability
\[
 \sum_{k=1}^n \sum_{\ell=1}^n K(u,v) U_i(k) U_j(\ell),
\]
and unconditionally,
\[
 \iint \sum_{k=1}^n \sum_{\ell=1}^n K(u,v) U_i(k) U_j(\ell) dU_i dU_j.
\]

\section{New upper bounds}

For a collection of sets $A$, we denote by $A^\flat = \cup_{a \in A} a$ the set of elements covered by one or more sets of the collection, so that for example, $\{\{1,2\}\}^\flat = \{1,2\}$ and $\{ \{1,2\}, \{1, 3\} \}^\flat = \{1,2,3\}$.

Given a partition $\cA$ of the link set $E(R)$ of a graph $R$ into nonempty sets, we denote by $\abs{\cA}$ the number of parts in the partition, and we set $\norm{\cA} = \sum_{A \in \cA} \abs{A^\flat}$ where $A^\flat = \cup_{e \in A} e$ equals the set of nodes covered by the node pairs of $A$. See Table~\ref{tab:TrianglePartition}.
\begin{table}[h]
\centering
\small
\begin{tabular}{lcc}
$\cA$ & $\abs{\cA}$ & $\norm{\cA}$ \\
\midrule
$\{ \{ e_{12}, e_{13} \}\}$ & 1 & 3 \\
$\{ \{ e_{12} \}, \{ e_{13} \} \}$ & 2 & 4
\end{tabular}
\qquad
\begin{tabular}{lcc}
$\cA$ & $\abs{\cA}$ & $\norm{\cA}$ \\
\midrule
$\{ \{ e_{12}, e_{13}, e_{23} \}\}$ & 1 & 3 \\
$\{ \{ e_{12} \}, \{ e_{13}, e_{23} \} \}$ & 2 & 5 \\
$\{ \{ e_{13} \}, \{ e_{12}, e_{23} \} \}$ & 2 & 5 \\
$\{ \{e_{23}\}, \{ e_{12}, e_{13} \} \}$ & 2 & 5 \\
$\{ \{ e_{12} \}, \{ e_{13} \},  \{ e_{23} \}\}$ & 3 & 6
\end{tabular}
\caption{\label{tab:TrianglePartition} Link partitions of a two-star (left) and a triangle (right) on node set $\{1,2,3\}$. Here $e_{ij} = \{i,j\}$.}
\end{table}

Lemma~\ref{the:OverlappingUnionBound} is a strengthened version of the standard union bound for sufficiently overlapping set collections.

\begin{lemma}[Overlapping union bound]
\label{the:OverlappingUnionBound}
Let $k \ge 2$ and let $C_1, \dots, C_k$ be finite sets such that for each $i$ there exists $j \ne i$ such that $C_i \cap C_j$ is nonempty. Then
\[
 \left\lvert{\bigcup_{i=1}^k C_i}\right\rvert
 \wle \sum_{i=1}^k C_i - k + 1.
\]
\end{lemma}
\begin{proof}
When $C_1 \cap C_2$ is nonempty,
\[
 \abs{C_1 \cup C_2}
 \weq \abs{C_1} + \abs{C_2} - \abs{ C_1 \cap C_2}
 \wle \abs{C_1} + \abs{C_2} - 1
\]
show that the claim is true for $k=2$. To proceed by induction, fix sets $C_1, \dots, C_{k+1}$ satisfying the property of the lemma. If necessary, relabel the sets so that $C_k \cup C_{k+1}$ is nonempty. Then
\[
 \abs{C_k \cup C_{k+1}}
 \weq \abs{C_k} + \abs{C_{k+1}} - \abs{C_k \cap C_{k+1}}
 \wle \abs{C_k} + \abs{C_{k+1}} - 1,
\]
so that
\[
 \sum_{i=1}^{k+1} C_i
 \wge \sum_{i=1}^{k-1} C_i + \abs{C_k \cup C_{k+1}} + 1
 \weq \sum_{i=1}^{k} \abs{D_i} + 1,
\]
where $D_i = C_i$ for $i \le k-1$ and $D_k = C_k \cup C_{k+1}$. Then by the induction assumption,
\[
 \sum_{i=1}^{k} \abs{D_i}
 \wge \left\lvert{\bigcup_{i=1}^k D_i}\right\rvert + k-1
 \weq \left\lvert{\bigcup_{i=1}^{k+1} C_i}\right\rvert + k-1,
\]
and hence
\[
 \sum_{i=1}^{k+1} C_i
 \weq \sum_{i=1}^{k} \abs{D_i} + 1
 \wge \left\lvert{\bigcup_{i=1}^{k+1} C_i}\right\rvert + k.
\]
Now the claim follows by induction.
\end{proof}

\begin{lemma}
\label{the:MinimumPartition}
For a any connected graph $R$,
\[
 \min_\cA \{ \norm{\cA} - \abs{\cA}\}
 \weq \abs{V(R)} - 1,
\]
where the minimum is taken over all partitions of $E(R)$ into nonempty sets.
\end{lemma}
\begin{rcomm}
For a disconnected graph $R$ consisting of two nonoverlapping triangles, the trivial link partition produces $\norm{\cA} - \abs{\cA} = 6 - 1 = 5$. The partition with two parts consisting of the link sets of the connected components gives $\norm{\cA} - \abs{\cA} = 6 - 2 = 4$. We might conjecture that in general,
\[
 \min_\cA \{ \norm{\cA} - \abs{\cA}\}
 \weq \abs{V(R)} - \#\text{connected components of $R$}.
\]
\end{rcomm}
\begin{proof}
Fix some partition $\cA$ of $E(R)$ into nonempty sets, and label its parts as $A_1,\dots, A_k$. Because $R$ is connected, there are no isolated nodes, and hence it follows that $\cup_i A_i^\flat = V(R)$. The connectedness of $R$ also implies that for every $i$ there exists some $j \ne i$ such that $A_i^\flat \cap A_j^\flat$ is nonempty.
Lemma~\ref{the:OverlappingUnionBound} can be hence applied to conclude that
\[
 \abs{V(R)}
 \weq \left\lvert{\bigcup_{i} A_i^\flat}\right\rvert
 \wle \sum_i \abs{A_i^\flat} - k + 1
 \wle \norm{\cA} - \abs{\cA} + 1.
\]
The claim follows because the above upper bound holds as equality for the trivial partition $\cA = \{ E(R) \}$ in which $\abs{\cA} = 1$ and $\norm{\cA} = \abs{V(R)}$. 
\end{proof}

\begin{lemma}
\label{the:GenericUpperBound}
For any graph $R$ with $V(R) \subset V(G)$, the probability that $G$ contains $R$ as a subgraph is bounded by
\[
 \pr(G \supset R)
 \wle \sum_\cA m^\abs{\cA} \prod_{A \in \cA} \frac{(\pi)_{\abs{A^\flat}, \abs{A}}}{(n)_\abs{A^\flat}}.
\]
where the sum is over all partitions $\cA$ of $E(R)$ into nonempty sets.
\end{lemma}

\begin{rcomm}
For a lower bound, consider the trivial partition $\cA_0 = \{A\}$ with $A = E(R)$. Then all maps from $\sigma: \cA \to [m]$ are constants, and
\[
 \sum_\sigma \prod_{A \in \cA_0} q_{\sigma(A)}^\abs{A} \frac{(x_{\sigma(A)})_\abs{A^\flat}}{(n)_\abs{A^\flat}}
 \weq \sum_{k=1}^m q_k^\abs{E(R)} \frac{(x_k)_\abs{V(R)}}{(n)_\abs{V(R)}}
 \weq \frac{m}{(n)_\abs{V(R)}} (\pi)_{\abs{V(R)}, \abs{E(R)}}.
\]
Now
\[
 \pr \left( G \supset R \right)
 \weq \pr \left( \bigcup_\cA \bigcup_\sigma \bigcap_{A \in \cA} \{ E(G_{\sigma(A)}) \supset A \} \right) \\
 \wge \pr \left( \bigcup_\sigma \bigcap_{A \in \cA_0} \{ E(G_{\sigma(A)}) \supset A \} \right).
\]
Maybe we can show with a second moment method to the right side that 
\[
 \pr \left( G \supset R \right)
 \wgesim n^{-\abs{V(R)}+1} (\pi)_{\abs{V(R)}, \abs{E(R)}}.
\]
\end{rcomm}

\begin{proof}
Observe that $G = \cup_{k=1}^m G_k$, where $G_k$ is the graph with node set $[n]$ and link set $\{e \in \binom{V_k}{2}: C_{e,k}=1\}$. We see that $G \supset R$ if and only if there exists a partition $\cA$ of $E(R)$ into nonempty sets, and an injective map $\sigma: \cA \to [m]$ such that $E(G_{\sigma(A)}) \supset A$ for all $A \in \cA$. Here $\sigma(A)$ refers to a layer which is responsible for realizing the links in $A$. Observe that
\begin{align*}
 \pr( E(G_k) \supset A )
 \weq \pr( V_k \supset A^\flat, \ C_{e,k} = 1 \ \text{for all} \ e \in A )
 \weq q_k^\abs{A} \frac{(x_k)_\abs{A^\flat}}{(n)_\abs{A^\flat}}
\end{align*}
Hence by the independence of $G_1,G_2, \dots, G_m$, it follows that
\begin{align*}
 \pr \left( G \supset R \right)
 &\weq \pr \left( \bigcup_\cA \bigcup_\sigma \bigcap_{A \in \cA} \{ E(G_{\sigma(A)}) \supset A \} \right) \\
 &\wle \sum_\cA \sum_\sigma \pr \left( \bigcap_{A \in \cA} \{ G_{\sigma(A)} \supset A \} \right) \\
 &\weq \sum_\cA \sum_\sigma \prod_{A \in \cA} q_{\sigma(A)}^\abs{A} \frac{(x_{\sigma(A)})_\abs{A^\flat}}{(n)_\abs{A^\flat}}.
\end{align*}
In the above sums and unions, the symbol $\sigma$ ranges over all injective maps from $\cA$ into $[m]$.  To simplify the last sum, we switch to labeled partitions. A labeled partition of $E(R)$ of size $b$ is a list $(A_1,\dots, A_b)$ of disjoint sets such that $\cup_{i=1}^b A_i = E(R)$. Because each (unlabeled) partition of size $b$ corresponds to $b!$ labeled partitions, and each injection from $\cA$ of size $b$ into $[m]$ corresponds to a list $(k_1,\dots, k_b)$ of distinct layers, we can condition on the size of the partition and write
\begin{align*}
 &\sum_\cA \sum_\sigma  \prod_{A \in \cA} q_{\sigma(A)}^\abs{A} \frac{(x_{\sigma(A)})_\abs{A^\flat}}{(n)_\abs{A^\flat}} \\
 &\weq \sum_{b = 1}^{\abs{E(R)}} \frac{1}{b!} \sum_{(A_1,\dots,A_b)} \sum_{(k_1, \dots, k_b) \in [m]^b_{\ne}}
\prod_{i=1}^b q_{k_i}^\abs{A_i} \frac{(x_{k_i})_\abs{A_i^\flat}}{(n)_\abs{A_i^\flat}} \\
 &\wle \sum_{b = 1}^{\abs{E(R)}} \frac{1}{b!} \sum_{(A_1,\dots,A_b)} \sum_{(k_1, \dots, k_b) \in [m]^b}
\prod_{i=1}^b q_{k_i}^\abs{A_i} \frac{(x_{k_i})_\abs{A_i^\flat}}{(n)_\abs{A_i^\flat}} \\
 &\weq \sum_{b = 1}^{\abs{E(R)}} \frac{1}{b!} \sum_{(A_1,\dots,A_b)} 
 \prod_{i=1}^b \left( \sum_{k=1}^m q_{k}^\abs{A_i} \frac{(x_{k})_\abs{A_i^\flat}}{(n)_\abs{A_i^\flat}} \right) \\
 &\weq \sum_{b = 1}^{\abs{E(R)}} \frac{m^b}{b!} \sum_{(A_1,\dots,A_b)} 
 \prod_{i=1}^b \frac{(\pi)_{\abs{A_i^\flat},\abs{A_i}}}{(n)_\abs{A_i^\flat}}.
\end{align*}
The last sum also equals
\[
 \sum_\cA m^\abs{\cA} \prod_{A \in \cA} \frac{(\pi)_{\abs{A^\flat}, \abs{A}}}{(n)_\abs{A^\flat}}.
\]
\end{proof}

\begin{lemma}\mnote{We get a lower bound by just taking one partition in the union into account.}
\label{the:KappaBound}
Assume that $m = O(n)$ and $(\pi)_{r,s} = O(1)$ for all $s \le \abs{E(R)}$ and $r \le \abs{V(R)}$. Then for any connected graph $R$ with $V(R) \subset V(G)$, the probability that $G$ contains $R$ as a subgraph is bounded by
\[
 \pr(G \supset R)
 \weq O(n^{-\abs{V(R)}+1}).
\]
\end{lemma}
\begin{proof}
Because $(n)_r \ge (1-r/n) n^r$, it follows that
\[
 \prod_{A \in \cA} (n)_\abs{A^\flat}
 \wge \prod_{A \in \cA} (1-\abs{A^\flat}/n) n^\abs{A^\flat}
 \wge (1-\abs{V(R)}/n) n^\norm{\cA}
 \wge 0.99 n^\norm{\cA}
\]
for all $n \ge 100 \abs{V(R)}$.
Moreover, by Lemma~\ref{the:MinimumPartition} we find that
\[
 m^\abs{\cA} n^{-\norm{\cA}}
 \weq \left(\frac{m}{n}\right)^{\abs{\cA}} n^{\abs{\cA} - \norm{\cA}}
 \wle \left(\frac{m}{n}\right)^{\abs{E(R)}} n^{-\abs{V(R)}+1}.
\]
Hence by Lemma~\ref{the:GenericUpperBound} we find that
\begin{align*}
 \pr(G \supset R)
 &\wle \sum_\cA m^\abs{\cA} \prod_{A \in \cA} \frac{(\pi)_{\abs{A^\flat}, \abs{A}}}{(n)_\abs{A^\flat}} \\
 &\wle 0.99 \sum_\cA m^\abs{\cA}  n^{-\norm{\cA}} \prod_{A \in \cA} (\pi)_{\abs{A^\flat}, \abs{A}} \\
 &\wle 0.99 \left(\frac{m}{n}\right)^{\abs{E(R)}} n^{-\abs{V(R)}+1} \sum_\cA \prod_{A \in \cA} (\pi)_{\abs{A^\flat}, \abs{A}}.
\end{align*}

Hence
\[
 \prod_{A \in \cA} \frac{(\pi)_{r, s}}{(n)_r}
 \wle c \norm{q}_\infty^\abs{E(R)} n^{-\norm{\cA}}
\]
where $c = \max_r (\pi)_{r, 0}$. Hence
\[
 \pr(G \supset R)
 \wle \sum_\cA m^\abs{\cA} \prod_{A \in \cA} \frac{(\pi)_{\abs{A^\flat}, \abs{A}}}{(n)_\abs{A^\flat}}
 \wle c \norm{q}_\infty^\abs{E(R)} \sum_\cA m^\abs{\cA} n^{-\norm{\cA}}
\]
\end{proof}




\section{Old proof stuff}

\begin{bcomm}

\subsubsection{Compound Poisson approximations for the upper bound (deprecated?)}

Next, by Lemma~\ref{the:CompoundBinPoisson}, we note that $\dtv( f, \tilde f_\tau ) \le \sum_{xy} n_{xy} p_{xy}^2  = \sum_k (\frac{X_k}{n-\tau})^2 \le \frac{M^2 n}{(m-\tau)^2} \le 4 M^2 m^{-2}n$ where $\tilde f _\tau = \CPoi( \lambda_{n,\tau}, \bar g_n )$ is a compound Poisson distribution with rate parameter $\lambda_{n,\tau} = \frac{n}{m-\tau} (P_n)_{10}$ and increment distribution $g_n$. Then by Lemma~\ref{the:CPoiPerturbation}, $\dtv(\tilde f_\tau, \tilde f_0) \le \abs{\lambda_{n,\tau} - \lambda_n} \le 2 M m^{-2} n \tau$ due to $(P_n)_{10} \le M$. Hence $\dtv(f, \tilde f_0) \le 4 M^2 m^{-2}n + 2 M m^{-2} n \tau$.

A natural coupling implies that $\dtv( \law(Q'_0, \dots, Q'_\tau), \law(Q''_0, \dots, Q''_\tau)) \le (4 M^2 m^{-2}n + 2 M m^{-2} n \tau) \tau \le 6M^2 m^{-2} n \tau^2$, where $(Q''_0, \dots, Q''_\tau)$ is defined similarly as in \eqref{eq:ExplorationQueueNew} but with $Z_1,Z_2,\dots$ replaced by independent $\tilde f_0$-distributed random integers $Z''_1, Z''_2, \dots$ Hence $\abs{\rho_\tau( f ) - \rho_\tau( \tilde f_0 )} \le (4 M^2 m^{-2}n + 2 M m^{-2} n \tau) \tau \le 6M^2 m^{-2} n \tau^2$.

The claim follows by noting that $\tau^2 \le \tau^2 \log t$ and $\tau \le \tau^2 \log \tau$ for $\tau \ge 3$.
\end{bcomm}

\subsubsection{Compound Poisson approximations (deprecated?)}
\begin{bcomm}
(i) Define modifications of the distribution $f_{\delta, \tau, \nu}$ in \eqref{eq:LowerOffspring} by
\[
 f_{\delta} = \law\Big( \sum_{(x,y) \in A} \sum_{k=1}^{n_{xy}} B_{xy}(k) T_{xy}(k) \Big),
 \quad
 \tilde f_{\delta} = \law\Big( \sum_{(x,y) \in A} \sum_{k=1}^{n_{xy}} \tilde B_{xy}(k) T_{xy}(k) \Big),
\]
where the random variables are mutually independent and such that $\law(B_{xy}(k)) = \Ber( (1-\delta) \frac{x}{m})$, $\law(\tilde B_{xy}(k)) = \Poi( (1-\delta) \frac{x}{m})$, and $\law(T_{xy}(k)) = \Bin^+(x-1,y)$. A natural coupling implies that
\[
 \dtv(f_{\delta}, f_{\delta, \tau, \nu})
 \wle \sum_{(x,y) \in A} (n_{xy} - n^a_{xy,\tau-1}) (1-\delta) \frac{x}{m}
 \wle M \abs{A} \frac{\tau\nu}{m}.
\]
Because $\dtv(\Ber(p), \Poi(p)) = p(1-e^{-p}) \le p^2$ for all $0 \le p \le 1$, a natural coupling implies that
\[
 \dtv(f_{\delta}, \tilde f_{\delta})
 \wle \sum_{(x,y) \in A} \sum_{k=1}^{n_{xy}} ((1-\delta) \frac{x}{m})^2
 \wle (\frac{M}{m})^2 \sum_{(x,y) \in A} n_{xy}
 \wle M^2 (n/m) m^{-1}.
\]
Hence
\[
 \dtv(f_{\delta, \tau, \nu}, \tilde f_{\delta})
 \wle M \abs{A} \frac{\tau \nu}{m} + M^2 (n/m) m^{-1}
 \wle M^2 \left(1 + \frac{n}{m} \right) \abs{A} \frac{\tau \nu}{m}.
\]
Now let us observe that $\law(\sum_{k=1}^{n_{xy}} \tilde B_{xy}(k) T_{xy}(k)) = \CPoi((1-\delta) n_{xy} \frac{x}{m}, \Bin^+(x-1,y))$, so by Lemma~\ref{the:CPoiSum} we see that $\law(\tilde f_{\delta}) = \CPoi((1-\delta)\lambda, g)$ with $\lambda = \sum_{(x,y) \in A} \frac{x}{m} n_{xy} = \frac{n}{m} \int_A x P_n(dx,dy)$ and $g = \int_A \Bin^+(x-1,y) \frac{x P_n(dx,dy)}{\int_A x P_n(dx,dy)}$. 

(ii) Define modifications of the distribution $f_{\delta, \tau, \nu}$ in \eqref{eq:LowerOffspring} by
\[
 \tilde f_{\delta, \tau, \nu}
 = \law \Big( \sum_{(x,y) \in A} \sum_{k=1}^{n^a_{xy,t-1}} \tilde B_{xy}(k) T_{xy}(k) \Big),
 \quad
 \tilde f_{\delta} = \law\Big( \sum_{(x,y) \in A} \sum_{k=1}^{n_{xy}} \tilde B_{xy}(k) T_{xy}(k) \Big),
\]
where $\law(\tilde B_{xy}(k)) = \Poi( (1-\delta) \frac{x}{m})$. Because $\dtv(\Ber(p), \Poi(p)) = p(1-e^{-p}) \le p^2$ for all $0 \le p \le 1$, and $\sum_{(x,y) \in A} \sum_{k=1}^{n^a_{xyt}} \le n$, a natural coupling implies that
\[
 \dtv(f_{\delta, \tau, \nu}, \tilde f_{\delta, \tau, \nu})
 \wle \sum_{(x,y) \in A} \sum_{k=1}^{n^a_{xyt}} \left( (1-\delta) \frac{x}{m} \right)^2
 \wle \frac{M^2}{m^2} n.
\]
We also find that, noting that $\pr( \tilde B_{xy}(k)) > 0) \le (1-\delta) \frac{x}{m}$, 
\[
 \dtv(\tilde f_{\delta, \tau, \nu}, \tilde f_{\delta})
 \wle \sum_{(x,y) \in A} (n_{xy} - n^a_{xy,\tau-1}) (1-\delta) \frac{x}{m}
 \wle M \abs{A} \frac{\tau\nu}{m}.
\]

\begin{rcomm}
A different approximation can be derived by noting that 
\[
 f_{\delta, \tau, \nu}
 \weq \law \Big( \sum_{(x,y) \in A} \sum_{k=1}^{N^a_{xy,t-1}} T_{xy}(k) \Big),
 \quad
 \tilde f_{\delta, \tau, \nu}
 \weq \law \Big( \sum_{(x,y) \in A} \sum_{k=1}^{\tilde N^a_{xy,t-1}} T_{xy}(k) \Big),
\]
where $\law(N^a_{xy,t-1}) = \Bin(n^a_{xy,t-1}, (1-\delta) \frac{x}{m})$ and $\law(\tilde N^a_{xy,t-1}) = \Poi(n^a_{xy,t-1} (1-\delta) \frac{x}{m})$, and applying the inequality \cite{LeCam_1960} $\dtv(\Bin (n,p), \Poi(np)) \le 9p$ (see Lemma~\ref{the:CompoundBinPoisson}). Then a natural coupling implies that
\[
 \dtv(f_{\delta, \tau, \nu}, \tilde f_{\delta, \tau, \nu})
 \wle \sum_{(x,y) \in A} 9 (1-\delta) \frac{x}{m}
 \wle 9 M \abs{A} m^{-1}.
\]
\end{rcomm}

Now a natural coupling implies that $\abs{ \rho_\tau(f_1) - \rho_\tau(f_2)} \le \tau \dtv(f_1,f_2)$. Hence, 
\[
 \rho_\tau(f_{\delta, \tau, \nu})
 \wge \rho_\tau( \tilde f_{\delta, \tau, \nu}) - M^2 (n/m) \tau m^{-1}.
\]
Now we cannot approximate $\rho_\tau( \tilde f_{\delta, \tau, \nu}) \approx \rho_\tau( \tilde f_{\delta})$ well enough. However, we note that
\[
 \rho_\tau(f_{\delta, \tau, \nu})
 \wge \rho_\tau( \tilde f_{\delta, \tau, \nu}) - M^2 (n/m) \tau m^{-1}
 \wge \rho( \tilde f_{\delta, \tau, \nu}) - M^2 (n/m) \tau m^{-1},
\]
and we can approximate $\rho(\tilde f_{\delta, \tau, \nu}) \approx \rho( \tilde f_{\delta})$ well. Here we can cite either BJR07 AND/OR Leskelä and Ngo.

\end{bcomm}

\rnote{Old proof stuff} where $\rho_k(Z)$ denotes the probability that a Galton--Watson process with offspring distributed according to $Z$ has total progeny larger than $k$, and where $Z_{n,\delta}^\pm$ are compound binomial distributions defined using stochastic representations
\begin{equation}
 \label{eq:CompoundBinomial}
 Z_{n,\delta}^{\pm} \ = \sum_{(x,y) \in A^0} \sum_{s=1}^{N^\pm_{xy}} T_{xy}(s),
\end{equation}
where the random variables appearing in the sum are mutually independent, $T_{xy}(s)$ is $\Bin^+(x-1,y)$-distributed for all $s$, and
\[
 N^-_{xy}  \eqst \Bin \Big((1-2\delta) nP(x,y), (1-\delta) \frac{x}{m} \Big),
 \quad
 N^+_{xy} \eqst \Bin \Big(nP_n(x,y), \frac{x}{m-\omega'} \Big).
\]
Let us also define compound Poisson distributions $\tilde Z_{n,\delta}^\pm$ using the same stochastic representation \eqref{eq:CompoundBinomial} but with $N^\pm_{xy}$ replaced by a Poisson-distributed random variable $\tilde N^\pm_{xy}$ having the same mean as $N^\pm_{xy}$. Using the total variation distance bound $\dtv(\Bin (n,p), \Poi(np)) \le 9p$ by Le Cam \cite{LeCam_1960}, we find that
\begin{rcomm}
For $\rho_t(Z) = \pr(T(Z) > t)$ where $T(Z)$ represents the total progeny of GW with offspring $Z$, we may prove that $\abs{\rho_t(Z) - \rho_t(Z')} \le t \dtv(Z,Z')$.
\end{rcomm}
\begin{equation}
 \label{2019-06-14-1++}
 \begin{aligned}
 \bigl| \rho_k ( Z_{n,\delta}^-) - \rho_k ( \tilde Z_{n,\delta}^- ) \bigr|
 &\wle 9 k \frac{M}{m}, \\
 \bigl| \rho_k ( Z_{n,\delta}^+ ) - \rho_k ( \tilde Z_{n,\delta}^+ ) \bigr|
 &\wle 9 k \frac{M}{m-\omega'}.
 \end{aligned}
\end{equation}
Because $\Bin^+(x-1,y)$ equals the Dirac measure at zero whenever $x \le 1$ or $y=0$, we find that we can replace $A_0$ by $A$ in \eqref{eq:CompoundBinomial}. With the help of Lemma~\ref{the:CPoiSum} we find that
\begin{align*}
 \law(\tilde Z_{n,\delta}^-) &\weq \CPoi( (1-\delta)(1-2\delta) \frac{n}{m} (P)_{10}, \bar g)
 \wto \CPoi( (1-\delta)(1-2\delta) \lambda, \bar g) \\
 \law(\tilde Z_{n,\delta}^+) &\weq \CPoi( \frac{n}{m-\omega'} (P_n)_{10}, \bar g_n)
 \wto \CPoi( \lambda, \bar g)
\end{align*}
where $\bar g$ is the distribution defined by \eqref{eq:MixedBinPlus} and $\bar g_n$ is defined using the same formula with $P$ replaced by $P_n$.

Let $\cY^+_{\delta}$, $\cY^{-}_{\delta}$, $\cY$ be compound Poisson distributions such that 
\begin{equation}
 \label{2019-06-24}
 \cY^\pm_\delta = \CPoi((1 \pm \delta) \lambda, \bar g),
 \quad
 \cY = \CPoi(\lambda, \bar g),
\end{equation}
where $\lambda = \mu (P)_{10}$ and $\bar g$ is the distribution defined in \eqref{eq:MixedBinPlus}.

From $P_n \weakto P$ , and (\ref{2019-06-14-1++}), recalling that $\frac{n}{m} \to \mu$, we obtain for  $k=k(n)$ such that $1 \le k \ll m$,
\begin{eqnarray}
\label{2019-06-25+1}
 &&
 \rho_k \bigl(\cY^{-}_{3\delta}\bigr)
 \wle \rho_k \bigl( \tilde Z_{n,\delta}^- \bigr)
 \wle \rho_k \bigl( Z_{n,\delta}^- \bigr) + o(1), \\
 &&
 \label{2019-06-25+2}
 \rho_k \bigl(\cY^+_{2\delta}\bigr)
 \wge \rho_k \bigl(\tilde Z_{n,\delta}^+ \bigr)
 \wge \rho_k \bigl(Z^+_{n,\delta} \bigr) - o(1).
\end{eqnarray}
Then (\ref{2019-06-25+1}), (\ref{2019-06-25+2}),
(\ref{2019-07-26+1}) imply
\[
 \rho_\omega \bigl( \cY^{-}_{3\delta} \bigr) - o(1)
 \wle \pr(C_i > \omega)
 \wle \rho_{\omega'} \bigl( \cY^+_{2\delta} \bigr)+o(1).
\]

Now the lower bound of \eqref{eq:MBGiant1} follows from $\rho_\omega(\cY^{-}_{3\delta}) \ge \rho(\cY^{-}_{3\delta})\to \rho(\cY)$ as $\delta\downarrow 0$.

For the upper bound, first fix any $\epsilon > 0$. Then find a large enough $k_0$ such that $\rho_{k_0}(\cY) \le \rho(\cY)+\epsilon$. Then fix a small $\delta_0 > 0$ so that $\rho_{k_0}(\cY^+_{2\delta}) \le \rho_{k_0}(\cY)+\epsilon$ for all $0 < \delta < \delta_0$. Then for all $k \ge k_0$ and all $0 < \delta < \delta_0$,
\[
 \rho_k(\cY^+_{2\delta})
 \wle \rho_{k_0}(\cY^+_{2\delta})
 \wle \rho(\cY) + 2 \epsilon.
\]
Hence
\[
 \rho_k \bigl( \tilde Z_{n,\delta}^+ \bigr)
 \wle \rho_k \bigl( \cY^+_{2\delta} \bigr)
 \wle \rho(\cY) + 2 \epsilon
\]

%

\begin{rcomm}
Any layer $G_k$ with type in $A \setminus A_0$ has size less than two or strength zero, and hence $E(G_k) = \emptyset$ with probability one. These layer are called trivial. Hence $E(G) = \cup_{k: (X_k,Y_k) \in A_0} E(G_k)$, and we may ignore the trivial layers from the component exploration. The set of nontrivial layers is denoted by $\DD_{0}$. The number of nontrivial layers in the $n$-th model equals $n P_n(A_0) = (P(A_0) +o(1))n \asymp n$.
\end{rcomm}

\emph{Proof of the upper bound in \eqref{2019-07-26+1}}. \bnote{Upper coupling bound for large component probability.}
Let $n'=nP_n(A_0)$ be the number of nontrivial layers in the $n$-th model. We assume without loss of generality that 
nontrivial layers are ordered by their size and labeled by 1,2,\dots,$n'$.

Given a node $w$, define the list $L_w$ of nodes using a BFS type exploration procedure. In the beginning all nodes are uncolored, all nontrivial layers are unmarked, and $L_w = \emptyset$. After a node is added to $L_w$ the node is colored white. We add $w$  to the list. Next we proceed recursively. We choose  the oldest (with respect to inclusion to $L_w$) white 
node, say  $u$, from $L_w$. For layers $i=1,2,\dots, n'$ such that $u \in V(G_i)$ and $G_i$ is unmarked, we mark $G_i$ (we say that $G_i$ is marked by $u$) and add to $L_w$ (in increasing order) all uncolored nodes of $G_i$ that are connected to $u$ by paths in $G_i$. We say that $G_i$ brings these nodes to the list and attach label $G_i$ to each of them.  Afterwards we color $u$ black. Nodes added to $L_w$ in this step are called children of $u$. We then chose the oldest white node from $L_w$, add to $L_w$ its children and color this node black etc. We stop when there are no more white nodes in $L_w$ or there are no more unmarked layers $G_i$ left. We denote $L_w= \{u_1,u_2,\dots\}$, where 
elements are listed in the order of their inclusion to the list ($u_i$ is older than $u_j$ for $i<j$ and $u_1=w$). We denote $L_{w,k}=\{u_1,\dots, u_k\}$ the set of $k$ oldest nodes of $L_w$. Note that $L_w$ is a subset of $C_{w}$. 
For any $u_i\in L_w$ with $i\ge 2$ there is unique $i^*\in[1,i)$ such that $u_i$ is a child of $u_{i^*}$ (equivalently, $u_{i^*}$ is the parent of  $u_i$). While constructing the list $L_w$ we keep track of the layers $G_{i_1}, G_{i_2},\dots$ that have been marked one after another ($G_{i_s}$ was marked before $G_{i_t}$ for $s<t$). For $u_j \in L_w$ the number $r=r(j)$ tells us that $u_j$ was brought to the list by $G_{i_r}$, the $r$-th member of the sequence ${\mathbb G}_w = \{G_{i_1}, G_{i_2},\dots\}$. A layer $G_{i_s}$ marked by $u\in L_w$ is called \emph{void} if $u$ has no neighbors in $G_{i_s}$ linked to $u$ by links labeled $G_{i_s}$ (in this case $G_{i_s}$ brings no children to $u$). Note that any $G_{i_j}$ is void with probability at most $1-{\hat q}$. A nontrivial layer $G_{i_s}$ is called \emph{regular} if $\cup_{j=1}^{s-1} G_{i_j}$ and $G_{i_{s}}$ intersect 
in a single point. Node $w$ is called \emph{$k$-regular} if $|L_w| \ge k$  and $G_{i_j}$ is regular for $j=2,3,\dots, r(k)$. The set of $k$-regular nodes of $G$ is denoted $W_k$. Note that the events $\{|C_{w}|\ge k, w\in W_k\}$ and $\{|L_w|\ge k, w\in W_k\}$ are equal.

The number of nodes brought to the list $L_w$ by a regular layer $G_{i_s}$ of type $(x,y)$ is $\Bin^+(x-1,y)$-distributed. For an irregular layer this number may be smaller, because $G_{i_s}$ cannot bring to $L_w$ those nodes of $G_{i_s}$ that have been colored in previous steps of the exploration. Therefore, as long as $k \le \omega'$, a coupling of the exploration process with the branching process $\gw(Z_n^+)$ shows that
\begin{equation}
\label{2019-07-09}
 \pr\{|L_w| \ge k\}
 \wle 
 \pr\{\gw(Z_n^+) \ge k\}
 \weq \rho_k(Z_n^+).
\end{equation}

Next we show that 
\begin{equation}
\label{2019-06-30}
 \pr\{|L_w|\ge k, w\notin W_k\}
 \wle 
 {\hat k}^2 M^2 m^{-1} + 2k^{-1},
\end{equation}
where ${\hat k}:=2k/{\hat q}$. For node $w$ with $|L_w| \ge k$ the event $\{w\notin W_k\}$ implies that either $G_{i_s}$ is irregular for some $s\le {\hat k}$ (this event we denote ${\cal A}_k$), or there are at least ${\hat k}-k+2$ void layers $G_{i_l}$ with $l \le  {\hat k}$ (this event we denote ${\cal B}_k$). Indeed, on the event 
${\overline {\cal A}}_k\cap \{w\notin W_k\} \cap \{|L_w|\ge k\}$, the index $j$ of the first observed irregular layer $G_{i_j}$ satisfies ${\hat k}<j\le r(k)$. But the inequality ${\hat k} < r(k)$ implies that among the first ${\hat k}$ layers from $\DD_w$ there are less than $k-1$ nonvoid ones as each nonvoid set contributes at least one new node to the list. Now (\ref{2019-06-30}) follows from the inequalities
\begin{eqnarray}
 \label{2019-06-30+1}
 \pr\{{\cal B}_k\} &\wle &\pr\{Y<k-1\} \wle 2k^{-1}, \\
 \label{2019-06-30+2}
 \pr\{{\cal A}_k\} &\wle & \sum_{2\le s\le {\hat k}} \pr\{D_{i_s} \ \text{is irregular} \}
 \wle ({\hat k}-1)\frac{({\hat k}-1)M^2}{m}.
\end{eqnarray}
Here $Y \eqst \Bin({\hat k}, {\hat q})$ and (\ref{2019-06-30+1}) follows by Chebyshev's inequality.
In (\ref{2019-06-30+2}) we estimated $\pr\{G_{i_s} \ \text{is irregular}\} \le ({\hat k}-1)M^2/(m-({\hat k}-1))$. Indeed, given $H_{s-1}=\cup_{1\le j\le s-1}G_{i_j}$, the size 
$|G_{i_s}|=t$ and the event that $G_{i_s}$ is marked by $u_j$, the probability that $G_{i_s}$ is irregular is the conditional probability 
$p^* = \pr\bigl\{ |H_{s-1}\cap D^*|\ge 2\,\bigr| u_j\in D^* \bigr\}$, where $D^*$ is a random subset of size $t$ of the set $[m] \setminus \{u_1,\dots, u_{j-1}\}$. For $|H_{s-1}|=h$ we have 
\begin{equation}
 \label{2019-07-01}
 p^*
 = 
 \frac{
 \pr\{|H_{s-1}\cap D^*|\ge 2, u_j\in D^*\}
 }
 {
 \pr\{u_j\in D^*\}
 }
 \le
 \frac{(h-j)(t-1)}{m-j}.
\end{equation}
The last fraction upper bounds the probability that $D^* \setminus \{u_j\}$ of size $t-1$ intersects with $
 H_{s-1} \setminus \{u_1,\dots, u_j\}$ of size $h-j$. Note that 
 $(h-j)/(m-j)\le h/m$ and  $h\le ({\hat k}-1)M$ and $t\le M$. Therefore, the right side of \eqref{2019-07-01} is at most $({\hat k}-1)M^2/m$. This shows (\ref{2019-06-30+2}) and we arrive to 
 (\ref{2019-06-30}). It follows from 
(\ref{2019-06-30}) that
\begin{equation}
\label{2019-07-24+1}
 k=o(\sqrt{m})
 \ \Rightarrow \
 \pr\{|L_w|\ge k, w\notin W_k\} = o(1).
\end{equation}
The same argument yields  
\begin{equation}
\label{2019-07-25}
 k=o(\sqrt{m})
 \ \Rightarrow \
 \pr\{|C_{w}|\ge k, w\notin W_k\}=o(1).
\end{equation}  
Indeed, for $w$ with $|C_{w}| \ge k$, the event $w\notin W_k$ 
implies that either $|L_w|<k$ or $|L_w|\ge k$ and 
$G_{i_j}$ is irregular for some $2 \le j\le r(k)$. The 
probability of the latter event is bounded by \eqref{2019-07-24+1}.
Furthermore, the event $\{|C_{w}|\ge k$, $|L_w|<k\}$ implies that an irregular layer $G_{i_s}$ has been marked by some $u_i \in L_w$ (note that $i<k$). 

\begin{rcomm}
-----x-x-x--x-- for preprint only---x-x-x-x-x-x-

Indeed we always have $L_w\subset C_w$. The situation where $L_w\not=C_w$ happens when some irregular $D_{i_s}$ brings a node  $u'$ to the list $L_w$ such that $u_j\in D_{i_t}$, for $t<s$, and $u_j$ is connected by an link labeled $D_{i_t}$ to another node $u''\in D_{i_t}$, which do not belong to $L_w$. Then $u''\in C_w$ and $u''\notin L_w$.

-------x-x-x-x-x-- end of prteprint only -------x-x-x-x-x- 
\end{rcomm}

The probability that the index $s$ of the first irregular layer $G_{i_s}$ satisfies $s \le {\hat k}$ is bounded by 
\eqref{2019-06-30+2}. On the other hand, the event $s>{\hat k}$  implies that at most $k-1$ elements of 
the list $L_w$  have marked at least ${\hat k}-k+2$ void sets before an irregular set was marked. The 
probability of such event is bounded by \eqref{2019-06-30+1}.

Finally, we observe that the events 
$\{|C_{w}| \ge k, w\in W_k\}$ and $\{|L_w|\ge k, w\in W_k\}$ are equal. 
Now \eqref{2019-07-24+1}, \eqref{2019-07-25} combined with \eqref{2019-07-09} imply
\[
 \pr\{|C_{w}|\ge k\}
 \weq \pr(|L_w|\ge k\} + o(1)
 \wle \rho_k(Z_n^+) + o(1).
\]
and hence the upper bound of \eqref{2019-07-26+1} follows.

\subsection{Old proof of the lower coupling bound in \eqref{2019-07-26+1}}
\rnote{Old stuff}

Denote $\cA_t = \{ \max_{(x,y) \in A_0} \abs{\cN^m_{xyt}} \le N\}$ and $\cB_t = \{\abs{\cQ_t} > 0\}$, and denote
$\cC_t = \cap_{s \le t} (\cA_s \cap \cB_s)$.

Because every explored layer overlaps the earlier marked layers at a single node, on the event that $v_t$ is the $t$-th explored node, $Z_t = \sum_{(x,y) \in A_0} Z_{xyt}$ where $Z_{xyt} = \sum_{k \in \cN^e_{xyt}} \deg_{v_t}(\bar G_k)$. Given $\cN^m_t$ and past history $\cF_{t-1}$, on the event $\cC_{t-1} \cap \cA_t$, the sizes $\abs{\cN^e_{xyt}}$, $(x,y) \in A_0$, are mutually independent and $\Bin( \abs{\cN^m_{xyt}}, (1-\delta)(1-\frac{t-1}{m}) )$-distributed. Furthermore, $\abs{\cN^m_{xyt}}$ are mutually independent and $\Bin( n^a_{xyt}, \frac{x}{m-(t-1)} )$-distributed given $\cF_{t-1}$. Hence the number of discovered nodes during step $t$ can be represented as
\[
 \law( Z_t \cond \cC_{t-1} \cap \cA_t )
 \weq \law(Z_t' \, \cond \cA_t' )
\]
where 
\[
 Z_t'
 \weq \sum_{(x,y) \in A_0} \sum_{k=1}^{N'_{xyt}} B_{xy}(k) T_{xy}(k),
 \qquad
 \cA_t' = \Big\{ \max_{(x,y)\in A_0} N'_{xyt} \le N \Big\}, 
\]
and where the random variables in the above sums are mutually independent, $\law(N'_{xyt}) = \Bin(n^a_{xyt}, \frac{x}{m-(t-1)})$, $\law(B_{xy}(k)) = \Ber( (1-\delta)(1-\frac{t-1}{m}))$, and $\law(T_{xy}(k)) = \Bin^+(x-1,y)$.

Note that $\abs{\cN^m_{xyt}} \eqst \Bin(n^a_{xyt}, \frac{x}{m-(t-1)})$, 
Assume that $t \le m/2$.
Because $n^{a}_{xyt} \le n$ and $\frac{x}{m-t+1} \le \frac{x}{m/2} \le 2 M m^{-1}$, we see that $\law(\abs{\cN^m_{xyt}}) \lest \Bin(n, 2 M m^{-1})$, and the moment generating function at one of the latter distribution is bounded by $e^{2M(e-1)m^{-1} n}$.
Therefore, Markov's inequality applied to $e^{\abs{\cN^m_{xyt}}}$ together with the union bound implies that $\pr( \cA_t^c ) \le m^{-2}$ when $N \ge \log(\abs{A_0} m^2) + 2M(e-1)m^{-1} n$ and $t \le m/2$.

Now
\[
 \pr( \cA_1, Q_1 > r)
 \weq \pr(\cA_1) \pr(Q_1 > r \cond \cA_1)
 \wge (1-m^{-2}) \pr(Q_1 > r \cond \cA_1)
\]
Now, using $\pr(A \cond B) \ge \pr(A,B) = \pr(A) - \pr(A, B^c) \ge \pr(A) - \pr(B^c)$ (or the total variation distance bound),
\begin{align*}
 \pr(Q_1 > r \cond \cA_1)
 \weq \pr(Z_1 > r \cond \cA_1)
 &\weq \pr(Z'_1 > r \, \cond \cA'_1 ) \\
 &\wge \pr(Z'_1 > r) - \pr( \cA'_1 ) \\
 &\weq \pr(Q'_1 > r) - \pr( \cA'_1 ) \\
\end{align*}

Next, by conditioning on $\cF_1$ which contains information about which layers were marked and explored in step 1, and the contents of the explored layers in step 1, we know $Q_1$,
\begin{align*}
 \pr( Q_2 > r, \cA_1, \cA_2)
 &\weq \pr( Q_1 + (Z_2 - 1) > r, Q_1 > 0, \cA_1, \cA_2)  \\
 &\weq \E \pr_{\cF_1}( Q_1 + (Z_2 - 1) > r, Q_1 > 0, \cA_1, \cA_2)  \\
 &\weq \E 1(\cA_1, Q_1>0) \pr_{\cF_1}( Q_1 + (Z_2 - 1) > r, \cA_2)  \\
 &\weq \E 1(\cA_1, Q_1>0) \pr_{\cF_1}( Q_1 + (Z_2 - 1) > r \cond \cA_2) \pr_{\cF_1}(\cA_2) \\
 &\weq \E 1(\cA_1, Q_1>0) \pr_{\cF_1}( Q_1 + (Z'_2 - 1) > r \cond \cA'_2) \pr_{\cF_1}(\cA_2) \\
 &\wge \E 1(\cA_1, Q_1>0) \Big( \pr_{\cF_1}( Q_1 + Z'_2 - 1 > r ) - \pr_{\cF_1}((\cA'_2)^c) \Big) \pr_{\cF_1}(\cA_2) \\
 &\weq \pr( Q_1 + Z'_2 - 1 > r, Q_1 > 0, \cA_1) - \pr( \cA_1, Q_1>0, (\cA'_2)^c ) \\
 &\wge \pr( Q_1 + Z'_2 - 1 > r, Q_1 > 0, \cA_1) - \pr( (\cA'_2)^c ) \\
 &\wge \pr( Q_1 + Z'_2 - 1 > r, Q_1 > 0 \cond \cA_1) \pr(\cA_1) - \pr( (\cA'_2)^c ) \\
 &\wge \pr( Q'_1 + Z'_2 - 1 > r, Q'_1 > 0 \cond \cA'_1) \pr(\cA_1) - \pr( (\cA'_2)^c ) \\
 &\wge \pr( Q'_2 > r \cond \cA'_1) \pr(\cA_1) - \pr( (\cA'_2)^c ) \\
 &\wge \Big( \pr( Q'_2 > r ) - \pr( (\cA'_1)^c ) \Big) \pr(\cA_1) - \pr( (\cA'_2)^c ) \\
 &\wge \pr( Q'_2 > r ) \pr(\cA_1) - \pr( (\cA'_1)^c ) - \pr( (\cA'_2)^c ) \\
\end{align*}

Note also that $\pr(\cA_1) = \pr(\cA'_1)$. Therefore,
\begin{align*}
 \pr( Q_2 > r, \cA_1, \cA_2)
 &\wge \pr( Q'_1 + Z'_2 - 1 > r, Q'_1 > 0 \cond \cA'_1) \pr(\cA'_1) - \pr( (\cA'_2)^c ) \\
 &\weq \pr( Q'_1 + Z'_2 - 1 > r, Q'_1 > 0, \cA'_1) - \pr( (\cA'_2)^c ) \\
 &\weq \pr( Q'_2 > r, \cA'_1) - \pr( (\cA'_2)^c ) \\
 &\wge \pr( Q'_2 > r ) - \pr( (\cA'_1)^c ) - \pr( (\cA'_2)^c ) \\
\end{align*}

We claim that for any $r \ge 0$ and any $0 \le t \le \omega$,
\[
 \pr( Q_t > r, \cA_1, \dots, \cA_t)
 \wge \pr( \hat Q_t > r ) - \sum_{s=1}^t \pr( \hat \cA_s^c ).
\]
Assume that the claim holds for up to $t-1$. Fix some sigma-algebra, for which $Q_{t-1}$ and the events $\cA_1,\dots, \cA_{t-1}$ are measurable. Then
\begin{align*}
 \pr( Q_t > r, \cA_1, \dots, \cA_t)
 &\weq \pr( Q_{t-1} + Z_t - 1 > r, Q_{t-1} > 0, \cA_1, \dots, \cA_t)  \\
 &\weq \E 1(\cA_1, \dots, \cA_{t-1}, Q_{t-1}>0) \pr_{\cF_{t-1}}( Q_{t-1} + Z_t - 1 > r, \cA_t),
\end{align*}
Now we let the hat variables be independent of everything else here,
\begin{align*}
 \pr_{\cF_{t-1}}( Q_{t-1} + Z_t - 1 > r, \cA_t)
 \weq \pr_{\cF_{t-1}}( Q_{t-1} + \hat Z_t - 1 > r, \hat \cA_t),
\end{align*}
so that
\begin{align*}
 \pr( Q_t > r, \cA_1, \dots, \cA_t)
 &\weq \pr( Q_{t-1} + Z_t - 1 > r, Q_{t-1} > 0, \cA_1, \dots, \cA_t)  \\
 &\weq \pr( Q_{t-1} + \hat Z_t - 1 > r, Q_{t-1} > 0, \cA_1, \dots, \cA_{t-1}, \hat\cA_t)  \\
\end{align*}

Fix $t \ge 1$. Fix some sets $\hat \cN_{s}^e \subset \hat \cN_{s}^d \subset \hat \cN_{s}^a \subset [n]$ and $\hat \cM_{s}^e \subset \hat \cM_{s}^d \subset [m]$, $1 \le s \le t-1$, a node $v \in [m]$ such that $v = \min \hat \cQ_{t-1}$ with $\hat \cQ_{t-1} = \hat \cM^d_{t-1} \setminus \hat \cM^e_{t-1}$, and some graphs $\hat G_k$, $k \in \cup_{s<t} \hat \cN_{s}^e$, such that $\abs{\hat \cN_{s}^d \cap \cN_{xy}} \le N$ for all $s < t$, and such that the event $\cE_{t-1}$ \dots has a nonzero probability. Also fix some sets $\hat \cN_{xyt}^a \subset [m] \setminus (\cup_{s<t}\hat \cN_{xys}^d)$ of sizes $n^{(a)}_{xyt}$, and let $\cF_{t1}$ be the event that $\cN^a_{xyt} = \hat \cN_{xyt}^a$ for all $(x,y) \in A_0$.

Given $\cE_{t-1} \cap \cF_{t1}$, the random graphs $G_k$, $k \in \cup_{xy} \hat\cN^a_{xyt}$, are independent, and distributed so that $V(G_k)$ is uniform among the $X_k$-sets of $[m] \setminus \hat \cM_{t-1}^e$, where $\hat \cM_{t-1}^e$ has size $t-1$. Hence, given $\cE_{t-1} \cap \cF_{t1}$, the random variables $\abs{\cN^d_{xyt}}$ are mutually independent and $\Bin(n^{(a)}_{xyt}, \frac{x}{m-(t-1)} )$-distributed.

Also fix some sets $\hat \cN_{xyt}^d \subset \hat \cN_{xyt}^a$, and let $\cF_{t2}$ be the event that $\cN^d_{xyt} = \hat \cN_{xyt}^d$ for all $(x,y) \in A_0$. Given $\cE_{t-1} \cap \cF_{t1} \cap \cF_{t2}$, the random graphs $G_k$, $k \in \cup_{xy} \hat\cN^d_{xyt}$, are mutually independent, and such that $V(G_k)$ is uniformly distributed among the $X_k$-sets of $[m] \setminus \hat \cM_{t-1}^e$ containing $v$.

conditional distribution of



\begin{rcomm}
The indicator variables $B^r_{t,k}$ guarantee that the node sets $V(G_k)\setminus \{v_t\}$ and hence also the neighborhoods $N_{v_t}(\bar G_k)$ for $k \in \cN^e_t$ are mutually disjoint. Hence the number of discovered nodes during step $t$ can be written as $Z_t = \sum_{(x,y) \in A_0} Z_{xyt}$ where
\[
 Z_{xyt}
 \weq \sum_{k \in \cN^e_{xyt}} T_{kt}
 \weq \sum_{k \in \cN^m_{xyt}} B^r_{kt} B^e_{kt} T_{kt}
 \weq \sum_{k \in \cN^a_{xyt}} B_{kt} T_{kt}
\]
where $T_{kt} = \abs{N_{v_t}(\bar G_k)}$ and $B_{kt} = B^m_{kt} B^r_{kt} B^e_{kt}$ with $B^m_{kt} = 1(V(G_k) \ni v_t)$. 

Let us fix $t \ge 1$, and consider the event that $\cQ_{t-1} = \hat \cQ_{t-1}$ for some set with $v_t = \min\hat \cQ_{t-1}$, and

Let us define a compound binomial distribution $\CBin(n,p,f)$ as the law of 
$\sum_{k=1}^N X_k$, where $\law(N) = \Bin(n,p)$ and $\law(X_k) = f$, and the random variables in the sum are mutually independent. Then
\[
 \law( Z_{xyt} \cond N^m_{xyt}, \dots )
 \weq \CBin\Big( N^m_{xyt}, (1-\delta)(1-\tfrac{t-1}{m}), \Bin^+(x-1,y)\Big)
\]
and
\[
 \law( Z_{xyt} \cond \dots )
 \weq \CBin\Big( n^a_{xyt}, (1-\delta)\tfrac{x}{m}, \Bin^+(x-1,y)\Big)
\]

where $B_{t,k} = B^r_{t,k} B^e_{t,k}$. The indicator variables $B^e_{t,k}$ are designed so that the random variables $B_{t,k}$, $k \in \cN^{m}_t$, are (conditionally) mutually independent with mean $(1-\delta)(1-\frac{t-1}{m})$, given suitable information. Hence the conditional law of $Z_t$ given suitable information involving the event that $\abs{\cN^m_{xyt}} = n^m_{xyt}$, is the law of
\[
 \sum_{(x,y) \in A_0} \sum_{k=1}^{n^{m}_{xyt}} B_{xyt}(k) T_{xyt}(k)
\]
where all random variables on the right are independent, $B_{xyt}(k)$ is $\Ber((1-\delta)(1-\frac{t-1}{m}))$, and $T_{xyt}(k)$ is $\Bin^+(x-1,y)$. Furthermore, observe that the random variables $\abs{\cN^m_{xyt}}$ are independent and $\Bin(n^a_{xyt}, \frac{x}{m-(t-1)})$-distributed given \dots Hence
\[
 \sum_{(x,y) \in A_0} \sum_{k=1}^{N^{e}_{xyt}} T_{xyt}(k)
\]
where $\law(N^{e}_{xyt}) = \Bin(n_{xy} - N(t-1), (1-\delta)\frac{x}{m})$.

\end{rcomm}

\begin{rcomm}
Denote $\cA^m_t = \{ \norm{\cN^m_t} \le N\}$ and $\cA_t = \cA^m_t \cap \{Q_t > 0\}$ be the event that exploration step $t$ is successful. Denote $\cA_{\le t} = \cap_{s \le t} \cA_t$. Given the event $\cA_{\le t-1}$, the conditional probability of $\cA^m_t$ is at least $1-m^{-2}$. We will show that
\[
 \law(Q_t \cond \cA_{\le t-1})
\]
Step 1 is successful with probability $\pr( \cA^m_1 ) \pr( Q_1 > 0 \cond \cA^m_1)$. Here $Q_1 = Q_0-1+Z_1$, with
\[
 Z_1
 \weq \sum_{(x,y) \in A_0} \sum_{k \in \cN^m_{xy1}} B_{1k} T_{xy1}(k),
\]
where $B_{1,k} = B^r_{1k} B^e_{1k}$. Observe next that
\[
 \dtv( \law(Z_1), \law(Z_1 \cond \norm{\cN^m_1} \le N))
 \wle \pr( \norm{\cN^m_1} > N )
 \le m^{-2}.
\]
Now the unconditional law of $Z_1$ is the same as the law of
\[
 \weq \sum_{(x,y) \in A_0} \sum_{k =1}^{\bar N_{xy1}} T_{xy1}(k),
\]
where $\bar N_{xy1} = \Bin( n^a_{xy1}, (1-\delta) \frac{x}{m}) = \Bin( n_{xy}, (1-\delta) \frac{x}{m})$.

The layer selection is constructed so that the node sets $V(G_k) \setminus \{v_t\}$, $k \in \cN^e_t$, are disjoint, and given suitable information including $\cN^m_t$, the sizes $\abs{\cN^e_{xyt}}$ are mutually independent and $\Bin(n^m_{xyt}, (1-\delta) \frac{x}{m-(t-1)})$-distributed.
\end{rcomm}

\begin{rcomm}
Denote by $\cF^-_{t,k}$ the sigma-algebra describing the information gained from the layers explored in steps up to $t-1$, the information about $\cN^m_t$, and the information about $V(G_j)$ for $j <k$ such that $j \in \cN^m_t$.

given by the information available \dots Then
$
 \pr( B^r_{t,k} = 1 \cond \cF^-_{t,k} )
 \weq p_r(\abs{\cH_{t,k}}, X_k, t)
$
and
$
 \pr( B_{t,k} = 1 \cond \cF^-_{t,k} )
 \weq (1-\delta)p_m(X_k,t)
$
As a consequence, the random variables $\{B_{t,k}: k \in \cN^m_t\}$ are conditionally independent given $\cN^m_t$, and $\Ber( (1-\delta)\frac{X_k}{m-(t-1)} )$-distributed, given $\cF^-_t$ and $\cN^m_t$.

.  Hence, the collections $\{B_{t,k}: k \in \cN^m_{xyt}\}$ are mutually independent, and consists of mutually independent random variables.
\end{rcomm}


\begin{rcomm} \scriptsize
The conditional distribution of $\abs{\cN^e_{xyt}}$ given $\norm{\cN^m_t} \le N$, the previous $t-1$ exploration step are successful, and  

The conditional distribution of $\abs{\cN^e_{xyt}}$ given suitable information is $\Bin( n^{a}_{xyt}, (1-\delta) n P(x,y) )$.

The probability is good when $\cN^m_{xy,s} \le N$ for all $s \le t$ and $(x,y) \in A_0$:

A simple computation shows that $p_1(h,x,t) \ge p_1(h,x,1) \ge (1-2\frac{h}{m})^M \frac{x}{m}$ for all $1 \le t \le m$ and all $x \le M$, when $m \ge 2M$. Hence $p_1(h,x,t) \ge (1-\delta) \frac{x}{m}$ for all $1 \le t \le m$ and all $h \le \delta_1 m$ with $\delta_1 = \frac12 (1-(1-\delta))^{1/M}$.

$\abs{\cH_{t,k}} \le 1 + M \sum_{s \le t} \abs{\cN^m_{s}} \le c_1 N t$ with $c_1 = (M+1) \abs{A_0}$ on the event that $\cN^m_{xy,s} \le N$ for all $s \le t$ and $(x,y) \in A_0$.

Hence $p_1(h,x,t) \ge (1-\delta) \frac{x}{m}$ for $t \le c_1^{-1} \delta_1 N^{-1} m$

\bigskip

Define $p_0(h,x,t) = \binom{m-h}{x-1} \binom{m-(t-1)}{x-1}^{-1}$. This is the probability that a random $x$-set, given that it intersects $\{v_1,\dots, v_t\}$ precisely at $v_t$ (the layer was marked during step $t$), does not overlap with a given $h$-set containing $\{v_1,\dots, v_t\}$. A simple computation shows that $p_0(h,x,t) \ge p_0(h,x,1) \ge (1-2h/m)^M$ for $m \ge 2M$. Hence $p_0(h,x,t) \ge (1-\delta)\frac{x}{m-(t-1)}$ for $h \le c_0 m$ with $c_0 = $.
\end{rcomm}

\paragraph{Some remarks}

Given that a layer $k$ of type $(x,y)$ is admissible during step $t$, the conditional distribution of $V(G_k)$ is uniform among the $x$-sets of $\{v_1,\dots, v_{t-1}\}^c$. Hence any layer $k \in \cN^a_{xyt}$ is marked with conditional probability $\frac{x}{m-(t-1)}$. The layers are conditionally independent. Hence the conditional distribution of $\abs{\cN^m_{xyt}}$ is $\Bin(n^{a}_{xyt}, \frac{x}{m-(t-1)})$.

Given that $k \in \cN^a_{xyt}$ and $\abs{\cH_{t,k}} = h$, the event $V(G_k) \cap \cH_{t,k} = \{v_t\}$ occurs with probability
$
 p_1(h,x,t)
 = \binom{m-h}{x-1} \binom{m-(t-1)}{x}^{-1}.
$
Hence, given that $k \in \cN^a_{xyt}$ and $\abs{\cH_{t,k}} = h$, the event $k \in \cN^e_{xyt}$ occurs with probability $p_1(h,x,t) p(h,x,t) = (1-\delta) n P(x,y)$. Hence any admissible layer of type $(x,y)$ becomes marked and selected for exploration with probability $(1-\delta) n P(x,y)$. Are these events conditionally independent with respect to suitable background information? We would like to conclude that the conditional distribution of $\abs{\cN^e_{xyt}}$ given suitable background information is $\Bin( n^{a}_{xyt}, (1-\delta) n P(x,y) )$.

The selection procedure guarantees that $\cZ_t = \cup_{k \in \cN^{e}_t} N_{v_t}(\bar G_k)$ is a disjoint union of (conditionally) mutually independent random sets with sizes $\BIn^+(X_k-1, Y_k)$-distributed. Hence the conditional distribution $\abs{\cZ_t}$ can be represented as the law of 
\[
 Z_t
 \weq \sum_{(x,y) \in A_0} \sum_{k = 1}^{\bar N_{xyt}} T_{xyt}(k)
\]
where $\law(\bar N_{xyt}) = \Bin( n^a_{xyt}, (1-\delta) n P(x,y) ) \gest \Bin( (1-2\delta)n_{xy}, (1-\delta) n P(x,y) )$ \rnote{requires integers} when $t \le c N^{-1} n$ with $c = \delta \min_{(x,y) \in A_0} P(x,y)$.

Fix a small $\delta \in (0,1)$. Fix $n_0$ large enough so that $\abs{\frac{P_n(x,y)}{P(x,y)}-1} \le \delta$ for all $n \ge n_0$. We denote $N = \floor{3 \log m}$. We study a modified exploration defined as follows. The exploration is controlled so that the number of discovered nodes is required to be bounded by $h \le \delta_1 m$.

\mnote{only essential layers are used}\\

\subsection{MB proof of lower bound}

\subsubsection{MB lower bound for $X_k=M$ and $Y_k=1$}

\begin{bcomm}
Denote by $\cN^a_s \supset \cN^m_s \supset \cN^e_s$ the layers available, marked, and explored (regular and admitted) at step $s$.

Fix a small $\delta \in (0,1)$, $t \lesim m \log^{-2} m$ and $\omega_1 \sim 3 \log m$, such that $\omega_1 t \le \delta n$.
Define $n^a_s = (1-\delta)n - (s-1) \omega_1$ (suitably rounded) for all $s=1,\dots, t$. Initialize: Set $\cQ \leftarrow \{i\}$, $\cN^a_0 \leftarrow [n]$. Step $s \ge 1$. Check that $\cQ_s \ne \emptyset$ that $\max_{1 \le r \le s-1}\abs{\cN^m_r} \le \omega_1$. If yes, proceed. Let $v_s$ be the node in the queue with the smallest index. Let $\cN^a_s$ = uniformly random $n^a_s$-set of $\cN^a_{s-1}$ (uniformly random conditionally on the previous events). Let $\cN^m_s = \{k \in \cN^a_s: V(G_k) \ni v_s\}$ .
Let $\cN^e_s$ be those marked layers $k \in \cN^m_s$ which are regular and admitted for exploration. Then remove $v_s$ and add the nodes in $N_{G_k}(v_s)$, $k \in \cN^e_s$, to the exploration queue.

Because all explored layers are disjoint from previously explored layers, on the event that the exploration proceeds to step $s$, the number of discovered nodes during step $s$ equals
\begin{align*}
 Z_s
 \weq (M-1) \abs{\cN^e_s}
 &\weq (M-1) \sum_{k \in \cN^m_s} 1(k \in \cN^e_s) \\
 &\weq (M-1) \sum_{k \in \cN^a_s} 1(k \in \cN^m_s) 1(k \in \cN^e_s).
\end{align*}
Given an event $\cE_{s-1}^+$ that $Q_{s-1} > 0$ and $\max_{r \le s-1} \abs{\cN^m_r} \le \omega_1$ and $\cN^a_{s-1} = A$ for some $A$ of size $n^a_{s}$, 
the number of marked layers $\abs{\cN^m_s}$ in step $s$ is distributed according to $\Bin(n^a_s, \frac{M}{m-(s-1)})$. Given $\cE_{s-1}^+$ and $\abs{\cN^m_s} = n_1$ for some $n_1 \le \omega_1$, the number of explored layers $\abs{\cN^e_s}$ is distributed according to $\Bin\left(n_1, (1-\delta) (1 - \frac{s-1}{m}) \right)$. \rnote{Problem for nonconstant layer types: Given earlier events, the law of $\cN^a_s$ is not uniform because conditionally on $Q_{s-1}>0$, larger and stronger layers are more likely to have been used in the previous steps.}
%
Define a mixed binomial distribution
\[
 f_s(r)
 \weq \sum_{n_1=0}^{n^a_s} \Bin\Big(n^a_s, \tfrac{M}{m-(s-1)}\Big)(n_1) \, \Bin\Big(n_1, (1-\delta) (1 - \tfrac{s-1}{m})\Big)(r).
\]
and let
\[
 \tilde f_s(r)
 \weq \frac{1(r \le \omega_1) f_s(r)}{\sum_{r' \le \omega_1} f_s(r')}.
\]
Define a modified queue length process by $Q'_s = 1$ and $Q'_s = 1(Q'_{s-1} > 0) (Q'_{s-1} - 1 + Z'_s)$ for $1 \le s \le t$ where $Z'_1,\dots,Z'_t$ are mutually independent and such that $\law(Z'_s) = \tilde f_s$ \rnote{times $M-1$}.
Then (for $M=2$)
\[
 \law( Z_s \cond \cE_{s-1}^+, \abs{\cN^m_s} \le \omega_1 )
 \weq \law(Z'_s). 
\]
Now
\begin{align*}
 \pr( Q_s > 0 )
 &\wge \pr( Q_s > 0, \cE_{s-1}^+, \abs{\cN^m_s} \le \omega_1 ) \\
 &\wge \pr( Q_s > 0 \cond \cE_{s-1}^+, \abs{\cN^m_s} \le \omega_1 )
 - \pr( \max_{r \le s} \abs{\cN^m_r} > \omega_1 ).
\end{align*}

Denote $\cA_s = \{ \abs{\cN^m_s} \le \omega_1 \}$ and $\cA_{\le s} = \cA_1 \cap \cdots \cap \cA_s$. Then for any $a>0$ and $b \ge 0$,
\begin{align*}
 \pr(Q_s=b \cond Q_{s-1}=a, \cA_{\le s})
 &\weq \pr( a+ Z_s = b \cond Q_{s-1}=a, \cA_{\le s}) \\
 &\weq \pr(a + Z'_s = b) \\
 &\weq \pr(Q'_s = b \cond Q'_{s-1}=a, \dots, Q'_0=1).
\end{align*}
Then
\begin{align*}
 \pr( Q_0=1, \dots, Q_s=b \cond \cA_{\le s} )
 &\weq \pr( Q_0=1, \dots, Q_{s-1}=a \cond \cA_{\le s} ) \pr(Q_s=b \cond Q_{s-1}=a, \cA_{\le s}) \\
 &\weq \pr( Q_0=1, \dots, Q_{s-1}=a \cond \cA_{\le s} ) \pr(Q'_s = b \cond Q'_{s-1}=a).
\end{align*}
Now by writing $\pr(B,  A_1,A_2) = \pr( B \cond A_1) - \pr(B \cond A_1) \pr(A_1^c) - \pr(B_1, A_1, A_2^c)$, it follows that $\pr( B \cond A_1,A_2) \ge \pr(B,  A_1,A_2) \ge \pr( B \cond A_1) - \pr(A_1^c) - \pr(A_2^c)$. By applying this generic inequality, we find that
\begin{align*}
 \pr( Q_0=1, \dots, Q_s=b \cond \cA_{\le s} )
 &\weq \pr( Q_0=1, \dots, Q_{s-1}=a \cond \cA_{\le s} ) \pr(Q'_s = b \cond Q'_{s-1}=a) \\
 &\wge \pr( Q_0=1, \dots, Q_{s-1}=a \cond \cA_{\le s-1} ) \pr(Q'_s = b \cond Q'_{s-1}=a) \\
 &\qquad - \pr(\cA_{\le s-1}^c) - \pr(\cA_{s}^c) \\
 &\wge \pr( Q_0=1, \dots, Q_{s-1}=a \cond \cA_{\le s-1} ) \pr(Q'_s = b \cond Q'_{s-1}=a) \\
 &\qquad - \sum_{r = 1}^s \pr(\cA_r^c).
\end{align*}
By applying induction and noting that $Q'_s$ is Markov, we find that
\begin{align*}
 \pr( Q_0=1, \dots, Q_s=b \cond \cA_{\le s} )
 \wge \pr( Q'_0=1, \dots, Q'_s=b ) - s \sum_{r = 1}^s \pr(\cA_r^c) 
\end{align*}
Especially,
\begin{align*}
 \pr( Q_s > 0 \cond \cA_{\le s} )
 \wge \pr( Q'_s > 0 ) - s \sum_{r = 1}^s \pr(\cA_r^c) 
\end{align*}
so that
\begin{align*}
 \pr( Q_s > 0 )
 \wge \pr( Q_s > 0, \cA_{\le s} )
 &\weq \pr( Q_s > 0 \cond \cA_{\le s} ) - \pr( Q_s > 0 \cond \cA_{\le s} ) \pr(\cA_{\le s}^c) \\
 &\wge \pr( Q'_s > 0 ) - s \sum_{r = 1}^s \pr(\cA_r^c) - \pr(\cA_{\le s}^c) \\
 &\wge \pr( Q'_s > 0 ) - (s+1) \sum_{r = 1}^s \pr(\cA_r^c).
\end{align*}
By choosing $\omega_1$ well, we can show that $\pr(\cA_s^c) \lesim m^{-3}$ for all $s \le t$, or $m^{-4}$ if needed. Then
\begin{align*}
 \pr( Q_t > 0 )
 &\wge \pr( Q'_t > 0 ) - O(t^2 m^{-3}).
\end{align*}

But now we see that $f_s$ is actually a binomial distribution,
\[
 f_s
 \weq \Bin\Big( n^a_s, \tfrac{M}{m-(s-1)} (1-\delta) (1 - \tfrac{s-1}{m}) \Big)
 \weq \Bin\Big( n^a_s, (1-\delta) \tfrac{M}{m} \Big),
\]

Then $f_s \gest f_t \gest \Bin\left( (1-2\delta), (1-\delta) (1 - \frac{M}{m}) \right)$ for $\omega_1 t \le \delta n$.

Note that $\dtv(f_s, \tilde f_s) = 2 \sum_{r > \omega_1}f_s(r)$. Note also that $\E e^N = (1 + (e-1)p)^n \le (1+2p)^n \le e^{2np}$ implies $\pr( N > a ) \le e^{-a} e^{2np}$ for $N$ being $\Bin(n,p)$-distributed. Hence $\dtv(f_s, \tilde f_s) \le 2 \pr( N > \omega_1 ) \le 2 e^{-\omega_1} e^{2n^a_s (1-\delta) M/m} \le  e^{-\omega_1} 2 e^{2M n/m} \le m^{-3} 2 e^{2M n/m}$ for $N$ being $f_s$-distributed and for $\omega_1 \ge 3 \log m$.

Denote by $\cE_t$ the event that $N^m_s \le \omega_1$ for all $1 \le s \le t \wedge T_i$. Then we can show that $\pr( \cE_t^c ) \lesim m^{-1}$ for $t \le m/2$.

\end{bcomm}

Given a root node, the list $L^* = \{u_1,u_2,\dots\}$ is constructed similarly as $L$ but now each explored node only accepts children brought by \emph{regular} layers. Moreover, not every regular layer is allowed to contribute to the list $L^*$. Permission to contribute is granted at random as follows. Let $G^*_{1}, G^*_{2},\dots$ denote the regular marked layers that were allowed to contribute to the list one after another during the exploration. Denote $H^*_{s} = \cup_{r=0}^s H_r$ where $H_0 = \{u_1\}$ and $H_r = V(G^*_r)$. Then a regular layer $G^*_{s+1}$ of size $x \ge 2$ marked by $u_t$ is allowed to contribute to the list $L^*$ with probability\mnote{$\le 1$ under restrictions} 
$p^*(|H^*_s|, x, t)$, where\mnote{$p^-_{x,\delta} = (1-\delta)\frac{x}{m}$}
\begin{equation}
 \label{2019-07-10}
 p^*(h,x,t) = \frac{(1-\delta)\frac{x}{m}}{p_1^*(h,x,t)},
 \qquad
 p_1^*(h,x,t) = \binom{m-h}{x-1}{\binom{m-t+1}{x}}^{-1}.
\end{equation}
Note that $p_1^*(h,x,t)$ is the probability that a uniformly random $x$-set $D$ in $[m] \setminus \{u_1,\dots, u_{t-1}\}$ contains node $u_t$ but no other node of a $h$-set $H$ of nodes (covered by previously marked layers). This is the probability that an $xy$-layer is discovered (contains $u_t$) and regular (no multi-overlap with the $h$ previously discovered nodes.) Then $p^*(h,x,t)$ equals the probability that an $xy$-layer is discovered, regular, and admitted.


During the exploration we make sure that $\abs{H^*_s} \le \delta m$ which implies $p^*_1( \abs{H^*_s}, x, t) \ge (1-\delta)\frac{x}{m}$.
%
To this aim we control the growth of the number of marked layers. Let $N^{(t)}_{xy}$ denote the number of $xy$-layers marked by $u_t$, and let $N^{(t)} = \sum_{xy} N^{(t)}_{xy}$.
We introduce an observer who monitors these numbers
and alerts at the first instance when $N^{(t)}_{xy} > \omega_1$ occurs.
Furthermore, each $u_t$ is only allowed to mark layers from certain collections $\DD^{(t)}_{x,y}\subset \DD_{x,y}$ defined as follows. For each $(x,y)\in A_0$  we select (at random) a collection $\DD^{(1)}_{x,y}$ of $(x,y)$-layers of size $n^{(1)}_{x,y} = (1-\delta) nP(x,y)$. \rnote{This is possible due to $\abs{\frac{P_n(x,y)}{P(x,y)}-1} \le \delta$.}
The root node $u_1$ is only allowed to mark layers from the collections $\DD^{(1)}_{x,y}$, $(x,y)\in A_0$. After the first exploration step (after $u_1$ has collected its children) we check whether $N^{(1)}_{x,y} \le \omega_1$ for each $(x,y)\in A_0$. If no alert was declared, we proceed to the next step of exploration. During exploration step $t \ge 2$ (if no alert was declared so far), node $u_t$ is only allowed to mark layers from $\cup_{xy} \DD^{(t)}_{xy}$, where each $\DD^{(t)}_{xy}$ is a (random) collection of unmarked layers from $\DD^{(t-1)}_{xy}$ of size $n^{(t)}_{xy} =  n^{(t-1)}_{xy} - \omega_1$. Note that $n^{(t)}_{xy} = (1-\delta) nP(x,y) - (t-1) \omega_1 \ge (1-2\delta) nP(x,y)$, \rnote{provided that $t \le \omega_1^{-1} n \delta P(x,y) \asymp n \log^{-1} n$}. The number $N^{(t)}_{xy}$ has binomial distribution $\Bin(n^{(t)}_{xy}, \frac{x}{m-t+1})$.  Let $\eta^{(t)}_{xy}$ denote the random variable $N^{(t)}_{xy}$ conditioned on the event $N^{(t)}_{xy} \le \omega_1$.
 
Denote by $\cE_{k}$ the event that $N^{(t)}_{xy} \le \omega_1$ for all $1 \le t < k$ and $xy$.
\rnote{$\cE_{i*(t)}$ is the event that all layer mark counts were $\le \omega_1$ until the layer mark count leading to the discovery of node $u_t$ }
Let us verify that
\begin{equation}
 \label{2019-07-06}
 \pr(N^{(t)}_{x,y} > \omega_1)
 \le cm^{-2}
 \quad \text{and} \quad
 \pr(\cE_{\floor{m/2}}^c)
 \le c |A_0| m^{-1}.
\end{equation}
for all $t \le m/2$ and for all $(x,y) \in A_0$.  Because the number of unmarked layers is at most $n$, and an unmarked layer of size $x$ is marked by $u_t$ with probability $\frac{x}{m-t+1} \le \frac{x}{m/2} \le 2 M m^{-1}$ (for $|L^*_w| < t$ we have $N^{(t)}_{x,y} = 0$), we see that $N^{(t)}_{x,y} \lest N^*$ for a generic $\Bin(n, 2 M m^{-1})$-distributed random integer $N^*$. Then $\pr(N^* > \omega_1) \le e^{-\omega_1} \E e^{N^*} \sim e^{-\omega_1} \E e^{\tilde N^*}$ where $\law(\tilde N^*) = \Poi(2 M \mu)$. Hence $\pr(N^{(t)}_{x,y} > \omega_1) \le e^{4M\mu} e^{-\omega_1}$ for all large enough values of $n$ and all $t \le m/2$. The union bound implies the second inequality. Now we ask the observer to stop the exploration $L^*_w$ at the first instance where $N^{(t)}_{x,y} > \omega_1$. It follows from \eqref{2019-07-06} that with probability $1 - O(m^{-1})$, the exploration will not be stopped by observer within the first $t = 1,2,\dots, \floor{m/2}$ steps (it may still terminate  for other reasons). 


Let $t \le m\log^{-2}m$. Recall that $i^*(t)$ indicates the exploration step during which $u_t$ was discovered, 
\rnote{$i^*(t)$ = exploration step when $u_t$ was discovered}
and $r(t)$ indicates the running label of the layer that was explored while $u_t$ was discovered. \rnote{$r(t)$ = number of layers explored before the one related to $u_t$} Hence $u_t$ was brought to the list $L^*_w$ while exploring node $u_{i^*(t)}$ and layer $G^*_{r(t)}$.
Observe that $r(t) \le N_1 + \dots + N_{i^*(t)}$ \rnote{$ \le \abs{A_0} \omega_1 \le \frac12 m$?}. From the second inequality of \eqref{2019-07-06} we obtain that
\begin{equation}
 \label{2019-07-09+3}
 \pr( \cE_{i^*(t)} )
 = 1 - O(m^{-1}),
 \quad
 \pr\big( |L^*_w|\ge t \, \bigr|\, \cE_{i^*(t)} \big)
 = \pr(|L^*_w| \ge t) + O(m^{-1}).
\end{equation}
Here the first inequality implies the second. The conditioning on $\cE_{i^*(t)}$ means that the observer has not stopped the exploration until $u_t$ was discovered and added to the list. On the event $\{|L^*_w| \ge t\} \cap \cE_{i^*(t)}$,
each of the sets $H^*_1 \subset H^*_2 \subset \cdots H^*_{r(t)}$ contains at most $M r(t)$ nodes, and \rnote{for $\omega_1 \sim 3 \log m$}
\[
 r(t)M
 \wle (N_1+\cdots +N_{i^*(t)})M
 \wle i^*(t) |A_0| \omega_1 M 
 \wle 3 M |A_0| m \log^{-1} m 
 =: h^*. 
\]
Note that $h^* = o(m)$. Hence for any $1\le s \le i^*(t)$ and $s-1\le h \le h^*$ we have $p^*_1(h,x,s) > (1-\delta)\frac{x}{m}$. Therefore, for each $1 \le s \le i^*(t)$, the probability that an $xy$-layer $G_k$ is marked by $u_s$, is regular, and allowed to contribute to $L^*_w$ is $(1-\delta)\frac{x}{m}$. The total number of children of $u_s$ is then
\begin{equation}
 \label{2019-07-10+1}
 \sum_{xy} \sum_{k=1}^{\eta^{(s)}_{xy}} \II^{(s)}_k(x,y) T^{(s)}_{k}(x,y)
\end{equation}
where $\II^{(s)}_k(x,y)$ is a Bernoulli random variable (independent of all the other random variables) with success probability 
\[
 p'_{s,\delta}
 := \frac{(1-\delta) \frac{x}{m}}{\pr(G_k \ \text{is marked by $u_s$} )}
 = \frac{(1-\delta) \frac{x}{m}}{x/(m-s+1)}
 \weq (1-\delta) \left(1 - \frac{s-1}{m} \right).
\]
Here $p'_{s,\delta}$ is the conditional probability that $G_k \in \DD^{(s)}_{x,y}$ is allowed to contribute to $L^*_w$ given that $G$ is marked by $u_s$. \rnote{Does not depend on layer type}

Let us compare the exploration process $L^*_w$ with the branching process $\cL$ which produces an ordered list of nodes $\{u_1,u_2,\dots\}$ and where the offspring number of $u_s$ is defined by \eqref{2019-07-10+1}, but with $\eta^{(s)}_{xy}$ replaced by $N^{(s)}_{xy}$. Because
\[
 \dtv\bigl( \eta^{(s)}_{xy}, N^{(s)}_{xy} \bigr)
 \wle \pr (N^{(s)}_{xy} > \omega_1)
 \wle c m^{-2},
\]
we find that
\begin{equation}
 \label{2019-07-25+2}
 \pr( |L^*_w| \ge t \ | \ \cE_{i^*(t)} )
 \weq \pr( |\cL| \ge t) + O(t/m^2).
\end{equation}
Furthermore, the total progenies of the branching processes are ordered by $\pr(|\cL|\ge k) \ge \pr(|Z^-|\ge k\}$. Indeed, we can represent the offspring  number of $\cL$ as
\[
 \sum_{xy} \sum_{s=1}^{N^{(j)}_{xy}} \II^{(j)}_s(x,y)T^{(j)}_{s}(x,y)
 = \sum_{xy} \sum_{s=1}^{{\bar N}^{(j)}_{xy}} T^{(j)}_{s}(x,y),
\]
where ${\bar N}^{(j)}_{xy}\eqst \Bin(n^{(j)}_{xy}, (1-\delta)\frac{x}{m})$. Observe that $j-1 \le \frac{\delta n P(x,y)}{\omega_1}$ implies $n^{(j)}_{x,y} \ge (1-2\delta)nP(x,y)$, and hence ${\bar N}^{(j)}_{x,y} \gest N^{-}_{x,y}$. Now, if we assume that
\[
 j-1 \wle \frac{\delta n \min_{(x,y) \in A_0} P(x,y)}{\omega_1},
\]
it follows that the $j$-dependent offspring numbers are bounded from below by $Z^-_{n,\delta}$ defined in \eqref{eq:CompoundBinomial}. Hence \rnote{up to $j$ at least} we find that
\[
 \pr( \abs{\cL} \ge k )
 \wge \rho_k( Z^-_{n,\delta} ).
\]

Now \eqref{2019-07-09+3}, \eqref{2019-07-25+2} imply $\pr( \abs{L^*} \ge k) \ge \pr( |Z^-| \ge k) + o(1)$. Finally, the simple inequality $\pr( C_i \ge k) \ge \pr( \abs{L^*} \ge k)$ shows that
\[
 \pr( C_i \ge k )
 \wge \rho_k( Z_{n,\delta}^- ) + o(1).
\]
This gives the lower bound of \eqref{2019-07-26+1}.
\qed

\subsection{Number of nodes in large components}

\bnote{The number of nodes contained in large components is approximately $\rho(\cY)$. Requires the variance analysis of component sizes, done using double branching processes.}
\begin{lemma}
\label{the:MBGiant2}
Under the same assumptions\mnote{Assume finite support $A$, but no need to assume $P(A_0)>0$?} as in Lemma~\ref{the:MBGiant1}, for any $1 \ll \omega \le n \log^{-2} n$,
\begin{eqnarray}
 \label{2019-08-05+1}
 &&
 m^{-1} |B^{\omega}(G^{(n)})| \prto \rho(\cY)
\end{eqnarray}
\end{lemma}
\begin{proof}
Using the shorthand notation $\II_w = 1_{\{|C_{w}|\ge \omega\}}$ we write
\begin{equation}
\label{2019-08-06+1}
 |B^{\omega}|
 =
 \sum_w {\mathbb I}_w,
 \qquad
 \binom{|B^{\omega}|}{2}=\sum_{\{u,w\}\subset W}
 \II_u \II_w.
\end{equation}
The first identity combined with (\ref{eq:MBGiant1}) yields
\begin{equation}
 \label{2019-08-10+4}
 \E|B^{\omega}|
 = m\rho(\cY)+o(m).
\end{equation} 
For $\rho(\cY)=0$ this implies \eqref{2019-08-05+1}. For $\rho(\cY)>0$ we establish (\ref{2019-08-05+1}) by showing that
$|B^{\omega}|$ concentrates around its mean.

(i) We will first prove the claim in the case where $\omega = \log m$. Let $\{x,y\}\subset W$ denote a pair of nodes selected uniformly at random. We show below that, uniformly with respect to $x,y$,
\begin{equation}
 \label{2019-08-06+2}
 \E (\II_x \II_y )
 \wle \rho(\cY) \times \rho(\cY) + o(1).
\end{equation}
Then \eqref{2019-08-06+2} combined with \eqref{2019-08-06+1}, \eqref{2019-08-10+4} implies $\E|B^{\omega}|^2 \le (\E |B^{\omega}|)^2+o(m^2)$. From the latter inequality we conclude that $\var |B^{\omega}|=o(m^2)$. Now
Chebyshev's inequality implies that for all $\gamma > 0$,
\begin{equation}
 \label{2019-08-17+3}
 \pr \Bigl\{ \bigl| |B^{\omega}|-\E |B^{\omega}|\bigr| > \gamma m \Bigr\}
 \wle (\gamma m)^{-2}\var(|B^{\omega}|)
 \weq o(1).
\end{equation}
Letting $\gamma \downarrow 0$ we obtain \eqref{2019-08-05+1}.

\subsection{MB proof of double upper bound}
\label{sec:MBDoubleUpperBound}
See Section~\ref{sec:DoubleUpperBound} and Algorithm~\ref{algo:UpperExploration}.

Proof of \eqref{2019-08-06+2}, that is,
\[
 \pr( C_i \ge \omega, C_j \ge \omega)
 \wle \rho(Y)^2 + o(1)
 \qquad \text{for $\omega \asymp \log m$}.
\] 
Fix $t \asymp \log m$.  Denote $\cL_{it} = \{ \abs{L_i} \ge t\}$ and $\cL_{it}^+ = \cL_{it} \cap \cR_{it}$ where $\cR_{it}$ is the event that the exploration from $i$ contains no multi-overlaps before discovering $t$ nodes ($i$ is $t$-regular).  Approximation \eqref{2019-07-25} and the fact that $\{C_i \ge t\} \cap \cR_{it} = \{\abs{L_i} \ge t\} \cap \cR_{it}$ imply
\begin{equation}
\label{2019-08-10}
 \begin{aligned}
 \pr( C_i \ge t, C_j \ge t)
 &\weq \pr( C_i \ge t, C_j \ge t, \cR_{it}, \cR_{it}) + o(1) \\
 &\weq \pr( \cL^+_{it}, \cL^+_{jt}) + o(1).
 \end{aligned}
\end{equation}

(i) We show that a moderate number $s = 2 \frac{t}{\hat q}$ of layer explorations suffices to discover $t$ nodes.  On the event $\abs{L_i} \ge t$, we denote by $\DD_{it}$ the set of layers explored in the $i$-exploration until $t$ nodes are discovered (we assume that initially node $i$ is discovered). Let $H_{it} = \cup_{k \in \DD_{it}} V(G_k)$. The event $\cL_{it}^+ \cap \{ \abs{\DD_{it}} > s\}$ implies that more than $s$ layers need to be explored to discover $t$ nodes in the $i$-exploration, and that no multi-overlaps occur during the first $s$ layer explorations. Hence less than $t-1$ of the first $s$ layer explorations are successful. The probability of such event is $o(1)$, see \eqref{2019-06-30+1}.  Now we have $\pr( \cL_{it}^+, \abs{\DD_{it}} > s)  = o(1)$. Because $\abs{\DD_{it}} \le s$ implies $\abs{H_{it}} \le Ms $, we now see that $\pr( j \in H_{it}  \cond \abs{\DD_{it}} \le s) \le M s/m \ll 1$. Hence
\begin{equation}
 \label{2019-08-10+1}
 \begin{aligned}
 \pr( \cL_{it}^+, \cL_{jt}^+)
 &\weq \pr( \cL_{it}^+, \cL_{jt}^+, \abs{\DD_{it}} \le s ) + o(1) \\
 &\weq \pr( \cL_{it}^+, \cL_{jt}^+, \abs{\DD_{it}} \le s, j \notin H_{it} ) + o(1).
 \end{aligned}
\end{equation}

(ii) We show that $i$-exploration and $j$-exploration are whp disjoint up to $t$-th node discovery. More precisely, we show that
\begin{equation}
 \label{2019-08-09}
 \pr( \cL_i^+, \cL_j^+, \abs{\DD_{it}} \le s, j \notin H_{it} )
 \weq \pr( \cL_i^+, \cL_j^+, \abs{\DD_{it}} \le s, j \notin H_{it}, \cS_{ijt}) + o(1).
\end{equation}
where 
$\cS_{ijt} = H_{it} \cap H_{jt} = \emptyset$. Let $\cS^*_{ijst}$ be the event that first $j$-explored layers up to $s$ are disjoint from $H_{it}$. For a uniformly random $x$-set Let $D \subset [m]$, independent of $H_{it}$, on the event $j \notin H_{it}$, we see that
\[
 \pr\big( D \cap H_{it} \ne \emptyset \cond  H_{it}, j \in D \big)
 \wle \abs{H_{it}} \frac{x-1}{m-1}
 \wle \abs{H_{it}} \frac{x}{m}.
\]
Hence, given $\{\abs{\DD_{it}} \le s\}  \cap \{j \notin H_{it}\}$, the conditional probability that the first layer $G_k$ explored by $j$ overlaps with $H_{it}$ at most $M^2 s m$. Here we used the fact that $j \notin H_{it}$ implies $G_k$ does not belong to $\DD_{it}$. Furthermore, given the event that layer $G_k$ is explored by the $r$-th node of the $j$-exploration and that all previously explored layers are disjoint from $H_{it}$ (conditional on this information, $V(G_k) \setminus \{v_r\}$ is a uniformly random $(X_k-1)$-set in $[m] \setminus \{v_1,\dots, v_r\}$ and $H_{it}$ is a subset of $[m] \setminus \{v_1,\dots, v_r\}$), the probability that $G_k$ overlaps $H_{it}$ is at most $(M-1)M{s}(m-r)^{-1} \le M^2 s(m-t)^{-1}$. Here we used the fact that  $v_r \notin H_{it}$ implies that $G_k$ is not in $\DD_{it}$. By the union bound applied to $({ {\cS^*_{ijst}}})^c$, we have
\[
 \pr \big( ( {\cS^*_{ijst}})^c \cond \abs{\DD_{it}} \le s, j \notin H_{it} \big)
 \wle M^2 s^2 (m-t)^{-1}
 \wll 1.
\]
Furthermore, on the event $\cS^*_{ijst}$, i.e., when the $j$-exploration does not encounter $H_{it}$ and thus $L_j$ is 
determined solely by the layers $\DD_j = \{D_{i_1}, D_{i_2},\dots\}$ (which are subsets of $H_{it}^c$), we have by \eqref{2019-06-30+2} that the event $\cL_{jt}^+ \cap \{ \abs{\DD_{jt}} > s\}$ has probability $o(1)$. But the event
$\{\abs{\DD_{jt}} \le s \} \cap \cS^*_{ijst}$ implies $\cS_{ijt}$. 
Hence
\begin{align*}
 \pr( \cL_i^+, \cL_j^+, \cS_{ijt}^c)
 &\weq \pr( \cL_i^+, \cL_j^+, \abs{\DD_{it}} \le s, \abs{\DD_{jt}} \le s, j \notin H_{it}, \cS_{ijt}^c) + o(1) \\
 &\wle  \pr( \cL_i^+, \cL_j^+, \abs{\DD_{it}} \le s, \abs{\DD_{jt}} \le s,  j \notin H_{it}, (\cS^*_{ijst})^c) + o(1) \\
 &\wle  \pr( \cL_i^+, \abs{\DD_{it}} \le s, j \notin H_{it}, (\cS^*_{ijst})^c) + o(1) \\
 &\weq o(1).
\end{align*}
We arrive at \eqref{2019-08-09}.

(iii) Finally, we recall that $\pr( \cL_{it}^+ ) \le \rho(\cY)+o(1)$, and we derive the estimate
\begin{equation}
 \label{2019-08-10+3}
 \pr\big( \cL_{jt}^+, \cS_{ijt} \ \big| \ \cL_{it}^+, \abs{\DD_{it}} \le s, j \notin H_{it} \big)
 \wle \rho(\cY) + o(1).
\end{equation}
This is the key estimate.

Conditioning MB style. Fix $s_1 \le s, t_1 \le t$, and let $\cE = \cE(A,B,C)$ be an event that in the $i$-exploration, the number of discovered nodes exceeds $t$ while exploring the $t_1$-th node and the $s_1$-th layer, $A$ equals the set of nodes explored before the $t_1$-th node, $B$ equals the set of first $s_1$ explored layers, and $C = \cup_{k \in B} V(G_k)$. Conditional on the event $\cE$, we know that \dots

Example. The number of discovered nodes exceeds $10$ while exploring the $2$-nd node and the $3$-rd layer, $A=\{1\}$ 
equals the set of nodes explored before the $2$-nd node, $B=\{6,7,9\}$ equals the set of first 3 explored layers, and $C = \cup_{k \in \{6,7,9\}} V(G_k)$.  On this event we know that the layers $k \in B^c$ do not contain the first explored node. But maybe we know more about these layers? If layers 6,7 were explored during the 1st node exploration step and layer 9 during the 2nd node exploration step, then we know that layers 1,2,3,4,5,8 do not contain the 2nd explored node.
\begin{rcomm}

\end{rcomm}

This is obtained by a similar argument as earlier, but now we put $N^+_x \eqst \Bin(n_{xy}, p^*_{x,t,Ms})$
and ${\tilde N}^+_{xy}$ to be Poisson with the same mean. Here given disjoint node sets $A,B \subset [m]$ of sizes
$|A|=a$, $|B|=b$ and a node $z \notin A \cup B$, we denote
\[
 p^*_{x,a,b}
 \weq \pr( z \in D, \, D \cap B = \emptyset )
 \weq \frac{x}{m-a} \frac{\binom{m-a-b}{x-1}}{\binom{m-a}{x-1}}.
\]
for a uniformly random $x$-set $D \subset [m] \setminus A$. Note that $p^*_x = x m^{-1}(1+o(1))$ for $2 \le x \le M$.
\begin{bcomm}
Note that for $a=t$ and $b = Ms$ and $x \le M$, with $t +1 \le m/2$,
\[
 \pr( D \cap B \ne \emptyset \cond D \ni z)
 \wle b \frac{x-1}{m-a-1}
 \wle M^2 s / (m-t-1)
 \wle 2 M^2 s/m.
\]
Hence in this case
\[
 (1-2 M^2 s/m) \frac{x}{m}
 \wle p^*_{x,a,b}
 \wle \frac{x}{m-t}
\]
\end{bcomm}

Finally, we use the inequality
\begin{equation}
 \label{2019-08-10+2}
 \pr \big( \cL_{it}^+, \cL_{jt}^+ , \abs{\DD_{it}} \le s, j \notin H_{it}, \cS_{ijt} \big)
 \wle \pr \big( \cL_{jt}^+,  \cS_{ijt} \cond \cL_{it}^+, \abs{\DD_{it}} \le s, j \notin H_{it} \big) \, \pr(\cL_{it}^+)
\end{equation}

Relation $\pr( C_i \ge t, C_j \ge t) \le \rho(Y)^2 + o(1)$ for $t = \omega \asymp \log m$ 
follows from \eqref{2019-08-10}, \eqref{2019-08-10+1}, \eqref{2019-08-09}, \eqref{2019-08-10+2}, and \eqref{2019-08-10+3}.

Next we prove (\ref{2019-08-05+1}) for general $1 \ll t \le n \log^{-2} n$, denoting $\bar\omega = \log m$. 
Let $N = \omega\vee{\bar\omega}$ and $\omega_2=\omega\wedge{\bar\omega}$. Then
$|B^{\omega}|-|B^{\bar{\omega}}|
=
 |B^{\omega_2}|
 -
 |B^{N}|\ge 0$.
Now \eqref{2019-08-10+4} implies
\[
 \E\bigl| |B^{\omega}|-|B^{\bar{\omega}}| \bigr|
 \weq \E \big( |B^{\omega_1}| - |B^{\omega_2}| \big)
 \weq o(m).
\]
Therefore, $|B^{\omega}| = |B^{\bar\omega}| + o_\pr(m)$, and the claim follows.
\end{proof}

\rnote{Giant component for finite layer type space.}
\begin{lemma}
\label{the:MBGiant3} 
Under the same assumptions and notations as in Theorem~\ref{the:Giant}, together with the extra assumption that the supports of $P$ and $(P_n)_{n \ge 1}$ are all contained in a finite set $A \subset \Z_+ \cap [0,1]$, the largest component size in $G^{(n)}$ is approximated by $m^{-1} N_1(G^{(n)}) \prto \rho$.
\end{lemma}
\begin{proof}
(i) Fix $\epsilon > 0$ and select $1 \ll \omega \ll n \log^{-1} n$. Lemma~\ref{the:BigComponents} implies $N_1(G^{(n)}) \le \max\{\abs{B_\omega(G^{(n)})}, \omega\}$. Hence for all sufficiently large values of $m$ so that $m^{-1} \omega \le \rho + \epsilon$,
\[
 \pr( m^{-1} N_1(G^{(n)}) > \rho + \epsilon )
 \wle \pr( m^{-1} \abs{B^\omega(G^{(n)})} > \rho + \epsilon ).
\]
The upper bound follows because the right side above tends to zero by Lemma~\ref{the:MBGiant2}.\rnote{The upper bound requires $P(A_0) > 0$, nontrivial layers must exist.}

(ii) Next we analyze the lower bound, assuming $\rho>0$ (the case $\rho=0$ follows immediately by Markov's inequality).
\begin{rcomm}
If $\rho(\cY)>0$, then $P(A_0) > 0$ for $A_0 = (\Z_+ \cap [2,\infty)) \times (0,1]$, because otherwise $\bar g$ and hence also $\cY = \CPoi(\lambda, \bar g)$ would both reduce to Dirac masses at zero.
\end{rcomm}

Recall that $\cY = \CPoi(\lambda, \bar g)$ and $\cY^-_\delta = \CPoi((1-\delta)\lambda, \bar g)$ as defined in \eqref{2019-06-24}, where $\lambda = \mu (P)_{10}$ and $\bar g$ is the distribution defined in \eqref{eq:MixedBinPlus}. By Lemma~\ref{the:CPoiPerturbation} it follows that $\cY^-_\delta \to \cY$ weakly as $\delta \to 0$. Because $\rho(\cY) < 1$, it follows \cite[Lemma 2.6]{Leskela_Ngo_2017} that $\rho(\cY^-_\delta) \to \rho(\cY)$ as $\delta \to 0$. Fix an arbitrary $\epsilon \in (0,1)$. Then we may choose a small enough $\delta > 0$ so that $\rho(\cY^-_\delta) \ge (1-\epsilon)\rho(\cY)$.

Now for each $n \ge 1$, we split the layers so that, for each layer type $(x,y) \in A$, out of the total of $n P_n(x,y)$ such layers, $\floor{\delta n P_n(x,y)}$ are colored blue, and the remaining layers red. We denote by $G^{(n)}_b$ the overlay graph generated by the blue layers, and $G^{(n)}_r$ the one with red layers. Then $G^{(n)} = G^{(n)}_b \cup G^{(n)}_r$. Now the $n$-th model has $\sum_{(x,y)\in A} \floor{\delta n P_n(x,y)} \sim \delta n$ blue layers and $\sim (1-\delta) n$ red layers. Moreover, the empirical layer type distributions of the blue and red models both converge weakly to $P$ as $n \to \infty$. Let $\omega = n^{2/3}$. Then by 
applying \eqref{2019-08-05+1} to the model with red layers we see that $B^\omega(G^{(n)}_r) \wge m\rho(\cY^-_\delta)+o_\pr(m)$ where $\cY^-_\delta = \CPoi((1-\delta)\lambda, \bar g)$. Hence
\begin{equation}
 \label{eq:RedBigComponents}
 \abs{B^\omega(G^{(n)}_r)} \wge (1-\epsilon)\rho(\cY) m + o_\pr(m).
\end{equation}

Then we will show that all nodes of $B^\omega(G^{(n)}_{r})$ belong to the same connected component of $G^{(n)}$ whp. Clearly, there are at most $\omega^{-1} n = n^{1/3}$ such components. Given a pair of distinct such components $C', C''$, for any blue layer $G_k$, the  probability \rnote{conditional probability given $G_{A_\delta^c}$} that $C', C''$ are connected by a link in $G_k$ is at least 
 \[
 p^* := \pr \big( V(G_k) \cap C' \ne \emptyset, V(G_k) \cap C'' \ne \emptyset | G_{A_\delta^c} \big)
 \cdot y 
 \ge (x)_2 n^{4/3} m^{-2} y(1+o(1)).
\]
A layer of size $x$ intersects $C'$ with probability $\abs{C'} x m^{-1}$, and given this event, the layer intersects $C''$ with probability $\abs{C''} (x-1) (m-1)^{-1}$.  A uniformly random node set of size $x$, selected independently of $C'$ and $C''$, given that $C'$ and $C''$ are disjoint, intersects both $C'$ and $C''$, with probability $\abs{C'} \abs{C''} \frac{(x)_2}{(m)_2}$. Hence the probability that there exists a $G_k$-link between $C'$ and $C''$ is at least $\abs{C'} \abs{C''} \frac{(X_k)_2}{(m)_2} Y_k \ge \frac{\omega^2}{m^2} (X_k)_2 Y_k = m^{-2} n^{4/3} (X_k)_2 Y_k$.

Now because $P(A_0) > 0$ and $P_n(x,y) \to P(x,y)$ for all $(x,y) \in A$, we may select a type $(x,y)$ with $x \ge 2$ and $y>0$ such that the set of blue layers of type $(x,y)$ has size $n_\delta \sim \delta P(x,y)n$. By the union bound, the probability\rnote{conditional probability given $G_{A_r}$?} that there exists a pair of distinct red components larger than $\omega$, not connected by link in blue layer of type $(x,y)$, is at most
\[
 \binom{\lceil n^{1/3}\rceil}{2}(1-p^*)^{n_\delta}
 \wle n^{2/3} e^{- n_\delta p^*}
\]
When we choose $\delta = n^{-1/6}$, it follows that 
\[
 n_\delta p^*
 \weq (1+o(1)) n^{-1/6} P(x,y)n p^*
 \wge (1+o(1)) n^{-1/6} P(x,y)n m^{-2} n^{4/3} (x)_2 y
 \wasymp n^{1/6},
\]
and hence $n^{2/3} e^{- n_\delta p^*} \to 0$. Hence whp, all red components larger than $\omega$ are connected to each other by a link in a blue layer of type $(x,y)$, and hence $N_1(G^{(n)}) \ge \abs{B^\omega(G^{(n)}_{r})}$ whp. By \eqref{eq:RedBigComponents}, it follows that $N_1(G^{(n)}) \ge (1-\epsilon)\rho(\cY) m + o_\pr(m)$. The claim follows because this holds for all $\epsilon > 0$.
\end{proof}

\subsection{Old (?) proof of Theorem~\ref{the:Giant}}

We start by selecting a large integer $M$ and defining a truncation map $\tau_M(x,y) = (x \wedge M, y)$ acting on the space of layer types $\Z_+ \cap [0,1]$. Let $G^{n,M}$ be the overlay graph defined with layer types truncated using $\tau_M$. Then the empirical layer type distribution of the truncated model equals $P^{n,M} = P_n \circ \tau_M^{-1}$ and converges according to $P^{n,M} \to P^M = P \circ \tau_M^{-1}$ weakly. As an intermediary step we will prove that
\begin{equation}
 \label{eq:GiantTruncatedSize}
 m^{-1} N_1(G^{n,M}) \prto \rho(\cY^M).
\end{equation}
where $\cY^M = \CPoi(\lambda^M, \bar g^M)$ with $\lambda^M = \mu (P^M)_{10}$ and $\bar g^M$ is defined as in \eqref{eq:MixedBinPlus} but with $P$ replaced by $P^M$.

Let $G^{M,L\pm,n}$ be the overlay graph generated by truncated layer sizes and discretized layer strengths according to $\sigma_{L\pm} \circ \tau_M$. Then this model has empirical layer type distribution $P^{M, L\pm, n} = P_n \circ \tau_M^{-1} \circ \sigma_{L\pm}^{-1}$ with support contained in the finite set $A = (\Z_+ \cap [0,M]) \times S_L$ for all $n$. Then $P^{M,L\pm,n} \to P^{M,L\pm}$ weakly as $n \to \infty$, where the limiting distribution $P^{M,L\pm} = P \circ \tau_M^{-1} \circ \sigma_{L\pm}^{-1}$ also has support contained in $A$.


By applying Lemma~\ref{Lemma2019-12-23} to $G^{M,L\pm,n}$, we conclude that\mnote{verify}
\begin{equation}
 \label{eq:GiantTruncatedSizeStrength}
 m^{-1} N_1( G^{M,L\pm,n} )
 \prto \rho \big(\cY^{M,L\pm} \big),
\end{equation}
where $\cY^{M,L\pm} = \CPoi(\lambda^M, \bar g^{M,L\pm})$, where $\bar g^{M,L\pm}$ is defined as in \eqref{eq:MixedBinPlus} but with $P$ replaced by $P^{M,L\pm}$. Now given any $\delta \in (0,1)$, we claim\mnote{verify} that there exist suitable $M$ and $L$ such that $\abs{\rho \big(\cY^{M,L\pm} \big) - \rho(\cY)} \le \delta \rho(\cY)$.

There is a natural coupling under which $G^{(n,M)}_{L-} \subset G^{(n,M)} \subset G^{(n)}$ and $N_1( G^{(n,M)}_{L-} ) \le N_1( G^{(n)} )$ with probability one. Therefore, by \eqref{eq:GiantTruncatedSizeStrength} we see that
\[
 m^{-1} N_1( G^{(n)} )
 \wge \rho(\cY^M_{L-}) + o_\pr(1)
 \wge (1-\delta) \rho(\cY) + o_\pr(1).
\]

\begin{lemma} 
\label{Lemma2019-12-23}
Assume that the limiting layer type distribution $P$ has a finite support $A$, $(P)_{21} > 0$, and that $X^{(n)}_k \le M$ for all $n \ge 1$ and all $k=1,\dots,n$. Assume also that $P_n(x,y) \to P(x,y)$ for all $(x,y) \in A$. Then the statements of Lemmas \ref{the:MBGiant2} and~\ref{the:MBGiant3} are true.
 \rnote{$ m^{-1} |B^{\omega}| \prto \rho(\cY)$ and $m^{-1} N_1(G^{(n)}) \prto \rho(\cY)$}
\end{lemma}

\begin{proof}
We first show that $m^{-1} N_1(G^{(n)}) \prto \rho(\cY)$.

Fix $n$ and let $A' = \{k \in [n]: (X_k^{(n)}, Y_k^{(n)}) \in A\}$ and $B' = [n] \setminus A'$. Let $N_1(G^{(n)}_{A'})$ be the largest component size in $G^{(n)}_{A'}$ with node set $[m]$ and link set $\cup_{k \in A'} E(G^{(n)}_k)$.
Because $\abs{\cup_{k \in B'} V(G^{(n)}_k)} \le \abs{B'} M$, it follows by Lemma~\ref{the:ComponentOverlayTruncation} that
\begin{equation}
 \label{2019-12-26}
 N_1(G^{(n)}_{A'})
 \wle N_1(G^{(n)})
 \wle |B_t(G^{(n)}_{A'})| + \abs{B'} M t +t
 \quad \text{for all $t \ge 0$}.
\end{equation}
\begin{rcomm}
We assume that $A$ is a finite set such that $P(A)=1$.  We also assume that $P_n \weakto P$ . Does it follow that $P_n(A) \to P(A) = 1$?

Because $A$ is finite, it is closed set in the natural topology of $\Z_+ \times [0,1]$. Here portmanteau goes the wrong way. MB explicitly assume that $P_n(x,y) \to P(x,y)$ for all $(x,y) \in A$, from which the claim follows.
 
\end{rcomm}

The assumptions imply that $P_n(A) \to P(A) = 1$, and hence the number of blue layers satisfies $n_B = (1-P_n(A))n = o(n)$. We observe that results (\ref{eq:MBGiant1}), (\ref{2019-08-05+1}), Lemma~\ref{the:MBGiant3}apply to $G_R$ because $n_B=o(n)$. In particular, the statement of Lemma~\ref{the:MBGiant3} remains true with $C$ replaced by $C_R$.  This shows the lower bound $|C| \ge m \rho(\cY) + o_\pr(m)$.
%
For the upper bound, note that $n_B \ll n$ implies that there exists $1 \ll \omega \ll n$ such that $n_B \omega \ll n$. By substituting $t = \omega$ in \eqref{2019-12-26}, we obtain $|C| \le |B_R^\omega| + o(m)$. Finally we apply (\ref{2019-08-05+1}) to $|B_R^\omega|$ and obtain $|C|\le m\rho(\cY)+o_\pr(m)$. 

Next we show that (\ref{2019-08-05+1}) holds.
The upper bound $|B^\omega|\le m\rho(\cY)+o_\pr(m)$ follows from the second inequality of (\ref{2019-12-26}) by the same argument as above. The lower bound $|B^{\omega}|\ge m\rho(\cY)+o_\pr(m)$ makes sense when $\rho (\cY)>0$. For $\rho (\cY)>0$ the lower bound follows from  $|C|\ge m\rho(\cY)+o_\pr(m)$ and the fact that $|B^{\omega}|\ge |C|$ \rnote{why, due to Lemma~\ref{the:BigComponents}?} provided that $\omega(n)\le n \log^{-2} n$ and $|C| \ge \frac12 m \rho(\cY)$.
\end{proof}

\subsection{Proof of Theorem~\ref{the:Percolation}}

(i) Assume that the layers are generated by model B (random $P_n$-distributed independent layer types). Then in layer percolation we obtain another instance of model B where the layer type distribution $P_n$ is replaced $\tilde P_n = P_n \circ t_\theta^{-1}$ where $t_\theta(x,y) = (x, \theta y)$. Now $\tilde P_n \to \tilde P = P \circ t_\theta^{-1}$ weakly. Moreover,
$(\tilde P_n)_{10} \to (\tilde P)_{10}$ because $(\tilde P_n)_{10} = (P_n)_{10}$ and $(P)_{10} = (\tilde P)_{10}$. Hence \eqref{eq:Percolation} for layer percolation follows by applying Theorem~\ref{the:Giant} to this modified instance of model B. 

(ii) The fact that $\rho_\theta > 0$ if and only if $R_0(\theta) > 1$ follows by basic branching process theory, by noting that $1-\rho_\theta$ equals the extinction probability of a Galton--Watson branching process with offspring distribution $\CPoi(\lambda, \bar g_{\theta})$ which has mean $R_0(\theta)$.

(iii) Assume next that $\sup_n (P_n)_{21} < \infty$, and let us compare the two percolation models. Each link $ij$ of $G^{(n)}$ is retained with probability $\theta$ in overlay percolation, and with probability $1- (1-\theta)^{s_{ij}} \ge \theta$ in layer percolation, where $s_{ij}$ is the number of layers covering node pair $ij$.  Hence we may couple the two percolation models so that $G^{(n)}(\theta) \subset \hat G^{(n)}(\theta)$ with probability one. Fix this coupling, and denote by $\Delta_n$ the number of node pairs linked in $\hat G^{(n)}$ but not in $G^{(n)}(\theta)$. By the union bound, the probability (in model~B) that node pair $ij$ is linked in $\hat G^{(n)}$ but not in $G^{(n)}(\theta)$ is bounded by
\[
 \pr( s_{ij} \ge 2 )
 \wle \frac12 \sumd_{k,\ell} \pr( ij \in E(G^{(n)}_k) \, \pr( ij \in E(G^{(n)}_\ell)
 \weq \binom{n}{2} \left( \frac{(P_n)_{21}}{(m)_2} \right)^2.
\]
Hence $\E \Delta_n \le \frac14 \frac{(n)_2}{(m)_2} (P_n)_{21}^2 = O(1)$, and it follows by Markov's inequality that $\pr( \Delta_n > \omega ) \ll 1$ for any $1 \ll \omega \ll \log n$.

In case $\rho_\theta = 0$, the almost sure upper bounds $\abs{C_1^{(n)}(\theta)} \le \abs{\hat C_1^{(n)}(\theta)}$ and $\abs{C_2^{(n)}(\theta)} \le \abs{\hat C_2^{(n)}(\theta)}$ imply that \eqref{eq:Percolation} holds also for overlay percolation. In case $\rho_\theta > 0$ we need to analyze a lower bound for $\abs{C_1^{(n)}(\theta)}$. Observe that $\abs{\hat C_1^{(n)}(\theta)} \le \omega_1 + \omega_2 \omega_3$ on the event $\cA = \cA_1 \cap \cA_2 \cap \cA_3$ where $\cA_1 = \{\abs{C_1^{(n)}(\theta)} \le \omega_1\}$, $\cA_2 = \{\Delta_n \le \omega_2\}$, and $\cA_3 = \{\abs{C_2^{(n)}(\theta)} \le \omega_3\}$.
Therefore,
\[
 \pr \big( \abs{C_1^{(n)}(\theta)} \le \omega_1 \big)
 \wle \pr( \abs{\hat C_1^{(n)}(\theta)} \le \omega_1 + \omega_2 \omega_3 )
 + \pr( \cA_2^c ) + \pr( \cA_3^c ).
\]
The right hand side above tends to zero when we choose $\omega_1 = (\rho_\theta - \epsilon) m$ for some $0 < \epsilon < \rho_\theta$ together with $1 \ll \omega_2 \ll m^{1/2}$ and $\log m \ll \omega_3 \ll m^{1/2}$. Hence it follows that $\pr \big( m^{-1} \abs{C_1^{(n)}(\theta)} > \rho_\theta-\epsilon \big) \to 1$.
\qed

\section{Leftovers}

\subsection{Old details about limiting degree distribution (deprecated)}
\begin{remark}
Recalling that the moment generating function of $\Bin(n,p)$ equals $( 1-p + p e^\theta)^n$, 
we find that the moment generating function of the limiting degree distribution is
\begin{align*}
 M_f(\theta)
 \weq \sum_{k=0}^\infty e^{-\lambda} \frac{\lambda^k}{k!} M_{g_{10}}(\theta)^k
 \weq \exp \left( \lambda \left( M_{g_{10}}(\theta) - 1 \right) \right),
\end{align*}
where
\begin{align*}
 M_{g_{10}}(\theta)
 \weq \int_{\Z_+ \times [0,1]} \left( 1-q+q e^\theta \right)^{y-1} \frac{y \, \pi(dy, dq)}{(\pi)_{10}}.
\end{align*}
\end{remark}

\begin{remark}
Let $(\tilde X, \tilde Q)$ be distributed according to $\frac{x \, \pi(dx, dq)}{(\pi)_{10}}$. Now
\[
 \E (\tilde X-1)_j \tilde Q^j
 \weq \int_{\Z_+ \times [0,1]} (x-1)_j q^j \ \frac{x \, \pi(dx, dq)}{(\pi)_{10}}
 \weq \int_{\Z_+ \times [0,1]} (x)_{j+1} q^j \ \frac{\pi(dx, dq)}{(\pi)_{10}}
 \weq \frac{(\pi)_{j+1,j}}{(\pi)_{10}},
\]
so that the moments of $g_{10}$ are given by
\[
 m_r( g_{10} )  \weq \sum_{j=1}^r c_{j, r} \E (\tilde X-1)_j \tilde Q^j
 \weq \sum_{j=1}^r c_{j, r} \frac{(\pi)_{j+1,j}}{(\pi)_{10}}.
\]
Hence $m_r(g_{10})$ is finite if and only if $(\pi)_{r+1,r}$ is finite. Because
\[
 \lambda m_r(g_{10})
 \wle m_r(f)
 \wle c_r \lambda^r m_r(g_{10}),
\]
it follows that the $f$ has a finite $r$-th moment if and only if $(\pi)_{r+1,r}$ is finite.

The mean of the limiting degree distribution equals
\[
 m_1(f)
 \weq \lambda m_1(g_{10})
 \weq \lambda \frac{(\pi)_{21}}{(\pi)_{10}}
 \weq \mu (\pi)_{21},
\]
and the variance (when exists) equals
\[
 m_2(f) - m_1(f)^2
 \weq \lambda m_2(g_{10})
 \weq \lambda \left( \frac{(\pi)_{21}}{(\pi)_{10}} + \frac{(\pi)_{32}}{(\pi)_{10}} \right)
 \weq  \mu\left( (\pi)_{21} + (\pi)_{32} \right).
\]

The compound Poisson distribution $f$ has a finite $r$-th moment if and only if the $r$-th moment of $g_{1,0}$ is finite, see Appendix~\ref{sec:CompoundPoisson}. See also Lemma~\ref{the:CrossFactorialMoments}. For a binomial distribution with parameters $x-1$ and $q$, we find that the $r$-th moment is bounded from above by
\[
 \text{const} \times \sum_{j=1}^r (x)_j q^j,
\]
so that the $r$-th moment of $g_{1,0}$ is bounded from above by
\[
 \text{const} \times \sum_{j=1}^r   \int_{\Z_+ \times [0,1]}  (x)_j q^j \ \frac{x \, \pi(dx, dq)}{(\pi)_{1,0}}
 \approx  \sum_{j=1}^r \frac{(\pi)_{j+1,j}}{(\pi)_{1,0}}
 \wle \approx (\pi)_{r+1,r}.
\]
This makes us conjecture that the degree distribution has a finite $r$-th moment iff $(\pi)_{r+1,r}$ is finite. See also Appendix~\ref{sec:BinomialKernel}.
\end{remark}

\subsection{Older details about the limiting degree distribution (deprecated)}

For a general compound Poisson distribution with rate $\lambda$ and summand distribution $f$, the mean equals $\lambda \mom_1(f)$ and the variance equals $\lambda \mom_2(f)$. To compute moments of a $q$-thinned distribution $T_q f$, note that $T_q f$ can be identified as the probability law of the random variable
\[
 \sum_{i=1}^{N-1} B_i
\]
where $N$ is $f$-distributed and $B_i$ are $\Ber$-distributed with a correct $q$-value.

The first moment of the $q$-thinned probability distribution $f$ on $\Z_+$ equals
\[
 \mom_1( T_q f )
 \weq \sum_{s \ge 1} s q_{s+1} f_s.
\]
Hence
\[
 \mom_1( T_q \hat \pi )
 \weq \sum_{s \ge 1} s q_{s+1} \hat \pi_s
 \weq \sum_{s \ge 1} s q_{s+1} \frac{(s+1)\pi_{s+1}}{\mom_1(\pi)}
 \weq \frac{\sum_{s \ge 1} (s-1)s q_{s} \pi_s}{\mom_1(\pi)},
\]
and therefore the first moment of the limiting degree distribution equals
\[
 \lambda \mom_1( T_q \hat \pi )
 \weq \mu \mom_1(\pi) \frac{\sum_{s \ge 1} (s-1)s q_{s} \pi_s}{\mom_1(\pi)}
 \weq \mu \sum_{s \ge 1} (s-1)s q_{s} \pi_s.
\]

\begin{remark}
Assume that $q_s=1$ identically. Then we obtain a so-called passive random intersection graph with $m$ nodes and $n$ layers. Then the upper bound equals
\[
 L_i
 \weq \sum_{k=1}^n 1_{V_k}(i) (x_k-1)_+
 \weq \sum_{s \ge 1} (s-1) \sum_{k: x_k = s} B_{ik}.
\]
where $B_{ik} = 1_{V_k}(i)$ are independent and $\Ber(\frac{x_k}{m})$-distributed. The mean equals
\[
 \E L_i
 \weq \sum_{s \ge 1} \pi^{(n)}_s n (s-1) \frac{s}{m}
 \weq m^{-1} n \sum_{s \ge 1} \pi^{(n)}_s (s-1) s.
\]
The variance equals
\begin{align*}
 \Var(L_i)
 &\weq \sum_{k=1}^n (x_k-1)_+^2 \Var(B_{ik}) \\
 &\weq \sum_{k=1}^n (x_k-1)_+^2 \frac{x_k}{m}(1-\frac{x_k}{m}) \\
 &\weq \sum_{s \ge 0} \sum_{k: x_k = s} (s-1)^2 \frac{s}{m}(1-\frac{s}{m}) \\
 &\weq m^{-1} n \sum_{s \ge 0} (s-1)^2 s(1-\frac{s}{m}) \pi^{(n)}_s.
\end{align*}
Assume that $m, n \to \infty$ and $m/n \to \beta \in (0,\infty)$. Assume also that $\pi^{(n)}_s \to \pi_s$ for every $s$, for some limiting probability distribution $\pi$. Under sufficient moment bounds, it follows that
\[
 \E L_i
 \wto \mu \sum_{s \ge 0} (s-1) s \pi_s
\]
and
\[
 \Var(L_i)
 \wto \mu \sum_{s \ge 0} (s-1)^2 s \pi_s.
\]
The mean and variance are not equal, so the limit is \emph{not} Poisson.
\end{remark}

\subsection{Two different biasings}

\begin{lemma}[Two different biasings]
\label{the:TwoBiasedMeasures}
Let $\pi$ be a probability measure, and let $0 \le \phi \le \psi$ be nonnegative functions such that $0 < \pi(\phi) \le \pi(\psi) < \infty$. Then the probability measures
\[
 \pi_\phi(dz)
 \weq \frac{\phi(z) \pi(dz)}{\pi(\phi)}
 \quad \text{and} \quad
 \pi_\psi(dz)
 \weq \frac{\psi(z) \pi(dz)}{\pi(\psi)}
\]
satisfy $\dtv(\pi_\phi,\pi_\psi) \le 1 - \frac{\pi(\phi)}{\pi(\psi)}$.
\end{lemma}
\begin{proof}
Note that
\begin{align*}
 \frac{\phi(z) }{\pi( \phi )} - \frac{\psi(z)}{\pi(\psi)}
 &\weq \left( \frac{1}{\pi( \phi )} - \frac{1}{\pi(\psi)} \right) \phi(z) - \frac{\psi(z) - \phi(z)}{\pi(\psi)} \\
\end{align*}
Because $0 \le \phi \le \psi$, it follows that
\begin{align*}
 \left| \frac{\phi(z) }{\pi( \phi )} - \frac{\psi(z)}{\pi(\psi)} \right|
 &\wle \left( \frac{1}{\pi( \phi )} - \frac{1}{\pi(\psi)} \right) \phi(z) + \frac{\psi(z)-\phi(z)}{\pi(\psi)} \\
\end{align*}
Integrating shows both sides of the above inequality agains $\pi$ now implies that
\begin{align*}
 \dtv( \pi_\phi, \pi_\psi )
 &\weq \frac12 \int \left| \frac{\phi(z) }{\pi( \phi )} - \frac{\psi(z)}{\pi(\psi)} \right| \pi(dz) \\
 &\wle \frac12 \left( \frac{1}{\pi(\phi )} - \frac{1}{\pi(\psi)} \right) \pi(\phi) + \frac12 \frac{\pi(\psi) - \pi(\phi)}{\pi(\psi)} \\
 &\weq 1 - \frac{\pi(\phi)}{\pi(\psi)}.
\end{align*}
\end{proof}

\subsection{Scale-dependent empirical distributions, unbounded case}

\begin{lemma}[Scale-dependent empirical distributions, unbounded case]
\label{the:ScaleEmpirical}
For each integer $n \ge 1$, let $\hat P_n = \frac{1}{n(n)} \sum_{k=1}^{m(n)} \delta_{X^{(n)}_k}$ be the empirical distribution of independent $P_n$-distributed random variables $X^{(n)}_1,\dots, X^{(n)}_{m(n)}$ in a measurable space $S$, with joint distribution $\pr_n = P_n^{\otimes m(n)}$.
\begin{enumerate}[(i)]
\item Assume that  $m(n) \to \infty$ and that $(P_n \circ \psi^{-1})_{n \ge 1}$ is uniformly integrable for some $\psi: S \to \R_+$. Then for any $\epsilon > 0$,
\begin{equation}
 \label{eq:ScaleEmpirical1}
 \pr_n\Big(  \Big| \int \psi d \hat P_n -  \int \psi d P_n \Big| > \epsilon  \Big)
 \wto 0.
\end{equation}

\item Assume in addition that $\inf_n \int \psi d P_n > 0$, and define $\psi$-biased distributions $P_n^*(dx) = \frac{\psi(x) P_n(dx)}{\int \psi d P_n}$, and
\[
 \hat P_n^*(dx)
 \weq
 \begin{cases}
   \frac{\psi(x) \hat P_n(dx)}{\int \psi d\hat P_n}, &\quad \text{if} \ \int \psi d\hat P_n > 0, \\
   \delta_0(dx), &\quad \text{else}.
 \end{cases}
\]
Then for all bounded measurable functions $\phi$ on $S$, 
\begin{equation}
 \label{eq:ScaleEmpirical2}
 \E_n \Big| \int \phi d \hat P^*_n -  \int \phi d P^*_n \Big| 
 \wto 0.
\end{equation}
\end{enumerate}
\end{lemma}
\begin{proof}
Fix $\epsilon > 0$. Let us use the shorthand $P \psi = \int \psi dP$ for integrals. Note that
\[
 \hat P_n \psi
 \weq \frac{1}{m(n)} \sum_{k=1}^{m(n)} \psi( X^{(n)}_k )
\]
is an average of independent $P_n \circ \psi^{-1}$-distributed random numbers with a common mean $P_n \psi$. Hence by Lemma~\ref{the:QWLLN_New}, it follows that for any $M > 0$,
\[
 \pr_n\Big(  \Big| \hat P_n \psi - P_n \psi \Big| > \epsilon  \Big)
 \wle 9 \frac{M^2}{\epsilon^2 m(n)} + 6 \frac{h_n(M)}{\epsilon},
\]
where $h_n(M) = \int \psi 1(\psi > M) d P_n$. Now \eqref{eq:ScaleEmpirical1} follows by uniform integrability.

(ii) Denote $b = \inf_n \int \psi d P_n > 0$, select a bounded $\phi: S \to \R$, and fix $\epsilon > 0$. Now denote $\Delta_{\phi\psi} = { \hat P_n(\phi\psi) - P_n(\phi\psi)}$ and $\Delta_{\psi} = { \hat P_n(\psi) - P_n(\psi)}$.
Fix some $0 < \delta \le b/2$, and consider the event that $\abs{\Delta_{\psi}} \le \delta$ and $\abs{\Delta_{\phi\psi}} \le \delta$. On this event $\hat P_n(\psi) > 0$, and by writing
\[
 \hat P^*_n \phi - P^*_n \phi
 \weq \frac{\hat P_n(\phi \psi)}{\hat P_n(\psi)} - \frac{P_n(\phi \psi)}{P_n(\psi)} 
 \weq \frac{\Delta_{\phi\psi} - \frac{P_n(\phi \psi)}{P_n(\psi)} \Delta_\psi}{P_n(\psi) + \Delta_\psi},
\]
we see that
\[
 \abs{\hat P^*_n \phi - P^*_n \phi}
 \wle \frac{\abs{\Delta_{\phi\psi}} + \supnorm{\phi} \abs{\Delta_\psi}}{ b - \abs{\Delta_\psi}}
 \wle \frac{2}{b} \left( 1 + \supnorm{\phi} \right) \delta.
\]
Therefore, for any $\delta \in (0,b/2)$ which is so small that $\frac{2}{b} \left( 1 + \supnorm{\phi} \right) \delta \le \epsilon$, it follows that
\[
 \pr_n \Big( \abs{\hat P^*_n \phi - P^*_n \phi} > \epsilon \Big)
 \wle \pr_n ( \abs{\Delta_{\psi}} > \delta ) + \pr_n ( \abs{\Delta_{\phi\psi}} > \delta ).
\]
Now $\pr_n( \abs{\Delta_{\psi}} > \delta) \to 0$ by \eqref{eq:ScaleEmpirical1}. Moreover, because also $(P_n \circ (\phi\psi)^{-1})_{n\ge 1}$ is uniformly integrable, we may apply \eqref{eq:ScaleEmpirical1} again to conclude that $\pr_n( \abs{\Delta_{\phi\psi}} > \delta) \to 0$. Hence \eqref{eq:ScaleEmpirical2} follows by noting that
\[
 \E_n \abs{\hat P^*_n \phi - P^*_n \phi}
 \wle \epsilon + 2 \supnorm{\phi} \pr_n \Big( \abs{\hat P^*_n \phi - P^*_n \phi} > \epsilon \Big)
\]
\end{proof}

\subsection{Empirical degree distribution}

The empirical degree distribution
\[
 \edeg(s)
 \weq \frac{1}{m} \sum_{i=1}^m 1( \deg_G(i) = s )
\]
is a random function which for any integer $s$ returns the fraction of nodes of degree $s$.

\begin{theorem}
\label{the:EmpiricalDegreeDistribution}
Consider a sequence of models parametrized by $(m_\nu, n_\nu, \pi_\nu)$ such that $m_\nu, n_\nu \to \infty$ with $\frac{n_\nu}{m_\nu} \to \eta \in (0,\infty)$. Assume also that $\pi_\nu \to \pi$ weakly with $(\pi_\nu)_{10} \to (\pi)_{10} \in (0,\infty)$, and that $(\pi_\nu)_{21} \ll n_\nu^{1/2}$ and $\int \big( x^3q \wedge x^2 \big) \pi_\nu(dx,dq) \ll n_\nu$. Then the empirical degree distribution of the graph converges according to
\[
 \dtv( \edegnu, f ) \wprnuto 0,
\]
where $f = \CPoi(\lambda, g)$ is a compound Poisson distribution with rate parameter $\lambda = \eta (\pi)_{10}$ and increment distribution $g = \int \Bin(x-1,q) \frac{x \pi(dx, dq)}{(\pi)_{10}}$.
\end{theorem}

\begin{proof}[Proof of Theorem~\ref{the:EmpiricalDegreeDistribution}]
Fix an $\epsilon > 0$ and choose a large enough integer $M$ so that $\sum_{s \ge M} f(s) \le \epsilon$.
Let us denote by $\aedegnu(s) = \E_\nu \edegnu(s)$ an averaged version of the empirical degree distribution, and observe that by exchangeability,
\[
 \aedegnu(s)
 \weq \pr_\nu( \deg_G(i) = s )
 \qquad \text{for all $i$}.
\]
Then by applying Lemma~\ref{the:TotalVariation},
\begin{align*}
 \dtv( \edeg_\nu, f )
 &\wle \sum_{s < M} \abs{ \edeg_\nu(s) - f(s) } + \sum_{s \ge M} f(s) \\
 &\wle \sum_{s < M} \abs{ \edeg_\nu(s) - \aedeg_\nu(s) } + \sum_{s < M} \abs{ \aedeg_\nu(s) - f(s) }
   + \sum_{s \ge M} f(s) \\
 &\wle M \norm{ \edeg_\nu - \aedeg_\nu }_\infty + 2 \dtv(\aedeg_\nu, f) + \sum_{s \ge M} f(s).
\end{align*}
Hence by our choice of $M$ it follows that
\[
 \dtv( \edeg_\nu, f )
 \wle M \norm{ \edeg_\nu - \aedeg_\nu }_\infty + 2 \dtv(\aedeg_\nu, f) + \epsilon,
\]
and for proving the claim it suffices to show that $\norm{ \edeg_\nu - \aedeg_\nu }_\infty \prnuto 0$ and $\dtv(\aedeg_\nu, f) \to 0$. This is what we will do next.

(i) We will verify that $\norm{ \edeg_\nu - \aedeg_\nu }_\infty \prnuto 0$. By Theorem~\ref{the:DegreeDecoupling} we know that for all $i \ne j$, the joint distribution of the degrees $D_i = \deg_G(i)$ and $D_j = \deg_G(j)$ satisfies
\[
 \dtv \Big( \law( D_i, D_j ), \ \law(D_i) \times \law(D_j) \Big)
 \wle \frac{2 n_\nu^2}{m_\nu(m_\nu)_2} (\pi_\nu)_{21}^2 + \frac{2 n_\nu}{(m_\nu)_2} (\pi_\nu)_{10}.
\]
Hence by Lemma~\ref{the:EmpiricalDistributionMean} it follows that for any $\delta > 0$,
\[
 \pr_\nu\left( \norm{ \edeg - \aedeg }_\infty > \delta \right)
 \wle \frac{2}{\delta^2} \left( \frac{1}{m_\nu}
  + \frac{2 n_\nu^2}{m_\nu(m_\nu)_2} (\pi_\nu)_{21}^2 + \frac{2 n_\nu}{(m_\nu)_2} (\pi_\nu)_{10} \right)
 \wto 0.
\]

(ii) To show that $\dtv(\aedegnu, f) \to 0$, note first that
\[
 \dtv(\aedegnu, f)
 \wle \dtv(\aedegnu, f_\nu) + \dtv(f_\nu, f),
\]
where $f_\nu = \CPoi(\lambda_\nu, g_\nu)$ is a compound Poisson distribution with rate parameter $\lambda_\nu = \frac{n_\nu}{m_\nu} (\pi_\nu)_{10}$ and increment distribution $g_\nu = \int \Bin(x-1,q) \frac{x \pi_\nu(dx, dq)}{(\pi_\nu)_{10}}$. By Theorem~\ref{the:DegreeApproximationQuantitative}, it follows that
\[
 \dtv( \aedeg_\nu, f_\nu ) \wto 0.
\]
Moreover, $\lambda_\nu \to \lambda$, and by Lemma~\ref{the:BiasedWeakConvergence} we see that $\frac{x \pi_\nu(dx, dq)}{(\pi_\nu)_{10}} \to \frac{x \pi(dx, dq)}{(\pi)_{10}}$ weakly. Because the map $(x,q) \mapsto \Bin(x-1,q)(s)$ from $\Z_+ \times [0,1]$ into $[0,1]$ is continuous and bounded for any $s \in \Z_+$, it follows that $g_\nu(s) \to g(s)$ for all $s$, which further implies that $\dtv(g_\nu, g) \to 0$. Now by Lemma~\ref{the:CPoiPerturbation} we conclude that $\dtv(f_\nu, f) \to 0$. Hence it follows that $\dtv(\aedeg_\nu, f) \to 0$.
\end{proof}

\subsubsection{Approximate independence}
The following result shows that the degrees of any two distinct nodes are asymptotically independent, under sufficient regularity. 
\begin{lemma}
\label{the:DegreeDecoupling}
For any graph model with parameters $(m,n,\pi)$, the joint distribution of the degrees $D_i = \deg_G(i)$ and $D_j = \deg_G(j)$ of any distinct nodes $i$ and $j$ satisfies
\[
 \dtv \Big( \law( D_i, D_j ), \ \law(D_i) \times \law(D_j) \Big)
 \wle \frac{2 n^2}{m(m)_2} (\pi)_{21}^2 + \frac{2 n}{(m)_2} (\pi)_{10}.
\]
\end{lemma}
\begin{proof}
Denote $D_i = \deg_G(i)$ and $D_{ik} = \deg_{G^k}(i)$. Define $L_i = \sum_k D_{ik}$. We also define random integers
\[
 D'_{ik} \weq B_{ik} A_{ik},
\]
where $\{B_{ik}, A_{ik}: i=1,\dots,n, k=1,\dots, m\}$ is a collection of independent random variables, which is assumed to be independent of $(V_1,\dots,V_n)$, such that $\law(B_{ik}) = \Ber(\frac{x_k}{m})$ and $\law(A_{ik}) = \Bin(x_k-1,q_k)$. This construction implies that $D_{ik} \eqst D'_{ik}$ for all $k$, and hence
\[
 L_i \weqst \sum_{k=1}^n D'_{ik}.
\]

Note that
\[
 ( D_{ik}, D_{jk} )
 \weq \Big( 1_{V_k}(i) D_{ik}, \, 1_{V_k}(j) D_{jk} \Big)
 \weqst \Big( 1_{V_k}(i) A_{ik}, \, 1_{V_k}(j) A_{jk} \Big).
\] 
This implies, by applying Lemma~\ref{the:Sampling}, that
\begin{align*}
 \dtv\big( \law( D_{ik}, D_{jk} ),  \law( D'_{ik}, D'_{jk} ) \big)
 &\wle \dtv\big( \law( 1_{V_k}(i), 1_{V_k}(j) ),  \law( B_{ik}, B_{jk} ) \big) \\
 &\wle 2\frac{x_k}{(m)_2}.
\end{align*}
This further implies, by Lemma~\ref{the:dtvMarginals}, that
\[
 \dtv\Big( \law( D_{ik}, D_{jk}: k \le n ), \ \law( D'_{ik}, D'_{jk}: k \le n ) \Big)
 \wle 2\sum_{k=1}^n \frac{x_k}{(m)_2}.
\]
This also implies that
\[
 \dtv\Big( \law( \sum_k D_{ik}, \sum_k D_{jk} ),  \law( \sum_k D'_{ik}, \sum_k D'_{jk} ) \Big)
 \wle 2\sum_{k=1}^n \frac{x_k}{(m)_2}.
\]
But now
\[
 \law( \sum_k D'_{ik}, \sum_k D'_{jk} )
 \weq \law( \sum_k D'_{ik} ) \times \law( \sum_k D'_{jk}) )
 \weq \law( L_i ) \times \law( L_j ).
\]
Hence
\[
 \dtv( \law(L_i, L_j), \ \law(L_i) \times \law(L_j) )
 \wle 2\sum_{k=1}^n \frac{x_k}{(m)_2}.
\]

Hence by Lemma~\ref{the:dtvMarginals}, we see that
\begin{align*}
 &\dtv( \law(D_i, D_j), \ \law(D_i) \times \law(D_j) ) \\
 &\quad \wle \dtv( \law(D_i, D_j), \ \law(L_i, L_j) ) \\
 &\quad\quad + \dtv( \law(L_i, L_j), \ \law(L_i) \times \law(L_j) ) \\
 &\quad\quad + \dtv( \law(L_i) \times \law(L_j), \ \law(D_i) \times \law(D_j) ) \\
 &\quad \wle 4 \dtv( \law(L_i), \law(D_i) ) + \dtv( \law(L_i, L_j), \ \law(L_i) \times \law(L_j) ).
\end{align*}
The claim follows because in the proof of \rnote{earlier} we saw that
\[
 \dtv( \law(D_i), \law(L_i) )
 \wle \frac{n^2}{2m(m)_2} (\pi)_{21}^2.
\]
\end{proof}

\subsection{Sparsity --- Deterministic layer types}


\begin{proposition}
\label{the:Sparsity}
$\pr(G_{ij}=1) \ll 1$ if and only if $\frac{n}{(m)_2} (\pi)_{21} \ll 1$ \rnote{$\mu_{21} \ll 1$} , in which case
\[
 \pr(G_{ij}=1)
 \weq (1+o(1)) \frac{n}{(m)_2} (\pi)_{21}, \rnote{\sim \mu_{21}}
\]
and the mean degree is approximately
\[
 \E \deg_G(i)
 \weq (1+o(1)) \frac{n}{m} (\pi)_{21}. \rnote{\sim m \mu_{21}}
\]
\end{proposition}
\begin{proof}
An exact formula for the link probability is
\[
 \pr(G_{ij}=1)
 \weq 1 - \prod_{k=1}^n ( 1 - \pr(G_{ij}^k = 1) )
 \weq 1 - \prod_{k=1}^n \left( 1 - p_{21}(k) \right).
\]
Hence the inequality $1-x \le e^{-x}$ yields a lower bound
\[
 \pr(G_{ij}=1)
 \wge 1 - \exp\left( -\sum_{k=1}^n p_{21}(k) \right)
 \weq 1 - e^{- \mu_{21}}.
\]
Hence $\pr(G_{ij}=1) \ll 1$ only if $\mu_{21} \ll 1$.

Assume next that $\mu_{21} \ll 1$. Then by applying the union bound to $\pr(G_{ij}=1) = \pr(\cup_k G_{ij}^k = 1)$ yields
\[
 \pr(G_{ij}=1)
 \wle \sum_k \pr(G_{ij}^k = 1)
 \weq \sum_{k} p_{21}(k)
 \weq \mu_{21}.
\]
Hence in this case $\pr(G_{ij}=1) \ll 1$. Inclusion--exclusion gives another lower bound
\begin{align*}
 \pr(G_{ij}=1)
 &\wge \sum_{k} \pr(G_{ij}^k = 1) - \frac12 \sumd_{k,\ell} \pr(G_{ij}^k = 1, G_{ij}^\ell=1)  \\
 &\weq \sum_{k} \pr(G_{ij}^k = 1) - \frac12 \sumd_{k,\ell} p_{21}(k) p_{21}(\ell) \\
 &\wge \mu_{21} - \frac12 \mu_{21}^2.
\end{align*}
Therefore, $\pr(G_{ij}=1) = \mu_{21} + O(\mu_{21}^2)$, and the expected degree satisfies $\E \deg_G(i) = (m-1) \pr(G_{ij}=1) = (1+o(1)) (m-1) \mu_{21}$.
\end{proof}

\subsection{Sparsity --- random layer types}

\begin{proposition}
Assume that $m^{(\nu)} \gg 1$, $n^{(\nu)} \lesim m^{(\nu)}$,  and that $P^{(\nu)} \to P$ weakly together with $(P^{(\nu)})_{10} \to (P)_{10} < \infty$. Then $\pr( G^{(\nu)}(i,j) = 1 ) \to 0$.
\end{proposition}
\begin{proof}

The proof of Proposition \ref{the:Sparsity} contains the bound
\[
 \pr_{X,Q}(G_{ij}=1)
 \wle n \frac{ (\pi^{(\nu)})_{21}}{(m)_2}.
\]
Taking expectations shows that
\[
 \pr(G_{ij}=1)
 \wle \frac{n}{(m)_2} (P^{(\nu)})_{21}
\]
Moreover, for any $M > 0$,
\begin{align*}
 \frac{n}{(m)_2} (P^{(\nu)})_{21}
 &\weq \frac{n}{(m)_2} \int (x)_2 q 1(x \le M) \, d P^{(\nu)} + \frac{n}{(m)_2} \int (x)_2 q 1(x > M) \, d P^{(\nu)} \\
 &\wle \frac{n}{(m)_2} M^2 + \frac{n}{m} \int x 1(x > M) \, d P^{(\nu)} \\
 &\wle \frac{n}{(m)_2} M^2 + \frac{n}{m} h(M),
\end{align*}
where $h(M) = \sup_{\nu \ge 1} \int x 1(x > M) \, d P^{(\nu)}$. Here $h(M) \to 0$ by uniform integrability. Hence the claim follows.
\end{proof}

\subsection{Old stuff about degree distribution, random layer types (deprecated)}

By Lemma~\ref{the:UIBoundSimple}, whp, $\int x^2 q \, d \pi \ll n^{1/\beta - 1}$ whenever $(X^2 Q)^{\beta}$ is UI for some $\beta > 0$, and $\int (x^3 q \wedge x^2) \, d\pi \ll n^{1/\beta' - 1}$ whenever $(X^3 Q \wedge X^2)^{\beta'}$ is UI for some $\beta' > 0$. Hence, it follows that
\[
 \dtv \Big( \law(\deg_G(i)), \CPoi(\lambda, g_{10}) \Big)
 \wll 1
 \qquad \text{whp}
\]
when  $(X^2 Q)^{\beta}$ is UI for some $\beta > 0$ such that $n \lesim m^{(3/2)\beta}$, and
$(X^3 Q \wedge X^2)^{\beta'}$ is UI for some $\beta' > 0$ such that $n \lesim m^{2\beta'}$.

Case: Assume that $n \wle c m^\alpha$ for some $\alpha \ge 0$. Then we choose $\beta = (2/3)\alpha$ and $\beta' = (1/2)\alpha$. Then we require that $(X^{4/3} Q^{2/3})^{\alpha}$ and $(X^{3/2} Q^{1/2} \wedge X)^{\alpha}$ are uniformly integrable. Hence, we require the $\alpha$-uniform integrability of
\[
 \max\{ X^{4/3} Q^{2/3}, \min\{ X^{3/2} Q^{1/2}, X \} \}
 \weq
 \begin{cases}
  X^{3/2} Q^{1/2}, &\quad Q \le X^{-1}, \\
  X, &\quad X^{-1} \le Q \le X^{-1/2}, \\
  X^{4/3} Q^{2/3}, &\quad Q \ge X^{-1/2}.
 \end{cases}
\]

Special case: $Q = X^{-\gamma}$. Then we require $X^{\gamma'}$ to be $\alpha$-UI where 
\[
 \gamma'
 \weq
 \begin{cases}
   \frac{4}{3} - \frac{2}{3} \gamma, &\quad 0 \le \gamma \le \frac12, \\
   1, &\quad \frac12 \le \gamma \le 1, \\
   \frac{3}{2} - \frac{\gamma}{2}, &\quad \gamma \ge 1.
 \end{cases}
\]

\begin{theorem}
Consider a sequence of models indexed by integers $\nu=1,2,\dots$ where the $\nu$-th model has $m_\nu$ nodes and $n_\nu$ layers of types $(X_{\nu, k}, Q_{\nu, k})$, $k=1,\dots,n_\nu$ which are independent random variables each distributed as a random variable $(X_\nu, Q_\nu)$. Assume that $n_\nu \le c m_\nu^\alpha$ for some scale-independent constants $c > 0$ and $\alpha \ge 0$, $(X_\nu, Q_\nu) \to (X,Q)$ in distribution for some random variable $(X,Q)$ in $\Z_+ \times [0,1]$, and that the collection of random variables $\{\phi(X_\nu, Q_\nu)^\alpha: \nu \ge 1\}$ is uniformly integrable, where
\[
 \phi(x,q)
 \weq
 \begin{cases}
  x^{3/2} q^{1/2}, &\quad 0 \le q \le x^{-1}, \\
  x, &\quad x^{-1} \le q \le x^{-1/2}, \\
  x^{4/3} q^{2/3}, &\quad x^{-1/2} \le q \le 1.
 \end{cases}
\]
Then with high probability, the sequence $(X_\nu, Q_\nu)$ is such that
\[
 \law( \deg_{G_\nu}(1) \cond (X_{\nu,k}, Q_{\nu,k}) )
 \wto \CPoi(\lambda, g_{10})
 \quad \text{weakly as $\nu \to \infty$}.
\]
\rnote{$\lambda \approx 0$ here}
\end{theorem}

\subsection{Old auxiliary results for degree distribution (deprecated?)}

\begin{rcomm}
Apply Lemma~\ref{the:EmpMixedBiased} this with $S_1 = \Z_+ \times [0,1]$, $S_2 = \Z_+$, $\phi(x,q) = x$, $P = P^{\nu}$ being the model layer type distribution, and $Q((x,q),A) = \Bin(x-1,q,A)$. Then $P^\phi Q$ is a probability measure on $\Z_+$ with density
\[
 g_{10}^{(\nu)}(t)
 \ := \ P^\phi Q(t)
 \weq \int \Bin(x-1,q)(t) \frac{x P^{\nu}(dx,dq)}{(P^{\nu})_{10}}
\]
and $\hat P^\phi Q$ is a random probability measure on $\Z_+$ with density
\[
 \hat g_{10}^{(\nu)}(t)
 \ := \ \hat P^\phi Q(t)
 \weq \frac{\frac{1}{n} \sum_{k=1}^n \Bin(X_k-1,Q_k)(t) X_k}{\frac{1}{n} \sum_{k=1}^n X_k}
 \weq \int \Bin(x,q-1)(t) \frac{\pi^\nu(dx,dq)}{(\pi^\nu)_{10}}.
\]
Then Lemma~\ref{the:EmpMixedBiased} shows that
\[
 \pr \left( \, \dtv(\hat g_{10}^{(\nu)}, g_{10}^{(\nu)}) > \epsilon \right)
 \wto 0
\]
when $(P^\nu)_{\nu \ge 1}$ is uniformly $\phi$-integrable for $\phi(x,q) = x$. See below for details.

Consider the probability kernel from $\Z_+ \times [0,1]$ to $\Z_+$ defined by $Q(x, q, A) = \Bin(x-1,q,A)$, where
$\Bin(x, q, A)$ denotes the binomial kernel (Appendix~\ref{sec:BinomialKernel}). Because the binomial kernel is stochastically monotone in its parameters, it follows that $Q(x, q, A) \le \Bin(x, q, A)$ for all upper sets $A$, and by Lemma~\ref{the:BinKernelChernoff},
\[
 Q(x,q,[M',\infty))
 \wle e^{-M'}
\]
for all $M' \ge 7a! M$ and all $(x,q)$ such that $(x)_a q^b \le M$. As a consequence, denoting $\phi(x,q) = (x)_a q^b$ together with $A = [M',\infty)$, the mixed binomial distribution $P^\phi Q$ satisfies
\begin{align*}
 &P^\phi Q([M',\infty)) \\
 &\weq \int_{\phi(x,q) \le M} Q(x,q, A) P^\phi(dx,dq) + \int_{\phi(x,q) > M} \nhquad Q(x,q, A) \, P^\phi(dx,dq) \\
 &\wle \sup_{(x,q): \phi(x,q) \le M} Q(x,q, A) + P^\phi\{(x,q): \phi(x,q) > M\} \\
 &\wle e^{-M'} + h_{P,\phi}(M),
\end{align*}
where
\begin{align*}
 h_{P,\phi}(M)
 &\weq \frac{1}{P\phi} \int \phi(x,q) 1(\phi(x,q) > M) P(dx,dq).
\end{align*}
Hence
\[
 P^\phi Q([M',\infty))
 \wle e^{-M'} + h_{P,\phi}(M)
\]
for all $M \ge 1$ and all $M' \ge 7 a!M$.

This shows that when $(P_\nu)_{\nu \ge 1}$ is uniformly $\phi$-integrable, then $(P_\nu^\phi Q)_{\nu \ge 1}$ is tight. Then by Lemma~\ref{the:EmpMixedBiased} we can show that $\dtv( \hat P^\phi_\nu Q, P^\phi_\nu Q) = o_\pr(1)$.
\end{rcomm}

Let $P$ be a probability measure on a measurable space $S_1$. Let $\phi: S_1 \to [0,\infty)$ be such that $0 < P\phi < \infty$, and define $P^\phi(dz) = \frac{\phi(z) P(dz)}{P\phi}$. Let $Q$ be a probability kernel from $S_1$ into a measurable space $S_2$. Define a probability distribution $P^\phi Q$ on $S_2$ by
\[
 P^\phi Q(A)
 \weq \int Q(z, A) P^\phi(dz)
 \weq \frac{\int Q(z, A) \phi(z) P(dz)}{ \int \phi(z) P(dz) }.
\]
If $Z_1,\dots,Z_n$ are independent $P$-distributed random variables in $S_1$, then define the empirical distribution by $\hat P = \frac{1}{n} \sum_k \delta_{Z_k}$. Then an empirical version of $P^\phi Q$ equals
\[
 \hat P^\phi Q(A)
 \weq \frac{\frac{1}{n} \sum_k Q(Z_k, A) \phi(Z_k)}{\frac{1}{n} \sum_k \phi(Z_k)}.
\]
We will assume that $P\{z: \phi(z) > 0\} = 1$ so that the above formula produces a well-defined probability measure on $S_2$. Below we denote $h_{P,\phi}(M) = (P\phi)^{-1} \int \phi(z) 1(\phi(z) > M) P(dz)$.

\begin{lemma}[Biased empirical mixture]
\label{the:EmpMixedBiased}
Assume that $S_2$ is countable, and that $\phi: S_1 \to [0,\infty)$ satisfies $P\{z: \phi(z) > 0\} = 1$ and $P\phi < \infty$. Then
\[
 \pr\left( \, \dtv(\hat P^\phi Q, P^\phi Q) > 2\epsilon \right)
 \wle \frac{72 \abs{K}^3 M^2}{\epsilon^2 (P \phi)^2 n} + \frac{24 \abs{K}^2 h_{P,\phi}(M)}{\epsilon}.
\]
for all $\epsilon, M > 0$ and $K \subset S_2$ such that $P^\phi Q(K^c) \le \epsilon$,

\end{lemma}
\begin{proof}
Denote $g = P^\phi Q$ and $\hat g = \hat P^\phi Q$. By Lemma~\ref{the:TotalVariation}, 
\begin{align*}
 \dtv(\hat g, g)
 &\wle \sum_{s \in K} \abs{\hat g(s) - g(s)} + g(K^c)
 \wle \abs{K} \max_{s \in K} \abs{\hat g(s) - g(s)} + g(K^c)
\end{align*}
Hence when $g(K^c) \le \epsilon$, it follows that
\begin{equation}
\label{eq:MixedBiased1}
\begin{aligned}
 \pr( \dtv(\hat g, g) > 2 \epsilon)
 &\wle \pr\left( \max_{s \in K} \abs{\hat g(s) - g(s)} > \frac{\epsilon}{\abs{K}} \right) \\
 &\wle \sum_{s \in K} \pr\left(  \abs{\hat g(s) - g(s)} > \frac{\epsilon}{\abs{K}} \right).
\end{aligned}
\end{equation}
Observe next that
\[
 \hat g(s) - g(s)
 \weq \underbrace{\hat g(s) - \frac{\hat P\phi}{P\phi} \cdot \hat g(s)}_{\Delta_1}
 + \underbrace{\frac{\hat P\phi}{P\phi} \cdot \hat g(s) - g(s)}_{\Delta_2}.
\]
Due to $\abs{\hat g(s)} \le 1$,
\[
 \abs{\Delta_1}
 \wle \Big\lvert \frac{\hat P\phi}{P\phi} - 1 \Big\rvert
 \weq \Big\lvert \frac{1}{n} \sum_k \frac{\phi(Z_k)}{P\phi} - 1 \Big\rvert,
\]
where the random variables $Y_k = \frac{\phi(Z_k)}{P\phi}$ are iid with mean $1$, so by applying Lemma~\rnote{the:QWLLN},
\begin{equation}
\label{eq:Delta1Bound}
\begin{aligned}
 \pr \left( \abs{\Delta_1} > \epsilon \right)
 &\wle \frac{9(M/P\phi)^2}{\epsilon^2 n} + \frac{6 \E Y_k 1(Y_k > M/P\phi)}{\epsilon} \\
 &\weq \frac{9M^2}{\epsilon^2 (P\phi)^2 n} + \frac{6 h_{P,\phi}(M) }{\epsilon}.
\end{aligned}
\end{equation}
Similarly,
\[
 \abs{\Delta_2}
 \weq \Big\lvert \frac{1}{n} \sum_k \frac{Q(Z_k,s) \phi(Z_k)}{P\phi}  - g(s) \Big\rvert,
\]
where the summands $Y_k' =  \frac{Q(Z_k,s) \phi(Z_k)}{P\phi}$ are iid with with mean $g(s)$, and satisfy
$\E Y_k' 1(Y_k' > M) \le \E Y_k 1(Y_k > M)$. Hence by applying Lemma~\rnote{the:QWLLN} again, we see that $\Delta_2$ obeys the same upper bound \eqref{eq:Delta1Bound} bound as $\Delta_1$.
%
Hence it follows that
\begin{align*}
 \pr \left( \big\lvert \hat g(s) - g(s) \big\rvert > \frac{\epsilon}{\abs{K}} \right)
 &\wle \pr\left( \abs{\Delta_1} > \frac{\epsilon}{2\abs{K}} \right) + \pr\left( \abs{\Delta_2} > \frac{\epsilon}{2\abs{K}} \right) \\
 &\wle 2 \left( \frac{9M^2}{(\epsilon/(2\abs{K}))^2 (P\phi)^2 n} + \frac{6 h_{P,\phi}(M) }{\epsilon/(2\abs{K})} \right).
\end{align*}
The claim follows by combining this bound with \eqref{eq:MixedBiased1}.

\end{proof}


\subsection{Local clustering --- Rooted graph approach}
A \new{rooted graph} is a pair $(G,v)$ where $G$ is a graph and $v \in V(G)$. We say that $(F,u)$ is a \new{rooted subgraph} of $(G,u)$ if $V(F) \subset V(G)$, $E(F) \subset E(G)$, and $u=v$. A rooted graph homomorphism from $(F,u)$ to $(G,v)$ is a function $\phi: V(F) \to V(G)$ such that $\phi(e) \in E(G)$ for all $e \in E(F)$, and $\phi(u) = v$. An injective homomorphism is called an embedding, and an bijective homomorphism is called an isomorphism. Two rooted graphs are called isomorphic if there exists an isomorphism between them. The set of $(F,u)$-isomorphic rooted subgraphs of a rooted graph $(G,v)$ is denoted by $\Sub_{(F,u)}(G,v)$. The set of embeddings from $(F,u)$ to $(G,v)$ is denoted by $\Emb_{(F,u)}(G,v)$. The set of automorphisms (isomorphism from a rooted graph to itself) of $(G,v)$  is denoted by $\Aut(G,v)$. Lowercase letters $\sub_{(F,u)}(G,v), \emb_{(F,u)}(G,v), \aut(G,v)$ denote the cardinalities of the aforementioned sets. A basic property is that $\emb_{(F,u)}(G,v) = \aut(F,u) \sub_{(F,u)}(G,v)$.

We denote by $K_3^\bullet$ a triangle rooted at any of its nodes, and by $K_{1n}^\bullet$ a star graph with $n$ leaves rooted at its hub node.  The {local clustering} of a node $v$ measures the relative proportion of linked nodes among the neighbors of $v$. More precisely, the \new{local clustering} of a graph $G$ at a node $v \in V(G)$ with degree at least 2 is defined by
\[
 \frac{\sub_{K_3^\bullet}(G,v)}{\sub_{K_{12}^\bullet}(G,v)}.
\]
This quantity also equals the probability that an unordered pair of neighbors of $v$ in $G$, selected uniformly at random, is linked in $G$.

Note that $\sub_{K_3^\bullet}(G,v)$ equals the number of linked (unordered) node pairs at distance one from $v$.
Also, $\sub_{K_{12}^\bullet}(G,v) = \binom{k}{2}$ for any node of degree $k$. Hence for a node of degree $k$,
\[
 \frac{\sub_{K_3^\bullet}(G,v)}{\sub_{K_{12}^\bullet}(G,v)}
 \weq \frac{\abs{E(G[N_G(v)])}}{\binom{k}{2}},
\]
where $N_G(v)$ denotes the set of neighbors of $v$. In terms of the adjacency matrix, 
\[
 \frac{\sub_{K_3^\bullet}(G,v)}{\sub_{K_{12}^\bullet}(G,v)}
 \weq \frac{\sum_{(w,w'): w < w'} G(v,w) G(v,w') G(w,w')}{\sum_{(w,w'): w < w'} G(v,w) G(v,w')}.
\]
Also, because $\aut(K_{12}^\bullet) = \aut(K_3^\bullet) = 2$, we see that
\[
 \frac{\sub_{K_3^\bullet}(G,v)}{\sub_{K_{12}^\bullet}(G,v)}
 \weq \frac{\emb_{K_3^\bullet}(G,v)}{\emb_{K_{12}^\bullet}(G,v)}.
\]

\subsection{General triangle density}

\begin{theorem}[Complicated, deprecated]
\label{the:TriangleDensity}
The probability that $G$ covers any particular triangle is approximated by $\pr\{ G \supset K_3 \} = U' - \epsilon$ where
\begin{equation}
 \label{eq:TriangleDensity}
 U'
 \weq \mu_{33} + 3 \mu_{21} \mu_{32} + \mu_{21}^3,
\end{equation}
and the approximation error is bounded by\mnote{$\ang{p_{ab}} = \frac{n}{(m)_a} (\pi)_{ab}$}
\[
 0 \wle \epsilon
 \wle 54 ( e^{\mu_{21}}-1 ) U' + 5 \mu_{21} \mu_{32}^2 + 15 \mu_{32}^2
 + 3 \ang{p_{21}p_{32}} + 6 \mu_{21} \ang{p_{21}^2}.
\]
\end{theorem}
\begin{proof}[Proof of Theorem~\ref{the:TriangleDensity}]
Denote the links of $K_{3}$ by $e_1, e_2, e_3$.  Let us first introduce the shorthand notations, $k,\ell \in [n]^3$,
\begin{align*}
 p(k) &\weq \pr( G^{k_1} \ni e_1, G^{k_2} \ni e_2, G^{k_3} \ni e_3 ), \\
 p(k\ell) &\weq \pr( G^{k_1} \ni e_1, G^{k_2} \ni e_2, G^{k_3} \ni e_3, G^{\ell_1} \ni e_1, G^{\ell_2} \ni e_2, G^{\ell_3} \ni e_3 ).
\end{align*}
Then by inclusion--exclusion, we find that $U - \Delta \le \pr \{ G \supset K_{3} \} \le U$ where
\[
 U
 \weq \sum_{k \in [n]^3} p(k)
\]
and
\begin{equation}
 \label{eq:TriangleInclusion}
 \Delta
 \weq \nhquad \sum_{k \in [n]^3} \sum_{\ell \in [n]^3: \, \ell \ne k} p(k \ell)
 \weq \Delta_1 + \Delta_{02} + \Delta_{03},
\end{equation}
where $\Delta_1$ is the sum of $p(k\ell)$ such that $\set{k} \ne \set{\ell}$, and $\Delta_{0s}$, $s=2,3$, is the sum of $p(k\ell)$ such that $k \ne \ell$ with $\set{k} = \set{\ell}$ being of cardinality $s$. By Lemma~\ref{the:GeneralUpperBound}, it follows that 
\[
 \Delta_1 \wle 54 ( e^{\mu_{21}}-1 ) U.
\]

Let us next study $\Delta_{03}$. Let us represent a pair of layer triples $(k,\ell)$ as a matrix
\[
 \begin{bmatrix}
  k_1 & k_2 & k_3 \\
  \ell_1 & \ell_2 & \ell_3 \\
 \end{bmatrix}.
\]
For any layer triple $k = (a,b,c)$ such that $\set{k} = 3$, there are 5 layer triples $\ell \ne k$ such that $\set{\ell} = \set{k}$. Out of these, two triples result in matrices
\[
 \begin{bmatrix}
  a & b & c \\
  b & c & a \\
 \end{bmatrix},
 \quad
 \begin{bmatrix}
  a & b & c \\
  c & a & b \\
 \end{bmatrix},
\]
where every layer appears in precisely two columns, and therefore $p(k\ell) = p_{32}(a) p_{32}(b) p_{32}(c)$. The remaining three triples result in matrices
\[
 \begin{bmatrix}
  a & b & c \\
  a & c & b \\
 \end{bmatrix},
 \quad
 \begin{bmatrix}
  a & b & c \\
  c & b & a \\
 \end{bmatrix},
 \quad
 \begin{bmatrix}
  a & b & c \\
  b & a & c \\
 \end{bmatrix},
\]
for which $p(k\ell)$ equals $p_{21}(a) p_{32}(b) p_{32}(c)$, $p_{32}(a) p_{21}(b) p_{32}(c)$, $p_{32}(a) p_{32}(b) p_{21}(c)$, respectively. Therefore, by symmetry, and due to $p_{32}(k) \le p_{21}(k)$, we conclude that
\begin{align*}
 \Delta_{03}
 &\weq \sumd_{a,b,c} \Big( 2 p_{32}(a) p_{32}(b) p_{32}(c) + 3 p_{21}(a) p_{32}(b) p_{32}(c) \Big) \\
 &\wle 5 \sumd_{a,b,c} p_{21}(a) p_{32}(b) p_{32}(c)
 \wle 5 \mu_{21} \mu_{32}^2.
\end{align*}

Let us next analyze $\Delta_{02}$. Fix any $a \ne b$, and let $W_{ab}$ be the set of pairs of layer triples $(k,\ell)$ such that $\set{k} = \set{\ell} = \{a,b\}$ and $k \ne \ell$. Then
\[
 \Delta_{02}
 \weq \sum_{ \{a,b\} \in \binom{[n]}{2}} \sum_{(k,\ell) \in W_{ab}} p(k\ell)
 \weq \frac12 \sumd_{a,b} \sum_{(k,\ell) \in W_{ab}} p(k\ell).
\]
If $k = (a,a,b)$, then $\ell$ must contain $b$ in position 1 or 2, corresponding to a matrix of type
\[
 \begin{bmatrix}
  a & a & b \\
  b & * & * \\
 \end{bmatrix}
 \quad\text{or}\quad
 \begin{bmatrix}
  a & a & b \\
  * & b & * \\
 \end{bmatrix},
\]
where $*$ indicates either $a$ or $b$. Hence both $a$ and $b$ appear in at least two columns of the matrix, and it follows that $p(k\ell) \le p_{32}(a) p_{32}(b)$ (recall that $p_{33}(a) \le p_{32}(a)$ for all $a$). This observation can be generalized to conclude that $p(k\ell) \le p_{32}(a) p_{32}(b)$ for all $(k,\ell) \in W_{ab}$. Moreover, because the set of $k$ such that $\set{k} = \{a,b\}$ has cardinality 6, we see that $\abs{W_{ab}} = 6 \cdot 5 = 30$, and therefore
\[
 \Delta_{02}
 \wle 15 \sumd_{a,b} p_{32}(a) p_{32}(b)
 \wle 15 \mu_{32}^2.
\]
Hence we conclude that
\begin{equation}
 \label{eq:TriangleDeltaOld}
 \Delta
 \wle 54 ( e^{\mu_{21}}-1 ) U + 5 \mu_{21} \mu_{32}^2 + 15 \mu_{32}^2.
\end{equation}

We will next approximate
\[
 U
 \weq \sum_a p_{33}(a) + 3 \sumd_{a,b} p_{21}(a) p_{32}(b) + \sumd_{a,b,c} p_{21}(a) p_{21}(b) p_{21}(c)
\]
using a simpler expression
\[
 U'
 \weq \mu_{33} + 3 \mu_{21}  \mu_{32} +  \mu_{21}^3.
\]
Because
\begin{align*}
 \sumd_{a,b} p_{21}(a) p_{32}(b)
 &\weq \sum_{a,b} p_{21}(a) p_{32}(b) - \sum_a p_{21}(a) p_{32}(a)
\end{align*}
and
\begin{align*}
 \sumd_{a,b,c} p_{21}(a) p_{21}(b) p_{21}(c)
 &\weq \sum_{a,b,c} p_{21}(a) p_{21}(b) p_{21}(c) - 6 \sumd_{a,b} p_{21}(a) p_{21}(b)^2
 - \sum_a p_{21}(a)^3,
\end{align*}
we find that
\[
 U'
 \weq U + 3 \sum_a p_{21}(a) p_{32}(a) + 6 \sumd_{a,b} p_{21}(a) p_{21}(b)^2 + \sum_{a} p_{21}(a)^3.
\]
Hence we see that $U \le U'$. The above equality also shows that
\begin{align*}
 U'-U
 &\weq 3 \ang{p_{21} p_{32}}
 + 6 \sum_{a,b} p_{21}(a) p_{21}(b)^2 - 5 \sum_{a} p_{21}(a)^3 \\
 &\wle 3 \ang{p_{21} p_{32}} + 6 \sum_{a,b} p_{21}(a) p_{21}(b)^2.
\end{align*}
%
Hence $U$ is bounded by $U' - \Delta' \le U \le U'$ with
\[
 \Delta'
 \weq 3 \ang{p_{21}p_{32}} + 6 \mu_{21} \ang{p_{21}^2}.
\]
By combining this with \eqref{eq:TriangleDelta}, we find that the triangle density is bounded by
\[
 U' - \epsilon
 \wle \pr\{G \supset K_3\}
 \wle U',
\]
where the error term $\epsilon = \Delta + \Delta' \ge 0$ satisfies the bound as claimed.
\end{proof}

\begin{rcomm}
For regular layers,
\[
 \frac{3 \mu_{21} \mu_{32} + \mu_{21}^3}{\mu_{33}}
 \wasymp m^{-2} n + m^{-3} n^2.
\]
Hence we need to assume that $n \ll m^{3/2}$ to get $U' \sim \mu_{33}$. Then we get
\[
 \epsilon
 \wlesim 54 \mu_{21} \mu_{33} + 5 \mu_{21} \mu_{32}^2 + 15 \mu_{32}^2
 + 3 \ang{p_{21}p_{32}} + 6 \mu_{21} \ang{p_{21}^2}.
\]
Then also
\[
 \frac{54 \mu_{21} \mu_{33} + 5 \mu_{21} \mu_{32}^2 + 15 \mu_{32}^2}{\mu_{33}}
 \wasymp m^{-2} n
\]
and
\[
 \frac{3 \ang{p_{21}p_{32}}}{\mu_{33}} + \frac{6 \mu_{21} \ang{p_{21}^2}}{\mu_{33}}
 \wlesim \supnorm{x} m^{-1} + \mu_{21} \supnorm{x} m^{-1}
 \wasymp \supnorm{x} m^{-1}.
\]
Hence
\[
 \pr(\cK_3)
 \wsim \mu_{33}
\]
for regular layers with $n \ll m^{3/2}$.
\end{rcomm}

\subsubsection{Simplified triangle density}
It is natural to assume that in a sparse model, the most likely way to form a triangle is by a single layer connecting all nodes of the triangle. For this we need to impose mild extra condition which guarantees there exist sufficiently many layers capable of creating a triangle. The expected number of layers containing any particular triangle as a subgraph equals
\[
 \mu_{33}
 \wasymp (\pi)_{33} m^{-3} n.
\]
In a sparse regime, the expected number of ordered triples of two-node layers jointly covering any particular triangle is approximately $\mu_{21}[2]^3$, where
\[
 \mu_{21}[2]
 \weq \sum_{k: x_k=2} p_{21}(k)
 \wasymp m^{-2} n_2 \qtwo,
\]
and where $n_2$ and $\qtwo$ are the number and average strength of two-node layers. It turns out that a sufficient extra condition is to require that $\mu_{33} \gg \mu_{21}[2]^3$, which is equivalent to
\begin{equation}
 \label{eq:ThreeLayers}
 (\pi)_{33} \wgg m^{-3} n^{-1} (n_2 \qtwo)^3.
\end{equation}
Observe that $(\pi)_{33} \wgg m^{-3} n^2$ is a simple sufficient condition for \eqref{eq:ThreeLayers}, due to $n_2 \le n$ and $\qtwo \le 1$. For this, it suffices that $(\pi)_{33} \gesim 1$ and $n \ll m^{3/2}$.

\begin{rcomm}
Theorem. Assume that $n \ll m^{3/2}$ and that $(\pi)_{21}, (\pi)_{32}, (\pi)_{33} \asymp 1$. Then $\pr(\cK_3) \sim \mu_{33} \sim (\pi)_{33} m^{-3} n$.

Proof. Note that
\begin{align*}
 \frac{U'}{\mu_{33}} - 1
 &\weq 3 \frac{\mu_{21} \mu_{32}}{\mu_{33}} + \frac{\mu_{21}^3}{\mu_{33}} \\
 &\wasymp m^{-2} n + m^{-3} n^2 \\
 &\wasymp m^{-1/2} m^{-3/2} n + \big( m^{-3/2} n \big)^2
 \wll 1,
\end{align*}
so that $U' = (1+o(1)) \mu_{33}$.

Bounds:
\[
 \ang{p_{21} p_{32}}
 \wle \left( \frac{\supnorm{x}}{m} \right)^2 \mu_{33}
\]
and for $(\pi)_{32} \asymp (\pi)_{33} \asymp 1$,
\[
 \ang{p_{21}^2}
 \wle 16 m^{-4} n + 2 \frac{\supnorm{x}}{m} \mu_{32}
 \wasymp \frac{\supnorm{x}}{m} \mu_{33}
\]
and
\[
 \mu_{32}^2
 \wle \mu_{21} \mu_{32}
 \wasymp \mu_{21} \mu_{33}.
\]
Hence by the theorem it follows that $\pr(\cK_3) \sim \mu_{33}$.

\end{rcomm}

\begin{theorem}
\label{the:TriangleDensitySimple}
If the maximum layer size is bounded by $\supnorm{x} \ll m n^{-2/3}$, the link density is at least $\mu_{21} \gesim n^{-4/3}$, and \eqref{eq:ThreeLayers} holds, then
\[
 \pr\{ G \supset K_3 \}
 \wsim \mu_{33}.
\]
\end{theorem}


The proof of the theorem is based on the following lemma.
\begin{lemma}
\label{the:ExpectedCoveringCounts}
For any integers $0 < a \le c$ and $0 < b \le d$ such that $ad \le bc$,
\begin{align*}
 \ang{p_{ba}}
 \wle (d!)^{a/c} n_{\ge d}^{1-a/c} \left(\frac{\supnorm{x}}{m}\right)^{b-ad/c} \ang{p_{dc}}^{a/c} + \ang{p_{ba}}_{[b,d)}.
\end{align*}
where $\ang{p_{ba}}_{[b,d)} = \sum_{b \le x_k < d} p_{ba}(k)$.
\end{lemma}
\begin{proof}
Observe first that
\[
 \ang{p_{ba}}_{[d,\infty)}
 \weq \sum_{x_k \ge d} \frac{(x_k)_b}{(m)_b} q_k^a
 \wle \sum_{x_k \ge d} \left(\frac{x_k}{m}\right)^b q_k^a
 \wle \left(\frac{\supnorm{x}}{m}\right)^{b-ad/c} \sum_{x_k \ge d} \left(\frac{x_k}{m}\right)^{ad/c} q_k^a.
\]
By Lemma~\ref{the:VectorNorms}, and the fact that $x^d \le d!(x)_d$ for $x \ge d$, it follows that
\begin{align*}
 \left( \sum_{x_k \ge d} \left(\frac{x_k}{m}\right)^{ad/c} q_k^a \right)^{c/a}
 \wle n_{\ge d}^{c/a-1} \sum_{x_k \ge d} \left(\frac{x_k}{m}\right)^{d} q_k^c
 \wle d! \, n_{\ge d}^{c/a-1} \sum_{x_k \ge d} \frac{(x_k)_d}{(m)_d} q_k^c.
\end{align*}
By combining the above two inequalities,
\begin{align*}
 \ang{p_{ba}}_{[d,\infty)}
 \wle (d!)^{a/c} n_{\ge d}^{1-a/c} \left(\frac{\supnorm{x}}{m}\right)^{b-ad/c} \ang{p_{dc}}^{a/c}.
\end{align*}
This yields the claim because $\ang{p_{ba}} = \ang{p_{ba}}_{[d,\infty)} + \ang{p_{ba}}_{[b,d)}$.
\end{proof}

\begin{proof}[Proof of Theorem~\ref{the:TriangleDensitySimple}]
Observe first that $\supnorm{x} \ll m n^{-2/3}$ implies $\supnorm{x} \ll m$, and
\[
 \mu_{21}
 \weq \sum_k \frac{(x_k)_2}{(m)_2} q_k
 \wle n \frac{\supnorm{x}^2}{m^2}
 \wll 1.
\]
Moreover,
\[
 \supnorm{p_{21}}
 \weq \max_k \frac{(x_k)_2}{(m)_2} q_k
 \wle \max_k \frac{x_k^2}{m^2} q_k
 \wle m^{-2} \supnorm{x}^2,
\]
so that by applying $\ang{p_{21}^2} \le \supnorm{p_{21}} \mu_{21}$, we find that
\[
 \frac{\ang{p_{21}^2}}{\mu_{21}^2}
 \wle \frac{\supnorm{p_{21}}}{\mu_{21}}
 \wle \frac{m^{-2} \supnorm{x}^2}{\mu_{21}}
 \wlesim \frac{m^{-2} \supnorm{x}^2}{n^{-4/3}}
 \wll 1.
\]
Now by Theorem~\ref{the:TriangleDensity} it follows that
\[
 \pr\{ G \supset K_3 \}
 \wsim \mu_{33} + 3 \mu_{21} \mu_{32} + \mu_{21}^3.
\]

By applying Lemma~\ref{the:ExpectedCoveringCounts} with $(a,b,c,d) = (1,2,3,3)$, we find that
\begin{equation}
 \label{eq:Bound2133}
 \mu_{21}
 \wle 6^{1/3} n^{2/3} \rho \mu_{33}^{1/3} + \mu_{21}[{[2,3)}].
\end{equation}
By applying Lemma~\ref{the:ExpectedCoveringCounts} with $(a,b,c,d) = (2,3,3,3)$, we find that
\begin{equation}
 \label{eq:Bound3233}
 \mu_{32}
 \wle 6^{2/3} n^{1/3} \rho \mu_{33}^{2/3},
\end{equation}
with $\rho = \frac{\supnorm{x}}{m}$. Now $(x+ y)^3 \le 2^{3-1}(x^3 + y^3)$ implies
\[
 \frac{\mu_{21}^3}{\mu_{33}}
 \wle 24 n^{2} \rho^3 + 4 \frac{\mu_{21}[{[2,3)}]^3}{\mu_{33}}
 \weq 24 n^{2} \rho^3 + 4 R,
\]
where $R = \frac{\mu_{21}[{[2,3)}]^3}{ \mu_{33}}$.
Furthermore,
\[
 \frac{\mu_{21}\mu_{32}}{\mu_{33}}
 \wle 6 n \rho^2 + 6^{2/3} n^{1/3} \rho \frac{ \mu_{21}[{[2,3)}] }{ \mu_{33}^{1/3} }
 \weq 6 n \rho^2 +  \Big( 6 n \rho^3 R \Big)^{1/3}.
\]
Because $\rho \ll n^{-2/3}$ and $R \ll 1$, it follows that $\frac{\mu_{21}^3}{\mu_{33}} + \frac{\mu_{21}\mu_{32}}{\mu_{33}} \ll 1$, and the claim follows.

\end{proof}

\subsection{Old stuff about the degree-dependent triangle densities}

Compare with Theorem~\ref{the:TriangleDensity}, and its simplified version Theorem~\ref{the:TriangleDensitySimple}. Compare with similar results for two-stars in Theorem~\ref{the:TwoStarDensityDeg} and Theorem~\ref{the:TwoStarDensityDeg}.

\begin{theorem}
\label{the:TriangleProbability}
For $D$ being the degree of any particular node of $K_3$,
\[
 \pr( D = t, G \supset K_3 )
 \weq \mu_{33} \, f \conv g_{33}(t-2) + \epsilon(t),
\]
where the approximation error is bounded by
\[
 \sum_{t \ge 0} \abs{\epsilon(t)}
 \wle \mu_{21} \mu_{32} + \mu_{21}^3
 + \mu_{33}^2
 + 4 \mu_{21}^{1/3} \mu_{33} b^{2/3}
 + 2 m \ang{p_{21} p_{33}}
\]
with
\[
 b
 \weq 2 + m \mu_{21} + m \frac{\mu_{44}}{\mu_{33}}
 \wle m \mu_{21} + 3 \supnorm{x}.
\]
\end{theorem}

\begin{rcomm}
Random layer types, $\mu_{sr} = \frac{n}{(m)_s} (P_n)_{sr}$:
\begin{align*}
 \sum_{t \ge 0} \abs{\epsilon_1(t)}
 \wle 3 \mu_{21} \mu_{32} + \mu_{21}^3
 \weq O(n^{-3})
\end{align*}
\begin{align*}
 \sum_{t \ge 0} \abs{\epsilon_2(t)}
 \wle \mu_{33}^2
 \weq O(n^{-4})
\end{align*}
\[
 \abs{\epsilon_{3}(t)}
 \wle t \mu_{21} \sum_k \pr( \cA_k )
 \weq t \mu_{21} \mu_{33}
 \weq O(t n^{-3})
\]
\begin{align*}
 \sum_{t > t_0} \abs{\epsilon_{3}(t)}
 &\wle 3 t_0^{-1} \E D 1(\cA_k)
 \wle \Big( (P_n)_{21} (P_n)_{33} + (P_n)_{33} + (P_n)_{44} \Big) O(n^{-2}) t_0^{-1}
\end{align*}
\begin{align*}
 \sum_{t \ge 0} \abs{\epsilon_4(t)}
 \wle 2 (m-1)n^{-1} \mu_{33} \mu_{21} 
 \weq O(n^{-3}).
\end{align*}

\begin{align*}
 \abs{\epsilon(t)}
 \weq O(n^{-3})
 \quad \text{for any fixed $t$}
\end{align*}
\begin{align*}
 \sum_{t \ge 0} \abs{\epsilon(t)}
 \weq O(n^{-3}) + O(t_0^2 n^{-3} + t_0^{-1} n^{-2} +t_0^{-1} (P_n)_{44} n^{-2})
\end{align*}
This is $o(n^{-2})$ when $(P_n)_{44} \ll n^{1/2}$, because then we may choose 
$1 + (P_n)_{44} \ll t_0 \ll n^{1/2}$.
\end{rcomm}

\begin{rcomm}
For regular layers,
\begin{align*}
 \mu_{21} \frac{\mu_{32}}{\mu_{33}} + \frac{\mu_{21}^3}{\mu_{33}} + \mu_{33}
 &\wasymp m^{-2}n + m^{-3} n^2 + m^{-3} n \\
 &\wasymp m^{-2}n + (m^{-3/2} n)^2,
\end{align*}
and $m \frac{\ang{p_{21} p_{33}}}{\mu_{33}} \le m^{-1} \supnorm{x}^2 $, and $b \le m \mu_{21} + 3 \supnorm{x}$ implies
\begin{align*}
 \mu_{21}^{1/2} b
 &\wle \mu_{21}^{1/2} (m \mu_{21} + 3 \supnorm{x}) \\
 &\wasymp (m^{-2} n)^{1/2} (m^{-1} n + \supnorm{x}) \\
 &\wasymp m^{-2} n^{3/2} + (m^{-2} n)^{1/2} \supnorm{x} \\
 &\wasymp (m^{-4/3} n )^{3/2} + (m^{-2} n)^{1/2} \supnorm{x}.
\end{align*}
Hence
\[
 \frac{\sum_{t \ge 0} \abs{\epsilon(t)}}{\mu_{33}}
 \wlesim m^{-2}n + (m^{-3/2} n)^2 + \left( (m^{-4/3} n )^{3/2} + (m^{-2} n)^{1/2} \supnorm{x} \right)^{2/3} + m^{-1} \supnorm{x}^2,
\]
which is small when $n \ll m^{4/3}$ and $\supnorm{x} \ll m n^{-1/2} \wedge m^{1/2}$. A sufficient condition for this is that $n \lesim m$ and $\supnorm{x} \ll m^{1/2}$.
\end{rcomm}

\begin{proof}
Let us denote by $\cA_k = \{G^k \supset K_3\}$  the event that all node pairs of the triangle are linked by layer $k$. We also denote $D = \deg_G(i)$ and $D_k = \deg_{G^k}(i)$ and $D_{-k} = \deg_{G^{-k}}(i)$ where $G^{-k} = \cup_{k' \ne k} G^k$. The proof is based on the following approximations (see also Lemma~\ref{the:TriangleDegree})
\begin{align}
 \pr( D = t, G \supset K_3 )
 \label{eq:TriangleDegree1} &\wapprox \pr( D = t, \cup_k \cA_k ) \\
 \label{eq:TriangleDegree2} &\wapprox \sum_k \pr( D = t, \cA_k) \\
 \label{eq:TriangleDegree3} &\wapprox \sum_k \pr( D_{-k} + D_k = t, \cA_k) \\
 \nonumber &\weq \sum_k \sum_{r+s=t} \pr( D_{-k} = r) \, \pr( D_k = s, \cA_k) \\
 \label{eq:TriangleDegree4} &\wapprox \sum_k \sum_{r+s=t} \pr(D=r) \, \pr( D_k = s, \cA_k ) \\
 \nonumber &\weq \mu_{33} \sum_{r+s=t} \pr(D=r) \, g_{33}(s-2) \\
 \nonumber &\weq \mu_{33} f \conv g_{33}(t-2).
\end{align}
Let $\epsilon_1,\dots,\epsilon_4$ be the approximation errors made in \eqref{eq:TriangleDegree1}--\eqref{eq:TriangleDegree4}.

(i) The approximation error in \eqref{eq:TriangleDegree1} equals
\[
 \epsilon_1(t)
 \weq \pr( D = t, G \supset K_3 ) - \pr( D = t, \cup_k \cA_k ).
\]
Because $\epsilon_1(t) \ge 0$, it follows that
\[
 \sum_{t \ge 0} \abs{\epsilon_1(t)}
 \weq \pr(G \supset K_3 ) - \pr(\cup_k \cA_k )
 \wle \pr( \cE_{12} ) + \pr( \cE_{111} ),
\]
where $\cE_{12}$ is the event that there exists one layer covering one link and a different layer covering two links, and $\cE_{111}$ is the event that three distinct layers cover the links of $K_3$. We write $p(abc) = \pr(\cG^a_{12}, \cG^b_{13}, \cG^c_{23})$ and note that
\begin{align*}
 \pr(\cE_{12})
 \wle \sumd_{a,b} \Big(p(aab) + p(aba) + p(baa)\Big) 
 &\weq 3 \sumd_{a,b} p_{21}(a) p_{32}(b)
\end{align*}
and
\begin{align*}
 \pr(\cE_{111})
 \wle \sumd_{a,b,c} p(abc)
 \weq \sumd_{a,b,c} p_{21}(a) p_{21}(b) p_{21}(c).
\end{align*}
Hence it follows that
\begin{align*}
 \sum_{t \ge 0} \abs{\epsilon_1(t)}
 &\wle 3 \sumd_{a,b} p_{21}(a) p_{32}(b) + \sumd_{a,b,c} p_{21}(a) p_{21}(b) p_{21}(c) \\
 &\wle 3 \mu_{21} \mu_{32} + \mu_{21}^3.
\end{align*}
\begin{rcomm}
Random layer types, $\mu_{sr} = \frac{n}{(m)_s} (P_n)_{sr}$,
\begin{align*}
 \E \sum_{t \ge 0} \abs{\epsilon_1(t)}
 \wle 3 \mu_{21} \mu_{32} + \mu_{21}^3
 \weq O(n^{-3}).
\end{align*}
\end{rcomm}

(ii) The approximation error in \eqref{eq:TriangleDegree2} equals
\[
 \epsilon_2(t)
 \weq \pr( \cup_k \cA_k, D = t) - \sum_k \pr( \cA_k,D = t)
\]
and is bounded (inclusion--exclusion) by
\[
 0
 \wle -\epsilon_2(t)
 \wle \sumd_{k,k'} \pr( \cA_k, \cA_{k'}, D = t).
\]
Hence
\[
 \sum_{t \ge 0} \abs{\epsilon_2(t)}
 \wle \sumd_{k,k'} \pr( \cA_k, \cA_{k'})
 \weq  \sumd_{k,k'} p_{33}(k) p_{33}(k') 
 \wle \mu_{33}^2.
\]
\begin{rcomm}
Random layer types, $\mu_{sr} = \frac{n}{(m)_s} (P_n)_{sr}$,
\begin{align*}
 \E \sum_{t \ge 0} \abs{\epsilon_2(t)}
 \wle \mu_{33}^2
 \weq O(n^{-4}).
\end{align*}
\end{rcomm}

(iii) The approximation error in \eqref{eq:TriangleDegree3} equals $\epsilon_3(t) = \sum_k \epsilon_{3k}(t)$, where
\begin{align*}
 \epsilon_{3k}(t)
 \weq \pr( D = t, \cA_k) - \pr( D_{-k} + D_k = t, \cA_k).
\end{align*}
\rnote{This part not needed for fixed $t$.}
Fix a number $t_0 > 0$. Note that by Markov's inequality,
\[
 \pr( D > t_0, \cA_k )
 \weq \pr( D 1(\cA_k) > t_0 )
 \wle t_0^{-1} \E D 1(\cA_k).
\]
Similarly, by noting that $D_{-k} \le D$ and $D_k \le D$,
\[
 \pr(D_{-k} + D_k > t_0, \cA_k)
 \wle t_0^{-1} \E (D_{-k} + D_k)1(\cA_k)
 \wle 2 t_0^{-1} \E D 1(\cA_k).
\]
As a consequence, we find that
\begin{align*}
 \sum_{t > t_0} \abs{\epsilon_{3k}(t)}
 &\wle \pr( D > t_0, \cA_k) + \pr( D_{-k} + D_k > t_0, \cA_k) \\
 &\wle 3 t_0^{-1} \E D 1(\cA_k).
\end{align*}
Furthermore, by noting that $D \le \sum_\ell D_\ell$, we see that
\begin{align*}
 \E D 1(\cA_k)
 &\wle \sum_{\ell \ne k} \E D_\ell 1(\cA_k) + \E D_k 1(\cA_k),
\end{align*}
\begin{rcomm}
Random layer types:
\begin{align*}
 \E D 1(\cA_k)
 \wle \sum_{\ell \ne k} \E D_\ell 1(\cA_k) + \E D_k 1(\cA_k),
\end{align*}
with
\begin{align*}
 \E D_k 1(\cA_k)
 &\weq 2 \pr(\cA_k) + (m-3) \pr( G_1 \ni \text{coathanger} ) \\
 &\weq 2 \frac{(P_n)_{33}}{(m)_3} + (m-3) \frac{(P_n)_{44}}{(m)_4} \\
 &\weq \Big(  (P_n)_{33} + (P_n)_{44} \Big) O(n^{-3}).
\end{align*}
and 
\[
 \sum_{\ell \ne k} \E D_\ell 1(\cA_k)
 \weq (n-1) (m-1) n^{-2} \mu_{21} \mu_{33}
 \weq (P_n)_{21} (P_n)_{33} O(n^{-3}).
\]
\begin{align*}
 \sum_{t > t_0} \abs{\epsilon_{3k}(t)}
 &\wle 3 t_0^{-1} \E D 1(\cA_k)
 \wle \Big( (P_n)_{21} (P_n)_{33} + (P_n)_{33} + (P_n)_{44} \Big) O(n^{-3}) t_0^{-1}.
\end{align*}
\begin{align*}
 \sum_{t > t_0} \abs{\epsilon_{3}(t)}
 &\wle 3 t_0^{-1} \E D 1(\cA_k)
 \wle \Big( (P_n)_{21} (P_n)_{33} + (P_n)_{33} + (P_n)_{44} \Big) O(n^{-2}) t_0^{-1}.
\end{align*}
\end{rcomm}
where
\[
 \sum_{\ell \ne k} \E D_\ell 1(\cA_k)
 \weq \sum_{\ell \ne k} (m-1) p_{21}(\ell) p_{33}(k)
 \wle m \mu_{21} p_{33}(k).
\]
Also, by noting that $D_k = 2 + \sum_{j \in V(K_3)^c} G^k_{ij}$ on the event $\cA_k$, we see that
\begin{align*}
 \E D_k 1(\cA_k)
 \weq 2 \pr(\cA_k) + \sum_{j \in V(K_3)^c} \E G^k_{ij} 1(\cA_k)
 \weq 2 p_{33}(k) + (m-3) p_{44}(k).
\end{align*}
We may now conclude that
\[
 \sum_{t > t_0} \abs{\epsilon_{3k}(t)}
 \wle 3 t_0^{-1} \Big( m \mu_{21} p_{33}(k) + 2 p_{33}(k) + mp_{44}(k) \Big),
\]
and by summing over $k$ it follows that
\[
 \sum_{t > t_0} \abs{\epsilon_{3}(t)}
 \wle 3 t_0^{-1} \left( 2 \mu_{33} + m \mu_{21} \mu_{33} + m \mu_{44} \right).
\]
\rnote{$3 t_0^{-1} ( O(n^{-2}) + O(n^{-2}) + O(m \mu_{44})$}


Denote $R_k = \deg_{G^k \cap G^{-k}}(i)$. Now, by noting that $D = D_{-k} + D_k$ on the event $R_k=0$, it follows that
\begin{align*}
 \epsilon_{3k}(t)
 \weq \pr( D = t, \cA_k, R_k>0) - \pr( D_{-k} + D_k = t, \cA_k, R_k>0).
\end{align*}
Because both probabilities on the right are at most $\pr( D_k \le t, \cA_k, R_k>0)$, it follows that
\[
 \abs{\epsilon_{3k}(t)}
 \wle \pr( D_k \le t, \cA_k, R_k>0).
\]
Next, $R_k>0$ implies that $ij \in E(G^\ell)$ for some node $j \in N_{G^k(i)} = N_k$ and some layer $\ell \ne k$.
Hence for any node set $U$ of size at most $t$,
\begin{align*}
 \pr( \cA_k, R_k>0, N_k = U)
 &\wle \sum_{j \in U} \sum_{\ell \ne k} \pr( \cA_k, ij \in E(G^\ell), N_k = U) \\
 &\weq \pr( \cA_k, N_k = U) \sum_{j \in U} \sum_{\ell \ne k}  \, \pr(ij \in E(G^\ell) ) \\
 &\weq \pr( \cA_k, N_k = U) \, \abs{U} \sum_{\ell \ne k} p_{21}(\ell)\\
 &\wle t \mu_{21} \, \pr( \cA_k, N_k = U).
\end{align*}
By summing over all node sets $U$ of size at most $t$, it follows that
\[
 \abs{\epsilon_{3k}(t)}
 \wle \pr( \cA_k, D_k \le t, R_k>0)
 \wle t \mu_{21} \, \pr( \cA_k, D_k \le t).
\]
\begin{rcomm}
\begin{align*}
 \abs{\epsilon_{3k}(t)}
 \wle t \mu_{21} \, \pr( \cA_k, D_k \le t).
\end{align*}
\[
 \abs{\epsilon_{3}(t)}
 \wle t \mu_{21} \sum_k \pr( \cA_k )
 \weq t \mu_{21} \mu_{33}
\]
\end{rcomm}
As a consequence,
\[
 \abs{\epsilon_{3}(t)}
 \wle t \mu_{21} \sum_k \pr( \cA_k )
 \weq t \mu_{21} \mu_{33},
\]
so that
\[
 \sum_{t \le t_0} \abs{\epsilon_{3}(t)}
 \wle t_0^2 \mu_{21} \mu_{33},
\]
\rnote{$\weq t_0^2 O(n^{-3})$}
Now we conclude that
\[
 \sum_{t \ge 0} \abs{\epsilon_{3}(t)}
 \wle \Big( t_0^2 \mu_{21} + 3 t_0^{-1} b \Big) \mu_{33},
\]
\rnote{$3 t_0^{-1} ( O(n^{-2}) + O(n^{-2}) + O(m \mu_{44}) + t_0^2 O(n^{-3})$}\\
\rnote{for $t_0 = n^{1/3}$, this is $O(n^{-7/6}) + O(n^{-1/3} m \mu_{44}) $}
with
\[
 b \weq 2 + m \mu_{21} + m \frac{\mu_{44}}{\mu_{33}}.
\]
By substituting $t_0 = \mu_{21}^{-1/3} b^{1/3}$ above, it follows that
\[
 \sum_{t \ge 0} \abs{\epsilon_{3}(t)}
 \wle 4 \mu_{21}^{1/3} b^{2/3} \mu_{33}.
\]

(iv) The approximation error in \eqref{eq:TriangleDegree4} equals $\epsilon_4(t) = \sum_k \epsilon_{4k}(t)$ where
\[
 \epsilon_{4k}(t)
 \weq \sum_{r+s=t} \Big( \pr(D=r) - \pr( D_{-k} = r) \Big) \pr(D_k = s, \cA_k).
\]
Because $D = D_{-k}$ on the event $D_k=0$, it follows that
\begin{align*}
 \abs{ \pr(D=r) - \pr( D_{-k} = r) }
 \wle \pr(D_k > 0),
\end{align*}
and hence
\[
 \abs{\epsilon_{4k}(t)}
 \wle \pr(D_k > 0) \sum_{r+s=t} \pr(D_k = s, \cA_k)
 \wle \pr(D_k > 0) \pr(\cA_k).
\]
\rnote{This is good for fixed $t$}

\begin{align*}
 \epsilon_{4k}(t)
 &\weq \sum_{r+s = t} \Big( \pr( D = r, D_{k} > 0) -\pr( D_{-k} = r, D_{k}>0 ) \Big) \pr(D_k = s, \cA_k) \\
 &\weq \pr(D_{k} > 0) \pr(\cA_{k}) \sum_{r+s = t} \Big( \pr( D^* = r) -\pr( D^*_{-k} = r ) \Big) \pr( \tilde D_{k} = s),
\end{align*}
where $D^*, D_{-k}^*$, and $\tilde D_{k}$ are\mnote{Clean out $D^*, D_{-k}^*$, and $\tilde D_{k}$} arbitrary mutually independent random integers distributed according to
$\law(D \cond D_{k} > 0)$, $\law(D_{-k} \cond D_{k} > 0)$, and $\law(D_{k} \cond \cA_{k})$, respectively.
Hence
\begin{align*}
 \sum_{t \ge 0} \abs{\epsilon_{4k}(t)}
 \wle \pr(D_{k} > 0) \pr(\cA_{k}) \sum_{t \ge 0} \Big( \pr( D^* + \tilde D_{k} = t) + \pr( D^*_{-k} + \tilde D_{k} = t) \Big),
\end{align*}
from which it follows that
\begin{align*}
 \sum_{t \ge 0} \abs{\epsilon_4(t)}
 \wle 2 \sum_{k} \pr(\cA_{k}) \pr(D_{k} > 0) 
 \wle 2 \sum_{k} \pr(\cA_{k}) \E D_{k}.
\end{align*}
Hence, by noting that $\E D_k = (m-1) p_{21}(k) \le m p_{21}(k)$, we see that
\begin{align*}
 \sum_{t \ge 0} \abs{\epsilon_4(t)}
 \wle 2 m \sum_k p_{21}(k) \, \pr( \cA_k)
 \weq 2 m \ang{p_{21} p_{33}}.
\end{align*}
\begin{rcomm}
Random layer types, $\mu_{sr} = \frac{n}{(m)_s} (P_n)_{sr}$, do the above analysis using unconditional probabilities, to obtain
\begin{align*}
 \E \sum_{t \ge 0} \abs{\epsilon_4(t)}
 \wle 2 \sum_k \pr(\cA_k) \E D_k
 \weq 2n \frac{(P_n)_{33}}{(m)_3}  (m-1) \frac{(P_n)_{21}}{(m)_2} 
 \weq 2n^{-1} \mu_{33} \mu_{21} (m-1)
 \weq O(n^{-3}).
\end{align*}
\end{rcomm}

\begin{rcomm}
NO MORE NEEDED.
For large $t$, we approximate
\begin{align*}
 \epsilon_{4k}(t)
 \wle \sum_{r+s=t} \pr(D=r) \pr(D_k = s, \cA_k) 
 &\weq \pr(\cA_k) \sum_{r+s=t} \pr(D=r) \pr(D_k = s \cond \cA_k) \\
 &\weq \pr(\cA_k) \pr(D + D^*_k = t),
\end{align*}
where $D^*_k$ is a random variable independent of $D$, such that $\law(D^*_k) = \law( D_k \cond \cA_k )$.
Similarly, the negative part of the error term is bounded by
\begin{align*}
 -\epsilon_{4k}(t)
 \wle \sum_{r+s=t} \pr(D_{-k}=r) \pr(D_k = s, \cA_k) 
 &\weq \pr(\cA_k) \pr(D_k + D^*_k = t),
\end{align*}
where we assume that $D^*_k$ is also independent of $D_k$. Hence, noting that $D_{-k} \le D$,
\begin{align*}
 \sum_{t > s} \abs{\epsilon_{4k}(t)}
 &\wle \sum_{t > s} \Big( \pr(\cA_k) \pr(D + D^*_k = t) + \pr(\cA_k) \pr(D_k + D^*_k = t) \Big) \\
 &\weq \pr(\cA_k) \Big(  \pr(D + D^*_k > s ) + \pr(\cA_k) \pr(D_k + D^*_k > s) \Big) \\
 &\wle 2 \pr(\cA_k) \pr(D + D^*_k > s) \\
 &\wle 2 \pr(\cA_k) s^{-1} \Big( \E D + \E D^*_k \Big) \\
 &\weq 2 s^{-1} \Big( \pr(\cA_k) \E D + \E D_k 1(\cA_k) \Big).
\end{align*}
By summing both sides of the above inequality over $k$, and recalling that $\pr(\cA_k) = 2\mu_{33}$, $\E D \le (m-1) \mu_{21}$, and
 $\E D_k 1(\cA_k)
 = 2p_{33}(k) + (m-3) p_{44}(k)$,
we find that
\begin{align*}
 \sum_{t > s} \abs{\epsilon_{4}(t)}
 \wle 2 s^{-1} \Big( m \mu_{21} \mu_{33} + 2 \mu_{33} + m \ang{p_{44}} \Big).
\end{align*}
As a conclusion, by noting that $\ang{p_{21} p_{33}} \le \left(\frac{\supnorm{x}}{m}\right)^2 \mu_{33}$
and $\ang{p_{44}} \le \frac{\supnorm{x}}{m} \mu_{33}$,
\begin{align*}
 \sum_{t > 0} \abs{\epsilon_{4}(t)}
 &\wle t_0 m \ang{p_{21} p_{33}} + 2 t_0^{-1} \Big( m \mu_{21} \mu_{33} + 2 \mu_{33} + m \ang{p_{44}} \Big) \\
 &\wle t_0 \supnorm{x}^2 m^{-1} \mu_{33} + 2 t_0^{-1} \Big( m \mu_{21} + 2 + \supnorm{x} \Big) \mu_{33} .
\end{align*}
\end{rcomm}
\end{proof}

\begin{lemma}
\label{the:TriangleDegree}
Let $K_3$ be a triangle with node set contained in $[m]$, and let $i$ be a node of $K_3$. Then for any layer $k$,
\begin{enumerate}[(i)]
\item $\pr( G^k \supset K_3) = p_{33}(k) = \frac{(x_k)_3}{(m)_3} q_k^3$,
\item
$
 \pr\left( \deg_{G^k}(i) = s, \, G^k \supset K_3 \right)
 \weq p_{33}(k) \dbin(x_k-3, q_k, s-2),
$
\item 
$
 \pr\left( \deg_{G^k}(i) \le s, \, G^k \supset K_3 \right)
 \wle 689 \frac{(s+1)^3}{(m)_3}.
$
\end{enumerate}
\end{lemma}
\begin{proof}
(i) The claim follows after noting that
\[
 \pr\Big( V(G^k) \supset V(K_3) \Big)
 \weq \frac{(x_k)_3}{(m)_3}
\]
and
\[
 \pr\Big( G^k \supset K_3 \cond V(G^k) \supset V(K_3) \Big)
 \weq q_k^3.
\]

(ii) Let us consider a node set $B$ of size $x_k$ which contains $A = V(K_3)$. On the event $\{G^k \supset K_3, V_k = B\}$, the number of $G^k$-neighbors of node $i$ equals
\[
 \deg_{G^k}(i)
 \weq \sum_{j \in B \setminus A} G^k_{ij} + 2.
\]
Because the random variables on the right side above are independent of $V_k$ and the link indicators $\{G^k_{ij}: i,j \in A\}$, it follows that
\begin{align*}
 &\pr( \deg_{G^k}(i) = s, \, G^k \supset K_3, \, V_k = B) \\
 &\weq \pr \Big( \sum_{j \in B \setminus A} G^k_{ij} + 2 = s \Big) \, \pr( G^k \supset K_3, \, V_k = B) \\
 &\weq \pr( Z + 2= s) \, \pr( G^k \supset K_3, \, V_k = B),
\end{align*}
where $Z$ is a generic $\Bin(x_k-3, q_k)$-distributed random integer. Hence the second claim follows by summing the above equality over $B$, and dividing the outcome by $\pr( G^k \supset K_3)$.

(iii) Let $Z'$ be a generic $\Bin(x_k, q_k)$-distributed random integer, and note that
$\law(Z') \le \law(Z+3)$ in the strong stochastic order, which can be verified
by a simple coupling of sums of Bernoulli random variables.
Then by applying (i) and (ii) we see by Lemma~\ref{the:ChernoffThirdMoment} that
\begin{align*}
 \pr( \deg_{G^k}(i) \le s, G^k \supset K_3)
 &\weq \pr( G^k \supset K_3) \, \pr( Z \le s-2) \\
 &\weq \pr( G^k \supset K_3) \, \pr( Z+3 \le s+1) \\
 &\wle \pr( G^k \supset K_3) \, \pr( Z' \le s+1) \\
 &\weq \frac{(x_k)_3}{(m)_3} q_k^3 \, \pr( Z' \le s+1) \\
 &\wle (m)_3^{-1} (x_k q_k)^3 \, \pr( Z' \le s+1) \\
 &\wle (m)_3^{-1} 689 (s+1)^3.
\end{align*}
\end{proof}
The following result shows that the conditional degree distribution of a node being part of a triangle is approximately the law of the random integer $2 + D + D'$ where the summands are independent, $D$ follows the degree distribution $f$, and $D'$ follows the distribution $g_{33}$ which corresponds to the excess degree arising from the knowledge that a node already is covered by a triangle.

\begin{rcomm}
Special case: $q_k=1$. Then $g_{33}(x) = g_3(x)$ where
\[
 g_3(x)
 \weq \frac{(x+3)_3 \pi(x+3)}{(\pi)_3},
 \quad x=0,1,2,\dots
\]
is the law of $\tilde X - 3$ where $\tilde X$ is $\tilde \pi_3$-distributed.
\end{rcomm}

\subsection{Stuff about two-star densities}

\begin{theorem}
\label{the:TwoStarDensityDegOld}
For $D$ being the degree of the hub of $K_{12}$,
\begin{equation}
 \label{eq:TwoStarProbabilityOld}
 \begin{aligned}
 &\pr( D = t, \, G \supset K_{12}) \\
 &\quad \weq
 \mu_{32} \, f \conv g_{32}(t-2)
 \ + \ \mu_{21}^2 f \conv g_{21} \conv g_{21}(t-2)
 \ + \ \epsilon(t),
 \end{aligned}
\end{equation}
where $f$ denotes the (nonasymptotic) degree distribution and $g_{21}, g_{32}$ denote the  (nonasymptotic) mixed binomial distributions defined by \eqref{eq:MixedBin}, and the approximation error is bounded by
\begin{align*}
 \sum_{t \ge 0} \abs{ \epsilon(t) }
 &\wle \Big( 15 b^{2/3} \mu_{21}^{1/3}  + \mu_{21} + 6 \mu_{21}^2 \Big) \Big( \mu_{32} + \mu_{21}^2 \Big) \\
   &\qquad + 4  m \ang{ p_{21} p_{32} } + (4 m \mu_{21} + 1) \ang{p_{21}^2},
\end{align*}
where
\[
 b \weq 1 + m \mu_{21} + m \frac{\mu_{32}}{\mu_{21}} + m \frac{\mu_{43}}{\mu_{32}}
 \wle m \mu_{21} + 3 \supnorm{x}.
\]
\end{theorem}

\begin{rcomm}
Because $b \le m \mu_{21} + 3 \supnorm{x}$, this bound shows that
\begin{align*}
 \frac{\sum_{t \ge 0} \abs{\epsilon(t)}}{\pr(\cK_{12})}
 &\wlesim b^{2/3} \mu_{21}^{1/3}  + \mu_{21} + \frac{m \ang{ p_{21} p_{32} } }{m^{-3}n}
 + \frac{(4 m \mu_{21} + 1) \ang{p_{21}^2}}{m^{-3}n} \\
 &\wlesim (m^{-1} n + \supnorm{x})^{2/3} (m^{-2}n)^{1/3}  + m^{-2} n + \supnorm{x}^2 m^{-1} + (4 m^{-1}n + 1) \supnorm{x} m^{-1}  \\
 &\wasymp (m^{-1} n + \supnorm{x})^{2/3} (m^{-2}n)^{1/3}  + \supnorm{x} m^{-2} n + \supnorm{x}^2 m^{-1} \\
 &\weq \left( (m^{-1} n + \supnorm{x}) m^{-1} n^{1/2} \right)^{2/3}  + \supnorm{x} m^{-2} n + \supnorm{x}^2 m^{-1} \\
 &\weq \left( \big(m^{-4/3} n \big)^{3/2} + \supnorm{x} m^{-1} n^{1/2} \right)^{2/3}
   + \supnorm{x} m^{-2} n + \supnorm{x}^2 m^{-1}.
\end{align*}
The relative error vanishes when $n \ll m^{4/3}$ and $\supnorm{x} \ll m^{1/2} \wedge m n^{-1/2}$. (Note that $m n^{-1/2} \ll m^2 n^{-1}$ here.) For $n \lesim m$ this is small iff $\norm{x}_\infty \ll m^{1/2}$.
\end{rcomm}

\begin{remark}
Assume that $n = O(m)$, $\norm{p_{21}}_\infty \ll m^{-1}$, and that the mean degree satisfies $m \mu_{21} \gesim 1$. Then 
\[
 \mu_{21}
 \wle n \norm{p_{21}}_\infty
 \wlesim m \norm{p_{21}}_\infty
 \wll 1,
\]
and
\[
 2 m \ang{p_{21} p_{32} }
 \wle 2 m \norm{p_{21}}_\infty \mu_{32}
 \wll \mu_{32},
\]
and
\[
 \Big(2 m \mu_{21} + 1\Big) \ang{p_{21}^2}
 \wlesim m \mu_{21} \ang{p_{21}^2} 
 \wle m \norm{p_{21}}_\infty \mu_{21}^2
 \wll \mu_{21}^2.
\]
Hence $\abs{\epsilon(t)} \ll \mu_{32} +  \mu_{21}^2$ as desired.

Assume that $n = O(m)$ and $(x_1,q_1), \dots, (x_n, q_n)$ are independent samples from $\law(X_\nu, Q_\nu)$ such that $(X_\nu)_2 Q_\nu)$ is UI with respect to the scale parameter.  Then by Lemma~\ref{the:IIDMaxima}, for any $\epsilon > 0$, we have $\max_k (x_k)_2 q_k \le \epsilon n \le \epsilon c m$ with high probability. Hence $\norm{p_{21}}_\infty \le  \epsilon c m^{-1}$ whp.
\end{remark}

\begin{lemma}
\label{the:DegreeTailTwoStar}
The degree $D$ of the hub of node of $K_{12}$ is bounded by, for any $s \ge 0$,
\[
 \pr(D > s, G \supset K_{12})
 \wle 3 s^{-1}  \left( 1 + m \mu_{21} + m \frac{\mu_{43}}{\mu_{32}} \right) \Big( \mu_{32} + \mu_{21}^2 \Big). 
\]
\end{lemma}
\begin{proof}
We assume that $K_{12}$ is the 2-star with node set $\{1,2,3\}$ and link set $\{12,13\}$.  By writing $D = \sum_{j \ne 1} 1( E(G) \ni 1j)$ and noting that $1(G \supset K_{12}) = 1( E(G) \ni 12, 13)$, we find that
\begin{align*}
 \E D 1(G \supset K_{12})
 &\weq \sum_{j \ne 1} \pr( E(G) \ni 12, 13, 1j) \\
 &\weq 2 \pr( G \supset K_{12}) + (m-3) \pr( G \supset K_{13} ) \\
 &\weq 2 \pr( G \supset K_{12}) + m \pr( G \supset K_{13} ),
\end{align*}
where $K_{13}$ is the 3-star with node set $\{1,2,3,4\}$ and link set $\{12,13,14\}$. Recall that
\[
 \pr (G \supset K_{12})
 \wle \mu_{32} + \mu_{21}^2.
\]
Next, denote $p(abc) = \pr( G_a \ni 12, G_b \ni 13, G_c \ni 14)$, and note that
\begin{align*}
 \pr (G \supset K_{13})
 &\wle \sum_{a,b,c} p(abc) \\
 &\weq \sum_{a} p(aaa) + 3 \sumd_{a,b} p(abb) + \sumd_{a,b,c} p(abc) \\
 &\weq \sum_{a} p_{43}(a) + 3 \sumd_{a,b} p_{21}(a) p_{32}(b) + \sumd_{a,b,c} p_{21}(a) p_{21}(b) p_{21}(c) \\
 &\wle \mu_{43} + 3 \mu_{21} \mu_{32} + \mu_{21}^3 \\
 &\weq \left( \frac{\mu_{43}}{\mu_{32}} + 3 \mu_{21} \right) \mu_{32} + \mu_{21} \mu_{21}^2.
\end{align*}
As a consequence,
\begin{align*}
 \E D 1(G \supset K_{12})
 &\wle \Big( 2 + 3 m \mu_{21} + m \frac{\mu_{43}}{\mu_{32}} \Big) \mu_{32} + \Big(2 + m \mu_{21} \Big) \mu_{21}^2 \\
 &\wle \Big( 2 + 3 m \mu_{21} + m \frac{\mu_{43}}{\mu_{32}} \Big) \Big( \mu_{32} + \mu_{21}^2  \Big).
\end{align*}
Now the claim follows by Markov's inequality
\[
 \pr( D > s,  G \supset K_{12} )
 \weq \pr( D 1(G \supset K_{12}) > s )
 \wle s^{-1} \E D 1(G \supset K_{12}).
\]
\end{proof}

\begin{rcomm}
\rnote{This part needed for bounding the $\sum{t \ge 0} \abs{\epsilon(t)}$ for the degree-dependent two-star density.}
(vii) By Lemma~\ref{the:DegreeTailTwoStar}, we see that
\[
 \sum_{t > t_0} \pr(D=t, \cK_{12})
 \wle 3 t_0^{-1}  \left( 1 + m \mu_{21} + m \frac{\mu_{43}}{\mu_{32}} \right) \Big( \mu_{32} + \mu_{21}^2 \Big). 
\]
Observe next that by Markov's inequality,
\[
 \sum_{t > t_0} f \conv g_{32}(t-2)
 \weq \pr( 2 + D + D_{32} > t_0 )
 \wle t_0^{-1} \Big( 2 + \E D + \E D_{32} \Big),
\]
where $D_{32}$ is an arbitrary $g_{32}$-distributed random integer which is independent of $D$. Note that $\E D \le m \mu_{21}$, and 
\[
 \E D_{32}
 \weq \sum_k (x_k-3) q_k \frac{p_{32}(k)}{\mu_{32}}
 \weq (m-3) \frac{\mu_{43}}{\mu_{32}}
 \wle m \frac{\mu_{43}}{\mu_{32}}.
\]
Hence
\[
 \sum_{t > t_0} f \conv g_{32}(t-2)
 \wle t_0^{-1} \Big( 2 + m \mu_{21} + m \frac{\mu_{43}}{\mu_{32}} \Big).
\]
Similarly,
\[
 \sum_{t > t_0} f \conv g_{21} \conv g_{21} (t-2)
 \weq \pr( 2 + D + D_{21} + D'_{21} > t_0 )
 \wle t_0^{-1} \Big( 2 + \E D + 2 \E D_{21} \Big),
\]
where $D_{21}, D'_{21}$ is are arbitrary $g_{21}$-distributed random integers which are independent of each other and $D$. Note that $\E D \le m \mu_{21}$, and 
\[
 \E D_{21}
 \weq \sum_k (x_k-2) q_k \frac{p_{21}(k)}{\mu_{21}}
 \weq (m-2) \frac{\mu_{32}}{\mu_{21}}
 \wle m \frac{\mu_{32}}{\mu_{21}}.
\]
Hence
\[
 \sum_{t > t_0} f \conv g_{21} \conv g_{21} (t-2)
 \wle  t_0^{-1} \Big( 2 + m \mu_{21} + 2 m \frac{\mu_{32}}{\mu_{21}} \Big).
\]
By combining these upper bounds we find that
\begin{align*}
 &\mu_{32} \sum_{t > t_0} f \conv g_{32} (t-2) + \mu_{21}^2 \sum_{t > t_0} f \conv g_{21} \conv g_{21} (t-2) \\
 &\wle t_0^{-1} \Big( 2 + m \mu_{21} + m \frac{\mu_{43}}{\mu_{32}} \Big) \mu_{32} + t_0^{-1} \Big( 2 + m \mu_{21} + 2 m \frac{\mu_{32}}{\mu_{21}} \Big) \mu_{21}^2 \\
 &\wle t_0^{-1} \Big( 2 + m \mu_{21} + 2 m \frac{\mu_{32}}{\mu_{21}} + m \frac{\mu_{43}}{\mu_{32}} \Big) \Big( \mu_{32} + \mu_{21}^2 \Big).
\end{align*}
As a conclusion, the approximation error $\epsilon(t)$ in \eqref{eq:TwoStarDegCombined} is bounded by
\begin{align*}
 \sum_{t > t_0} \abs{ \epsilon(t) }
 &\wle 3 t_0^{-1}  \left( 1 + m \mu_{21} + m \frac{\mu_{43}}{\mu_{32}} \right) \Big( \mu_{32} + \mu_{21}^2 \Big) \\
 &\qquad + t_0^{-1} \Big( 2 + m \mu_{21} + 2 m \frac{\mu_{32}}{\mu_{21}} + m \frac{\mu_{43}}{\mu_{32}} \Big) \Big( \mu_{32} + \mu_{21}^2 \Big) \\
 &\wle 5 t_0^{-1}  \left( 1 + m \mu_{21} + m \frac{\mu_{32}}{\mu_{21}} + m \frac{\mu_{43}}{\mu_{32}} \right) \Big( \mu_{32} + \mu_{21}^2 \Big).
\end{align*}

(viii) By combining the bounds of (vi) and (vii), it follows that
\begin{align*}
 \sum_{t \ge 0} \abs{ \epsilon(t) }
 &\wle (6 + \mu_{21} + t_0 + 9 t_0^2 ) \mu_{21}
   \Big( \mu_{32} + \mu_{21}^2 \Big) \\
   &\quad + 4  m \ang{ p_{21} p_{32} } + (4 m \mu_{21} + 1) \ang{p_{21}^2} \\
 &\quad + 5 t_0^{-1}  \left( 1 + m \mu_{21} + m \frac{\mu_{32}}{\mu_{21}} + m \frac{\mu_{43}}{\mu_{32}} \right) \Big( \mu_{32} + \mu_{21}^2 \Big).
\end{align*}
For $t_0 \ge 1$, it follows that
\begin{align*}
 \sum_{t \ge 0} \abs{ \epsilon(t) }
 &\wle ( 10 \mu_{21} t_0^2 + 5 b t_0^{-1} ) \Big( \mu_{32} + \mu_{21}^2 \Big)
 + (6 + \mu_{21} ) \mu_{21} \Big( \mu_{32} + \mu_{21}^2 \Big) \\
   &\quad + 4  m \ang{ p_{21} p_{32} } + (4 m \mu_{21} + 1) \ang{p_{21}^2},
\end{align*}
where $b = 1 + m \mu_{21} + m \frac{\mu_{32}}{\mu_{21}} + m \frac{\mu_{43}}{\mu_{32}}$. The above bound is valid for all $t_0 \ge 1$.  Because $b \ge m \mu_{21}$, we see that $\left( \frac{b}{\mu_{21}} \right)^{1/3} \ge m^{1/3} \ge 1$. By substituting $t_0 = \left( \frac{b}{\mu_{21}} \right)^{1/3}$, we see that
\begin{align*}
 10 \mu_{21} t_0^2 + 5 b t_0^{-1}
 &\weq 10 \mu_{21} \left( \frac{b}{\mu_{21}} \right)^{2/3} + 5 b \left( \frac{b}{\mu_{21}} \right)^{-1/3} \\
 &\weq 15 b^{2/3} \mu_{21}^{1/3}.
\end{align*}
Then we conclude that
 \begin{align*}
 \sum_{t \ge 0} \abs{ \epsilon(t) }
 &\wle \Big( 15 b^{2/3} \mu_{21}^{1/3}  + \mu_{21} + 6 \mu_{21}^2 \Big) \Big( \mu_{32} + \mu_{21}^2 \Big) \\
   &\qquad + 4  m \ang{ p_{21} p_{32} } + (4 m \mu_{21} + 1) \ang{p_{21}^2}.
\end{align*}
\end{rcomm}

\begin{lemma}
\label{the:MeanDegreeCond}
For any graph $R$ such that $V(R) \subset [m]$ and any $i \in V(R)$,
\[
 \E \deg_{G_k}(i) 1(G_k \supset R)
 \weq d p_{r,s}(k) + (r - 1 - d) p_{r,s+1}(k) + (m-r) p_{r+1,s+1}(k),
\]
where $r = \abs{V(R)}$, $s = \abs{E(R)}$, and $d=\deg_R(i)$. \rnote{This lemma still needed?}
\end{lemma}
\begin{proof}
Denote by $K_{\{i,j\}}$ the complete graph on node set $\{i,j\}$, and note that the adjacency matrix of $G_k$ can be represented as $G_k(i,j) = 1( G_k \supset K_{\{i,j\}})$. As a consequence,
\[
 G_k(i,j) 1(G_k \supset R)
 \weq 1(G_k \supset K_{\{i,j\}}) 1(G_k \supset R)
 \weq 1(G_k \supset R \cup K_{\{i,j\}}),
\]
and it follows that
\begin{align*}
 \E \deg_{G_k}(i) 1(G_k \supset R)
 &\weq \E \sum_{j \ne i} G_k(i,j) 1(G_k \supset R)
 \weq \sum_{j \ne i} \pr( G_k \supset R \cup K_{\{i,j\}} ).
\end{align*}
Observe next that
\[
 ( \abs{V(R \cup K_{\{i,j\}})}, \abs{E(R \cup K_{\{i,j\}})})
 \weq
 \begin{cases}
  (r,s), &\quad j \in N_R(i), \\
  (r,s+1), &\quad j \in V(R) \setminus N_R(i), \\
  (r+1,s+1), &\quad j \in V(R)^c,
 \end{cases}
\]
where $N_R(i)$ the set of neighbors of node $i$ in graph $R$. Hence the claim follows by splitting the sum above into three parts, and recalling that $\pr(G_k \supset R) = p_{r,s}(k)$ for any graph $R$ with node set contained in $[m]$.
\end{proof}

\begin{remark}[Simplified two-star density]
\label{rem:SimplifiedTwostar}
We get a simplified approximation for the two-star density when
\begin{equation}
 \label{eq:TwostarDensitySimple}
 \ang{p_{21}^2}
 \wll \mu_{32} + \mu_{21}^2.
\end{equation}
Because $(x)_2^2 \le 2 x (x)_3$ for $x \ge 3$, and $(m)_2^2 \ge m (m)_3$ for all integers $m \ge 0$, and $(m_2) \ge \frac12 m^2$ for $m \ge 2$, we find that
\begin{equation}
 \label{eq:p21squared}
 \begin{aligned}
 \ang{p_{21}^2}
 &\weq \frac{4}{(m)_2^2} \sum_{k: x_k=2} q_k^2 + \sum_{k: x_k \ge 3} \frac{(x_k)_2^2}{(m)_2^2} q_k^2 \\
 &\wle \frac{4n}{(m)_2^2} + 2 \sum_{k: x_k \ge 3} \frac{x_k}{m} \frac{(x_k)_3}{(m)_3} q_k^2 \\
 &\wle 16 m^{-4} n + 2 \frac{\norm{x}_\infty}{m} \ang{ p_{32}}.
 \end{aligned}
\end{equation}
Hence \eqref{eq:TwostarDensitySimple} holds when  $\norm{x}_\infty \ll m$ and $\mu_{21}^2 + \mu_{32} \gg m^{-4} n$. For the latter condition it suffices to assume that $(\pi)_{21} \gg n^{-1/2}$ or $(\pi)_{32} \gg m^{-1}$.

The fact that $(k!)^{-1} x^k \le (x)_k \le x^k$ for all $x \ge k$ shows that (the sums below are unrestricted when $\min_k x_k \ge 3$ uniformly)
\begin{align*}
 \mu_{21} &\wasymp m^{-2} \sum_{k: x_k \ge 2} x_k^2 q_k, \\
 \mu_{32} &\wasymp m^{-3} \sum_{k: x_k \ge 3} x_k^3 q_k^2, \\
 \ang{p_{21}^2} &\wasymp m^{-4} \sum_{k: x_k \ge 2} x_k^4 q_k^2.
\end{align*}
Now Lemma~\ref{the:VectorNorms} tells that $\ang{p_{21}^2} \lesim \mu_{21}^2$ but no more.
On the other hand, denoting by $n_2$ the number of layers of size 2,
\begin{align*}
 \ang{p_{21}^2}
 &\wasymp m^{-4} \sum_{k: x_k \ge 2} x_k^4 q_k^2 \\
 &\weq m^{-4} \sum_{k: x_k \ge 3} x_k^4 q_k^2 + m^{-4} \sum_{k: x_k = 2} x_k^4 q_k^2 \\
 &\wlesim m^{-1} \norm{x}_\infty \mu_{32}  + m^{-4} n_2.
\end{align*}
Hence for \eqref{eq:TwostarDensitySimple} it suffices that $\supnorm{x} \ll m$ and $m^{-2} (n_2)^{1/2} \ll \mu_{21}$.

An alternative sufficient condition for \eqref{eq:TwostarDensitySimple} is to require that
the numbers $z_k = (x_k)_2 q_k$ satisfy $\sum_k z_k^2 \ll ( \sum_k z_k )^2$. Note that $\sum_k z_k^2 \ll (\sum_k z_k)^2$ is equivalent to $\frac{1}{n} \sum_k \tilde z_k^2 \ll n$, where $\tilde z_k = \frac{(x_k)_2 q_k}{(\pi)_{21}}$. A sufficient condition for this is that $(\pi)_{21} \gesim 1$ and that the empirical distribution of $\tilde z_k$ is uniformly integrable (with respect to the scale parameter). This can be verified by applying Lemma~\ref{the:EmpiricalUIPower}.
\end{remark}

\begin{remark}[Simplifying conditions]
If $n = O(m)$ and the layer sizes are independent $P_\nu$-distributed random numbers and $(P_\nu)_{\nu \ge 1}$ is uniformly integrable, then by Lemma~\ref{the:IIDMaxima}, it follows that $\max_{1 \le k \le n} x_k \ll_{\pr} m$. We also see that the condition $m n p_{21} \gg m^{-1} n^{1/2}$ is valid whenever $n \lesim m$ and the mean degree is bounded away from zero.
\end{remark}

\begin{example}[Constant layer size]
\label{exa:ConstantLayerSize}
Let $m=n \gg 1$ and assume that all layers have size $x$. Denote $\bar q_r = \frac{1}{n} \sum_k q_k^r$. Then
\[
 np_{21}
 \weq \frac{(x)_2}{(m)_2} n \bar q_1
 \wsim (x)_2 \bar q_1 n^{-1},
\]
and
\[
 np_{32}
 \weq \frac{(x)_3}{(m)_3} n \bar q_2
 \wsim (x)_3 \bar q_2 n^{-2},
\]
and
\[
 \delta
 \weq \frac{(x)_2^2}{(m)_2^2} n \bar q_2
 \wsim (x)_2^2 \bar q_2 n^{-3}.
\]

For $x=2$, we get $n p_{32}=0$,  $(np_{21})^2 \sim 4 \bar q_1^2 n^{-2}$, and 
$\delta \sim 4 \bar q_2 n^{-3}$. In this case sparsity is guaranteed, and Theorem~\ref{the:TwoStarDensityDeg} tells us that
\begin{align*}
 \pr\{ G \supset K_{12} \}
 \wsim (np_{21})^2 - \delta
 \wsim 4 \bar q_1^2 n^{-2} - 4 \bar q_2 n^{-3}.
\end{align*}
Can we say that $\delta \ll (np_{21})^2$ in this case? This is equivalent to $\frac{\bar q_2}{\bar q_1^2} \ll n$. Note that $q_k \le 1$ implies $\bar q_2 \le \bar q_1$. Hence we see the following sufficient conditions:
\[
 \min_k q _k \gg n^{-1}
 \quad \implies \quad
 \bar q_1 \gg n^{-1}
 \quad \implies \quad
 \delta \ll (np_{21})^2.
\]
Recall also that Jensen's inequality implies $\bar q_1 \le (\bar q_2)^{1/2}$, but this tells nothing interesting here.
\end{example}

\begin{example}[Ambient layer containing all nodes]
\label{exa:AmbientLayer}
Let $m = n$, and assume that $x_1 = m$ and $x_k = 7$ for $k \ge 2$. Assume that $q_k = x_k^{-\beta}$ for some $0 < \beta < 1$. Then the expected number of layers covering any particular link is
\[
 n p_{21}
 \weq \sum_{k} \frac{(x_k)_2}{(m)_2} q_k
 \weq m^{-\beta} + \sum_{k \ne 1} \frac{(7)_2}{(m)_2} 7^{-\beta}
 \wsim m^{-\beta},
\]
and the expected number of layers covering any particular two-star is
\[
 n p_{32}
 \weq \sum_{k} \frac{(x_k)_3}{(m)_3} q_k^2
 \weq m^{-2\beta} + \sum_{k \ne 1} \frac{(7)_3}{(m)_3} 7^{-2\beta}
 \wsim m^{-2\beta}.
\]
Moreover,
\[
 \delta
 \weq \sum_{k} \left( \frac{(x_k)_2}{(m)_2} q_k \right)^2
 \weq m^{-2\beta} + \sum_{k \ne 1} \left( \frac{(7)_2}{(m)_2} 7^{-\beta} \right)^2
 \wsim m^{-2\beta}.
\]
Hence in this case $(np_{21})^2 \sim np_{32} \sim \delta \ll 1$. Theorem~\ref{the:TwoStarDensityDeg} tells us that
\begin{align*}
 \pr\{ G \supset K_{12} \}
 \wsim n p_{32} + (np_{21})^2 - \delta
 \wsim m^{-2\beta}.
\end{align*}
A more detailed computation shows that
\[
 n p_{32} + (np_{21})^2 - \delta
 \weq m^{-2\beta} + 2 (7)_2 7^{-\beta} m^{-\beta-1} + O(m^{-2}),
\]
but this more accurate approximation is useless unless the statement of Theorem~\ref{the:TwoStarDensityDeg} is refined to say something about the size of the approximation error.
\end{example}

\begin{example}
Let $m \asymp n$ and assume that one layer has size $n^{1/2} \ll r \ll n$ and the other layers have size 2. Assume that all layers have unit strength $q_k=1$. Then
\[
 n p_{21}
 \weq \sum_{k} \frac{(x_k)_2}{(m)_2} q_k
 \weq (n-1) \cdot \frac{(2)_2}{(m)_2} + 1 \cdot \frac{(r)_2}{(m)_2}
 \wsim \frac{r^2}{n^2}
\]
and (note that $(2)_3=0$ by definition)
\[
 n p_{32}
 \weq \sum_{k} \frac{(x_k)_3}{(m)_3} q_k^2
 \weq (n-1) \cdot \frac{(2)_3}{(m)_3} + 1 \cdot \frac{(r)_3}{(m)_3}
 \wsim \frac{r^3}{n^3}.
\]
We also have
\[
 \delta
 \weq \sum_{k} \left( \frac{(x_k)_2}{(m)_2} q_k \right)^2
 \weq (n-1) \cdot \frac{(2)_2^2}{(m)_2^2} + 1 \cdot \frac{(r)_2^2}{(m)_2^2}
 \wsim \frac{r^4}{n^4}.
\]
Hence $\delta \sim (n p_{21})^2 \ll n p_{32} \ll n p_{21} \ll 1$. This example shows that $\max_k x_k \ll m$ with $n p_{21} \ll 1$ is not in general sufficient for concluding that $\delta \ll  (n p_{21})^2$. In this case Theorem~\ref{the:TwoStarDensityDeg} tells us that
\begin{align*}
 \pr\{ G \supset K_{12} \}
 \wsim n p_{32} + (np_{21})^2 - \delta
 \wsim n p_{32}.
\end{align*}
But this means that here indeed
\[
 \pr\{ G \supset K_{12} \} \wsim n p_{32} + (np_{21})^2.
\]
The fact that $\delta \sim (np_{21})^2$ is not valid does not matter because $\delta \ll np_{32}$. Is it possible to have scenario where $np_{32}, (np_{21})^2, \delta$ are all of the same order? Yes, see Example~\ref{exa:AmbientLayer}.
\end{example}

\begin{example}[Small and large layers]
Assume that $q_k=1$ identically. Assume that $n_1$ of the layers have size $r_1$ and the remaining $n_2 = n - n_1$ layers have size $r_2$, for some $2 \le r_1 \le r_2 \le m$. In this case the sparsity condition is equivalent to requiring that $m \gg 1$ and
\[
 n_1 (r_1)_2 + n_2 (r_2)_2 \wll m^2.
\]

We want to study when
\[
 \sum_{k=1}^n \left[ \frac{(x_k)_2}{(m)_2} q_k \right]^2
 \wll \left[ \sum_{k=1}^n \frac{(x_k)_2}{(m)_2} q_k \right]^2.
\]
In this case this is equivalent to
\begin{equation}
 \label{eq:Wanted}
 n_1 (r_1)_2^2 + n_2 (r_2)_2^2
 \wll \Big( n_1 (r_1)_2 + n_2 (r_2)_2 \Big)^2.
\end{equation}

(i) If the numbers of both small and large layers are $n_1, n_2 \gg 1$, then $n_1 \ll n_1^2$ and $n_2 \ll n_2^2$ imply \eqref{eq:Wanted}.

(ii) Assume that the number of small layers is $n_1 \asymp 1$, and the number of large layers $n_2 \gg 1$. Then the left side of~\eqref{eq:Wanted} is $\sim n_2 (r_2)_2^2$, and the right side of~\eqref{eq:Wanted} is $\sim n_2^2 (r_2)_2^2$. Hence \eqref{eq:Wanted} holds.

(iii) Assume that there are $n_1 \gg 1$ small layers and $n_2 \asymp 1$ large layers. Then $n_1 \sim n$ and $n_1 (r_1)_2^2 \ll n_1^2 (r_1)_2^2$. Then \eqref{eq:Wanted} becomes equivalent to
\[
 (r_2)_2^2
 \wll \Big( n_1 (r_1)_2 + n_2 (r_2)_2 \Big)^2,
\]
or also
\[
 (r_2)_2
 \wll n_1 (r_1)_2 + n_2 (r_2)_2,
\]
or also
\[
 (r_2)_2
 \wll n (r_1)_2.
\]
Hence in this case \eqref{eq:Wanted} is equivalent to the condition that the ratio of large and small layer sizes is bounded by $\frac{r_2}{r_1} \ll n^{1/2}$.

Hence \eqref{eq:Wanted} fails when there are $n_1 \gg 1$ small layers and $n_2 \asymp 1$ large layers, and size of the large layers is $r_2 \gesim n^{1/2} r_1$. In this case the mean number of layers covering any particular link is
\begin{align*}
 \sum_{k=1}^n \frac{(x_k)_2}{(m)_2} q_k
 &\weq (m)_2^{-1} \Big( n_1 (r_1)_2 + n_2 (r_2)_2 \Big) \\
 &\wasymp m^{-2} \Big( n (r_1)_2 + (r_2)_2 \Big) \\
 &\wasymp m^{-2} r_2^2,
\end{align*}
and sparsity is equivalent to $r_2 \ll m$. How does the empirical layer size distribution look in this case?
\end{example}

Denote
\[
 \phi(K)
 \weq \sup \left( \frac{1}{n} \sum_{k=1}^n (x_k)_2 1( (x_k)_2 > K ) \right),
\]
where the supremum is taken with respect to all scale parameters.

Observe now that
\begin{align*}
 \sum_{k=1}^n (x_k)_2^2 1( (x_k)_2 > K )
 \wle \Big( \sum_{k=1}^n (x_k)_2 1( (x_k)_2 > K ) \Big)^2
 \weq n^2 \phi(K)^2.
\end{align*}
and
\begin{align*}
 \sum_{k=1}^n (x_k)_2^2 \, 1( (x_k)_2 \le K )
 \wle n K^2,
\end{align*}
so that
\[
 \sum_{k=1}^n (x_k)_2^2
 \wle n K^2 + n^2 \phi(K)^2.
\]
On the other hand, for $x_k \ge 2$,
\[
 \sum_{k=1}^n (x_k)_2
 \wge n,
\]
so that
\[
 \frac{\sum_{k=1}^n (x_k)_2^2}{\left( \sum_{k=1}^n (x_k)_2 \right)^2} 
 \wle \frac{n K^2 + n^2 \phi(K)^2}{n^2} 
 \weq n^{-1} K^2 + \phi(K).
\]
Now for $K = n^{1/3}$, the right side vanishes.

\subsection{Transitivity spectrum --- Rooted graph notations}

The \new{transitivity spectrum} of $G$ is the function $t_G: S \to [0,1]$ defined by
\[
 t_G(k)
 \weq \frac{1}{v_k(G)} \sum_{v \in V_k(G)} \frac{\sub_{K_3^\bullet}(G,v)}{\sub_{K_{12}^\bullet}(G,v)},
\]
where $S = \{k \ge 2: v_k(G) > 0\}$.

Observe that
\begin{align*}
 &\sum_{v \in V_{k}(G)} \emb( (K_3,1), (G,v) ) \\
 &\weq \sum_{v \in V_{k}(G)} \abs{ \{ \phi \in \Emb(K_3, G): \phi(1)=v\} } \\
 &\weq \sum_{v \in V_{k}(G)} \sum_{w_1 \in V(G)} \sum_{w_2 \in V(G)} \abs{ \{ \phi \in \Emb(K_3, G): \phi(1)=v, \phi(2) = w_2, \phi(3) = w_3\} } \\
 &\weq \sum_{v \in V_{k}(G)} \sum_{\ell_1 \ge 0} \sum_{\ell_2 \ge 0} \sum_{w_1 \in V_{\ell_1} (G)} \sum_{w_2 \in V_{\ell_2}(G)} \abs{ \{ \phi \in \Emb(K_3, G): \phi(1)=v, \phi(2) = w_2, \phi(3) = w_3\} } \\
 &\weq \sum_{v \in V_{k}(G)} \sum_{\ell_1 \ge 0} \sum_{\ell_2 \ge 0} \sum_{w_1 \in V_{\ell_1} (G)} \sum_{w_2 \in V_{\ell_2}(G)} 1( G[\{v,w_1,w_2\}] \isom K_3 ).
\end{align*}
Now define (for graphs such that the denominators are nonzero)
\[
 f_{k,\ell_1,\ell_2}(G)
 \weq \frac{1}{v_k(G)} \frac{1}{v_{\ell_1}(G)} \frac{1}{v_{\ell_2}(G)}
 \sum_{v \in V_{k}(G)} \sum_{w_1 \in V_{\ell_1} (G)} \sum_{w_2 \in V_{\ell_2}(G)} 1( G[\{v,w_1,w_2\}] \isom K_3 )
\]
as the probability (given a graph sample $G$) that an ordered node triple $(v^*, w_1^*, w_2^*)$ with degrees $k, \ell_1, \ell_2$, selected uniformly at random, induces a triangle in $G$. Then we may write
\[
 \sum_{v \in V_{k}(G)} \emb( (K_3,1), (G,v) )
 \weq v_k(G) \sum_{\ell_1 \ge 0} \sum_{\ell_2 \ge 0} v_{\ell_1}(G) v_{\ell_2}(G) f_{k,\ell_1,\ell_2}(G),
\]
and as a consequence, the transitivity spectrum can be expressed as
\[
 t_G(k)
 \weq \frac{1}{k(k-1)} \sum_{\ell_1 \ge 0} \sum_{\ell_2 \ge 0} v_{\ell_1}(G) v_{\ell_2}(G) f_{k,\ell_1,\ell_2}(G).
\]

\subsection{Discussion on the variational approach of Stegehuis et al.}

The heuristic variational approach in \cite[Chapter 6]{Stegehuis_2019_thesis} finds an approximate value for $t_G(k)$ by replacing the sum on the right with a term which maximizes the sum. The key heuristic observation is that for many random graph models, there exists unique $(\ell_1^*,\ell_2^*$ which maximize the term $v_{\ell_1}(G) v_{\ell_2}(G) f_{k,\ell_1,\ell_2}(G)$, in which case
\[
 t_G(k)
 \wapprox \frac{1}{k^2} \, v_{\ell_1^*}(G) v_{\ell_2^*}(G) f_{k,\ell_1^*,\ell_2^*}(G).
\]

For the active RIG (under usual limiting assumptions, with limiting degree distribution having a power law with density exponent $\tau > 2$), whp,
\[
 f_{k,\ell_1,\ell_2}(G)
 \wapprox \text{const} \times k \frac{\ell_1}{m} \frac{\ell_2}{m} 
\]
and
\[
 v_{\ell}(G) 
 \weq n \frac{v_\ell(G)}{n}
 \wapprox \text{const} \times n \ell^{-\tau},
\]
so that for $\ell_1 \approx n^{\alpha_1}$ and $\ell_2 \approx n^{\alpha_2}$,
\begin{align*}
 v_{\ell_1}(G) v_{\ell_2}(G) f_{k,\ell_1,\ell_2}(G)
 &\wapprox \text{const} \times n \ell_1^{-\tau} n \ell_2^{-\tau} k \frac{\ell_1}{m} \frac{\ell_2}{m} \\
 &\wapprox \text{const} \times k n^2 m^{-2} \ell_1^{1-\tau} \ell_2^{1-\tau} \\
 &\wapprox \text{const} \times k n^2 m^{-2} n^{(\alpha_1+\alpha_2)(1-\tau)}
\end{align*}
The right side is maximized by taking $\alpha_1=\alpha_2=0$, which leads to
\[
 v_{\ell_1}(G) v_{\ell_2}(G) f_{k,\ell_1,\ell_2}(G)
 \wapprox \text{const} \times k
\]
and
\[
 t_G(k)
 \wapprox \frac{\text{const} \times k}{k(k-1)} 
 \wapprox \text{const} \times k^{-1}.
\]


\subsection{Local limit conjecture}

First define a locally finite partially labeled rooted random tree $T^\bullet$ as follows. Choose some node as the root. Individuals in even generations (including the root, representing nodes) have offspring distribution $\Poi(\lambda)$ with $\lambda = \lim \frac{n}{m} \E X$. Every individual in an odd generation (representing a layer) is attached a random label $(X^*,Q^*)$ sampled independently from the $X$-biased distribution of $(X,Q)$. Every individual in an odd generation with label $(X^*,Q^*)$ produces a deterministic number $X^*-1$ of children. See Figure~\ref{fig:LabeledTree}. Having generated the partially labeled rooted tree, we define a locally finite rooted random graph $G^\bullet$ as follows. First we attach each layer (odd generation node) $k$ with label $(x_k, q_k)$ in $T^\bullet$ an independent collection of $\{0,1\}$-valued random variables $\{C_{ij}^k\}$ with mean $q_k$, independently of other layers. Then we define the node set of $G^\bullet$ as the set of even-generation nodes in $T^\bullet$, and the root of $G^\bullet$ to be the root of $T^\bullet$. We declare a pair of distinct nodes $i$ and $j$ in $V(G^\bullet)$ to be linked if there exists a layer $k$ and a 2-hop path $i \to k \to j$ in $T^\bullet$ such that $C_{ij}^k = 1$. Finally, we redefine $G^\bullet$ to be the connected component of the root in $G^\bullet$. Then $G^\bullet$ is a random instance of a locally finite connected rooted graph. 

\begin{conjecture}
Under sufficient regularity, $G^\bullet$ is a local weak limit of the random intersection graph model.
\end{conjecture}

Note by replacing $q_k$ by $p q_k$ above, we might obtain the local limit of the $p$-thinned random intersection graph corresponding to bond percolation.

\begin{figure}[h]
\centering
\begin{tikzpicture} [
 font = \small
]
\node[circle, draw](n0) at (0,0) {};
\node[rectangle, draw](c11) at (-1,1) {(4,0.6)};
\node[rectangle, draw](c12) at (1,1) {(2,0.2)};
\node[circle, draw](n11) at (-2,2) {};
\node[circle, draw](n12) at (-1,2) {};
\node[circle, draw](n13) at (-0,2) {};
\node[circle, draw](n14) at (+2,2) {};
\node[rectangle, draw](c21) at (-2,3) {(3,0.1)};
\node[rectangle, draw](c22) at (+2,3) {(6,0.2)};
\node[circle, draw](n21) at (-2.5,4) {};
\node[circle, draw](n22) at (-1.5,4) {};
\node[circle, draw](n23) at (+1.0,4) {};
\node[circle, draw](n24) at (+1.5,4) {};
\node[circle, draw](n25) at (+2.0,4) {};
\node[circle, draw](n26) at (+2.5,4) {};
\node[circle, draw](n27) at (+3.0,4) {};
\draw (n0) -- (c11);
\draw (n0) -- (c12);
\draw (c11) -- (n11);
\draw (c11) -- (n12);
\draw (c11) -- (n13);
\draw (c12) -- (n14);
\draw (n11) -- (c21);
\draw (n14) -- (c22);
\draw (c21) -- (n21);
\draw (c21) -- (n22);
\draw (c22) -- (n23);
\draw (c22) -- (n24);
\draw (c22) -- (n25);
\draw (c22) -- (n26);
\draw (c22) -- (n27);
\end{tikzpicture}
\caption{\label{fig:LabeledTree} A sample of a partially labeled random tree. Circles represent nodes (in the original model description) and rectangles layers.}
\end{figure}

%
%
\begin{figure}[h]
\centering
\begin{tikzpicture}
  \tikzstyle{every node}=[ultra thick, circle, minimum width=6pt,
    minimum height=3pt, inner sep=0pt, outer sep=0pt, align=center]
  \tikzstyle{node_0} = [fill=red]        
  \tikzstyle{node_1} = [fill=blue]       
    \node [node_0] (v10) at (5,5) {};
    \node [node_1] (v11) at (5,7) {};
\end{tikzpicture}
\end{figure}


%
%
\begin{figure}[h]
\centering
\begin{tikzpicture}
  \tikzstyle{every node}=[ultra thick, ellipse, minimum width=6pt,
    minimum height=3pt, inner sep=0pt, outer sep=0pt, align=center]
  \tikzstyle{vline} = [dashed, thin, black]
  \tikzstyle{node_0} = [fill=red]        
  \tikzstyle{node_1} = [fill=blue]       
  \tikzstyle{node_2} = [fill=blue]       
  \tikzstyle{node_3} = [fill=blue]       
  \tikzstyle{node_i} = [fill=black!50]  
  \tikzstyle{link} = [thick, black!50]
  \tikzstyle{link_comp} = [thick, blue]
  \tikzstyle{layer} = [draw, black, fill=blue!10, opacity=1, ultra thin]

  \begin{scope}
    \foreach \h [count=\i] in {0,40,80,120,160} {
      \pgftransformcm{0.7}{0}{0.25}{0.1}{\pgfpoint{0}{\h}}
      \draw [black!50, fill=black!4, step=1cm] (0,0) rectangle (10,10);
      \node at (1, 4.5) {$G^\i$};
    }
  \pgftransformcm{0.7}{0}{0.25}{0.1}{\pgfpoint{0}{-80}}
  \draw [black!50, fill=blue!5, step=1cm] (0,0) rectangle (10,10);
  \node at (1, 4.5) {$G$};
  \end{scope}
  \draw [layer] (5.7,0.5) ellipse [x radius=0.95, y radius=0.4]; 
  \draw [layer] (4.5,1.9) ellipse [x radius=0.8, y radius=0.3]; 
  \draw [layer] (6.5,3.2) ellipse [x radius=0.8, y radius=0.3]; 
  \draw [layer] (6.5,4.7) ellipse [x radius=1.2, y radius=0.4]; 
  \draw [layer] (3.9,6.0) ellipse [x radius=1.1, y radius=0.3]; 

  \begin{scope} 
    \pgftransformcm{0.7}{0}{0.3}{0.1}{\pgfpoint{0}{0}}
    \node [node_0] (v10) at (5,5) {};
    \node [node_1] (v11) at (5,7) {};
    \node [node_1] (v12) at (7,4) {};
    \node [node_2] (v13) at (7,3) {};
    \draw [link_comp] (v10) -- (v11);
    \draw [link_comp] (v10) -- (v12);
    \draw [link_comp] (v11) -- (v12);
    \draw [link_comp] (v12) -- (v13);
  \end{scope}
  \begin{scope} 
    \pgftransformcm{0.7}{0}{0.3}{0.1}{\pgfpoint{0}{40}}
    \node [node_0] (v20) at (5,5) {};
    \node [node_1] (v21) at (4.7,4) {};
    \node [node_i] (v22) at (3.4,5.5) {};
    \node [node_i] (v23) at (3.4,6.5) {};
    \draw [link_comp] (v20) -- (v21);
    \draw [link] (v22) -- (v23);
  \end{scope}
  \begin{scope} 
     \pgftransformcm{0.7}{0}{0.3}{0.1}{\pgfpoint{0}{80}}
     \node [node_1] (v32) at (7,4) {};
     \node [node_i] (v33) at (8.5,2) {};
     \node [node_i] (v34) at (8.9,2.5) {};
     \draw [link] (v33) -- (v34);
  \end{scope}
  \begin{scope} 
    \pgftransformcm{0.7}{0}{0.3}{0.1}{\pgfpoint{0}{120}}
     \node [node_1] (v42) at (7,4) {};
     \node [node_2] (v46) at (6,7) {};
     \node [node_2] (v43) at (7,7) {};
     \node [node_2] (v44) at (7.7,4) {};
     \node [node_3] (v45) at (8.5,4) {};
     \draw [link_comp] (v42) -- (v43);
     \draw [link_comp] (v42) -- (v44);
     \draw [link_comp] (v42) -- (v46);
     \draw [link_comp] (v43) -- (v44);
     \draw [link_comp] (v43) -- (v45);
     \draw [link_comp] (v44) -- (v45);
    \end{scope}
  \begin{scope} 
    \pgftransformcm{0.7}{0}{0.3}{0.1}{\pgfpoint{0}{160}}
     \node [node_1] (v51) at (4.7,4) {};
     \node [node_2] (v52) at (4,4) {};
     \node [node_2] (v53) at (4.5,2) {};
     \node [node_3] (v54) at (3,5) {};
     \node [node_i] (v55) at (3,3) {};
     \node [node_i] (v56) at (3,4) {};
     \node [node_i] (v57) at (2.3,4.9) {};
     \draw [link_comp] (v51) -- (v52);
     \draw [link_comp] (v51) -- (v53);
     \draw [link_comp] (v52) -- (v54);
     \draw [link_comp] (v53) -- (v54);
     \draw [link] (v55) -- (v56);
     \draw [link] (v56) -- (v57);

  \end{scope}
  \draw [vline] (v10) -- (v20);
  \draw [vline] (v12) -- (v32);
  \draw [vline] (v32) -- (v42);
  \draw [vline] (v21) -- (v51);
  
  \begin{scope} 
  \pgftransformcm{0.7}{0}{0.3}{0.1}{\pgfpoint{0}{-80}}
    \node [node_0, label=below:{\scriptsize $i$}] (a10) at (5,5) {};
    \node [node_1, label=below:{\scriptsize $j'$}] (a12) at (7,4) {};
    \node [node_1, label=below:{\scriptsize $j$}] (a51) at (4.7,4) {};
    \node [node_1] (a11) at (5,7) {};
    \node [node_2] (a13) at (7,3) {};
    \draw [link_comp] (a10) -- (a11);
    \draw [link_comp] (a10) -- (a12);
    \draw [link_comp] (a11) -- (a12);
    \draw [link_comp] (a12) -- (a13);

     \node [node_2] (a46) at (6,7) {};
     \node [node_2] (a43) at (7,7) {};
     \node [node_2] (a44) at (7.7,4) {};
     \node [node_3] (a45) at (8.5,4) {};

     \draw [link_comp] (a12) -- (a43);
     \draw [link_comp] (a12) -- (a44);
     \draw [link_comp] (a12) -- (a46);
     \draw [link_comp] (a43) -- (a44);
     \draw [link_comp] (a43) -- (a45);
     \draw [link_comp] (a44) -- (a45);

     \node [node_2] (a52) at (4,4) {};
     \node [node_2] (a53) at (4.5,2) {};
     \node [node_3] (a54) at (3,5) {};
     \draw [link_comp] (a10) -- (a51);
     \draw [link_comp] (a51) -- (a52);
     \draw [link_comp] (a51) -- (a53);
     \draw [link_comp] (a52) -- (a54);
     \draw [link_comp] (a53) -- (a54);
     
    \node [node_i] (a22) at (3.4,5.5) {};
    \node [node_i] (a23) at (3.4,6.5) {};
    \draw [link] (a22) -- (a23);

     \node [node_i] (a33) at (8.5,2) {};
     \node [node_i] (a34) at (8.9,2.5) {};
     \draw [link] (a33) -- (a34);

     \node [node_i] (a55) at (3,3) {};
     \node [node_i] (a56) at (3,4) {};
     \node [node_i] (a57) at (2.3,4.9) {};
     \draw [link] (a55) -- (a56);
     \draw [link] (a56) -- (a57);

  \end{scope}
\end{tikzpicture}
\caption{\label{fig:Union} Graph $G = G^1 \cup \cdots \cup G^5$ obtained as a union of five graphs corresponding to five layers. 
Nodes $i,j,j'$ belong to multiple layers. Node~$i$ belongs to layers 1 and 2, node~$j$ to layers 2 and 5, and node~$j'$ to layers $1,3,4$. \rnote{Todo: Create a plot series describing the layer exploration, where in step 1 the red node fully explores all layers that it belongs to, beyond its nearest neighbors.}}
\end{figure}
\clearpage


\clearpage

\newcommand{\myGlobalTransformation}[2] {
  \pgftransformcm{0.7}{0}{0.7}{1}{\pgfpoint{#1cm}{#2cm}}
}

\newcommand{\gridThreeD}[3]
{
    \begin{scope}
        \myGlobalTransformation{#1}{#2};
        \draw [#3, step=2cm] grid (8,8);
    \end{scope}
}

\tikzstyle myBG=[line width=3pt, opacity=1.0]

\newcommand{\drawLinewithBG}[2]
{
    \draw[white, myBG]  (#1) -- (#2);
    \draw[black, very thick] (#1) -- (#2);
}

\newcommand{\graphLinesHorizontal}
{
    \drawLinewithBG{1,1}{7,1};
    \drawLinewithBG{1,3}{7,3};
    \drawLinewithBG{1,5}{7,5};
    \drawLinewithBG{1,7}{7,7};
}

\clearpage


\end{document}